\documentclass{amsart}

\usepackage{amssymb, amsfonts, amsmath, amsthm,bm, amscd}
\usepackage[dvipsnames]{xcolor}
\usepackage{mathrsfs}
\usepackage{hyperref}
\usepackage{graphicx}
\usepackage{caption}
\usepackage{subcaption}
\usepackage{wrapfig}
\usepackage[margin=1.00 in]{geometry}
\usepackage{enumitem, multicol}
\usepackage{tikz}
\usetikzlibrary{cd}
\usepackage{nicefrac, bigints}
\usepackage{array}

\newcommand{\para}[1]{\medskip\noindent\textbf{#1.}}

\theoremstyle{definition}
\newtheorem{theorem}{Theorem}[section]
\newtheorem{maintheorem}{Theorem}

\newtheorem{definition}[theorem]{Definition}
\newtheorem{lemma}[theorem]{Lemma}

\newtheorem{corollary}[theorem]{Corollary}
\newtheorem{proposition}[theorem]{Proposition}
\newtheorem{remark}[theorem]{Remark}

\newtheorem{claim}[theorem]{Claim}

\newtheorem*{note}{Note}
\newtheorem{construction}[theorem]{Construction}

\makeatletter
\DeclareRobustCommand{\cev}[1]{%
  {\mathpalette\do@cev{#1}}%
}
\newcommand{\do@cev}[2]{%
  \vbox{\offinterlineskip
    \sbox\z@{$\m@th#1 x$}%
    \ialign{##\cr
      \hidewidth\reflectbox{$\m@th#1\vec{}\mkern4mu$}\hidewidth\cr
      \noalign{\kern-\ht\z@}
      $\m@th#1#2$\cr
    }%
  }%
}
\makeatother

\newcommand{\Addresses}{{
  \bigskip
  \footnotesize
\noindent  Aaron Calderon, \textsc{Department of Mathematics, University of Chicago}\par\nopagebreak
  \textit{E-mail address}: \texttt{aaroncalderon@uchicago.edu}
  
  \noindent James Farre, \textsc{Department of Mathematics, Universit{\"a}t Heidelberg}\par\nopagebreak
  \textit{E-mail address}: \texttt{jfarre@mathi.uni-heidelberg.de}
  }}

\newcommand{\NN}{\mathbb{N}}
\newcommand{\ZZ}{\mathbb{Z}}
\newcommand{\RR}{\mathbb{R}}

\newcommand{\CC}{\mathbb{C}}
\newcommand{\HH}{\mathbb{H}}

\newcommand{\ith}{^\text{th}}
\newcommand{\inverse}{^{-1}}

\DeclareMathOperator{\Mod}{Mod}

\newcommand{\MF}{\mathcal{MF}}

\newcommand{\ML}{\mathcal{ML}}

\newcommand{\GL}{\mathsf{GL}}
\newcommand{\SL}{\mathsf{SL}}
\DeclareMathOperator{\im}{im}

\DeclareMathOperator{\sech}{sech}

\newcommand{\sB}{\mathscr{B}}
\newcommand{\T}{\mathcal{T}}
\newcommand{\M}{\mathcal{M}}

\newcommand{\cH}{\mathcal{H}}

\newcommand{\PT}{\mathcal{PT}}
\newcommand{\PM}{\mathcal{PM}}
\newcommand{\PoT}{\mathcal{P}^1\mathcal{T}}
\newcommand{\PoM}{\mathcal{P}^1\mathcal{M}}
\newcommand{\QT}{\mathcal{QT}}
\newcommand{\QM}{\mathcal{QM}}
\newcommand{\QoT}{\mathcal{Q}^1\mathcal{T}}
\newcommand{\QoM}{\mathcal{Q}^1\mathcal{M}}
\newcommand{\sing}{\underline{\kappa}}

\newcommand{\fcH}{\underline{\mathcal H}}
\newcommand{\epN}[1]{\mathcal{N}_\epsilon(#1)}
\newcommand{\calH}{\mathcal{H}}
\newcommand{\SH}{\mathcal{SH}}
\DeclareMathOperator{\Sp}{\mathsf{Sp}}
\DeclareMathOperator{\val}{val}
\DeclareMathOperator{\Isom}{Isom}
\DeclareMathOperator{\PSL}{\mathsf{PSL}}

\DeclareMathOperator{\inj}{inj}
\DeclareMathOperator{\Area}{Area}
\newcommand{\tlambda}{\widetilde{\lambda}}
\newcommand{\tX}{\widetilde{X}}
\newcommand{\partiall}{\partial_\lambda}

\newcommand{\tS}{\widetilde{S}}
\newcommand{\tSp}{\widetilde{\Sp}}

\newcommand{\tlambdaa}{\widetilde{\lambda_{\arc}}}
\newcommand{\lambdaa}{\lambda_{\arc}}
\newcommand{\taua}{\tau_{\arc}}
\newcommand{\ttau}{\widetilde{\tau}}
\newcommand{\ttaua}{\widetilde{\tau_{\arc}}}
\newcommand{\sigl}{\sigma_\lambda}
\newcommand{\ac}{\mathfrak{s}}
\newcommand{\acarc}{\mathfrak{a}}
\newcommand{\cO}{\mathcal{O}}
\newcommand{\Defl}{\mathcal{D}_\lambda}

\DeclareMathOperator{\Ol}{\mathcal{O}_\lambda}
\DeclareMathOperator{\Il}{I_\lambda}

\DeclareMathOperator{\Per}{Per}

\newcommand{\arc}{\underline{\alpha}}
\newcommand{\arcwt}{\underline{A}}
\newcommand{\arcb}{\underline{\beta}}
\newcommand{\arcwtb}{\underline{B}}
\newcommand{\arcc}{\underline{\gamma}}

\newcommand{\Arc}{\mathscr{A}}
\newcommand{\Arcfill}{\mathscr{A}_\text{fill}}

\DeclareMathOperator{\Stab}{Stab}
\newcommand{\crowns}{\{\underline{c}\}}
\newcommand{\Crown}{\mathcal{C}}
\newcommand{\Fol}{\mathcal{F}}
\newcommand{\cutsurf}{\Sigma}
\newcommand{\Base}{\mathscr{B}(S \setminus \lambda)}
\DeclareMathOperator{\ThH}{\omega_{\calH}}
\DeclareMathOperator{\ThSH}{\omega_{\SH}}
\DeclareMathOperator{\res}{res}
\DeclareMathOperator{\trem}{trem}

\DeclareMathOperator{\Eq}{Eq}

\renewcommand{\Re}{\operatorname{Re}}
\renewcommand{\Im}{\operatorname{Im}}

\begin{document}

\title[Shear-shape cocycles for measured laminations]
{Shear-shape cocycles for measured laminations\\and ergodic theory of the earthquake flow}

\author{Aaron Calderon}
\author{James Farre}
\date{\today}

\begin{abstract}
We extend Mirzakhani's conjugacy between the earthquake and horocycle flows to a bijection, demonstrating conjugacies between these flows on all strata and exhibiting an abundance of new ergodic measures for the earthquake flow.
The structure of our map indicates a natural extension of the earthquake flow to an action of the upper-triangular subgroup $P < \SL_2\RR$ and we classify the ergodic measures for this action as pullbacks of affine measures on the bundle of quadratic differentials.
Our main tool is a generalization of the shear coordinates of Bonahon and Thurston to arbitrary measured laminations.
\end{abstract}

\vspace*{-2.5em}
\maketitle
\vspace*{-2em}

\section{Main results}

\subsection{Conjugating earthquake and horocycle flow}
This paper deals with two notions of unipotent flow over the moduli space $\M_g$ of Riemann surfaces. The first is the {\em Teichm{\"u}ller horocycle flow}, defined on the bundle $\QoM_g$ of unit area quadratic differentials $q$ by postcomposing the charts of the flat metric $|q|$ by the parabolic transformation
{\scriptsize $\begin{pmatrix}
1 & s \\
0 & 1
\end{pmatrix}$}.
This flow is ergodic with respect to a finite measure induced by Lebesgue in local period coordinates \cite{Masur_IETsMF, Veech_IETs} and is a fundamental object of study in Teichm{\"u}ller dynamics.

The second is the {\em earthquake flow} on the bundle $\PoM_g$, whose fiber is the sphere of unit-length measured geodesic laminations on a hyperbolic surface. The earthquake flow is defined as a generalization of twisting about simple closed curves, or by postcomposing hyperbolic charts by certain piecewise-isometric transformations. While this flow is more mysterious, earthquakes are a familiar tool in Teichm{\"u}ller theory, playing a central role in Kerckhoff's proof of the Nielsen realization conjecture \cite{Kerckhoff_NR}, for example.

These two flows are both assembled from families of Hamiltonian flows
(extremal length for horocycle \cite{PapaHam} and hyperbolic length for earthquake \cite{Kerckhoff_NR, Wolpert, BonSoz})
and exhibit similar non-divergence properties \cite{MWnondiv}, but the horocycle flow belongs properly to the flat-geometric viewpoint and the earthquake flow to the hyperbolic one.
All the same, in \cite[Theorem 1.1]{MirzEQ} Mirzakhani established a bridge between the two worlds, demonstrating a measurable conjugacy between the earthquake and horocycle flows. 
Consequently, the earthquake flow is ergodic with respect to the measure class of Lebesgue on $\PoM_g$.

In this article, we deepen this connection between flat and hyperbolic geometry, proving that the correspondence can be further upgraded to yield new results on both the ergodic theory of the earthquake flow and the structure of Teichm{\"u}ller space.

\begin{maintheorem}\label{mainthm:conjugacy}
Mirzakhani's conjugacy extends to a bijection
\[\mathcal{O}: \PoM_g \leftrightarrow \QoM_g\]
that conjugates earthquake flow to horocycle flow.
\end{maintheorem}

The moduli space of quadratic differentials is naturally partitioned into {\em strata} $\QoM_g(\sing)$, disjoint subsets parametrizing unit-area differentials with zeros of order $\sing = (\kappa_1, \ldots, \kappa_n)$. Similarly, for any $\sing$ we may define the {\em regular locus} $\PoM_g^{\text{reg}}(\sing)$ 
\label{ind:reglocus}
to be the set of $(X, \lambda)$ where $\lambda$ cuts $X$ into ideal polygons with $(\kappa_1 + 2, \ldots, \kappa_n + 2)$ many sides, each with a cyclic symmetry of that order.

With this notation, Mirzakhani's conjugacy can more precisely be stated as the existence of a bijection 
\[\PoM_g^{\text{reg}}(1^{4g-4}) \leftrightarrow \QoM_g^{\text{nsc}}(1^{4g-4})\]
taking earthquake flow to horocycle flow, where the superscript $\text{nsc}$ specifies the (full-measure) sublocus of the stratum consisting of those differentials with no horizontal saddle connections.
\label{ind:stratumNSC}

One of our main applications of Theorem \ref{mainthm:conjugacy} is to produce an analogue of Mirzakhani's conjugacy for components of strata (even those coming from global squares of Abelian differentials), confirming a conjecture of Alex Wright \cite[Remark 5.6]{Wright_MirzEQ} (see also \cite[Problems 12.5 and 12.6]{Wright_Mirz}).

\begin{maintheorem}\label{mainthm:strata}
For every $\sing$, the map $\cO$ restricts to a bijection
\[\PoM_g^{\text{reg}}(\sing) \leftrightarrow \QoM_g^{\text{nsc}}(\sing)\]
that takes earthquake to horocycle flow and (generalized) stretch rays to Teichm{\"u}ller geodesics.
\end{maintheorem}

While strata of holomorphic quadratic differentials are generally not connected, for $g \neq 4$ their connected components are classified by whether or not they consist of squares of abelian differentials and the parity of the induced spin structure (both of which depend only on the horizontal foliation when there are no horizontal saddles), as well as hyperellipticity \cite{KZ, Lanneau}.
\footnote{In genus $4$, there are certain strata whose components have only been characterized via algebraic geometry \cite{CMexceptional}.}
The bijection $\cO$ respects both the horizontal direction and the $\Mod(S)$ action, so Theorem \ref{mainthm:strata} can be refined to describe the preimages of these components.

As an immediate consequence of Theorem \ref{mainthm:strata}, the earthquake flow is ergodic with respect to the pushforward by $\cO\inverse$ of the Masur-Veech measure on any component of any stratum of quadratic differentials.

\subsection{Geodesic flows and $P$-invariant measures}\label{subsec:Paction}
Pulling back the Teichm{\"u}ller geodesic flow via $\cO$ allows us to specify a family of ``dilation rays'' which serve as a geodesic flow for the earthquake flow's parabolic action and in many cases project to geodesics for Thurston's Lipschitz asymmetric metric.
Combining dilation rays and the earthquake flow therefore gives a action of the upper triangular subgroup $P < \SL_2\RR$ on $\PoM_g$ by ``stretchquakes.'' See Section \ref{subsec:Thurston_geos}.

Due in part to the failure of $\cO$ to be continuous, the stretchquake action on $\PoM_g$ is not by homeomorphisms but rather by measurable bijections.
More precisely, it preserves the $\sigma$-algebra obtained by pulling back the Borel $\sigma$-algebra of $\QoM(S)$ along $\mathcal O$.
In a forthcoming sequel \cite{shshII}, the authors show that $\cO$ is actually a measurable isomorphism with respect to the Borel $\sigma$-algebra on $\PoM_g$ and that the stretchquake action restricted to each $\PoM_g^{\text{reg}}(\sing)$ is by homeomorphisms; see also Remark \ref{rmk:measurability}.

\begin{remark}
In fact, Arana-Herrera and Wright have recently shown that there is {\em no} continuous map conjugating the earthquake flow to horocycle flow, at least when $\PoM_g$ and $\QoM_g$ are equipped with their standard topologies \cite{AHW_EQ}.
\end{remark}

In their foundational work on $\SL_2\RR$-invariant ergodic measures on the moduli space of flat surfaces, Eskin and Mirzakhani \cite[Theorem 1.4]{EskMirz} proved that the support of any $P$-invariant ergodic measure on $\QoM_g$ is locally an affine manifold cut out by linear equations in period coordinates.
Our conjugacy translates this classification into a classification of ergodic measures for the extension of the earthquake flow defined above:

\begin{maintheorem}\label{mainthm:AIS}
Every stretchquake-invariant ergodic measure is the pullback of an affine measure.
\end{maintheorem}
\begin{proof}
If $\nu$ is a stretchquake-invariant ergodic measure on $\PoM_g$, then $\cO_*\nu$ is a $P$-invariant ergodic measure on $\QoM_g$, which is affine by \cite[Theorem 1.4]{EskMirz}.
\end{proof}

Using this correspondence we obtain a geometric rigidity phenomenon for stretchquake-invariant ergodic measures on $\PoM_g$: the generic point is made out of a fixed collection of regular ideal polygons.

\begin{corollary}\label{cor:reg_ae}
For any stretchquake-invariant ergodic probability measure $\nu$ on $\PoM_g$, there is some $\sing$
so that $\nu$-almost every $(X, \lambda)$ lies in $\PoM_g^{\text{reg}}(\sing)$.
\end{corollary}

This in particular implies that the dynamics of the stretchquake action with respect to any ergodic probability measure are measurably the same as its restriction to a stratum, on which we can identify dilation rays as (directed, unit-speed) geodesics for the Lipschitz asymmetric metric on $\T(S)$ (see Proposition \ref{prop:stretch_reg}).

\begin{remark}
We note that general ergodic measures for the stretchquake action can look quite different than the Lebesgue measure class on $\PoM_g$, even when pushed down to $\M_g$.

For example, if $\nu$ gives full measure to $\PoM^{\text{reg}}_g(4g-4)$ then a $\nu$-generic point is obtained by gluing together a single regular ideal $(4g-2)$-gon; in particular, the injectivity radius at the center of the polygon can be arbitrarily large, allowing $g \to \infty$. This implies that $\nu$ gives zero mass to (the restriction of $\PoM_g$ to) sufficiently thin parts of moduli space, as any $(X,\lambda)$ where $X$ has a very short pants decomposition has injectivity radius uniformly bounded above.
\end{remark}

\begin{remark}
While an important result of \cite{EskMirz} is that any $P$-invariant ergodic measure on $\QoM_g$ is actually $\SL_2\RR$-invariant, the  circle action on $\QoM_g$ (corresponding to rotating a quadratic differential) does not have an obvious geometric interpretation on $\PoM_g$.  See also \cite[Problems 12.3 and 12.4]{Wright_Mirz}
\end{remark}

\subsection{Dual foliations from hyperbolic structures}\label{subsec:orthointro}
A foundational result of Gardiner and Masur (Theorem \ref{thm:GM} below) states that quadratic differentials are parametrized by their real and imaginary parts; equivalently, their vertical and horizontal foliations (or laminations).
In particular, the real-analytic submanifold $\Fol^{uu}(\lambda)$ 
\label{ind:unstable}
of all quadratic differentials with horizontal lamination $\lambda$ can be identified with the space $\MF(\lambda)$ of foliations that bind together with $\lambda$.
\label{ind:MF(lambda)}
See Section \ref{sec:introproof} for a formal definition.
As the horocycle flow preserves the horizontal foliation, it induces a flow on $\MF(\lambda)$.

Mirzakhani's conjugacy and our extension therefore both follow from the construction of flow-equivariant maps that assign to a hyperbolic surface $X$ and a measured lamination $\lambda$ a ``dual'' measured foliation.

For maximal laminations $\lambda$, this dual is the {\em horocyclic foliation} $F_\lambda(X)$ 
\label{ind:horfol}
introduced by Thurston \cite{Th_stretch}, obtained by foliating the spikes of each triangle of $X \setminus \lambda$ by horocycles and extending across the leaves of $\lambda$. The measure of an arc transverse to $F_\lambda(X)$ is then the total distance along $\lambda$ between horocycles meeting the arc at endpoints.
As $F_\lambda(X)$ necessarily binds $S$ together with $\lambda$, this defines a map
\[F_\lambda: \T(S) \rightarrow \MF(\lambda).\]

We endow $\MF(\lambda)$ with the real-analytic structure coming from its identification with $\Fol^{uu}(\lambda)$.  The main engine of Mirzakhani's conjugacy is the following theorem of Bonahon \cite{Bon_SPB} and Thurston \cite{Th_stretch}; see also Section \ref{subsec:introshsh} for a discussion of her interpretation of this result.

\begin{theorem}[Bonahon, Thurston]\label{thm:horohomeo}
For any maximal $\lambda$, the horocyclic foliation map $F_\lambda$ is a real-analytic homeomorphism which takes the earthquake in $\lambda$ to the horocycle flow restricted to $\MF(\lambda) \cong \Fol^{uu}(\lambda)$ in a time-preserving way.
Moreover, the family $\{F_\lambda\}$ is equivariant with respect to the $\Mod(S)$ action. That is, $F_{g\lambda}(gX) = g F_\lambda(X)$ for all $g \in \Mod(S)$.
\end{theorem}

When $\lambda$ is not maximal the horocyclic foliation is no longer defined. The first thing one might try is to simply choose a completion of $\lambda$, but this approach is too na{\"i}ve. Indeed, this would require choosing a completion of every lamination, which necessarily destroys $\Mod(S)$--equivariance because laminations (and differentials) can have symmetries.
\footnote{For example, take $\gamma$ to be a simple closed curve; completions of $\gamma$ correspond to triangulations of $X \setminus \gamma$ where the boundaries are shrunk to cusps (up to a choice of spiraling about each side of $\gamma$). The space of such triangulations carries a rich $\Stab(\gamma)$ action, and a computation shows that the horocyclic foliations for two completions in the same $\Stab(\gamma)$ orbit need not be equal.}
Such a map will not descend to moduli space and is therefore unsuitable for our applications.
Besides, for our purposes it is important that the geometry of the subsurfaces of $X \setminus \lambda$ predicts the singularity structure of the corresponding differential.

If one restricts their attention to the case when $\lambda$ is filling and cuts $X$ into regular ideal polygons then there is a canonical notion of horocyclic foliation. While this construction is equivalent on the regular locus to the more general procedure we describe just below, any attempt to prove Theorem \ref{mainthm:strata} with this restricted viewpoint would necessarily rely on ($\Mod(S)$--equivariant) descriptions of the loci of surfaces built from regular polygons, as well as the intersection of $\Fol^{uu}(\lambda)$ with strata, results which (to the knowledge of the authors) were heretofore unknown. Compare Corollary \ref{cor:fillML_cap_unstable} and Section \ref{subsec:FNandDT}.

We therefore place no restrictions on the topological type or the complementary geometry of $\lambda$. Following a suggestion of Yi Huang (communicated to us by Alex Wright), we prove that the correct analogue of the horocyclic foliation for non-maximal $\lambda$ is the {\em orthogeodesic foliation} $\Ol(X)$, whose leaves are the fibers of the closest point projection to $\lambda$ and whose measure is given by length of the projection to $\lambda$.
As in the maximal case, the orthogeodesic foliation binds together with $\lambda$, inducing a map
\[\Ol: \T(S) \rightarrow \MF(\lambda).\]
See Section \ref{sec:ortho} for a more detailed discussion of this construction.

\begin{maintheorem}\label{mainthm:orthohomeo}
For any $\lambda \in \ML(S)$, the orthogeodesic foliation map $\Ol$ is a homeomorphism which takes the earthquake in $\lambda$ to the horocycle flow restricted to $\MF(\lambda) \cong \Fol^{uu}(\lambda)$ in a time-preserving way. Moreover, the family $\{\Ol\}$ is equivariant with respect to the $\Mod(S)$ action. That is, $\mathcal{O}_{g\lambda}(gX) = g \Ol(X)$ for all $g \in \Mod(S)$.
\end{maintheorem}

Although $\MF(\lambda)$ does not have an obvious smooth structure, the map $\Ol$ still exhibits a surprising amount of regularity; see Theorem \ref{mainthm:coordinates}.

The proof of Theorem \ref{mainthm:orthohomeo} requires generalizing Bonahon's machinery of transverse cocycles to new combinatorial objects called ``shear-shape cocycles'' which capture the essential structure of the orthogeodesic foliation; see Section \ref{subsec:introshsh} just below.
The space of shear-shape cocycles forms a common coordinatization of both $\T(S)$ and $\MF(\lambda)$ that is compatible with the map $\Ol$ and reveals an abundance of structure encoded in the orthogeodesic foliation:
\begin{itemize}
\item When $\lambda$ cuts $X$ into regular ideal polygons, the orthogeodesic and horocyclic foliations agree.
\item The locus of points of $X$ which are closest to at least two leaves of $\lambda$ forms a piecewise geodesic spine for $X \setminus \lambda$ which captures the geometry and topology of the complementary subsurfaces (see Theorem \ref{thm:arc=T(S)_crown}). Moreover, this spine is exactly the diagram of horizontal separatrices for the quadratic differential with horizontal foliation $\lambda$ and vertical foliation $\Ol(X)$.
\item For every measure $\mu$ on $\lambda$, the intersection of $\mu$ and $\Ol(X)$ is the hyperbolic length of $\mu$ on $X$.
\item The pullback of Teichm{\"u}ller geodesics with no horizontal saddle connections are geodesics with respect to Thurston's Lipschitz (asymmetric) metric (Proposition \ref{prop:stretch_reg}).
\end{itemize}

A statement similar to Theorem \ref{mainthm:orthohomeo} is probably also true for any (unmeasured) geodesic lamination, but for technical reasons regarding compatibility of complementary subsurfaces and the spiraling behavior of $\lambda$ we have restricted ourselves to the measured setting. See Remark \ref{rmk:geod_lamination}.

The orthogeodesic foliation map can also be thought of as relating the hyperbolic and extremal length functions $\ell_\lambda(\cdot)$ and $\text{Ext}_{\lambda}(\cdot)$ for any fixed $\lambda$. Indeed, a seminal theorem of Hubbard and Masur \cite{HubMas} states that the natural projection 
\[\pi: \Fol^{uu}(\lambda) \to \T(S)\]
that records only the complex structure underlying a differential is a homeomorphism. Combining this with the fact that the extremal length of $\lambda$ on $Y$ is exactly the area of the differential $\pi^{-1}(Y)$, we deduce that

\begin{corollary}\label{cor:deflate}
For every $\lambda \in \ML(S)$, the map $\pi \circ \Ol$ is a $\Stab(\lambda)$--equivariant self-homeomorphism of $\T(S)$ that takes the hyperbolic length function $\ell_\lambda(\cdot)$ to the extremal length function $\text{Ext}_{\lambda}(\cdot)$.
\end{corollary}

\subsection{Acknowledgments}
The authors would firstly like to thank Alex Wright for providing the germ of this project and useful suggestions, as well as for helpful comments on a preliminary draft of this paper. We are grateful to Francis Bonahon, Feng Luo, Howie Masur, and Jing Tao for lending their expertise and for enlightening discussions.

The authors would also like to thank those who contributed helpful comments or with whom we had clarifying conversations, including Daniele Alessandrini, Francisco Arana-Herrera, Mladen Bestvina, Jon Chaika, Vincent Delecroix, Valentina Disarlo, Spencer Dowdall, Ben Dozier, Aaron Fenyes, Ser-Wei Fu, Curt McMullen, Mareike Pfeil, Beatrice Pozzetti, John Smillie, and Sam Taylor. We also want to thank the anonymous referee for their careful reading and useful comments.

Finally, the authors are indebted to Yair Minsky for his dedicated guidance and mentorship throughout all stages of this project, as well as his generous listening and insightful comments.

The first author gratefully acknowledges support from NSF grants DGE-1122492,
DMS-161087, and DMS-2005328, and travel support from NSF grants DMS-1107452, -1107263, and -1107367 ``RNMS: Geometric Structures and Representation Varieties'' (the GEAR Network). The second author gratefully acknowledges support from NSF grants DMS-1246989, 
DMS-1509171, 
 DMS-1902896, 
 DMS-161087, 
 and DMS-2005328. 

Portions of this work were accomplished while the authors were visiting MSRI for the Fall 2019 program ``Holomorphic Differentials in Mathematics and Physics,'' and the authors would like to thank the venue for its hospitality and excellent working environment.
Part of this material is based upon work supported by the National Science Foundation under Grant No. DMS-1928930 while the second author participated in a program hosted by the Mathematical Sciences Research Institute in Berkeley, California, during the Fall 2020 semester on ``Random and Arithmetic Structures in Topology."
The second author would also like to thank the both the  Department of Mathematics at the University of Utah and the Mathematics Institute at Universit\"at Heidelberg for their hospitality and rich working environments.

\setcounter{tocdepth}{1}
\tableofcontents

\section{About the proof}\label{sec:introproof}

Given Theorem \ref{mainthm:orthohomeo} which associates to $(X,\lambda)$ a dual foliation $\Ol(X)$ describing the geometry of the pair, it is not difficult to prove Theorems \ref{mainthm:conjugacy} and \ref{mainthm:strata}. First, we recall the relationship between differentials, foliations, and laminations in a little more detail.  

The space of measured foliations (up to equivalence) on a closed surface $S$ of genus $g\ge2$ is denoted $\MF(S)$. There is a canonical identification \cite{Levitt} between $\MF(S)$ and $\ML(S)$, the space of measured laminations on $S$; throughout this paper we will implicitly pass between the two notions at will, depending on our situation.  By $\QT_g$ and $\QoT_g$ we mean the bundle of holomorphic quadratic differentials over the Teichm\"uller space and the locus of unit area quadratic differentials, respectively. We similarly denote $\PT_g = \T(S) \times \ML(S)$ and $\PoT_g$ the locus of pairs $(X, \lambda)$ where $\lambda$ has unit length on $X$.
\label{ind:PTandQT}

To every $q\in \QT_g$ one may associate the real measured foliation $|\Re(q)|$ which measures the total variation of the real part of the holonomy of an arc; the imaginary foliation $|\Im(q)|$ is defined similarly. These foliations have vertical, respectively horizontal, trajectories, and so we will also refer to them as the vertical and horizontal foliations (or laminations) of $q$ and write 
\[q = q(|\Re(q)|, |\Im(q)|).\]
A foundational theorem of Gardiner and Masur implies that the real and imaginary foliations completely determine $q$, and that given any two foliations which ``fill up'' the surface, one can integrate against their measures to recover a quadratic differential.

A pair of measured foliations/laminations $(\eta, \lambda)$ is said to {\em bind} $S$ if for every $\gamma \in \ML(S)$,
\label{ind:bind}
\[i(\gamma,\eta) + i(\gamma, \lambda) >0,\]
where $i( \cdot ~, \cdot)$ is the geometric intersection pairing. In the literature, such pairs are sometimes called {\em filling,} though we choose to distinguish the topological notion of filling from the measure-theoretic notion of binding.

\begin{theorem}[{\cite[Thereom 3.1]{GM}}]\label{thm:GM}
There is a $\Mod(S)$--equivariant homemomorphism
\[\QT(S) \cong \MF(S) \times \MF(S) \setminus \Delta\]
where $\Delta$ is the set of all non-binding pairs $(\eta, \lambda)$.
In particular, the set $\Fol^{uu}(\lambda)$ of all quadratic differentials with $|\Im(q)| = \lambda$ may be identified with $\MF(\lambda)$, the set of foliations which together bind with $\lambda$.
\end{theorem}

\begin{proof}[Proof of Theorems \ref{mainthm:conjugacy} and \ref{mainthm:strata}]
By definition, there is a $\Mod(S)$--equivariant projection $\PT_g \rightarrow \ML(S)$ with fiber $\T(S)$. Theorem \ref{thm:GM} implies there is a $\Mod(S)$--equivariant projection $\QT_g \rightarrow \ML(S)$ whose fiber over $\lambda$ may be identified with $\MF(\lambda)$. Applying Theorem \ref{mainthm:orthohomeo} on the fibers therefore yields an equivariant bijection
\[\cO: \PT_g \leftrightarrow \QT_g\]
\label{ind:O}
which takes unit-length laminations to unit-area differentials (Corollary \ref{cor:sigl_into_SH+}), and quotienting by the $\Mod(S)$ action proves Theorem \ref{mainthm:conjugacy}.

Furthermore, we observe that the spine of the orthogeodesic foliation of a regular ideal $(k+2)$-gon is just a star with $k+2$ edges, which corresponds to the separatrix diagram of a zero of order $k$ when there are no horizontal saddle connections. Thus $\cO$ restricts to the promised conjugacy on strata (Theorem \ref{mainthm:strata}).
\end{proof}

\begin{remark}\label{rmk:measurability}
Mirzakhani's conjugacy is defined on the Borel subset $ \PT_g^{\text{reg}}(1^{4g-4})\subset \PT_g$ of full Lebesgue measure and is moreover Borel measurable on its domain of definition.  The latter assertion is a consequence of a stronger result, namely that $\PT_g^{\text{reg}}(1^{4g-4})\to \QT_g$ is continuous (with respect to the subspace topology on $\PT_g^{\text{reg}}(1^{4g-4})$). 

While convergence of measured laminations (in measure) does not typically imply Hausdorff convergence of the supports, whenever a sequence $\{\lambda_n\}$ of maximal measured laminations converges to a maximal measured lamination $\lambda$, then $\lambda_n$ is eventually carried (snugly) on a maximal train track also carrying $\lambda$.  From here, it is not difficult to deduce that $\lambda_n\to \lambda$ in the Hausdorff topology \cite{BonZhu:HD} and thus the horocyclic foliations $F_{\lambda_n}(X)$ converge to $F_{\lambda(X)}$.
Intuitively, the leaves of $\lambda_n$ intersect the leaves of $\lambda$ with small angle (depending on the specific surface on which they are realized), so the orthogonal directions become more parallel.

In forthcoming work \cite{shshII}, the authors extend these ideas and prove that $\cO$ is (everywhere) Borel measurable with Borel measurable inverse by identifying a countable partition of $\PT_g$ and $\QT_g$ into Borel subsets on which $\cO$ is homeomorphic. See also Section \ref{sec:future}.
\end{remark}

In general, the compact edges of the spine of a pair $(X, \lambda)$ correspond exactly to horizontal saddle connections in the differential $\cO(X,\lambda)$. 
This observation allows us to prove that the generic point for a $P$-invariant ergodic probability measure on $\PoM_g$ consists of pairs $(X, \lambda)$ where $\lambda$ cuts $X$ into a fixed set of regular ideal polygons.

\begin{proof}[Proof of Corollary \ref{cor:reg_ae}]
Using our conjugacy, the desired statement is equivalent to the fact that any $P$-invariant ergodic probability measure on $\QoM_g$ is (a) supported in a single stratum and (b) gives 0 measure to the set of differentials with horizontal saddle connections. 

The first statement is implied by ergodicity, while the second follows from the fact that the measure is actually $\SL_2\RR$-invariant \cite{EMM}.
Indeed, for any quadratic differential $q$, the Lebesgue measure of the set of directions $\theta$ such that $e^{i \theta}q$ has a saddle connection is $0$, so Fubini's theorem implies (b).

\end{proof}

Refining the proof by considering connected components of strata, we see that we can also conclude that $\nu$-almost every pair has the same orientability, spin, and hyperellipticity properties.

\subsection{Shear-shape coordinates}\label{subsec:introshsh}
Our strategy to prove Theorem \ref{mainthm:orthohomeo} follows Mirzakhani's intepretation of Theorem \ref{thm:horohomeo}, in which she clarifies the relationship between Thurston's geometric perspective on the horocyclic foliation and Bonahon's powerful analytic approach in terms of transverse cocycles.
Namely, she shows that the horocyclic foliation map $F_\lambda$ is compatible with shearing coordinates for both hyperbolic structures and measured foliations. To motivate our construction, we give a brief outline of Mirzakhani's proof below.

A (real-valued) {\em transverse cocycle} for $\lambda$ is a finitely additive signed measure on arcs transverse to $\lambda$ that is invariant under isotopy transverse to $\lambda$; observe that transverse measures are themselves transverse cocycles. These objects equivalently manifest as transverse H\"older distributions, cohomology classes, or weight systems on snug train tracks \cite{Bon_GLTHB, Bon_SPB, Bon_THDGL}. 
The space $\cH(\lambda)$ of transverse cocycles forms a finite dimensional vector space which carries a natural homological intersection pairing which is non-degenerate when $\lambda$ is maximal. The intersection pairing then identifies a ``positive locus'' $\cH^+(\lambda)\subset \cH(\lambda)$ cut out by finitely many geometrically meaningful linear inequalities. See also Section \ref{subsec:trans_cocy}.

In \cite[Theorem A]{Bon_SPB}, Bonahon proved that for any maximal geodesic lamination $\lambda$ there is a real-analytic homeomorphism $\sigl: \T(S)\to \cH^+(\lambda)$ that takes a hyperbolic metric to its ``shearing cocycle,'' which essentially records the signed distance along $\lambda$ between the centers of ideal triangles in the complement of $\lambda$. Mirzakhani then constructed a homeomorphism $\Il$ (essentially by a well-chosen system of period coordinates) that coordinatizes $\MF(\lambda)$ by $\cH^+(\lambda)$ and for which the following diagram commutes \cite[\S\S 5.2, 6.2]{MirzEQ}:
\begin{equation}\label{diagram:M}
\begin{tikzcd}
\T(S) \arrow[rr, "F_{\lambda}"] \arrow[dr, "\sigl"']
&  & \MF(\lambda)
    \arrow[dl, "\Il"]\\
& \cH^+(\lambda) 
\end{tikzcd}
\end{equation}
Since $F_\lambda = \Il\inverse\circ \sigl$ is a composition of homeomorphisms, it is itself a homeomorphism. As the construction of the horocyclic foliation requires no choices, the family $\{F_{\lambda}\}$ is necessarily $\Mod(S)$--equivariant.
Finally, a direct computation shows that $\sigl$ transports the earthquake in $\lambda$ to translation in $\cH^+(\lambda)$ by $\lambda$, and $\Il$ similarly takes horocycle to translation, demonstrating Theorem \ref{thm:horohomeo}.

\para{Shear-shape cocycles}
When $\lambda$ is not maximal, the space of transverse cocycles is no longer suitable to coordinatize hyperbolic structures (or transverse foliations). Indeed, in this case the vector space $\cH(\lambda)$ has dimension less than $6g-6$ and its intersection form may be degenerate; this is a consequence of the fact that the Teichm{\"u}ller space of $S \setminus \lambda$ now has a rich analytic structure that transverse cocycles cannot see.

In order to imitate Diagram \eqref{diagram:M} and its concomitant arguments for arbitrary $\lambda \in \ML(S)$, we therefore introduce the notion of \emph{shear-shape cocycles} on $\lambda$.
Roughly, a shear-shape cocyle consists of finitely additive signed data on certain arcs transverse to $\lambda$ together with a weighted arc system that cuts $S \setminus \lambda$ into cells; this pair is also required to satisfy a certain compatibility condition mimicking features of the orthogeodesic foliation (Definition \ref{def:shsh_axiom}).
Generalizing results of Luo \cite[Theorem 1.2 and Corollary 1.4]{Luo}, 
we show that such an arc system is equivalent to a hyperbolic structure on $S \setminus \lambda$ (Theorem \ref{thm:arc=T(S)_crown}), so shear-shape cocycles may equivalently be thought of as transverse data together with a compatible hyperbolic structure on the complementary subsurface(s).
Like transverse cocycles, shear-shape cocycles also admit realizations as cohomology classes or weight systems on certain train tracks (Definition \ref{def:shsh_cohom} and Proposition \ref{prop:ttcoords}).

\begin{remark}
We note that only certain classes of arcs admit consistent weights when measured by a shear-shape cocycle, whereas transverse cocycles provide a measure to any arc transverse to $\lambda$. While this subtlety is exactly what allows us to understand how to relate shear-shape cocycles with the geometry of complementary subsurfaces, it also presents a number of technical challenges throughout the paper.
\end{remark}

Unlike transverse cocycles, the space $\SH(\lambda)$ of shear-shape cocycles is not a vector space, instead forming a principal $\cH(\lambda)$ bundle over a contractible analytic subvariety of $\T(S \setminus \lambda)$ (Theorem \ref{thm:shsh_structure}). All the same, the cohomological realization of shear-shape cocycles equips $\SH(\lambda)$ with an intersection form 
\[\ThSH: \SH(\lambda)\times \cH(\lambda) \to \RR\]
that identifies a ``positive locus'' $\SH^+(\lambda)$ and equips both $\SH(\lambda)$ and $\SH^+(\lambda)$ with piecewise-integral-linear structures. The positive locus forms a $\cH^+(\lambda)$ cone-bundle over the same subvariety of $\T(S \setminus \lambda)$ (Proposition \ref{prop:SH+_structure}) and fits into a familiar-looking commutative diagram: 
\begin{equation}\label{diagram}
\begin{tikzcd}
\T(S) \arrow[rr, "\Ol"] \arrow[dr, "\sigl"']
&  & \MF(\lambda)
    \arrow[dl, "\Il"]\\
& \SH^+(\lambda) & 
\end{tikzcd}
\end{equation}
where $\sigl$ and $\Il$ record shearing data along $\lambda$ as well as shape data in the complementary subsurfaces. These maps can be thought of as a common generalization of Bonahon and Mirzakhani's shear coordinates as well as Fenchel--Nielsen and Dehn--Thurston coordinates adapted to a pants decomposition (see Section \ref{subsec:FNandDT}).
In the case when $\lambda$ is orientable, the map $\Il$ can also be viewed as an extension of Minsky and Weiss's description of the set of Abelian differentials with given horizontal foliation \cite[Theorem 1.2]{MW_cohom}.
\footnote{Technically, \cite{MW_cohom} investigates the family of Abelian differentials with a fixed horizontal foliation and fixed topological type of horizontal separatrix diagram, whereas our map applies to quadratic differentials (whether or not they are globally the square of an Abelian differential) and packages together all possible types of separatrix diagrams.}

The conjugacy of Theorem \ref{mainthm:orthohomeo} is then a consequence of the following structural theorem, which is an amalgam of the main technical results of the paper (compare Theorems \ref{thm:Ilhomeo}, \ref{thm:hyp_main}, and \ref{thm:diagram_commutes}).

\begin{maintheorem}\label{mainthm:coordinates}
For any measured lamination $\lambda$, Diagram \eqref{diagram} commutes and all arrows are $\Stab(\lambda)$--equivariant  homeomorphisms.
Moreover,
\begin{itemize}
\item $\sigl$ is (stratified) real-analytic and transports the earthquake flow to translation by $\lambda$ and the hyperbolic length of $\lambda$ to $\ThSH(\cdot, \lambda)$. 
\item The weighted arc system underlying $\sigl(X)$ records the hyperbolic structure $X \setminus \lambda$ under the correspondence of Theorem \ref{thm:arc=T(S)_crown}.
\item $\Il$ is piecewise-integral-linear and transports horocycle flow to translation by $\lambda$ and intersection with $\lambda$ to $\ThSH(\cdot, \lambda)$.
\item The weighted arc system underlying $\Il(\eta)$ records the compact horizontal separatrices of $q(\eta, \lambda)$.
\end{itemize}
\end{maintheorem}

In the course of our proof, we also describe new ``shape-shifting deformations'' of hyperbolic surfaces which generalize Bonahon and Thurston's cataclysms by shearing along a lamination while also varying the hyperbolic structures on complementary pieces. See Section \ref{subsec:shapeshift_deform}.

One particularly interesting family of deformations is obtained by dilation. The space $\SH^+(\lambda)$ admits a natural scaling action by $\RR_{>0}$, and since both earthquake and horocycle flow are carried to translation in coordinates, this scaling action indicates extensions of each to $P$ actions. A quick computation (Lemma \ref{lem:Il_geo}) shows that the pullback of a dilation ray by $\Il$ is (a variant of) the Teichm{\"u}ller geodesic flow, so the $P$ action on the flat side is just the standard $P$ action on $\QT_g$.

On the hyperbolic side, these dilation rays define our extension of the earthquake flow, and correspond to families of hyperbolic metrics on which the length of $\lambda$ is scaled by a uniform factor. They are therefore natural candidates for (directed, unit-speed) geodesics for the Lipschitz asymmetric metric on $\T(S)$, and in some cases we can identify them as such (see Propositions \ref{prop:stretch_reg} and \ref{prop:quad_geo}, as well as Remarks \ref{rmk:thurston_geo_general} and \ref{rmk:PW}).

\begin{remark}
Over the course of the paper we formalize the notion that shear-shape coordinates for hyperbolic structures are essentially the ``real part'' of period coordinates for $\PT_g$.
Interpreting $\sigl(X) + i\lambda$ as a complex weight system on a train track, Theorem \ref{mainthm:AIS} implies that the support of every stretchquake-invariant ergodic measure on $\PoM_g$ is locally an affine measure in train track charts.
See Lemma \ref{lem:pers_as_ttwts}.
\end{remark}

\para{Coordinatizing horospheres}
Since the Thurston intersection form $\ThSH$ captures both the hyperbolic length of and geometric intersection with $\lambda$, the coordinate systems of Theorem \ref{mainthm:coordinates} also allow us to give global descriptions of the level sets of these functions.
In particular, we can recover Gardiner and Masur's description of extremal length horospheres \cite[p.236]{GM} as well as Bonahon's description of the hyperbolic length ones (which is implicit in the structure of shear coordinates for maximal completions).

\begin{corollary}\label{cor:horocoords}
Suppose that $\lambda$ supports $k$ ergodic transverse measures $\lambda_1, \ldots, \lambda_k$. Then for all 
$L_1, \ldots, L_k \in \RR_{>0}$, the level sets
\[\{X \in \T(S) \mid \ell_X(\lambda_i) = L_i \text{ for all }i\}
\text{ and }
\{ \eta \in \MF(\lambda) \mid i(\eta, \lambda_i) = L_i \text{ for all }i\}\]
are both homeomorphic to $\RR^{6g-6-k}$.
\end{corollary}

Analyzing this coordinatization more closely, we see that in fact both level sets can be described as affine bundles of dimension $\dim_{\RR} \cH^+(\lambda) - k$ over the same subvariety of $\T(S \setminus \lambda)$ as underlies $\SH(\lambda)$.

From this refinement we are able to describe the intersection of the leaf $\Fol^{uu}(\lambda)$ with strata.
The decomposition of period coordinates into real and imaginary parts shows that this intersection (when not empty) is locally homeomorphic to $\RR^d$, where $d$ is the complex dimension of the stratum; our work shows that these local homeomorphisms patch together to a global one. Compare \cite[Theorem 1.2]{MW_cohom}.

\begin{corollary}\label{cor:fillML_cap_unstable}
Suppose that $\lambda$ is a filling measured lamination that cuts a surface into polygons with $\kappa_1 +2, \ldots, \kappa_n +2$ many sides, and let $\varepsilon = +1$ if $\lambda$ is orientable and $-1$ otherwise. Let $\QT_g(\sing; \varepsilon)$ denote the union of the components of the stratum $\QT_g(\sing) \subset \QT_g$ that either are ($\varepsilon = +1$) or are not ($\varepsilon = -1$) global squares of Abelian differentials. Then
\[\{ q \in \QT_g(\sing; \varepsilon) \mid |\Im(q)| = \lambda \} \cong \cH^+(\lambda) \cong \RR^d\]
where $d$ is the complex dimension of $\QT_g(\sing; \varepsilon)$.
\end{corollary}
\begin{proof}
Theorem \ref{mainthm:coordinates} indicates that the metric graph of compact horizontal separatrices of $q(\eta,\lambda)$ is encoded by the weighted arc system underlying $\Il(\eta)$.  These weighted arc systems are organized in a piecewise-linear subvariety $\Base$ of a product of \emph{weighted filling arc complexes} that encode the combinatorics of how a zero of order $\kappa_i$ can split up into lower order zeros joined by horizontal saddle connections (see Sections \ref{sec:arc_cx}, \ref{subsec:shsh_axiom}, \ref{subsec:aq}, and Figure \ref{fig:arccx}).
For differentials in the indicated set, there are no compact horizontal separatrices, and so the underlying arc system is always the \emph{empty} (filling) arc system $\emptyset \in \Base$.
In other words, the image of $\{ q \in \QT_g(\sing; \varepsilon) \mid |\Im(q)| = \lambda \}$ in coordinates is just the fiber over $\emptyset$, where Proposition \ref{prop:SH+_structure} identifies $\SH^+(\lambda)$ as an $\cH^+(\lambda)$-bundle over $\Base$.

The second isomorphism $\cH^+(\lambda) \cong \RR^d$ is just a dimension count (see Lemmas \ref{lem:Eulerchar_crowns} and \ref{lem:trans_cocycle_dim} in particular).
\end{proof}

In general, we see that $\Fol^{uu}(\lambda) \cap \QT_g(\sing; \varepsilon)$ forms a $ \cH^+(\lambda)$ bundle over a union of faces of an arc complex of $S \setminus \lambda$.
As a consequence we find that the only obstruction to completeness of any such leaf comes from zeros colliding along a horizontal saddle connection (see also \cite[Theorem 11.2]{MW_cohom}).
This global description of $\Fol^{uu}(\lambda) \cap \QT_g(\sing; \varepsilon)$ also allows the import of arguments from homogeneous dynamics to investigate equidistribution in both $\QoM_g$ and $\PoM_g$ and their strata; see the discussion in Section \ref{sec:future}.

\subsection{Generalized Fenchel--Nielsen coordinates}\label{subsec:FNandDT}
Our shear-shape coordinates for hyperbolic structures can be thought of as interpolating between the classical Fenchel--Nielsen coordinates adapted to a pants decomposition and Bonahon and Thurston's shear coordinates.
In both cases, one remembers the shapes of the complementary subsurfaces (pairs of pants and ideal triangles, respectively) and the space of all hyperbolic structures with given complementary shapes is parametrized by gluing data (twist/shear parameters).

For general $\lambda$, there is a map
\[\textsf{cut}_\lambda:\T(S) \to \T(S \setminus \lambda)\]
that remembers the induced hyperbolic structure on each complementary subsurface. Theorem \ref{thm:hyp_main} then implies that the image of $\textsf{cut}_\lambda$ is a real-analytic subvariety $\Base$ of $\T(S \setminus \lambda)$ consisting of those structures satisfying a ``metric residue condition'' (see Lemma \ref{lem:aXBase}). In the case where each component of $\lambda$ is either non-orientable or a simple closed curve, $\Base$ is just the space of hyperbolic structures for which the two boundary components of the cut surface corresponding to a simple curve component of $\lambda$ have equal length. 
Theorem \ref{thm:hyp_main} together with the structure of $\SH^+(\lambda)$ also allows us to identify the fiber $\textsf{cut}_\lambda\inverse(Y)$ over any $Y \in \Base$ with the gluing data $\cH^+(\lambda)$\footnote{See the discussion around \eqref{eqn:defH+} in regards to the positivity condition for disconnected $\lambda$; in essence, $\cH^+(\lambda)$ is the product of $\cH^+(\lambda_i)$ for each non-closed minimal component together with the twisting data around simple closed curves.} (though not in a canonical way).

We summarize this discussion in the following triptych:
\begin{equation}\label{diagram:T(S)asfibration}
\begin{array}{ccccccc}
\textbf{Fenchel--Nielsen}
	&&& \textbf{Shear-shape}
	&&& \textbf{Shear}\\
\begin{tikzcd}
\RR^{3g-3} \arrow[r]
& \T(S) \arrow[d]\\
& \RR_{>0}^{3g-3}
\end{tikzcd}
	&&& \begin{tikzcd}
\cH^+(\lambda) \arrow[r]
& \T(S) \arrow[d]\\
& \Base
\end{tikzcd}
	&&& \begin{tikzcd}
\cH^+(\lambda) \arrow[r]
& \T(S) \arrow[d]\\
& \{\text{pt}\}
\end{tikzcd}\\
\lambda \text{ a pants decomposition}
	&&& \lambda \text{ arbitrary}
	&&& \lambda \text{ maximal}
\end{array}
\end{equation}
In each coordinate system, $\T(S)$ is the total space of a fiber bundle over a base space of allowable shape data on the subsurface complementary to $\lambda$, while the fiber consists of gluing data.

A completely analogous picture also holds for foliations transverse to $\lambda$, demonstrating $\Il$ as a common generalization of both Dehn--Thurston and Mirzakhani's shear coordinates.

\subsection{Fenchel--Nielsen and Dehn--Thurston via shears and shapes}
In order to give the reader a concrete example of shear-shape coordinates, we include here a discussion of our construction for $\lambda = P$ a pants decomposition. In this case, we see that shear-shape coordinates are just a (mild) reformulation of the classical Fenchel--Nielsen and Dehn--Thurston ones.

First we consider a hyperbolic structure $X$. A pair of pants in $X \setminus P$ is typically parametrized by its boundary lengths $(a,b,c)$, or equivalently, by the alternating side lengths of either of the right angled hexagons coming from cutting along seams.
The orthogeodesic foliation on a pair of pants picks out either a pair or triple of seams (those which are realized as leaves of $\cO_P(X)$), each weighted by the length of a boundary arc consisting of endpoints of leaves of $\cO_P(X)$ isotopic to the seam. 
See Figure \ref{fig:FNDT}. In this case, these lengths are just simple (piecewise) linear combinations of the boundary lengths and the metric residue condition defining $\mathscr{B}(S \setminus P)$ just states that the boundaries that are glued together must have the same length.
See Figure \ref{fig:FNDT}.

\begin{figure}[ht]
\begin{tikzpicture}
    \draw (0, 0) node[inner sep=0] {\includegraphics{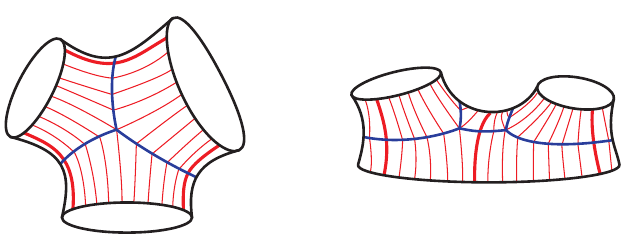}};
    \node at (-5.1, 1){$a$};
    \node at (-1.3, 1){$b$};
    \node at (-3.1,-2.1){$c$};
    \node at (-4.8,-1)[red]{$\frac{a+c-b}{2}$};
    \node at (-1.4,-1)[red]{$\frac{b+c-a}{2}$};
    \node at (-3.3,1.6)[red]{$\frac{a+b-c}{2}$};
    \node at (1.3, 1.1){$a$};
    \node at (4.6, 1){$b$};
    \node at (.5,-.7){$c$};
    \node at (1.2, -1.1)[red]{$a$};
    \node at (4.8, -1.1)[red]{$b$};
    \node at (2.8,-1.3)[red]{$\frac{c-a-b}{2}$};
\end{tikzpicture}
\caption{The orthogeodesic foliation on pairs of pants. Note that the weight of each bolded arc is a linear combination of the boundary lengths, hence the correspondence between shear-shape and Fenchel--Nielsen/Dehn--Thurston coordinates. If any of the weights is zero, the orthogeodesic foliation only picks out the two seams with non-zero weights.}
\label{fig:FNDT}
\end{figure}
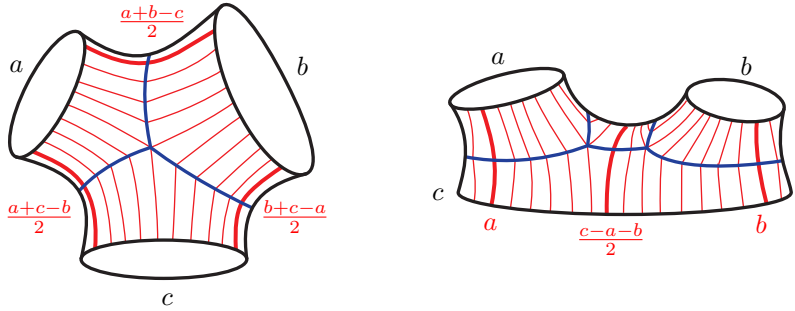

The space $\cH^+(P)$ reduces to a sum of the twist spaces for each curve of $P$, and so Theorem \ref{prop:SH+_structure} implies that $\SH^+(P)$ is a principal $\RR^{3g-3}$ bundle over $\mathscr{B}(S \setminus P) \cong \RR^{3g-3}_{>0}$.
The transverse data recorded by this twist space then describes the signed distance between certain reference points in pairs of right-angled hexagons in $\tX$ that are adjacent to the same curve of $\tilde P$, which is the same as the twist parameter measured by the appropriate choice of Fenchel--Nielsen coordinates.
\footnote{Fenchel--Nielson coordinates always involve some choice of section of the space of twists over the length parameters, and so have only the structure of a principal $\RR^{3g-3}$ bundle over $\RR_+^{3g-3}$.}

We can similarly recognize $\Il:\MF(P)\to \SH^+(P)$ as Dehn--Thurston coordinates. Note first that there are no essential simple closed curves in the complement of $P$, so $\MF(P)$, the space of measured foliations that intersect every curve in the support of $P$, is the same as the space of measured foliations not contained in the support of $P$.
Now from any integral point $\sigma\in \SH^+(P)$ we can construct a multicurve $\alpha$ with prescribed intersection and twisting parameters as follows: the weighted arc system describes how strands of $\alpha$ pass between and meet the components of $ P$, while the transverse data recorded by $\cH^+(P) \cong \RR^P$ describes the extent that strands of $\alpha$ wrap around components of $P$. This procedure is clearly reversible and can easily be extended to transverse foliations using a family of standard train tracks on each pair of pants (see \cite[\S2.6]{PennerHarer}).
As in the hyperbolic case, one can easily pass between these coordinates and the standard Dehn--Thurston ones just by replacing the count of strands of $\alpha$ going from one boundary to the other with the total intersection of $\alpha$ with each boundary.

\section{Outline of the paper}
The rest of this paper is roughly divided into four parts, corresponding to the orthogeodesic foliation, shear-shape cocycles, and shear-shape coordinates for flat and hyperbolic structures, as well as a collection of further directions for investigation (Section \ref{sec:future}).
While the constructions of $\Il$ and $\sigl$ both rely on foundational results established in the first two parts, we have attempted to direct the reader eager to understand our coordinates to the most important statements of these sections.

We expect that the reader is familiar with many of the standard constructions of Teichm{\"u}ller theory, as well as the definitions of both the earthquake and horocycle flows; we recommend \cite[Section 4]{MWnondiv} for a particularly lucid overview of the relevant objects.
We also refer the reader to \cite{CB} and \cite[Section 8]{Thurston:notes} for more on laminations and to \cite{PennerHarer} for a comprehensive introduction to train tracks.

\para{\S\S \ref{sec:crowns}--\ref{sec:arc_cx}: The orthogeodesic foliation}
Cutting along a lamination results in a (possibly disconnected) hyperbolic surface $\Sigma$ with crown boundary, and in Section \ref{sec:crowns} we recall some useful information about the Teichm{\"u}ller spaces of such surfaces.
One particularly important definition is that of the ``metric residue'' of a crown end, which is a generalization of boundary length and plays an important role in cohomological constraints on the shape data of shear-shape cocycles (Lemma \ref{lem:sum_res=0}).

With these preliminaries established, in Section \ref{sec:ortho} we discuss in more detail the orthogeodesic foliation and the hyperbolic geometry of $X$ in a neighborhood of $\lambda$. In this section we also give a geometric interpretation of the map in Corollary \ref{cor:deflate} that relates hyperbolic and extremal length.

The most important result of this part occupies Section \ref{sec:arc_cx}, in which we show that the orthogeodesic foliation restricted to $\Sigma$ completely determines its hyperbolic structure.
More explicitly, dual to each compact edge of the spine of $\Ol(X)$ is a packet of properly isotopic arcs joining non-asymptotic boundary components of $\Sigma$. By assigning geometric weights to each of these packets we can therefore combinatorialize the restriction of $\Ol(X)$ to $\Sigma$ by a weighted, filling arc system. 

Using a geometric limit argument, in Theorem \ref{thm:arc=T(S)_crown} we prove that the map which associates to a hyperbolic structure on $\Sigma$ the associated arc system is a $\Mod(\Sigma)$--equivariant stratified real-analytic homeomorphism between $\T(\Sigma)$ and a certain type of arc complex for $\Sigma$, generalizing a theorem of Luo \cite{Luo} for surfaces with totally geodesic boundary (see also \cite{Mondello,Do,Ushijima}). 
Moreover, by construction this map records both the combinatorial structure of the spine of $\Ol(X)$ as well as the metric residue of the crowns of $\Sigma$.

Theorem \ref{thm:arc=T(S)_crown} is used extensively throughout the paper in order to pass between the combinatorial data of a weighted arc system, the restriction of $\Ol(X)$ to $\Sigma$, and the corresponding hyperbolic structure on $\Sigma$. 
The proof is independent of the main line of argument; as such, the reader is encouraged to understand the statement, but may wish only to skim the proof.

\para{\S\S \ref{sec:shsh_def}--\ref{sec:tt_shsh}: The space of shear-shape cocycles}
The second part of the paper is devoted to our construction of shear-shape cocycles for a given $\lambda$ and an analysis of the space $\SH(\lambda)$ of all shear-shape cocycles.
Upon reaching this section, the reader may find it useful to glance ahead to either Section \ref{sec:flat_map} or \ref{sec:hyp_map} to instantiate our definitions.

After reviewing structural results on transverse cocycles, in Section \ref{sec:shsh_def} we give both cohomological and axiomatic definitions of shear-shape cocycles (Definitions \ref{def:shsh_cohom} and \ref{def:shsh_axiom}, respectively), both predicated on some underlying weighted arc system on $\Sigma$. In Proposition \ref{prop:shsh_defsagree} we prove these definitions agree.
Using the cohomological description, we observe a constraint on the weighted arc systems that can underlie a shear-shape cocycle coming from metric residue conditions (Lemma \ref{lem:sum_res=0}); this can also be thought of as a generalization of the fact that one can only glue together totally geodesic boundary components of the same length (compare Lemma \ref{lem:aXBase}).

Letting $\Base$ denote the subvariety of the filling arc complex of $\Sigma$ cut out by the aforementioned residue conditions, we show in Section \ref{sec:shsh_structure} that the space $\SH(\lambda)$ of shear-shape cocycles forms a bundle of transverse cocycles over $\Base$ with some additional structure (Theorem \ref{thm:shsh_structure}) whose total space is a cell of dimension $6g-6$ (Corollary \ref{cor:shsh_dimension}).
In this section we also introduce the Thurston intersection form on $\SH(\lambda)$ (Section \ref{subsec:Th_form}) and prove that the positive locus $\SH^+(\lambda)$ it defines is itself a bundle over $\Base$ (Proposition \ref{prop:SH+_structure}).

Finally, in Section \ref{sec:tt_shsh} we give train track coordinates for the space of shear-shape cocycles. The train tracks we use give a preferred decomposition of arcs on $S$ into pieces that are measurable by shear-shape cocycles and as such give a useful way of specifying shear-shape cocycles by a finite amount of data. The weight space for a train track is also a natural model in which to consider local deformations of a shear-shape cocycle, a feature which we exploit in Section \ref{sec:shapeshift_def}. In Section \ref{subsec:PIL} we discuss how the piecewise integral linear structure induced by train track charts endows $\SH^+(\lambda)$ with a well defined integer lattice and preferred measure in the class of Lebesgue.

The reader willing to accept the structure theorems can adequately navigate the remaining two parts of the paper using weight systems on (augmented) train tracks as a local description of the structure of shear-shape space.

\para{\S\S \ref{sec:flat_map} and \ref{sec:flatflows}: Coordinates for transverse foliations}
At this point, we have established the structure necessary to coordinatize foliations transverse to $\lambda$ by shear-shape cocycles.

A measured foliation $\eta\in \MF(\lambda)$ determines a holomorphic quadratic differential $q = q(\eta, \lambda) \in \Fol^{uu}(\lambda)$ via Theorem \ref{thm:GM}, and we begin by specifying an arc system $\arc(q)$ that records the horizontal separatrices of $q$.
We then build a train track $\tau$ carrying $\lambda$ from a triangulation by saddle connections (Construction \ref{constr:ttfromtri}); augmenting $\tau$ by the arc system $\arc(q)$ then allows us to realize the periods of the triangulation as a (cohomological) shear-shape cocycle $\Il(\eta)$.
This identification also gives a useful formula for $\Il(\eta)$ as a weight system on the augmented train track $\tau$ (Lemma \ref{lem:pers_as_ttwts}). 

We then show that one can rebuild $q$ just from the train track weights defined by $\Il(\eta)$; a similar (but more technical) argument then gives that $\Il(\eta) \in \SH^+(\lambda)$ (Proposition \ref{prop:Il_takes_int_to_Thurston}).
This reconstruction technique together with the structure of shear-shape space therefore allows to deduce that $\Il$ is a homeomorphism onto its image.
At the end of this section, we explain how the work done in the fourth and final part of the paper implies that $\Il$ surjects onto $\SH^+(\lambda)$ (Theorem \ref{thm:Ilhomeo}), and why we choose to prove surjectivity this way. See Remark \ref{rmk:Il_surjective} in particular.

Since $\Il$ essentially yields period coordinates, it is not surprising that (a variant of) Teichm\"uller geodesic flow is given in coordinates by dilation (Lemma \ref{lem:Il_geo}), while the Teichm\"uller horocycle flow is translation by $\lambda$ (Lemma \ref{lem:Il_hor}).  
We also naturally recover the ``tremor deformations'' introduced in \cite{CSW} as translation by measures $\mu$ supported on $\lambda$ that are not necessarily absolutely continuous with respect to $\lambda$ (Definition \ref{def:Il_trem}).
Figure \ref{fig:CSWdict} details a dictionary between the language of \cite{CSW} and our own.

\para{\S\S \ref{sec:hyp_overview}--\ref{sec:shsh_homeo}: Coordinates for hyperbolic structures} 
In the final part of the paper, we use the geometry of the orthogeodesic foliation to coordinatize hyperbolic structures via shear-shape cocycles.

From Theorem \ref{thm:arc=T(S)_crown}, we know that the combinatorialization of $\Ol(X)$ on each subsurface $S\setminus \lambda$ by a weighted arc system completely encodes the geometry of the pieces.
Cutting $X\setminus \lambda$ further along the orthogeodesic realization of each such arc, we obtain a family of (partially ideal) right-angled polygons. The orthogeodesic foliation equips each polygon with a natural family of basepoints, one on each of its sides adjacent to $\lambda$, that vary analytically in $\T(S\setminus \lambda)$.
We are thus able to define a ``shear'' parameter between (some pairs of) degenerate polygons, and this shear data assembles together with the ``shape'' data on each subsurface to give instructions for gluing the polygonal pieces back together to obtain $X$.

In Section \ref{sec:hyp_overview} we state the main Theorem \ref{thm:hyp_main}, that the shear-shape coordinate map $\sigl: \T(S)\to \SH^+(\lambda)$ is a homeomorphism, supply an outline of its proof, and derive some immediate corollaries.
The construction of $\sigl$ is given in Section \ref{sec:hyp_map}, where we formalize the discussion from the previous paragraph.
We also prove that the central Diagram \eqref{diagram} commutes (Theorem \ref{thm:diagram_commutes}), which then implies that $\sigl$ takes hyperbolic length to the Thurston intersection form (Corollary \ref{cor:sigl_into_SH+}).

Section \ref{sec:shapeshift_def} is the most technical part of the paper. In it, we define the ``shape-shifting'' cocycles (Proposition \ref{prop:shapeshift_cocycle}) along which a hyperbolic structure can be deformed (Theorem \ref{thm:shsh_open}); these deformations are generalizations of Thurston's cataclysms or Bonahon's shear deformations.  
Although the construction of a shape-shifting deformation is rather involved, we attempt to keep the reader informed of the geometric intuition that guides the construction throughout.
Finally, in Section \ref{sec:shsh_homeo} we assemble all of the necessary ingredients to prove  Theorem \ref{thm:hyp_main}.
That the earthquake along $\lambda$ is given by translation by $\lambda$ in $\SH^+(\lambda)$ (Corollary \ref{cor:eq=translation}) is an immediate consequence of the construction of shape-shifting deformations as generalizations of cataclysms.
We then discuss how the action of dilation in coordinates can sometimes be identified with directed geodesics in Thurston's asymmetric metric (Propositions \ref{prop:stretch_reg} and \ref{prop:quad_geo}).

\section{Crowned hyperbolic surfaces}\label{sec:crowns}
When a hyperbolic surface is cut along a geodesic multicurve, the (completion of the) resulting space is a compact hyperbolic surface with compact, totally geodesic boundary. When the same surface is cut along a geodesic lamination, the (completion of the) complementary subsurface can have non-compact ``crowned boundaries.'' 
This section collects results about hyperbolic structures on such ``crowned surfaces'' as well as the relationship between properties of the lamination and the topology of its complementary subsurfaces.

\begin{remark}
Throughout this section and the following, we reserve $S$ to denote a closed surface. If $\lambda$ is a geodesic lamination, then $S\setminus\lambda$ denotes the metric completion of the complementary subsurfaces to $\lambda$ (with respect to some auxiliary hyperbolic metric); we will refer to the topological type of a component of $S \setminus \lambda$ by $\Sigma$. Hyperbolic metrics on $S$ and $\Sigma$ will be denoted by $X$ and $Y$, respectively.
\end{remark}

\para{Hyperbolic crowns}
While less familiar than surfaces with boundary, crowned hyperbolic surfaces naturally arise by uniformizing surfaces with boundary and marked points on the boundary. They are also intricately related to meromorphic differentials on Riemann surfaces with high order poles (see, e.g., \cite{Gupta_wild}).

A {\em hyperbolic crown} with $c_k$ spikes is a complete, finite-area hyperbolic surface with geodesic boundary that is homeomorphic to an annulus with $c_k$ points removed from one boundary component.
In the hyperbolic metric, the circular boundary component corresponds to a closed geodesic and each interval of the other boundary becomes a bi-infinite geodesic running between ideal vertices; compare Figure \ref{fig:truncation}.

In general, a {\em hyperbolic surface with crowned boundary} is a complete, finite-area hyperbolic surface with totally geodesic boundary; the boundary components are either compact or hyperbolic crowns.
We record the topological type of a crowned surface of genus $g$ with $b$ closed boundary components and $k$ crowns with $c_1, \ldots, c_k$ many spikes as $\cutsurf_{g,b}^{\crowns}$, where $\crowns = \{c_1, \ldots, c_k\}$.
\label{ind:crownedsurf}

\begin{remark}
Ideal polygons may be considered as crowned surfaces of genus 0 with a single (crowned) boundary component. All of the results in this section hold for both crowned surfaces with nontrivial topology as well as for ideal polygons, but their proofs are slightly different. 
Our citations of \cite{Gupta_wild} are all for the case when $\cutsurf$ is not an ideal polygon; for the corresponding statements for ideal polygons, see \cite[Section 3.3]{Gupta_wild} or  \cite{HTTW}.
\end{remark}

Every crowned surface $Y$ with non-cyclic (and non-trivial) fundamental group contains a ``convex core'' obtained by cutting off its crowns along a geodesic multicurve \cite[Lemma 4.4]{CB}.
When $Y$ has type $\cutsurf_{g,b}^{\crowns}$, this core is a subsurface of genus $g$ with $b+k$ closed boundary components.
Since each crown with $c_i$ spikes may be decomposed into $c_i$ ideal hyperbolic triangles by introducing leaves wrapping around the totally geodesic boundary component, we have the following expression for the area:
\begin{equation}\label{eqn:area_crown}
\frac{1}{\pi} \text{Area} \left( Y \right)
= 4g-4+2b + \sum_{i=1}^k (c_i+2).
\end{equation}
Note that one can triangulate an ideal polygon of $c$ sides into $(c-2)$ ideal triangles, and so the above formula also holds for ideal polygons.

\para{The metric residue}
While crown ends (and ideal polygons) do not have well-defined boundary lengths, one can define a natural generalization when there are an even number of spikes.
This turns out to be a fundamental invariant that controls when crowns can be glued together along a lamination (Lemma \ref{lem:aXBase}).

Let $\Crown$ be a hyperbolic crown or an ideal polygon with $c$ spikes, where $c$ is even. One can then orient $\Crown$, that is, pick an orientation of the boundary leaves so that the orientations of asymptotic leaves agree.
Truncating each spike of $\Crown$ along a horocycle based at the tip of the spike yields a surface with a boundary made up of horocyclic segments $h_1, \ldots, h_c$ and geodesic segments $g_1, \ldots, g_c$. See Figure \ref{fig:truncation}.

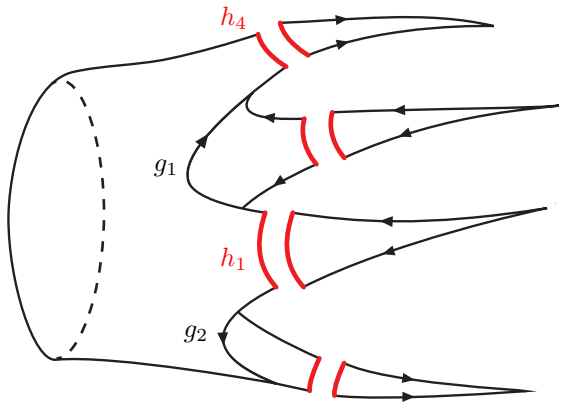
\begin{figure}[ht]
\centering
\begin{tikzpicture}
    \draw (0, 0) node[inner sep=0] {\includegraphics{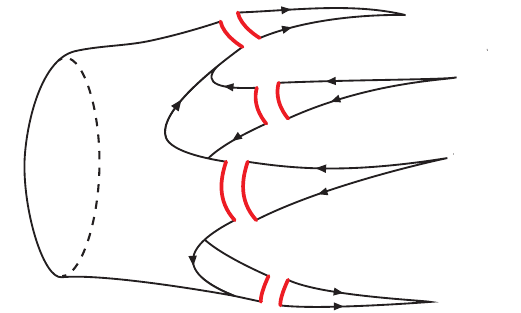}};
    \node at (-1.9,.5) {$g_1$};
    \node at (-1,-.7) [red]{$h_1$};
    \node at (-1.5,-1.7) {$g_2$};
    \node at (-1,2.5) [red]{$h_4$};
\end{tikzpicture}
    \caption{Truncating an (oriented) crown to compute its metric residue.}
    \label{fig:truncation}
\end{figure}

\begin{definition}[Definition 2.9 of \cite{Gupta_wild}]\label{def:metric_res}
Let $\Crown$ be either an oriented hyperbolic crown or an oriented ideal polygon with an even number of spikes. Then its {\em metric residue} $\res(\Crown)$ is
\[\res(\Crown) = \sum_{i=1}^c \varepsilon_i \ell(g_i)\]
\label{ind:metricres}
where $\varepsilon_i$ is positive if the truncated crown lies on the left of $g_i$ and negative if it lies on the right.
\end{definition}

Since changing the truncation depth of a spike increases the length of two adjacent sides, the metric residue evidently does not depend on the choice of truncation \cite[Lemma 2.10]{Gupta_wild}.
Observe also that flipping the orientation of $\Crown$ flips the sign of its metric residue.

Similarly, define the metric residue of an oriented totally geodesic boundary component $\beta$ of $Y$ to be $\pm\ell(\beta)$, where the sign depends on whether $Y$ lies to the left of $\beta$ (positive) or right (negative).

\para{Deformation spaces of crowned surfaces}
We now record some useful facts about the Teichm{\"u}ller spaces of crowned hyperbolic surfaces.

Given any crowned hyperbolic surface $Y$, one can obtain a natural compactification $\widehat{Y}$ by adding on an ideal vertex at the end of each spike of each crown. The corresponding (topological) surface $\widehat{\cutsurf}_{g,b}^{\crowns}$ then has $b+k$ boundary components with $c_i$ marked points on the $(b+i)\ith$ boundary component.
A {\em marking} of a crowned hyperbolic surface $Y$ is a homeomorphism
\[f: \widehat{\cutsurf}_{g,b}^{\crowns} \rightarrow \widehat{Y}\]
\label{ind:compactifiedcrowned}
which takes boundary marked points to ideal vertices.
We think of the boundary marked points as having distinct labels, so different identifications of the boundary points of $\widehat{\cutsurf}_{g,b}^{\crowns}$ with the spikes of $Y$ yield different markings.  The Teichm{\"u}ller space of a crowned hyperbolic surface $\cutsurf_{g,b}^{\crowns}$ is then defined to be the space of all marked hyperbolic metrics on $\cutsurf_{g,b}^{\crowns}$, up to isotopies which fix the totally geodesic boundary components pointwise and fix each ideal vertex of each crown.

As noted above, any crowned hyperbolic surface $\cutsurf_{g,b}^{\crowns}$ contains an uncrowned subsurface which serves as its convex core. Therefore, the Teichm{\"u}ller space of a crowned hyperbolic surface may be parametrized by the Teichm{\"u}ller space of its convex core together with parameters describing each crown and how it is attached.
A precise version of this dimension count is recorded below.

\begin{lemma}[Lemma 2.16 of \cite{Gupta_wild}]\label{lem:Teich_crown}
Let $\cutsurf= \cutsurf_{g,b}^{\crowns}$ be a crowned hyperbolic surface or an ideal polygon. Then $\T(\cutsurf) \cong \RR^{d}$, where
\begin{equation}\label{eqn:Teich_crown_dim}
d = 6g-6 + 3b + \sum_{i=1}^k (c_i + 3).
\end{equation}
\end{lemma}

Fixing the length of any closed boundary component of $\cutsurf_{g,b}^{\crowns}$ cuts out a codimension $1$ subvariety of $\T(\cutsurf)$. Similarly, the subspace of surfaces with fixed metric residues at an even--spiked crown has codimension one.
The following proposition ensures that the intersections of the level sets of length and metric residue are topologically just cells of the proper dimension:

\begin{proposition}[Corollary 2.17 in \cite{Gupta_wild}]\label{prop:res_mnfld}
Let $\cutsurf = \cutsurf_{g,b}^{\crowns}$ be a crowned surface or an ideal polygon. 
Let $\beta_1,\ldots, \beta_b$ denote the closed boundary components of $\cutsurf$ and let $\Crown_1, \ldots, \Crown_e$ denote the crown ends which have an even number of spikes. Fix an orientation of each crown end. Then for any $(L_i) \in \RR_{>0}^{b}$ and any $(R_j) \in \RR^{e}$, 
\[ \left\{ (Y, f) \in \T(\cutsurf) \,|\,
\ell(\beta_i) = L_i \text{ and }
\res(\Crown_j) = R_j \text{ for all } i,j \right\} 
\cong \RR^{d-b-e}\]
where $d$ is as in \eqref{eqn:Teich_crown_dim}.
\end{proposition}

\para{Topology}
When a crowned surface $\Sigma$ comes from cutting a closed surface $S$ along a geodesic lamination $\lambda$, we can relate the topology of $\lambda$ to the topology of $\Sigma$.

Recall that the Euler characteristic of a lamination 
\label{ind:X(lambda)}
is defined to be alternating sum of the ranks of its {\v C}ech cohomology groups, viewing $\lambda$ as a subset of $S$. Below, we compute the Euler characteristic of a geodesic lamination in terms of the topological type of its complementary subsurfaces.

\begin{lemma}\label{lem:Eulerchar_crowns}
Let $\lambda$ be a geodesic lamination on $S$. Then the total number of spikes of $S \setminus \lambda$ equals $-2 \chi(\lambda)$.
\end{lemma}

We also record the corresponding formula for later use. Suppose that $\overline{ S \setminus \lambda} = \cutsurf_1 \cup \ldots \cup \cutsurf_m$; then 
\begin{equation}\label{eqn:Eulerchar_crowns}
\chi(\lambda) = -\frac{1}{2} \sum_{j=1}^m \sum_{i=1}^{k_j} c^j_{i}
\end{equation}
where $\{ c^{j}_1, \ldots, c^j_{k_j} \}$ denotes the crown type of $\cutsurf_j$.

\begin{proof}
Fix some train track $\tau$ which carries $\lambda$ and has the same topological type; in Section \ref{subsec:ortho_foliation} below, this is referred to as {\em snug} carrying of $\lambda$ on $\tau$. Lemma 13 of \cite{Bon_THDGL} states that for any such train track, $\chi(\lambda) = \chi(\tau)$, and so it suffices to compute the Euler characteristic of $\tau$.

Splitting the switches of $\tau$ if necessary, we may assume that $\tau$ is trivalent (observe that this operation preserves the Euler characteristic). Then each spike of $S \setminus \lambda$ corresponds to a unique switch of $\tau$, and each switch corresponds to three half--edges, so
\[\# \text{ spikes}(S \setminus \lambda) 
= \# \text{ switches}(\tau)
= \frac{2}{3} \cdot \# \text{ edges}(\tau).\]
Plugging this into the formula $\chi(\tau) = \# \text{ switches}(\tau) - \# \text{ edges}(\tau)$ proves the claim.
\end{proof}

In general, the relationship between the boundary components of $S \setminus \lambda$ and $\lambda$ can be rather involved. For example, one can construct a lamination on a closed surface of genus $g \ge 2$ consisting of 3 leaves, two of which are non-isotopic simple closed curves and one leaf which spirals onto each of the closed leaves. In this scenario, there is not a precise correspondence between closed leaves of $\lambda$ and totally geodesic boundary components of its complementary subsurface.

When $\lambda$ also supports a measure of full support, however, things become much nicer. In particular, each component of $\lambda$ is minimal (in that every leaf is dense in the component) and so the closed leaves of $\lambda$ are all isolated. In this case, there is a natural 1-to-2 correspondence between closed leaves of $\lambda$ and totally geodesic boundary components of $S \setminus \lambda$.

\para{N.B} So that we do not have to deal with possible spiraling behavior of $\lambda$, we henceforth restrict our discussion to those laminations that support a measure. See also Remark \ref{rmk:geod_lamination}.

\section{The orthogeodesic foliation}\label{sec:ortho}
In this section we construct the \emph{orthogeodesic foliation} $\Ol(X)\in \MF(\lambda)$ of a hyperbolic surface $X$ with respect to $\lambda$ and describe some of its basic properties.

\subsection{The spine of a hyperbolic surface}\label{subsec:ortho_spine}

We begin by describing the othogeodesic foliation restricted to subsurfaces $Y$ complementary to $\lambda$. Let $Y$ be a finite area hyperbolic surface with totally geodesic boundary, possibly with crowned boundary. As we are most interested in the $Y$ coming from cutting a closed surface along a lamination, we also assume that $Y$ has no annular cusps.

\begin{definition}
The orthogeodesic foliation $\cO_{\partial Y}(Y)$
\label{ind:orthorelboundary}
of $Y$ is the (singular, piecewise-geodesic) foliation of $Y$ whose leaves are fibers of the closest point projection to $\partial Y$.
\end{definition}

Near $\partial Y$, the leaves of $\cO_{\partial Y}(Y)$ are geodesic arcs meeting $\partial Y$ orthogonally. To understand the global structure of the foliation, however, we need to determine how the leaves extend into the interior of $Y$. In particular, we must understand the locus of points that are closest to multiple points of $\partial Y$.

To that end, for any point $x \in Y$, define the {\em valence} of $x$ to be
\[\val(x) := \#\{y \in \partial Y : d(x,y) = d(x, \partial Y) \}.\]
\label{ind:valence}
The {\em (geometric) spine} $\Sp(Y)$ 
\label{ind:spine}
of $Y$ is the set of points of $Y$ with valence at least $2$, and has a natural partition into subsets $\Sp_k(Y)$, where $x \in \Sp_k(Y)$ if it is equidistant from exactly $k$ points in $\partial Y$.
For the rest of the section we fix a hyperbolic surface $Y$ and refer to $\Sp(Y)$ and $\Sp_k(Y)$ simply as $\Sp$ and $\Sp_k$. 

It is not hard to see that $\Sp$ is a properly embedded, piecewise geodesic $1$-complex with some nodes of valence $1$ removed (equivalently, a ribbon graph with some half-infinite edges).
Indeed, $\Sp$ decomposes into a \emph{finite core} $\Sp^0$
\label{ind:spinecore}
 and a finite collection of open geodesic rays; since we assumed $Y$ had no annular cusps, each ray corresponds with a spike of a crowned boundary component.
See \cite[Section 2]{Mondello} for a discussion of the structure of the spine of a compact hyperbolic surface with geodesic boundary in which $\Sp^0 = \Sp$.

We record below a summary of this discussion; see also Figure \ref{fig:spinecrowns}.

\begin{figure}[ht]
    \centering
    \includegraphics[scale=1]{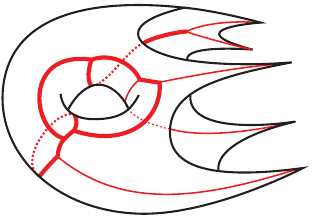}
    \caption{The spine of a hyperbolic surface with crowned boundary. Note that the finite core $\Sp^0$ (represented in bold) contains a spine for the convex core of the surface.}
    \label{fig:spinecrowns}
\end{figure}

\begin{lemma}
The finite core $\Sp^0$ is a piecewise geodesically embedded graph, whose edges correspond to the components of $\Sp_2$ with finite hyperbolic length and vertex set $\cup_{k\ge3}\Sp_k$. Each geodesic ray of $\Sp\setminus\Sp^0$ exits a unique spike of $Y$.
\end{lemma}

By definition, the orthogeodesic foliation $\mathcal{O}_{\partial Y}(Y)$ has $k$-pronged singular leaves emanating from $\cup_{k\ge3} \Sp_k$ for $k\ge 3$.
The nonsingular leaves of $\mathcal{O}_{\partial Y}(Y)$ glue along $\Sp_2(Y)$ (usually at an angle) and can be smoothed by an arbitrarily small isotopy supported near $\Sp$.
As the geometry of $\Sp$ interacts nicely with the leaves of $\mathcal O_{\partial Y}(Y)$, we generally prefer to think about $\mathcal O_{\partial Y}(Y)$ as a piecewise geodesic singular foliation rather than as a smooth one.  When convenient, we will pass freely between the orthogeodesic foliation and a smoothing.

We observe that there is also an isotopy supported in the ends of the spikes of $Y$ and restricting to the identity on $\partial Y$ that maps leaves of the orthogeodesic foliation to horocycles based at the tip of the spike. This equivalence between the orthogedesic and horocyclic foliations in spikes is of vital importance in Sections \ref{sec:hyp_map}--\ref{sec:shsh_homeo} as it allows us to adapt many of Bonahon and Thurston's arguments to this setting.

\begin{remark}
One can check that, for regular ideal polygons, the isotopy in spikes extends to a global isotopy between the orthogeodesic foliation and the symmetric partial foliation by horocycles.
\end{remark}

Following the leaves of the orthogeodesic foliation in the direction of $\Sp$ defines a deformation retraction of $Y$ onto $\Sp$; let $r:Y \to \Sp$ be the map fully collapsing $Y$ onto $\Sp$.
For $x$ and $y$ in the same component of $\Sp_2$, the leaves $r^{-1}(x)$ and $r^{-1}(y)$ of $\cO_{\partial Y}(Y)$ are properly isotopic. We may therefore associate to each edge $e$ of $\Sp_2$ the (proper) isotopy class of $r\inverse(x)$ for $x \in e$; we call this the {\em dual arc} $\alpha_e$ to $e$.
\label{ind:dualarc}

There is a distinguished representative of $\alpha_e$ that is geodesic and orthogonal to both $\partial Y$ and $e$; compare Figure \ref{fig:collar_bounds}. By abuse of notation, we henceforth identify $\alpha_e$ with its orthogeodesic representative and define
\[\arc(Y): = \bigcup_{e \subset \Sp_2^0 }\alpha_e. \]
\label{ind:arcsys}

\begin{lemma}\label{lem:arcs_fill}
The metric completion of the surface with corners $Y\setminus \arc(Y)$ is homeomorphic to a union of closed disks and closed disks with finitely many points on the boundary removed.  That is, $\arc(Y)$ fills $Y$.
\end{lemma}
\begin{proof}
Each component of $Y\setminus \arc(Y)$ deformation retracts onto a component of the metric completion of $\Sp\setminus \arc(Y)$.  By the duality of arcs and edges of $\Sp_2^0$,  each component of $\Sp\setminus \arc(Y)$ is contractible.
\end{proof}
The orthogeodesic foliation also comes with a natural transverse measure: the measure of an arc $k$ transverse to (a smoothing of) $\mathcal O_{\partial Y}(Y)$ is defined on small enough transverse arcs $k$ first by isotoping the arc into $\partial Y$ transversely to $\cO_{\partial Y}(Y)$ and then measuring the hyperbolic length there.  
Locally, the orthogeodesic foliation admits a reflection about each edge of $\Sp$, so by restricting $k$ to those leaves of $\cO_{\partial Y}(Y)$ that intersect a given edge, we can use this symmetry to see that the measure of $k$ is the same after a transverse isotopy onto either boundary component of $Y$.
Extending to all transverse arcs by additivity defines a transverse measure on $\mathcal O_{\partial Y}(Y)$.

To each component $e$ of $\Sp_2^0$ we associate the length $c_e>0$ of either component of $r\inverse (e)\cap\partial Y$; the transverse measure of $e$ is exactly $c_e$.
Anticipating the contents of the next section (see, e.g., Theorem \ref{thm:arc=T(S)_crown}), we define the formal sum
\begin{equation}\label{eqn:hyp_arc_def}
\arcwt(Y) : = \sum _{e\subset \Sp_2^0} c_e \alpha_e.
\end{equation}

\subsection{The orthogeodesic foliation}\label{subsec:ortho_foliation}
Now that we have described the orthogeodesic foliation on each component of $S \setminus \lambda$, we can glue these pieces together along the leaves of $\lambda$ to get a foliation of $S$.

\begin{construction}
Let $X\in \T(S)$ and $\lambda$ be a geodesic lamination on $X$.  Cutting $X$ open along $\lambda$ taking the metric completion of each component, we obtain a union of hyperbolic surfaces with totally geodesic boundary (possibly with crowned boundary).  On each such component $Y$, we construct the orthogeodesic foliation $\mathcal O_{\partial Y}(Y)$ as described in Section \ref{subsec:ortho_spine} above.

A standard fact from hyperbolic geometry \cite[Lemma 5.2.6]{CEG} shows that the line field defined by (a smoothing of) the orthogeodesic foliation forms a Lipschitz line field on $X\setminus \lambda$.  Since $\lambda$ has measure $0$, this line field is integrable near $\lambda$, so the partial foliation defined on $X\setminus \lambda$ extends across the leaves of $\lambda$. This defines a measured foliation $\Ol (X)\in \MF(S)$, hence a map $\Ol: \T(S) \to \MF(S)$.
\label{ind:Ol}
\end{construction}

Later, we prove Lemma \ref{lem:binding} that $\lambda$ and $\Ol(X)$ bind, allowing us to restrict the codomain of $\Ol$ to $\MF(\lambda)$.
Ultimately, our goal is to show that $\Ol$ is a homeomorphism onto $\MF(\lambda)$.

\para{Geometric train tracks}
We now consider the geometry of $\Ol(X)$ in a neighborhood of $\lambda$.  The following is a modification of an important construction of Thurston \cite[Chapter 8.9]{Thurston:notes}.

\begin{construction}\label{const:geometric_tt} Let $\epsilon>0$ be small enough so that the $\epsilon$-neighborhood $\epN \lambda$
\label{ind:geott}
is topologically stable. The orthogeodesic foliation $\Ol(X)$ restricts to a foliation of $\epN\lambda$ without singular points, and collapsing the leaves yields a quotient map $\pi: \epN\lambda \rightarrow \tau$ where $\tau$ can be $C^1$-embedded in $\epN\lambda$ as a train track carrying $\lambda$ in $X$.  By changing $\epsilon$, we may assume that $\tau$ is trivalent.
\footnote{In the literature, trivalent train tracks are also called ``generic.''}
Then $\tau = \tau(\lambda, X, \epsilon)$ is a \emph{geometric train track}.
\end{construction}

We sometimes refer to $\epN\lambda$ as a \emph{train track neighborhood} of $\lambda$ and the leaves of $\Ol(X)|_{\epN\lambda}$ as \emph{ties}.
\label{ind:ties}
A train track neighborhood coming from Construction \ref{const:geometric_tt} is a union of bands and annuli foliated by ties glued together along the ties that collapse to switches of $\tau$.
We recall that if $\lambda$ meets every tie of $\tau$ and there is no path between spikes of $S \setminus \epN\lambda$ that is contained in $\epN\lambda \setminus \lambda$, then $\tau$ is said to {\em snugly} carry $\lambda$.
Equivalently, $\tau$ snugly carries $\lambda$ if and only if $S \setminus \lambda$ and $S \setminus \tau$ have the same topological type. With this definition, it is clear that the geometric train tracks constructed above always carry $\lambda$ snugly.

Using the geometry of $\pi: \epN\lambda\to \tau$, the branches of $\tau$ admit a well defined notion of length. 
Indeed, let $b\subset \tau$ be a branch, and choose a lift $\widetilde {b}$ to the universal cover $\tX$.  
Let $\ell, \ell'\subset \tlambda$ be leaves of the elevation $\tlambda$ of $\lambda$ to $\tX$ that meet $\pi\inverse(\widetilde b)\subset \epN{\tlambda}$ in segments $g$ and $g'$.
Since $\Ol(X)$ is equivalent to a horocyclic foliation in $\epN\lambda$, transporting $g$ along the leaves of ${\mathcal{O}_{\tlambda}(\tX)}$ near $\widetilde b$ onto $g'$ is isometric, so $\ell_X(g) = \ell_X(g')$.
We may therefore define the \emph{length} of $b$ (along $\lambda$) as 
\[\ell_X(b) := \ell_X(g)\]
\label{ind:ttlength}
for any $g$ as above. Similarly, for any branch $b\subset \tau$, the ties of $\epN\lambda$ collapsing to $b$ all have the same integral with respect to $\lambda$.  Define \[\lambda(b) := \lambda (k)\]
for any tie $k\subset \Ol(X)|_{\pi\inverse(b)}$; this is equivalently the weight deposited by $\lambda$ on $b$ in its $\tau$ train track coordinates.

\begin{lemma}\label{lem:length_computation}
For any hyperbolic structure $X$ and any measure $\lambda'$ on $\lambda$, we have  $i(\lambda', \Ol(X)) = \ell_X(\lambda')$.
\end{lemma}

\begin{proof}
Using Construction \ref{const:geometric_tt}, find a geometric train track $\pi: \epN\lambda\to \tau$ snugly carrying $\lambda$ on $X$.
By definition, the intersection pairing is given by the integral over $X$ of the product measure $d\lambda' \otimes d\Ol(X)$, whose support is contained entirely in the train track neighborhood $\epN\lambda$.  
For each branch $b\subset \tau$, the integral of this measure on $\pi\inverse(b)$ is just $\lambda'(b)\ell_X(b)$, so 
\begin{align*}
i(\lambda', \Ol(X)) & = \iint_X d\lambda' \otimes d\Ol(X) &\\
 &= \iint_{\epN\lambda} d\lambda'\otimes d\Ol(X) = \sum_{b \subset \tau} \iint_{\pi\inverse(b)} d\lambda' \otimes d\Ol(X) \\
 & = \sum_{b\subset \tau} \lambda'(b)\ell_X(b).
\end{align*}

On the other hand, $\ell_X(\lambda')$ is the integral over $X$ of the measure $d\lambda'\otimes dl_{\lambda'}$, locally the product of the transverse measure $\lambda'$ and $1$-dimensional Lebesgue measure $l_{\lambda'}$ on the support of $\lambda'$.
Since $\lambda'$ is supported in $\lambda$, the integral of $d\lambda'\otimes dl_{\lambda'}$ is equal to the integral of $d\lambda'\otimes dl_\lambda$, and again the support of the product measure is contained in $\epN\lambda$.  
On each thickened branch $\pi\inverse(b)\subset \epN\lambda$, the integral of $d\lambda'\otimes dl_\lambda$ is $\lambda'(b)\ell_X(b)$, giving the equality 
\[\ell_X(\lambda') = \sum_{b\subset \tau} \lambda'(b)\ell_X(b).\]
This completes the proof of the lemma.
\end{proof}

With this computation, we can now show that $\lambda$ and $\Ol(X)$ together bind $S$.

\begin{lemma}\label{lem:binding}
For any $X \in \T(S)$ and $\lambda \in \ML(S)$, we have $\Ol(X)\in \MF(\lambda)$.
\end{lemma}

\begin{proof}
Suppose that $\eta$ is an measured lamination so that $i(\eta,\lambda) = 0$; without loss of generality, we may assume that $\eta$ is ergodic. Then one of two things must be true: either $\eta$ is supported on $\lambda$ or its support is disjoint from $\lambda$.  In the first case, $i(\eta,\Ol(X)) = \ell_X(\eta)>0$ by Lemma \ref{lem:length_computation}.

If $\eta$ is disjoint from $\lambda$ then it is contained in a component $Y$ of $\overline{X\setminus \lambda}$, and we need only show that $i(\eta, \Ol(X)) >0$. Scaling the measure of $\eta$ as necessary, let us assume that $\ell_X(\eta) = \ell_Y(\eta) = 1$.
Now we recall that the set of weighted simple closed curves is dense in the space of measured laminations on $Y$.
By homogeneity and continuity of the intersection pairing, it therefore suffices to find some uniform $\epsilon>0$ such that \[i(\gamma, \Ol(X))\ge \epsilon \ell_X(\gamma)\]
for every simple closed curve $\gamma\subset Y$. Indeed, once we have  demonstrated such a bound we may approximate $\eta$ arbitrarily well by weighted curves $\gamma / \ell_X(\gamma)$ to deduce the desired bound on $i(\eta, \Ol(X))$.

So let $Y_0$ be the convex hull of $r\inverse(\Sp^0)$; $Y_0$ is compact and the inclusion of $Y_0$ into $Y$ is a homotopy equivalence. Any simple closed geodesic $\gamma$ in $Y$ is contained in $Y_0$, and since $Y$ deformation retracts onto the component of $\Sp$ contained in $Y$, $\gamma$ is homotopic to a concatenation of edges in $\Sp^0$.  

Give $\Sp^0$ a metric making its edges $e$ have length $c_e= i(e, \Ol(X))$; then the inclusion $\Sp^0 \to Y_0$ with this metric induces an equivariant quasi-isometry on universal covers (this follows because they are both Gromov hyperbolic and $\pi_1(Y)$ acts cocompactly and properly discontinuous on each). The geodesic lengths of closed curves in $\Sp$ and in $Y_0$ are therefore comparable, so that there is some $\epsilon>0$ so that 
\[i(\gamma, \Ol(X)) = \ell_{\Sp^0}(\gamma) \ge \ell_X(\gamma) \epsilon,\]
demonstrating the desired uniform bound.
\end{proof}

\subsection{Deflation}\label{subsec:deflation}
For a given pair $(X, \lambda)\in \T(S)\times \ML(S)$, the pair of laminations $\Ol(X)$ and $\lambda$ bind by Lemma \ref{lem:binding}. By Theorem \ref{thm:GM}, there is a unique quadratic differential $q= q(\Ol(X), \lambda)$, holomorphic on some Riemann surface $Z$, whose real and imaginary foliations are $\Ol(X)$ and $\lambda$, respectively. 
In this section we define a \emph{deflation} map $\Defl: X\to Z$ that allows us to make direct comparisons between the hyperbolic geometry of $X$ and the singular flat geometry $q$.
\label{ind:deflate}

An informal description of $\Defl$ is that it ``deflates'' the subsurfaces of $X\setminus \lambda$, retracting them to $\Sp$ along the leaves of $\Ol(X)$, while it ``inflates'' along the leaves of $\lambda$ according to the transverse measure.
The orthogeodesic foliation in a neighborhood of $\lambda$ assembles into the vertical foliation of the resulting quadratic differential metric and $\Defl$ maps $\Sp\subset X$ to the horizontal separatrices; compare Figure \ref{fig:deflate}.

\begin{figure}[ht]
    \centering
    \begin{tikzpicture}
    \draw (0, 0) node[inner sep=0]{\includegraphics{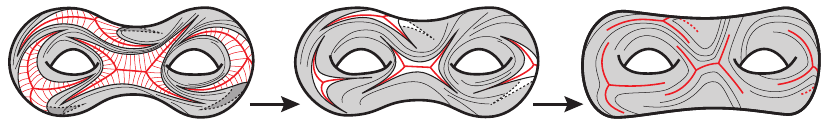}};
    \node at (-4.85, -1.25) {$(X, \lambda)$};
    \node at (4.85, -1.25) {$q(\Ol(X), \lambda)$};
\end{tikzpicture}
    \caption{Inflating a lamination and deflating its complementary components.}
    \label{fig:deflate}
\end{figure}

\begin{remark}
This heuristic description of $\Defl$ can be made precise by grafting $X$ along $\lambda$ (see, e.g., \cite{DumasCP1}) and then collapsing the hyperbolic pieces along the leaves of $\Ol(X)$.
In particular, $\Defl$ is {\em not} the grafting map.
\end{remark}

\begin{proposition}\label{prop:deflation}
Given a marked hyperbolic structure 
\footnote{Throughout the paper we suppress markings in our notation, but reintroduce them here to state the proposition precisely.}
$[f:S\to X]\in \T(S)$ and $\lambda\in \ML(S)$, let $[g: S\to Z] \in \T(S)$ be the marked complex structure on which $q(\Ol(X), \lambda)$ is holomorphic.
There is a map
\[\Defl : X \to Z\] that is a homotopy equivalence restricting to an isometry between $\Sp^0$ with its  metric induced by integrating the edges against $\Ol(X)$ and the graph of horizontal saddle connections of $q(\Ol(X), \lambda)$ with the induced path metric. Moreover, $\Defl \circ f\sim g$ and ${\Defl}_*\Ol(X) = \Re(q)$ and ${\Defl}_*\lambda = \Im(q)$ as measured foliations.
\end{proposition}

\begin{proof}
Construction \ref{const:geometric_tt} supplies us with a geometric train track $\pi: \mathcal N_\epsilon (\lambda) \to \tau$.
On the preimage $\pi\inverse(b)$ of each closed branch $b$ of $\tau$ we integrate the two measures $\Ol(X)|_{\mathcal N_\epsilon(\lambda)}$ and $\lambda$ giving $\pi\inverse(b)$ the structure of a  bi-foliated Euclidean rectangle of length $\ell_X(b)$ and height $\lambda(b)$ 
These rectangles glue along their `short' sides $\{\pi\inverse(s): \text{$s$ is a switch of $\tau$}\}$ to give  $\mathcal N_\epsilon(\lambda)$ the structure of  a bi-foliated Euclidean band complex.

The map $\pi$ extends to a self--homotopy equivalence of $X$ homotopic to the identity preserving the orthogeodesic foliation leafwise.
This means that the boundary of $\mathcal N_\epsilon(\lambda)$ admits a natural retraction onto $\Sp$ by collapsing the leaves of the orthogeodesic foliation in the complement of $\mathcal N_\epsilon(\lambda)$, and we take the quotient generated by this equivalence relation to obtain a new surface $Y$ with its complex structure described below.

On each rectangle $\pi\inverse (b)$, the bi-foliated Euclidean structure gives local coordinates to $\CC$ away from the singular points of $\Ol(X)$ locally mapping $\Ol(X)$ to $|dx|$ and $\lambda$ to $|dy|$, thought of as measured foliations on the plane. 
These coordinate patches glue together along the spine to give local coordinates away from the points of valence $\ge 3$. Moreover, these charts preserve $|dx|$ and $|dy|$, so the transitions must be of the form $z \mapsto \pm z+ \alpha$ for some $\alpha\in \CC$. 
We have therefore built a Riemann surface $Z$ equipped with a half-translation structure away from the vertices of $\Sp$, which become cone points of cone angle equal to $\pi\cdot \val(v)$.
Edges of $\Sp$ join vertices along horizontal trajectories representing all horizontal saddle connections on $q$; their lengths in the singular flat metric are given by the integral over $\Ol(X)$.  Thus $\Defl$ induces an isometry of metric graphs, as claimed.
\end{proof}

This explicit description of the quadratic differential associated to the pair $(X,\lambda)$ by the map $\cO$ from the introduction will be useful in order to prove in Theorem \ref{thm:diagram_commutes} that Diagram \eqref{diagram} commutes.

\section{Cellulating crowned Teichm{\"u}ller spaces}\label{sec:arc_cx}

We now define a certain arc complex which combinatorializes the structure the orthogeodesic foliation on complementary subsurfaces. The main result of this section is Theorem \ref{thm:arc=T(S)_crown}, which shows that this arc complex is equivariantly homeomorphic to the Teichm{\"u}ller space of the complementary surface.
In particular, this shows that the restriction of the orthogeodesic foliation to each component of $S \setminus \lambda$ completely determines the hyperbolic structure on that piece.

Before stating the theorem, we must first set up our combinatorial analogue for Teichm{\"u}ller space. This appears as Definition \ref{def:arccx} after a series of auxiliary constructions.

Suppose that $\cutsurf = \cutsurf_{g,b}^{\crowns}$ is a finite-area hyperbolic surface with boundary and without annular cusps. A properly embedded arc $I \rightarrow \Sigma$ is {\em essential} if $I$ cannot be isotoped (through properly embedded arcs) into $\partial \Sigma$ or into a spike. 
The {\em arc complex $\Arc(\cutsurf, \partial \cutsurf)$ of $\cutsurf$ rel boundary}
\label{ind:arccx}
is the (simplicial, flag) complex whose vertices are isotopy classes of simple essential arcs of $\cutsurf$.
Vertices span a simplex in $\Arc(\cutsurf, \partial \cutsurf)$ if and only if there exist a collection of pairwise disjoint representatives for each isotopy class.
The {\em filling arc complex} $\Arcfill(\cutsurf, \partial \cutsurf)$
\label{ind:fillarccx} is the subset of $\Arc(\cutsurf, \partial \cutsurf)$ consisting only of those arc systems which cut $\cutsurf$ into a union of topological disks.

The geometric realization $|\Arc(\cutsurf, \partial \cutsurf)|$ of $\Arc(\cutsurf, \partial \cutsurf)$ is obtained by declaring every simplex to be a regular Euclidean simplex of the proper dimension; note that the topology of $|\Arc(\cutsurf, \partial \cutsurf)|$ obtained from the metric structure is in general different from the standard simplicial topology (see, e.g., \cite{BowdEpst}).
The geometric realization $|\Arcfill(\cutsurf, \partial \cutsurf)|$
\label{ind:geofillarccx}
is then the subspace of filling arc systems equipped with the subspace topology induced by the metric structure.

\begin{definition}\label{def:arccx}
The {\em weighted filling arc complex $|\Arcfill(\cutsurf, \partial \cutsurf)|_{\RR}$ of $\cutsurf$ rel boundary} is the set of all weighted multi-arcs of the form 
\[\arcwt = \sum c_i \alpha_i\]
where $\arc=\bigcup \alpha_i \in \Arcfill(\cutsurf, \partial \cutsurf)$ and $c_i >0$ for all $i$.
\end{definition}

Throughout, we will use $\alpha$ to denote a single arc, and $\arc$ to denote an (unweighted) multi-arc. The symbol $\arcwt$ will be reserved to denote a weighted multi-arc.

\begin{note}
If $\cutsurf$ is an ideal hyperbolic polygon, then the empty arc system fills $\cutsurf$ and we consider it as an element of $|\Arcfill(\cutsurf, \partial \cutsurf)|_{\RR}$. If $\cutsurf$ is not a polygon, then the empty arc system never fills. 
\end{note}

So long as $\cutsurf$ is not an ideal polygon, $|\Arcfill(\cutsurf, \partial \cutsurf)|_{\RR}$ is just $|\Arcfill(\cutsurf, \partial \cutsurf)| \times \RR_{>0}$. When $\cutsurf$ is an ideal polygon, then $|\Arcfill(\cutsurf, \partial \cutsurf)|_{\RR}$ is homeomorphic to the open cone on the filling arc complex:
\[\big( |\Arcfill(\cutsurf, \partial \cutsurf)| \times \RR_{\ge 0} \big)
/ \big( |\Arcfill(\cutsurf, \partial \cutsurf)| \times \{0\}\big).\]
See Figure \ref{fig:arccx_eg} for an example of $|\Arcfill(\cutsurf, \partial \cutsurf)|_{\RR}$ in the case when $\cutsurf$ is an ideal pentagon.

\begin{remark}\label{rmk:arc_graph_duality}
The standard duality between arc systems and ribbon graphs (see, e.g., \cite{Mond_handbook}) assigns to every $\arcwt \in |\Arcfill(\cutsurf, \partial \cutsurf)|_{\RR}$ a metric ribbon graph spine for $\cutsurf$ (with some infinitely long edges if $\cutsurf$ has crowns).
One could of course translate the cell structure of $|\Arcfill(\cutsurf, \partial \cutsurf)|_{\RR}$ into a cellulation of an appropriate space of marked metric ribbon graphs.

While the arc complex definition is more practical for our definition of shear-shape cocycles, the dual ribbon graph picture allows us to immediately understand how to record the geometry of the horizontal trajectories of a quadratic differential (see Section \ref{sec:flat_map}).
\end{remark}

\begin{figure}[ht]
\centering
\begin{minipage}{0.42\textwidth}
  \centering
\includegraphics[height=6cm]{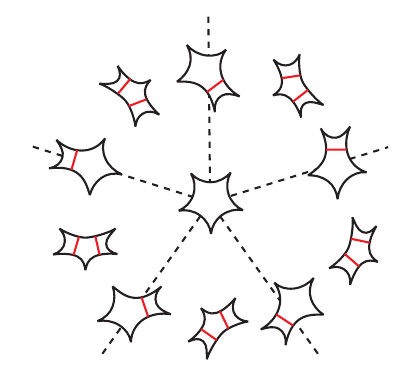}
\subcaption{The weighted arc complex of an ideal\\
pentagon rel its boundary.}\label{fig:arccx_eg}
\end{minipage}
\hspace{.5 cm}
\begin{minipage}{0.53\textwidth}
  \centering
\begin{tikzpicture}
    \draw (0, 0) node[inner sep=0] {\includegraphics[height=6cm]{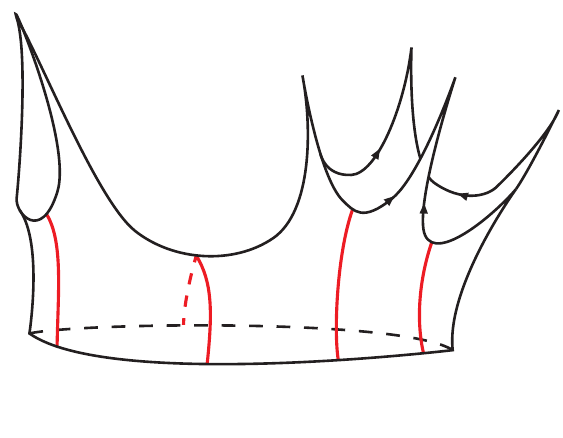}};
    \node at (-2.5,-1) [red]{$c_1\alpha_1$};
    \node at (-1.1,-.1) [red]{$c_2\alpha_2$};
    \node at (.25,-1) [red]{$c_3\alpha_3$};
    \node at (1.4,-1) [red]{$c_4\alpha_4$};
    \node at (-3.8,-1.5) {$\beta$};
    \node at (2.5,1) {$\Crown$};
    \node at (-.5, -2.3){$\ell_{\arcwt}(\beta) = c_1 + 2c_2 + c_3+c_4$};
    \node at (3.5, -.5) {$\res_{\arcwt}(\Crown)$};
    \node at (3.7,-.9) {$= c_4 -c_3$};
\end{tikzpicture}
\subcaption{The combinatorial length and residue associated to a weighted filling arc system $\arcwt$.}\label{fig:arcwt_lengthres}
\end{minipage}
\caption{Arc complexes and combinatorial geometry.} \label{fig:arccx}
\end{figure}

\para{Combinatorial geometry}
Now that we have defined our combinatorial analogue of Teichm{\"u}ller space, we can also define combinatorial notions of both length and metric residue.

Suppose that $\beta$ is a compact boundary component of $\cutsurf$ and $\arcwt \in |\Arcfill(\cutsurf, \partial \cutsurf)|_{\RR}$; then we define the $\arcwt$--length $\ell_{\arcwt}(\beta)$ of $\beta$ to be the sum of the weights of the arcs of $\arcwt$ incident to $\beta$ (counted with multiplicity, so that if both endpoints of $\arc$ lie on $\beta$ then its weight is counted twice).

Similarly, let $\Crown$ be an oriented crowned boundary component with an even number of spikes.
Then the edges of $\Crown$ are partitioned into those that that have the surface lying on their left and those which have the surface on their right; call these edges positively and negatively oriented, respectively.
The $\arcwt$--residue $\res_{\arcwt}(\Crown)$
\label{ind:combres}
of $\Crown$ is then defined to be the sum of the weights of the arcs incident to each positively oriented edge of $\Crown$ minus the sum of the weights of the arcs incident to the negatively oriented edges (where both sums are again taken with multiplicity). See Figure \ref{fig:arcwt_lengthres} for an example calculation.

We have now come to the most important object of this section, and a foundational result of this paper that allows us to pass between hyperbolic metrics,  orthogeodesic foliations, and metric graphs embedded in flat structures.

\begin{construction}
Let $Y$ be a crowned hyperbolic surface. As discussed in Section \ref{subsec:ortho_spine}, the orthogeodesic foliation determines a spine for $Y$ together with a dual (filling) arc system $\arc(Y)$. Weighting each dual arc by integrating the measure induced by $\cO_{\partial Y}(Y)$ over the corresponding edge of $\Sp$ (compare \eqref{eqn:hyp_arc_def}) therefore defines a map 
\begin{align*}
    \arcwt: \T(\cutsurf) \to |\Arcfill(\cutsurf, \partial \cutsurf)|_{\RR}.
\end{align*}
\label{ind:arcwt}
\end{construction}

When $\cutsurf$ has compact boundary, \cite[Theorem 1.2 and Corollary 1.4]{Luo} states that $\arcwt(\cdot)$ is a $\Mod(\cutsurf)$-equivariant stratified real-analytic homeomorphism; see also \cite{Mondello, Do, Ushijima}.
Our aim is to generalize Luo's theorem to surfaces with crowned boundary.
While the arguments of \cite{Luo} can probably be adapted to this setting, we prefer to use some elementary hyperbolic geometry to realize $|\Arcfill(\cutsurf, \partial \cutsurf)|_{\RR}$ as a subcomplex sitting ``at infinity'' of the weighted filling arc complex of a surface with compact boundary.

\begin{theorem}\label{thm:arc=T(S)_crown}
Let $\cutsurf$ be a crowned hyperbolic surface. Then the map
\[\arcwt: \T(\cutsurf) \to |\Arcfill(\cutsurf, \partial \cutsurf)|_{\RR}\]
is a $\Mod(\cutsurf)$--equivariant stratified real analytic homeomorphism.
Moreover, let $\beta_1,\ldots, \beta_b$ denote the closed boundary components of $\cutsurf$ and  $\Crown_1, \ldots, \Crown_e$ the crown ends which have an even number of spikes. Fix an orientation of each $\Crown_j$. Then the map above identifies the level sets
\[ \left\{ (Y, f) \in \T(\cutsurf) \,|\,
\ell(\beta_i) = L_i \, , \, 
\res(\Crown_j) = R_j \right\} 
\cong 
\left\{ \arcwt \in |\Arcfill(\cutsurf, \partial \cutsurf)|_{\RR} : 
\ell_{\arcwt}(\beta_i) = L_i \, , \, 
\res_{\arcwt}(\Crown_j) = R_j \right\} \]
for any $(L_i) \in \RR_{>0}^{b}$ and any $(R_j) \in \RR^{e}$.
\end{theorem}

The remainder of this section is devoted to deducing Theorem \ref{thm:arc=T(S)_crown} from \cite[Theorem 1.2, Corollary 1.4]{Luo} and \cite[Section 2.4]{Mondello}. Our plan is to appeal to the aforementioned references to prove that for a given maximal arc system $\arc$, the map $\arcwt(\cdot)$  extends to a real analytic map $\arcwt_{\arc}:\T(\cutsurf) \to \RR^{\arc}$ that agrees with $\arcwt(\cdot)$ on the locus of hyperbolic surfaces whose spine has dual arc system contained in $\arc$ (Lemma \ref{lem:arcwt_analytic}).  We show that $\arcwt(\cdot)$ is a homeomorphism by building a continuous right inverse $Y: |\Arcfill(\cutsurf, \partial \cutsurf)|_{\RR}\to \T(\cutsurf)$; $Y(\arcwt)$ is obtained as a geometric limit metrics on a larger compact surface with boundary as some arcs are pinched to spikes.

Endowing $\cutsurf$ with an auxiliary hyperbolic metric, we take $\cutsurf^\circ$ to be the surface with geodesic and horocyclic boundary components obtained by truncating the tips of the spikes.
Let $\arcc$ be the union of horocyclic boundary components of $\cutsurf^\circ$ and double $\cutsurf^\circ$ along $\arcc$ to obtain a (topological) surface $D\cutsurf$ and an identification of  $\cutsurf^\circ$ with a subsurface of $D\cutsurf$ taking $\partial \cutsurf^\circ \setminus \arcc$ into $\partial D\cutsurf$; see Figure \ref{fig:double_surface}.

\begin{figure}[ht]
    \centering
\begin{tikzpicture}
    \draw (0, 0) node[inner sep=0] {\includegraphics{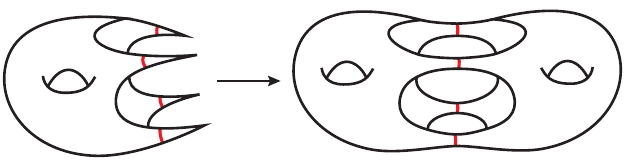}};
    \node at (-5.4, -.7){\large $\Sigma$};
    \node at (5.4, -.7){\large $D\Sigma$};
    \node at (-2.5,.1) [red]{$\arcc$};
    \node at (2.5, 0) [red]{$\arcc$};
\end{tikzpicture}
    \caption{The truncation of a crowned surface $\Sigma$ along $\arcc$ and its double $D\Sigma$.}
    \label{fig:double_surface}
\end{figure}

Let $\arcwt = \sum c_i \alpha_i$ be a weighted filling arc system on $\cutsurf$ and let $\arcb$ be the mirror image of $\arc$ in $D\cutsurf$, so that $\arc\cup \arcc \cup \arcb$ is a filling arc system on $D\cutsurf$.
For each $t>0$, define 
\[\arcwtb_t = \sum c_i \beta_i + t \sum \gamma_i + \sum c_i \alpha_i  \in|\Arcfill(D\cutsurf, \partial D\cutsurf)|_{\RR}.\]
Since $D\cutsurf$ is compact, we can apply \cite[Corollary 1.4]{Luo} which states that there is a unique hyperbolic structure $X_t\in \T(D\cutsurf)$ whose natural weighted arc system coincides with $\arcwtb_t$.  
\begin{remark}
It will be convenient to assume that $\arc$ is maximal, formally adding arcs of weight $0$ to $\arcwt$ (and $\arcwtb_t$) as necessary.
\end{remark}

Our goal is now to show that that $(X_t)$ converges as $t\to \infty$ to a surface $Y\in \T(\cutsurf)$ such that $\arcwt(Y) =\arcwt$.
The convergence is geometric: we take basepoints $x_t\in X_t$ lying outside of the ``thin parts'' of the subsurface corresponding to $\Sigma^\circ$ and extract a  geometric limit of $(X_t, x_t)$ as $t\to \infty$.  
The limit metric $Y$ has spikes corresponding to $\underline \gamma$ and so defines a point in $\T(\cutsurf)$. Moreover, $Y$ inherits a filling arc system naturally identified with $\arc$, which is necessarily realized orthogeodesically.

We begin with an estimate on the lengths of orthogeodesic arcs.

\begin{lemma}\label{lem:arclength_compact}
If $X$ is a hyperbolic metric on a compact surface with totally geodesic boundary and $\arcwt(X)= \sum c_i\alpha_i$, then
\[\min\left\{\log3, 2\tanh\inverse\left(\frac{\tanh(\log\sqrt 3)}{\cosh(c_i/2)}\right)\right\}\le\ell_X(\alpha_i)\le \frac{2\pi}{c_i},\] for each $i$.
\end{lemma}

\begin{proof}
Any leaf of the orthogeodesic foliation properly homotopic to $\alpha_i$ has hyperbolic length at least $\ell_X(\alpha_i)$. Thus the embedded ``collar'' about $\alpha_i$ consisting of all leaves of the orthogeodesic foliation in the same homotopy class of $\alpha_i$ has area at least $c_i\ell_X(\alpha_i)$ (see Figure \ref{fig:collar_bounds}). 
On the other hand, the Gauss--Bonnet theorem bounds area of the collar above by $2\pi$, so we get a bound
\[\ell_{X}(\alpha_i) \le \frac{2\pi}{c_i}.\]

Now we would like to find a lower bound for $\ell_X(\alpha_i)$ in terms of $c_i$; for notational convenience we fix $i$ and set $\alpha = \alpha_i$ and $c= c_i$. 
Assume that $\ell_X (\alpha) <  \log 3$.
Let $H$ be a component of $X\setminus \arc$ meeting $\alpha$; then there is a unique point $u\in H$ equidistant from all boundary components of $X$ meeting $H$. 
There is also a universal lower bound to the distance from $u$ to any such boundary component, given by $\log\sqrt3$, the radius of the circle inscribed in an ideal triangle. Thus the leaf of $\cO_{\partial X}(X)$ through $u$ has length at least $\log(3)$.

Since $\ell_X(\alpha) < \log3$, we  know that there is a  leaf of the orthogeodesic foliation parallel to $\alpha$ with length $\log 3$.  
Using a formula relating the lengths of the sides of a hyperbolic tri-rectangle \cite[Theorem 2.3.1]{Buser}, the distance $c_0$ from $\alpha$ and this leaf is given by
\begin{equation}\label{eqn:tri-rectangle}
    \tanh(\ell_X(\alpha)/2) =\frac{ \tanh(\log \sqrt3)}{\cosh{(c_0)}}.
\end{equation} 
Now this expression is decreasing in $c_0$, and $x\mapsto \tanh\inverse(x)$ is increasing. We have that $c>2c_0$ by definition (see Figure \ref{fig:collar_bounds}), so the lemma follows.
\end{proof}
 
\begin{figure}[ht]
    \centering
    \begin{tikzpicture}
    \draw (0, 0) node[inner sep=0]{\includegraphics{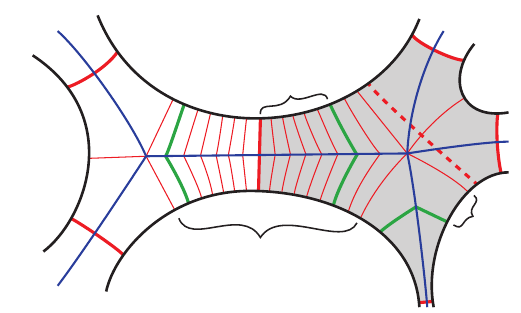}};
    \node at (0, -.8)[red]{$\alpha$};
    \node at (0,-1.5){$c$};
    \node at (.5,1.2) {$c_0$};
    \node at (1.7, 1.4)[red]{$\alpha_j$};
    \node at (4.2, -1.1){$\le \log 2$};
    \node at (3.8, 1.1){$H_u$};
\end{tikzpicture}
\caption{A foliated collar of width $c$ about an orthogeodesic arc $\alpha$.  If the arc is shorter than $\log 3$, then there are (bold green) leaves of this collar of length equal to $\log 3$. For a very short arc $\gamma_i$, the distance between the longest leaf of its collar and the leaf of length $\log 3$ is at most $\log 2$.  The dashed arc $\alpha_j$ has weight $0$ and corresponds to one of two possible choices of maximal completion of $\arc$.}
    \label{fig:collar_bounds}
\end{figure}

For any arc $\gamma_i$ of $\arcc$, some elementary estimates similar to those given in the proof of Lemma \ref{lem:arclength_compact} (compare Equation \eqref{eqn:tri-rectangle}) give $\ell_t(\gamma_i) = O(e^{-t/2})$.  
If $\alpha_i$ appears in $\arcwtb_t$ with coefficient $c_i = 0$, then Lemma \ref{lem:arclength_compact} provides a lower bound of $\log 3$ for the length $\ell_{t}(\alpha_i)$ of $\alpha_i$ on $X_t$. We also have the following upper bound:

\begin{lemma}\label{lem:zero_arc_length}
If $c_j = 0$ for some $j$, then for $t$ large enough, we have  
\[ \log 3\le \ell_t(\alpha_j) \le 2\sum c_i + 8\pi\sum\frac{1}{c_i} + |\arcc|\log 144 .  \]
\end{lemma}
\begin{proof}
We remove all arcs of $\arc\cup \arcc\cup \arcb$ with positive weight from $X_t$ and let $H_t$ be (the metric completion of) the right-angled polygon component that contains $\alpha_j$.  Our strategy is to find a path of controlled length contained in $\partial H_t$ joining the endpoints of $\alpha_j$.

Notice that $\partial H_t$ alternates between segments of $\partial X_t$ and arcs of $\arc\cup \arcc\cup \arcb$ with positive weight.  From Lemma \ref{lem:arclength_compact}, the total length of segments coming from arcs of $\arc\cup \arcb$ is at most $8\pi\sum1/c_i$, because each arc of $\arc\cup \arcb$ can appear at most two times on $\partial H_t$.  Similarly, from the construction of our coordinate system, the total length of the segments coming from $\partial X_t$ that correspond to collars of arcs in $\arc\cup \arcb$ is at most $2\sum c_i$.

Suppose some arc $\gamma_i$ of $\arcc$ forms a segment of $\partial H_t$.  The distance between the leaf of the orthogeodesic foliation parallel to $\gamma_i$ with length $\log(3)$ and the singular, longest leaf parallel to $\gamma_i$ has distance uniformly bounded above by $\log 2$ for large  values of $t$  (see Figure \ref{fig:collar_bounds}).
Truncate $H_t$ by removing the leaves of the orthogeodesic foliation parallel to $\gamma_i$ with length at most $\log 3$ to obtain a new (non-convex) geodesic polygon $H_t^\circ$.
An application of the Collar Lemma \cite[Theorem 4.1.1]{Buser} to the double $DX_t$ along its boundary shows that $\alpha_j$ does not enter the region of $H_t$ that we removed. 

Each arc $\gamma_i$ of $\arcc$ contributed at most $2t+O(e^{-t/2})$ to the length of $\partial H_t$.  However, after truncating, each  $\gamma_i$ contributes at most $2(\log2 + \log 3 +\log 2) = \log 144$ to the length of $\partial H_t^\circ$.  Putting together all of our estimates completes the proof.
\end{proof}

For each $\alpha_i\in \arc$ with positive coefficient $c_i$ in $\arcwtb_t$, the orthogeodesic length $\ell_t (\alpha_i)$ of $\alpha_i$ on $X_t$ is bounded above and below by the positive real numbers independent of $t$ provided by Lemma \ref{lem:arclength_compact}.  If $c_i = 0$ for some $i$, then Lemma \ref{lem:zero_arc_length} provides bounds on $\ell_t(\alpha_i)$ independent of $t$.    
Therefore, there exists a subsequence $t_k$ tending to infinity such that $(\ell_{t_k}(\alpha_i))$ converges to a positive number $\ell_i$ for each $i$, while $\ell_t(\gamma_i)= O(e^{-t/2})$ for each $\gamma_i\in \arcc$.

The metric completion of  $X_{t_k}\setminus (\arc\cup\arcc\cup\arcb)$  is a collection of  hyperbolic right-angled hexagons, each with three non-adjacent sides that correspond to arcs of $\arc\cup\arcc\cup\arcb$.  
The lengths of these sides determine uniquely an isometry class of right-angled hexagons, which we have just proved converge to (degenerate) right-angled hexagons in which the edges corresponding to arcs of $\gamma$ become spikes in the limit. 
The (degenerate) right-angled hexagons glue along $\arc$ to form a complete hyperbolic surface $Y$ homeomorphic to $\cutsurf$ with a maximal filling arc system labeled by $\arc$ and realized orthogeodesically on $Y$. That is, we have constructed a surface $Y(\arcwt)=Y \in \T(\cutsurf)$.

\begin{lemma}\label{lem:inverse_arcwt}
We have an equality $\arcwt(Y(\arcwt)) = \arcwt$.
\end{lemma}
\begin{proof}
By construction, the length of the projection of every edge of the spine of $X_t$ dual to an arc of $\arc$ was constant along the sequence $(X_{t_k})$ converging geometrically to $Y(\arcwt)$. The lemma follows.
\end{proof}

In order to show that the inverse $Y(\cdot)$ is well-defined, we will need the following statement, which refines the relationship between the coefficients of $\arcwtb_t$ and the lengths of its arcs.

Let  $\underline \delta =\arc\cup \arcc\cup \arcb$ denote the support of $\arcwtb_t$.  According to \cite[Theorem 1.2]{Luo}, the lengths of the closest-point projections of the edges of the spine to dual the arcs of $\underline \delta$ (i.e., the coefficients of the weighted arc system) extend to an analytic local diffeomorphism $\arcwtb_{\underline \delta}:\T(D\cutsurf)\to \RR^{\underline \delta}$ whose image is a convex cone with finitely many sides.\footnote{The ``projection length'' associated to each arc of $\underline \delta$ (called the ``radius coordinate'' in \cite{Luo} and the ``width'' in \cite{Mondello}) is positive when that arc is dual to an edge of the spine of a surface $X\in \T(D\cutsurf)$.} Now we show that analyticity extends to infinity.  

\begin{lemma}\label{lem:arcwt_analytic}
For each maximal filling arc system $\arc$ defining a cell of full dimension in $\Arcfill(\cutsurf, \partial \cutsurf)$, there is an analytic map \[\arcwt_{\arc} : \T(\cutsurf) \to \RR^{\arc}\]
such that if the spine of $Y\in \T(\cutsurf)$ has dual arc system contained in $\arc$, then $\arcwt_{\arc}(Y) = \arcwt(Y)$.
\end{lemma}
\begin{proof}
The orthogeodesic length functions associated to our maximal arc system $\underline \delta = \arc\cup\arcc\cup \arcb$ on $D\cutsurf$ form an analytic parameterization of $\T(D\cutsurf)$, which we denote by $\ell_{\underline \delta}: \T(D\cutsurf) \to \RR_{>0}^{\underline \delta}$.
We have a commutative diagram of analytic embeddings
\begin{equation}\label{eqn:arc_radius}
\begin{tikzcd}
\RR_{>0}^{\underline \delta} \arrow[rr, "\arcwtb_{\underline \delta}\circ\ell_{\underline \delta}\inverse"] 
&  & \RR^{\underline \delta}\\
& \T(D\cutsurf) \arrow[ul, "\ell_{\underline \delta}"']   \arrow[ur, "\arcwtb_{\underline \delta}"].
\end{tikzcd}
\end{equation}
An explicit formula for $\arcwtb_{\delta}\circ \ell\inverse$ can be recovered from \cite[Section 2.4]{Mondello}, which produces an analytic mapping $G: \RR_{>0}^{\arc\cup\arcb}\to \RR^{\arc\cup\arcb}$ that describes how $\arcwtb_{\underline \delta}$ behaves when the arcs corresponding to $\arcc$ have length close to $0$.  More precisely, let $\pi_{\arc\cup \arcb}: \RR^{\underline \delta} \to \RR^{\arc\cup\arcb}$ be the coordinate projection.
Then for $x_{\underline \delta} = (x_{\arc}, x_{\arcc}, x_{\arcb}) \in \RR_{>0}^{\arc\cup\arcb}\times \RR_{\ge 0}^{\arcc}$, we have
\begin{equation}\label{eqn:small_t}
\pi_{\arc\cup\arcb}\circ B_{\underline \delta} \circ \ell_{\underline \delta}^{-1}(x_{\underline \delta})
= G(x_{\arc},x_{\arcb}) + E
\end{equation}
uniformly on compact subsets of $\RR_{>0}^{\arc\cup\arcb}\times \RR_{\ge 0}^{\arcc}$, where $E$ is a vector whose entries are all of order $O \left (\max_{\gamma\in \arcc}\{x_{\gamma}^2\}\right)$.

Restricting to the locus of symmetric surfaces $\{X\in \T(D\cutsurf) : \ell_{\alpha_i}(X) = \ell_{\beta_i}(X), ~ \forall i\}$, the map $G$ therefore induces an analytic map $F: \RR_{>0}^{\arc} \to \RR^{\arc}$.
Again, we have an analytic parameterization $\ell_{\arc}:\T(\cutsurf)\to \RR_{>0}^{\arc}$ by length functions and a diagram
\begin{equation}\label{eqn:arc_radius_cut}
\begin{tikzcd}
\RR_{>0}^{\arc} \arrow[rr, "F"] 
&  & \RR^{\arc}\\
& \T(\cutsurf) \arrow[ul, "\ell_{\arc}"']   \arrow[ur, "F\circ \ell_{\arc}"].
\end{tikzcd}
\end{equation}
So take $\arcwt_{\arc} = F\circ \ell_{\arc}$; it follows from the definitions that if the dual arc system to the spine of a surface $Y\in \T(\cutsurf)$ is contained in $\arc$, then $\arcwt_{\arc}(Y) = \arcwt(Y)$. This completes the proof of the lemma.
\end{proof}

A priori, $Y(\arcwt)$  depends on the subsequence $X_{t_k}$ converging geometrically to $Y(\arcwt)$.  However,

\begin{lemma}\label{lem:Y_defined_cont}
The limit $Y(\arcwt)$ does not depend on choice of subsequence $X_{t_k}$, i.e., $X_t\to Y$.  Moreover, $Y: |\Arcfill(\cutsurf, \partial \cutsurf)|_{\RR} \to \T(\cutsurf)$ is continuous.
\end{lemma}

\begin{proof}
Throughout this proof, we let $\pi := \pi_{\arc\cup\arcb}$ be the coordinate projection from the proof of Lemma \ref{lem:arcwt_analytic}.

Let $s>0$ and  $X_{t,s}\in \T(D\cutsurf)$ be the surface obtained from $X_t$ by keeping all lengths of arcs of $\arc\cup \arcb$ fixed and taking $\ell_{\gamma_i} (X_{t,s}) := \ell_{\gamma_i}(X_{t+s})$,  for each $\gamma_i\in \arcc$.  Note that $\ell_{\gamma_i}(X_{t+s})=O(e^{-(s+t)/2})$.  
By construction of $X_{t,s}$, the lengths of arcs of $\arc\cup \arcb$ agree with those of $X_t$, so  \eqref{eqn:small_t} gives 
\[ \pi (\arcwtb_{\underline \delta}(X_t)) - \pi(\arcwtb_{\underline \delta}(X_{t,s})) = O(e^{-(s+t)}).\]
Recall that $\pi(\arcwtb_{\arc}(X_t))= \pi(\arcwtb_t)$ is constant for all $t>0$, so that  
\[\pi (\arcwtb_{\underline \delta}(X_{s+t})) - \pi(\arcwtb_{\underline \delta}(X_{t,s})) = O(e^{-(s+t)}),\]
as well.  
Since $\arcwtb_{\underline \delta}$ is open,  analytic, and $\{\pi(\ell_{\underline \delta}(X_t)): t>0\}\subset\RR_{>0}^{\arc\cup \arcb}$ lies in a compact set (Lemmas \ref{lem:arclength_compact} and \ref{lem:zero_arc_length}), we can adjust the lengths of arcs $\alpha_i$ and $\beta_i$ of $\arc\cup \arcb$ in $X_{t,s}$ by $O(e^{-(s+t)})$ to obtain $X_{s+t}$.  Thus, for any $t_k\to \infty$, the lengths $(\ell_{t_k}(\arc\cup \arcb))$ form a Cauchy sequence, hence converge.  Thus any two subsequential geometric limits (with basepoints away from the spikes of  the subsurface associated with $\cutsurf^\circ$) coincide, which proves that $Y(\arcwt)$ is well defined.

To see that $Y(\cdot)$ is continuous, let $\arcwt_k\to \arcwt$; by passing to a subsequence, we may assume that $\arcwt_k$ are in the closure of the cell associated to a maximal filling arc system $\arc$.  Let $\overline \arcwt_k$ and $\overline \arcwt$ be the mirror images (with corresponding weights) of $\arcwt_k$ and $\arcwt$ in $D\cutsurf$, respectively.
We build two families of approximating surfaces $X_k, X_k^k \in \T(D\cutsurf)$ corresponding to the weighted arc systems
\[\overline \arcwt + k \sum\gamma_i + \arcwt
\text{ and }
\overline \arcwt_k + k \sum \gamma_i +\arcwt_k \]
on $D\cutsurf$, respectively. By \cite[Theorem 1.2]{Luo} (alternately, the proof of Lemma \ref{lem:arcwt_analytic}), each $X_k^k$ is close to $X_k$ in $\T(D\cutsurf)$, hence $X_k^k$ and $X_k$ have the same geometric limit $Y(\arcwt)\in \T(\cutsurf)$, which is what we wanted to show.
\end{proof}

We now have all of the pieces in place to complete the proof of Theorem \ref{thm:arc=T(S)_crown}.

\begin{proof}[Proof of Theorem \ref{thm:arc=T(S)_crown}]
By Lemma \ref{lem:Y_defined_cont}, $Y(\cdot)$ is well defined and continuous, and by Lemma \ref{lem:inverse_arcwt},  $Y(\cdot)$ is a right inverse to $\arcwt(\cdot)$; in particular, $Y(\cdot)$ is injective.  
For a given maximal arc system $\arc$, the open orthant $U_{\arc}= \RR_{>0}^{\arc}\subset\RR^{\arc}$ is identified with the interior of a top dimensional cell of $|\Arcfill(\cutsurf,\partial \cutsurf)|_{\RR}$. Some of the hyperplanes in $\partial \RR_{\ge0}^{\arc}$ are identified with the interior of cells associated with non-maximal filling arc systems contained in $\arc$; let $\overline{U_{\arc}}$ denote the closure of $U_{\arc}$ in $|\Arcfill(\cutsurf,\partial \cutsurf)|_{\RR}$.

Then $Y(\cdot)$ defines a  continuous bijection $\overline{U_{\arc}}\to \overline{Y(U_{\arc})}$, and this identification is homeomorphic, because $\arcwt_{\arc}$ supplies an analytic inverse on $\overline{Y(U_{\arc})}$, by Lemma \ref{lem:arcwt_analytic}.
Since these homeomorphisms glue along the combinatorics of $|\Arcfill(\cutsurf,\partial \cutsurf)|_{\RR}$, the map $Y(\cdot)$ is the desired global homeomorphic inverse to $\arcwt(\cdot)$.

Again, by Lemma \ref{lem:arcwt_analytic}, $A(\cdot)$ is analytic restricted to the relative interior of the image under $Y(\cdot)$ of each cell of $|\Arcfill(\cutsurf, \partial \cutsurf)|_{\RR}$, demonstrating the stratified real analytic structure.   
That level sets of the residue functions are mapped to one another is an exercise in unpacking the definitions.
\end{proof}

\section{Transverse and shear-shape cocycles}\label{sec:shsh_def}

We now define the main protagonists of this paper, the {\em shear-shape cocycles} on a measured lamination.
In Section \ref{subsec:shsh_cohom}, we give a first definition of shear-shape cocycles in terms of the cohomology of an augmented neighborhood of $\lambda$, twisted by its local orientation (Definition \ref{def:shsh_cohom}).
While this definition has technical merit (and exactly parallels the construction of period coordinates for quadratic differentials, a fact which we exploit in Section \ref{sec:flat_map}), it is impractical to use.
We rectify this deficiency in Section \ref{subsec:shsh_axiom} by giving a second formulation which parallels Bonahon's axiomatic approach to transverse cocycles (compare Definitions \ref{def:trans_axiom} and \ref{def:shsh_axiom}). The main result of this section, Proposition \ref{prop:shsh_defsagree}, proves that these two definitions agree.

The reader may find it helpful to consult Sections \ref{sec:flat_map} or \ref{sec:hyp_map} while digesting these definitions so as to have a concrete model of shear-shape cocycles in mind.

\subsection{Transverse cocycles}\label{subsec:trans_cocy}
As shear-shape cocycles generalize Bonahon's transverse cocycles, we begin by recalling two equivalent definitions of transverse cocycles for geodesic laminations which we generalize in Sections \ref{subsec:shsh_cohom} and \ref{subsec:shsh_axiom} below.

\begin{remark}
We have chosen to present transverse cocycles in a way that anticipates our construction of shear-shape cocycles. The reader is advised that our treatment is ahistorical, and in particular omits the fascinating (and quite subtle) relationship between transverse cocycles and transverse H{\"o}lder distributions. For more on this correspondence, see \cite{Bon_GLTHB}, \cite{Bon_THDGL}, and \cite{Bon_SPB}.
\end{remark}

The first definition we consider is cohomological. 
Let $\lambda$ be a measured lamination on $S$; then an orientation of $\lambda$ is a continuous choice of orientation of the leaves of $\lambda$.
If $N$ is any snug neighborhood of $\lambda$, then one may take a corresponding (snug) neighborhood $\widehat N$ of the orientation cover $\hat \lambda$ of $\lambda$. Let $\iota$ be the covering involution of $\widehat N \rightarrow N$, and let $H^1(\widehat{N}, \partial \widehat{N}; \RR)^-$ denote the $-1$ eigenspace for the action of $\iota^*$,
\label{ind:transcochom}

\begin{definition}\label{def:trans_cohom}
With all notation as above, a {\em transverse cocycle} for $\lambda$ is an element of $H^1(\widehat{N}, \partial \widehat{N}; \RR)^-$.
We use $\calH(\lambda)$ to denote the set of all transverse cocycles for $\lambda$.
\label{ind:transcoc}
\end{definition}

With the definition above it is clear that $\calH(\lambda)$ is a vector space, and if $\lambda$ is a union of sublaminations $\lambda_1, \ldots, \lambda_L$, then the space of transverse cocycles splits as
\[\calH(\lambda) = \bigoplus_{\ell=1}^L \calH(\lambda_\ell).\]
We record the dimension of $\cH(\lambda)$ below.

\begin{lemma}[Theorem 15 of \cite{Bon_THDGL}] \label{lem:trans_cocycle_dim}
The space of transverse cocycles forms a vector space of real dimension 
$-\chi(\lambda) + n_0(\lambda),$
where $n_0(\lambda)$ is the number of orientable components of $\lambda$.
\label{ind:orcomps}
\end{lemma}

When working with individual transverse cocycles, the above definition is rather unwieldy. Instead, it is often more useful to think of a transverse cocycle as a function on actual arcs instead of on homology classes.

\begin{definition}\label{def:trans_axiom}
Let $\lambda \in \ML(S)$. A {\em transverse cocycle}  $\sigma$ for $\lambda$ is a function which assigns to every arc $k$ transverse to $\lambda$ a real number $\sigma(k)$ such that
\begin{enumerate}
    \item[{(H0)}] (support): If $k$ does not intersect $\lambda$ then $\sigma(k)=0$.
    \item[{(H1)}] (transverse invariance): If $k$ and $k'$ are isotopic transverse to $\lambda$ then $\sigma(k) = \sigma(k')$.
    \item[{(H2)}] (finite additivity): If $k = k_1 \cup k_2$ where $k_i$ have disjoint interiors then $\sigma(k) = \sigma(k_1) + \sigma(k_2)$.
\end{enumerate}
\label{ind:transcocaxioms}
\end{definition}

The reader familiar with train tracks will recognize that these rules resemble those governing weight systems on train tracks; see Section \ref{sec:tt_shsh} for a continuation of this discussion.

We direct the reader to \cite{Bon_THDGL} or \cite[\S 3]{Bon_SPB} for a proof of the equivalence of Definitions \ref{def:trans_cohom} and \ref{def:trans_axiom} (our proof of Proposition \ref{prop:shsh_defsagree}, the corresponding statement for shear-shape cocycles, can also be adapted to prove this equivalence).

\subsection{Shear-shape cocycles as cohomology classes}\label{subsec:shsh_cohom}

Our first definition of a shear-shape cocycle is as a cohomology class on an appropriate augmented orientation cover, paralleling Definition \ref{def:trans_cohom}. This viewpoint allows us to deduce global structural results about spaces of shear-shape cocycles (Lemma \ref{lem:shsh_compat}) and also reveals implicit constraints on the structure of individual shear-shape cocycles (Lemma \ref{lem:sum_res=0}).

Suppose that $\arc$ is a filling arc system for $S \setminus \lambda$. For each arc $\alpha_i \in \arc$, choose an arc $t_i$ which meets $\alpha_i$ exactly once and is disjoint from $\lambda \cup \arc \setminus \{\alpha_i\}$. We call such an arc $t_i$ a {\em standard transversal to $\alpha_i$}. Compare Figure \ref{fig:shshdefs_SH3} below.
\label{ind:standard transversal}
An {\em orientation} of $\lambda \cup \arc$
\label{ind:orientlamplusarc}
is a continuous orientation of the leaves of $\lambda$ together with a choice of orientation on each $t_i$ such that
$t_i$ can be isotoped transverse to $\alpha_i$ into $\lambda$ so that the orientations agree.
Most pairs $\lambda \cup \arc$ are not orientable, but each has an {\em orientation double cover} $\widehat{\lambda} \cup \hat{\arc}$
\label{ind:doublecover}
(the reader should have in mind the orientation cover of a quadratic differential). We note that if $\lambda \cup \arc$ is orientable then $\lambda$ itself must be.

Consider a snug neighborhood $\epN\lambda$ of $\lambda$ on some hyperbolic surface $X$; since $X \setminus \lambda$ and $X \setminus \epN\lambda$ have the same topological type, we can identify the arc system $\arc$ as an arc system on $X \setminus \epN\lambda$. In particular, taking a small neighborhood $\epN\arc$ of $\arc$ we see that there is a correspondence between complementary components of $X \setminus (\lambda \cup \arc)$ and $X \setminus \epN(\lambda \cup \arc)$.
We will refer to any neighborhood $N_{\arc}$ of $\lambda \cup \arc$ whose complementary components have the same topological type as $X \setminus (\lambda \cup \arc)$ as a {\em snug neighborhood.}
\label{ind:snugnbhd}

Now let $N_{\arc}$ be a snug neighborhood of $\lambda \cup \arc$; then the cover
$\widehat{\lambda} \cup \hat{\arc} \rightarrow \lambda \cup \arc$
extends to a covering
$\widehat{N}_{\arc} \rightarrow N_{\arc}$ with covering involution $\iota$. By definition of the orientation cover, each standard transversal $t_i$ lifts to a pair of distinguished homology classes
\[t_i^{(1)}, t_i^{(2)}\in
H_1(\widehat{N}_{\arc}, \partial \widehat{N}_{\arc}; \RR)\]
such that
$\iota_* \, t_i^{(1)} = - t_i^{(2)}$.

The odd cocycles $H^1(\widehat{N}_{\arc}, \partial \widehat{N}_{\arc}; \RR)^-$
\label{ind:shshhom}
for the covering involution $\iota^*$ now provide a local cohomological model for the space of shear-shape cocycles on $\lambda$.
Observe that for each $i$ and each $\sigma \in H^1(\widehat{N}_{\arc}, \partial \widehat{N}_{\arc}; \RR)^-$, we have
\[\sigma(t_i^{(1)}) = - \iota^* \sigma (t_i^{(1)})
= - \sigma ( \iota_* t_i^{(1)} ) = \sigma(t_i^{(2)}).\]

\begin{definition}\label{def:shsh_cohom}
Let $\lambda \in \ML(S)$. A {\em shear-shape cocycle} for $\lambda$ is a pair $(\arc, \sigma)$ where $\arc = \sum \alpha_i$ is a filling arc system on $S \setminus \lambda$ and $\sigma \in H^1(\widehat{N}_{\arc}, \partial \widehat{N}_{\arc}; \RR)^-$ is such that the values $\sigma (t_i^{(j)})$ are all positive.\footnote{By Poincar\'e--Lefschetz duality, we have a linear isomorphism $H^1(\widehat{N}_{\arc}, \partial \widehat{N}_{\arc}; \RR) \cong H_1(\widehat{N}_{\arc};\RR)$ mapping the odd cocycles for $\iota^*$ to the odd cycles for $\iota_*$.  Compare with \cite[\S\S 4.1 and 4.4]{BonDre}, where a theory of (appropriately generalized) transverse (co-)cycles are applied to give shear-type coordinates for some higher rank Teichm\"uller spaces.}
\end{definition}

Let $\cutsurf_1 \cup \ldots \cup \cutsurf_m$ denote the components of $S \setminus \lambda$; then we define the {\em weighted arc system underlying $\sigma$}
\[\arcwt := \sum \sigma (t_i^{(j)}) \alpha_i \in \prod_{j=1}^m |\Arcfill(\cutsurf_j, \partial \cutsurf_j)|_{\RR}.\]
We denote the set of all shear-shape cocycles for $\lambda$ by $\SH(\lambda)$, the set of all shear-shape cocycles with underlying arc system $\arc$ by $\SH^\circ(\lambda; \arc)$, and the set of all shear-shape cocycles with underlying weighted arc system $\arcwt$ by $\SH(\lambda;{\arcwt})$.
\label{ind:SHspaces}
Often, we will leave the arc system implicit and just say that $\sigma$ is a shear-shape cocycle for $\lambda$.

\begin{remark}
By Theorem \ref{thm:arc=T(S)_crown}, a filling weighted arc system $\arcwt$ is the same data as a marked hyperbolic structure on each component of $S \setminus \lambda$.
In Sections \ref{sec:hyp_overview}--\ref{sec:shsh_homeo} below, we prove that (so long as $\sigma$ satisfies a positivity condition) these metrics glue together to give a complete hyperbolic metric on $S$.
\end{remark}

Our definition of shear-shape cocycle {\em a priori} depends on the choice of auxiliary neighborhood $N_{\arc}$ of $\lambda \cup \arc$.
However, it is not hard to see that 

\begin{lemma}\label{lemma:shsh_nbhds_iso}
The spaces of shear-shape cocycles defined by different snug neighborhoods are linearly isomorphic. Moreover, any two choices of snug neighborhoods define the same underlying weighted arc system.
\end{lemma}
\begin{proof}
Given two nested, snug neighborhoods $N_{\arc}' \subset N_{\arc}$ there is a deformation retraction of $N_{\arc}$ onto $N_{\arc}'$ (this comes from the assumption of snugness). This induces an isomorphism
\begin{equation}\label{eqn:cohom_isom}
    H^1(\widehat{N}_{\arc}, \partial \widehat{N}_{\arc}; \RR) \cong H^1(\widehat{N}_{\arc}', \partial \widehat{N}_{\arc}'; \RR)
\end{equation}
which also identifies the $-1$ eigenspaces of the covering involution. Therefore, we may identify the shear-shape cocycles defined by $N_{\arc}$ with those defined by $N_{\arc}'$. 
To see that the weights on $\arc$ do not depend on the choice of $N_{\arc}$, we note that the deformation retraction of $N_{\arc}$ onto $N_{\arc}'$ takes standard transversals to standard transversals, and hence the value of the cocycle on the transversals does not change as we change neighborhoods.

Now given any two snug neighborhoods $N_{\arc}$ and $N_{\arc}'$ of $\lambda \cup \arc$, one may take a common refinement $N_{\arc}''$ of $N_{\arc}$ and $N_{\arc}'$ and apply \eqref{eqn:cohom_isom} to deduce that the spaces of shear-shape cocycles defined by $N_{\arc}$ and $N_{\arc}'$ are linearly isomorphic and define the same underlying arc system.
\end{proof}

In view of this lemma, throughout the sequel we will change the neighborhood $N_{\arc}$ carrying $\sigma$ at will.

As the orientation cover of $\lambda$ naturally embeds into $\widehat N_{\arc}$, we may identify $\calH(\lambda)$ with a subspace of $H^1(\widehat{N}_{\arc}, \partial \widehat{N}_{\arc}; \RR)$. Since any element of $\calH(\lambda)$ evaluates to 0 on each standard transversal, we can add and subtract transverse cocycles from shear-shape cocyles without changing the underlying weighted arc system. We therefore have the following analogue of Lemma \ref{lem:trans_cocycle_dim}:

\begin{lemma}\label{lem:shsh_compat}
Let $\arcwt$ be the weighted arc system underlying some shear-shape cocycle. Then $\SH(\lambda;\arcwt)$ is an affine space modeled on the vector space $\calH(\lambda)$. In particular, $\dim_{\RR} \left( \SH(\lambda; \arcwt) \right) = -\chi(\lambda) + n_0(\lambda).$
\end{lemma}

\para{Homological constraints on residues}
When $\lambda$ is orientable (or more generally, contains orientable components), there are homological constraints governing which weighted arc systems may underlie a shear-shape cocycle.
Passing between arc systems and hyperbolic strutures on complementary subsurfaces (via Theorem \ref{thm:arc=T(S)_crown}), these homological constraints govern when two structures can be glued together along $\lambda$. 

For example, if $\lambda$ is a simple closed curve then in order to glue a hyperbolic structure on $S \setminus \lambda$ along $\lambda$, the lengths of the boundary components must have equal length. Tracing through the combinatorialization by weighted arc systems, this implies that the $\arcwt$-length of the boundary components must be the same.
The following lemma generalizes this observation to the case when $S \setminus \lambda$ has crowned boundary (compare Lemma \ref{lem:aXBase} below for a similar discussion using hyperbolic geometry).

\begin{lemma}\label{lem:sum_res=0}
Suppose that $\sigma$ is a shear-shape cocycle for $\lambda$ with underlying weighted arc system $\arcwt$, and let $\mu$ be an orientable component of $\lambda$. Then the sum of the (signed) residues of the boundary components incident to $\mu$ is $0$.
\end{lemma}
\begin{proof}
For any component $\mu$ of $\lambda$, let $\partial(\mu)$ denote the boundary components (either closed or crowned) resulting from cutting along $\mu$. For the purposes of this proof, let $\alpha(\mu)$ denote the sub-arc system of $\arc$ consisting of those arcs with endpoints on $\mu$.

Pick an orientation on $\mu$; this induces an orientation on each boundary component $\Crown \in \partial(\mu)$, and hence gives the metric residue of each such $\Crown$ a definite choice of sign. Since we are eventually going to prove that the sum of these residues is $0$, it does not matter which orientation of $\mu$ we pick.

As $\mu$ is orientable, picking an orientation on $\mu$ is also equivalent to picking one of the lifts $\hat{\mu}$ of $\mu$ in the orientation cover $\widehat{\lambda} \cup \hat{\arc}$. Let $\widehat{\alpha({\mu})}$ denote the set of all lifts of arcs of $\alpha(\mu)$ which meet $\hat{\mu}$.
Then since severing $\widehat{\alpha({\mu})}$ disconnects $\hat{\mu}$ from the rest of $\widehat{\lambda} \cup \hat{\arc}$, there is a relation
\[\sum_{\hat{\alpha}_i \in \widehat{\alpha({\mu})}} \varepsilon_i \hat{t}_i = 0 \text{ in } H_1(\widehat{N}_{\arc}, \partial \widehat{N}_{\arc}; \ZZ)\]
where $\varepsilon_i$ is 1 if $\hat{\alpha}_i$ is on the left-hand side of $\hat{\mu}$ and $-1$ on the right-hand side, and $\hat{t}_i$ is the (relative homology class of the) oriented standard transversal corresponding to $\hat{\alpha}_i$. See Figure \ref{fig:ressum=0}.

\begin{figure}[ht]
    \centering
\begin{tikzpicture}
    \draw (0, 0) node[inner sep=0] {\includegraphics{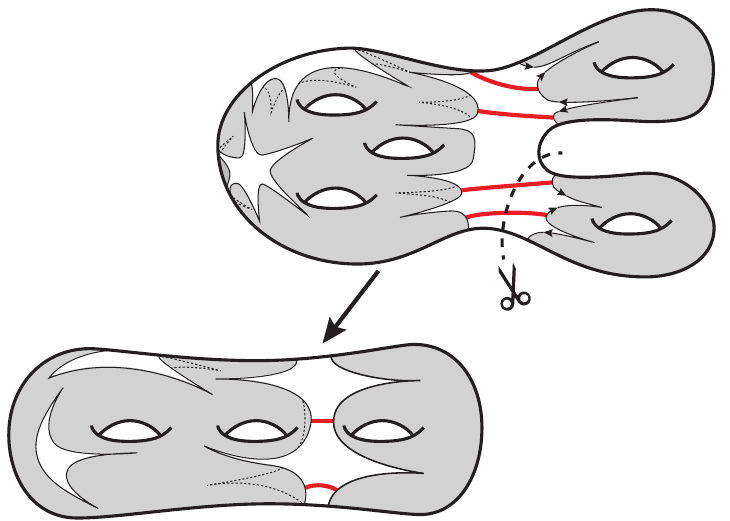}};
    \node at (-.7,-2.5) [red] {$t_1$};
    \node at (-.7,-3.5) [red] {$t_2$};
    \node at (2,.3) [red] {$\hat{t}_1$};
    \node at (2.25,1.5) [red] {$\hat{t}_2$};
    \node at (3.5, -1.2) [red] {\Large $[\hat{t}_1] - [\hat{t}_2] = 0$ in $H_1$};
    \node at (1.5,-3) {\Large $\mu$};
    \node at (5.5,.5) {\Large $\widehat{\mu}$};
\end{tikzpicture}
    \caption{Severing ties with one of the lifts $\widehat{\mu}$ of an orientable component $\mu$ of $\lambda$. This partition induces a relation in homology, hence a restriction on shear-shape cocycles. In this figure the top surface contains $\widehat{\lambda}$ while the bottom contains $\lambda$; the shaded regions are neighborhoods of these laminations.}
    \label{fig:ressum=0}
\end{figure}

Therefore, for any cohomology class $\sigma \in H^1(\widehat{N}_{\arc}, \partial \widehat{N}_{\arc}; \ZZ)$, and in particular any shear-shape cocycle,
\begin{equation}\label{eqn:ressum=0}
\sum_{\hat{\alpha}_i \in \widehat{\alpha({\mu})}} \varepsilon_i \sigma(\hat{t}_i) = 0.
\end{equation}
Now $\varepsilon_i$ is positive when the arc is on the left--hand side of $\hat{\mu}$, equivalently (equipping $\mu \subset S$ with the corresponding orientation) when $S \setminus \lambda$ is on the left--hand side of $\mu$. Similarly, $\varepsilon_i$ is negative when the complementary subsurface lies to the right of $\mu$. Unraveling the definitions and partitioning the arcs of $\alpha(\mu)$ into their incident boundary components, we see that \eqref{eqn:ressum=0} is equivalent to the statement that
\[\sum_{\Crown \in \partial(\mu)} \res_{\arcwt}(\Crown)
= \sum_{\alpha_i \in \alpha(\mu)} \varepsilon_i c_i =0,\]
which is what we wanted to prove.
\end{proof}

\begin{remark}\label{rmk:geod_lamination}
As with transverse cocycles, one can define shear-shape cocycles for any geodesic lamination, not just those which support transverse measures.
The analogue of Lemma \ref{lem:sum_res=0} is more complicated in this case, as the corresponding homological relations may involve both the $\hat{t}_i$ and other relative cycles (for example, consider when $\lambda$ contains a geodesic spiraling onto a closed leaf).
We have omitted such a discussion as this level of generality will not be needed for our purposes.
\end{remark}

\subsection{Shear-shape cocycles as functions on arcs}\label{subsec:shsh_axiom}

In analogy with Definition \ref{def:trans_axiom}, we can also view shear-shape cocycle as functions on transverse arcs which satisfy certain properties. While this definition is more involved, it is more convenient for the calculations of Sections \ref{sec:hyp_map}--\ref{sec:shsh_homeo} and better reflects the process of ``measuring'' arcs by a shear-shape cocycle.

As indicated by Lemma \ref{lem:sum_res=0}, we must first cut out the space of all possible weighted arc systems underlying a shear-shape cocycle. Denote the complementary subsurfaces of $\lambda \in \ML(S)$ by $\cutsurf_1, \ldots, \cutsurf_m$, and set
\[ \Base := \Big \{ \arcwt \in
\prod_{j=1}^m |\Arcfill(\cutsurf_j, \partial \cutsurf_j)|_{\RR} 
\, \Big| \,
\sum_{\mathcal{C} \in \partial(\mu)} \res_{\arcwt}( \mathcal{C}) = 0 \text{ for all orientable components } \mu \subset \lambda \Big\}
\]
\label{ind:base}
where we recall that $\partial(\mu)$ denotes the set of boundary components of $S \setminus \lambda$ resulting from cutting along $\mu$.

By Theorem \ref{thm:arc=T(S)_crown}, we can reinterpret $\Base$ as the set of all hyperbolic structures on $S \setminus \lambda$ so that the metric residues of the boundary components resulting from any orientable component $\mu$ of $\lambda$ sum to zero.
We note that when each component of $\lambda$ is nonorientable, $\Base$ is just the product of the Teichm{\"u}ller spaces of the complementary subsurfaces. When $\lambda$ is a simple closed curve, then $\Base$ consists of those metrics on $S \setminus \lambda$ where the two boundary components have the same length.

Using this reinterpretation together with Lemma \ref{lem:Teich_crown}, we see that $\Base$ is topologically just a cell:

\begin{lemma}\label{lem:base_dim}
Let $\lambda \in \ML(S)$ with $S \setminus \lambda = \cutsurf_1 \cup \ldots \cup \cutsurf_m$. Then $\Base \cong \RR^{d}$, where
\[d =  - n_0(\lambda)+ \sum_{j=1}^m \dim ( \T(\cutsurf_j))\]
where $n_0(\lambda)$ is the number of orientable components of $\lambda$.
\end{lemma}
\begin{proof}
Let $\mu_1, \ldots, \mu_{n_0(\lambda)}$ denote the orientable components of $\lambda$ and fix an arbitrary orientation on each. Then the lemma follows from the observation that $\Base$ is a fiber bundle over 
\[\prod_{i=1}^{n_0(\lambda)} 
\left\{ (R^i_k) \in \RR^{|\partial(\mu_i)|}
\, \bigg| \,
\sum_k R^i_k = 0 \right\} \]
with fibers equal to
\[\left\{
[Y, f] \in \prod_{j=1}^m \T(\cutsurf_j)
\, \bigg| \, 
\res(\Crown_k) = R^i_k \text{ for each } \Crown_k \in \partial (\mu_i)
\right\}.\]
By Proposition \ref{prop:res_mnfld}, the fibers are each homeomorphic to $\RR^d$, where
\[d =
\left(\sum_{j=1}^m \dim ( \T(\cutsurf_j)) \right) - 
\left(\sum_{i=1}^{n_0(\lambda)} |\partial (\mu_i)| \right).\]
Totalling the dimensions of base and fiber give the desired result.
\end{proof}

We can now present our second definition of shear-shape cocycles.

\begin{definition}\label{def:shsh_axiom}
Let $\lambda \in \ML(S)$. A {\em shear-shape cocycle} for $\lambda$ is a pair $(\sigma, \arcwt)$ where $\arcwt$ is a weighted filling arc system 
\[\arcwt = \sum_{i=1}^n c_i \alpha_i \in \Base\]
and $\sigma$ is a function which assigns to every arc $k$ transverse to $\lambda$ and disjoint from $\arc:= \cup \alpha_i$ a real number $\sigma(k)$, satisfying the following axioms:
\begin{enumerate}
    \item [{(SH0)}] (support): If $k$ does not intersect $\lambda$ then $\sigma(k)=0$.
    \item [{(SH1)}] (transverse invariance): If $k$ and $k'$ are isotopic through arcs transverse to $\lambda$ and disjoint from $\arc$, then $\sigma(k) = \sigma(k')$.
    \item [{(SH2)}] (finite additivity): If $k = k_1 \cup k_2$ where $k_i$ have disjoint interiors, then $\sigma(k) = \sigma(k_1) + \sigma(k_2)$.
    \item[(SH3)] ($\arcwt$-compatibility): 
    Suppose that $k$ is isotopic rel endpoints and transverse to $\lambda$ to some arc which may be written as $t_i \cup \ell$, where $t_i$ is a standard transversal and $\ell$ is disjoint from $\arc$. Then the loop $k \cup t_i \cup \ell$ encircles a unique point $p$ of $\lambda \cap \arc$, and 
    \[\sigma(k) = \sigma(\ell) + \varepsilon c_i\]
    where $\varepsilon$ denotes the winding number of $k \cup t_i \cup \ell$ about $p$ (where the loop is oriented so that the edges are traversed $k$ then $t_i$ then $\ell$). See Figure \ref{fig:shshdefs_SH3}.
\end{enumerate}
\end{definition}

While axiom (SH3) may seem convoluted upon first inspection, its entire effect is to prescribe how the value $\sigma(k)$ evolves as an endpoint of $k$ passes through an arc of $\arc$. The sign change records whether the map induced by $k = t_i \cup \ell$ from the oriented simplex into $S$ is orientation-preserving or -reversing.

\begin{remark}
In Section \ref{sec:tt_shsh} (Proposition \ref{prop:ttcoords} in particular), we show that there exists a choice of ``smoothing'' for $\arc$ which resolves condition (SH3) into an additivity condition. This is equivalent to prescribing that an arc $k$ may only be dragged over a point of $\lambda \cap \arc$ in one direction.
\end{remark}

The equivalence between Definitions \ref{def:shsh_cohom} and  \ref{def:shsh_axiom} is essentially the same as the equivalence of the cohomological and axiomatic definitions of transverse cocycles \cite[pp. 248--9]{Bon_SPB}.
However, the $\arcwt$-compatibility condition (axiom (SH3)) contributes new technical difficulties, and so we have included a full proof for completeness.

\begin{proposition}\label{prop:shsh_defsagree}
The cohomological and axiomatic definitions of shear-shape cocycles agree.
\end{proposition}
\begin{proof}
Suppose first that $\sigma$ is a cohomological shear-shape cocycle, that is, a cohomology class of the orientation cover $\widehat N_{\arc}$ of $N_{\arc}$ that is anti-invariant under the covering involution and that gives positive weight to the canonical lifts of the standard transversals of each arc of a filling arc system $\arc$.
We begin by building from $\sigma$ a function $f_{\sigma}$; the basic idea is to restrict an arc to a neighborhood of $\lambda$, resulting in a relative homology class, and to set $f_\sigma$ to be $\sigma$ evaluated on this class.

Suppose that $k$ is any arc transverse to $\lambda$ and disjoint from $\arc$. Choose a small neighborhood $N_{\arc}$ of $\lambda \cup \arc$ so that $k$ meets $\partial N_{\arc}$ transversely and $\partial k \cap N_{\arc} = \emptyset$; then $k|_{N_{\arc}}$ is a union of arcs with endpoints on $\partial N_{\arc}$.
Each arc $k_i$ of $k|_{N_{\arc}}$ has two distinguished, oriented lifts $k_i^{(1)}$ and $k_i^{(2)}$ to $\widehat N_{\arc}$ that cross $\widehat{\lambda}$ from right to left. As in Section \ref{subsec:shsh_cohom}, these distinguished lifts satisfy
\begin{equation}\label{eqn:involutionaction}
\iota_* ( [ k_i^{(1)}]) = - [k_i^{(2)}]
\end{equation}
in $H_1(\widehat N_{\arc}, \partial \widehat N_{\arc}; \ZZ)$, where $\iota$ is the covering involution of $\widehat N_{\arc} \rightarrow N_{\arc}$. In particular
$\sigma(  [ k_i^{(1)}]) = \sigma([k_i^{(2)}])$
since $\sigma$ is anti-invariant under $\iota$. We therefore set
\[f_\sigma(k) := \sigma( [ k ]) \]
where $[k]$ is the homology class of either lift of $k|_{N_{\arc}}$ to $\widehat N_{\arc}$.

We now prove that $f_\sigma$ satisfies the axioms of Definition \ref{def:shsh_axiom}:
\begin{enumerate}
    \item [{(SH0)}] If $k$ does not intersect $\lambda$ then $k|_{N_{\arc}}$ is empty and $[k]=0$, implying $f_\sigma(k)= 0$.
    \item [{(SH1)}] If $k$ and $k'$ are isotopic through arcs transverse to $\lambda$ and disjoint from $\arc$ then $k|_{N_{\arc}}$ and $k'|_{N_{\arc}}$ are properly isotopic. One can lift this isotopy to the orientation cover to deduce that $[k] = [k']$ for the correct choice of lifts, so $f_\sigma(k) = f_\sigma(k')$.
    \item [{(SH2)}] Suppose that $k = k_1 \cup k_2$; then so long as $N_{\arc}$ is small enough it is clear that $k|_{N_{\arc}} = k_1|_{N_{\arc}} \cup k_2|_{N_{\arc}}$. 
    Therefore, since a lift of $k|_{N_{\arc}}$ consists of the union of lifts of $k_1|_{N_{\arc}}$ and $k_2|_{N_{\arc}}$, we see that $[k]=[k_1] + [k_2]$ and hence the corresponding equality of $f_\sigma$ values also holds.
    \item[(SH3)] Finally, suppose that $k$ is isotopic (rel endpoints and transverse to $\lambda$) to $\ell \cup t_i$. 
    Without loss of generality, we assume that the restriction of each of $k, \ell, t_i$ to $N$ is a single properly embedded arc (if not, simply break the arcs into smaller pieces and apply (SH1) and (SH2) repeatedly). We also assume the restrictions are all disjoint (even at their endpoints), appealing to (SH1) as necessary.
    
    The isotopy between $k$ and $\ell \cup t_i$ induces a map from a disk $\Delta$ to $N_{\arc}$ so that $\partial \Delta \subset \partial N_{\arc} \cup k \cup \ell \cup t_i$.
    Refining $N_{\arc}$, isotoping the arcs, and homotoping the map as necessary, we may assume that $\Delta$ embeds into $N_{\arc}$, and therefore must occur in one of the configurations shown in Figure \ref{fig:shshdefs_SH3} below.

\begin{figure}[ht]
    \centering
\begin{tikzpicture}
    \draw (0, 0) node[inner sep=0] {\includegraphics{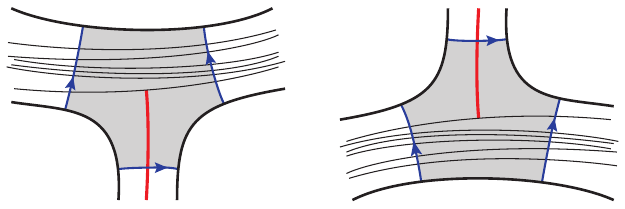}};
    \node at (-2.8,-1.9) [red] {$\alpha_i$};
    \node at (2.8,1.9) [red] {$\alpha_i$};
    \node at (-.3, 1) {$\widehat{\lambda}$};
    \node at (-5, 2){$\widehat{N}_{\arc}$};
    \node at (-4, -.5){$\widehat{\Delta}$};
    \node at (1.7, .5){$\widehat{\Delta}$};
    \node at (-3.8, 1.6)[blue]{$k$};
    \node at (-1.9, 1.6)[blue]{$\ell$};
    \node at (-2, -1.1)[blue]{$t_i$};
    \node at (3.5, 1.1)[blue]{$t_i$};
    \node at (1.9, -1.6)[blue] {$k$};
    \node at (3.8, -1.6)[blue]{$\ell$};
    \node at (-2.8, -2.5){$[k] = [t_i]+[\ell]$};
    \node at (2.8, -2.5){$[k] + [t_i]=[\ell]$};
\end{tikzpicture}
    \caption{Possible configurations of the disk $\widehat \Delta$ and the corresponding homological relations. }
    \label{fig:shshdefs_SH3}
\end{figure}

Now choose one of the lifts $\widehat \Delta \subset \widehat N_{\arc}$ of $\Delta$; this choice specifies lifts of the arcs $k$, $\ell$, and $t_i$ and therefore (after equipping the lifts with their canonical orientations) relative homology classes $[k]$, $[\ell]$, and $[t_i]$. As these lifts together with $\partial \hat N_{\arc}$ bound the disk $\widehat \Delta$, we therefore get the equality 
\[ [k] = [\ell ] \pm [t_i]\]
where the sign is determined by the relative configuration of the arcs. Inspection of Figure \ref{fig:shshdefs_SH3} reveals that the sign coincides with the winding number of the loop $k \cup t_i \cup \ell$ about $p$.
\end{enumerate}

\bigskip

Now suppose that $(\sigma, \arcwt)$ is an axiomatic shear-shape cocycle in the sense of Definition \ref{def:shsh_axiom}. Pick a snug neighborhood $N_{\arc}$ of $\lambda \cup \arc$; our task to show that the function $(k \mapsto \sigma(k))$ is indeed a cocycle (on the orientation cover, and is anti-invariant under the covering involution).

We first show that $\sigma$ naturally defines a cochain on $\widehat N$ relative to $\partial \widehat N_{\arc}$ which is anti-invariant by $\iota^*$. Recall that any arc in the orientation cover comes with a canonical orientation.
We may then assign to any oriented arc $\hat k$ properly embedded in $\widehat N_{\arc}$ the value $\pm\sigma(k)$, where $k$ is the image of $\hat k$ under the covering projection and where the sign is positive if $\hat k$ is oriented according to the canonical orientation and negative otherwise. To the (canonically oriented lifts of the) standard transversals $t_i$ we assign the value $c_i$. Anti-invariance then follows by construction (compare \eqref{eqn:involutionaction}).

To see that this cochain is actually a cocycle, we show that it evaluates to 0 on every boundary. For the purposes of this argument, it will be convenient to realize $H^1(\widehat{N}_{\arc}, \partial \widehat{N}_{\arc}; \RR)$ in terms of simplicial (co)homology.
The neighborhood $N_{\arc}$ may be triangulated as depicted in Figure \ref{fig:triangulate_tt} (compare \cite[Figure 1]{BonSoz}). 
In such a triangulation, each point of $\lambda \cap \arc$ and each switch of $N_{\arc}$ corresponds to a unique triangle, while the remaining branches each contribute a rectangle which is in turn subdivided into two triangles.
This triangulation clearly lifts to an ($\iota$-invariant) triangulation of $\widehat{N}_{\arc}.$

\begin{figure}[ht]
\centering
\begin{tikzpicture}
    \draw (0, 0) node[inner sep=0] {\includegraphics{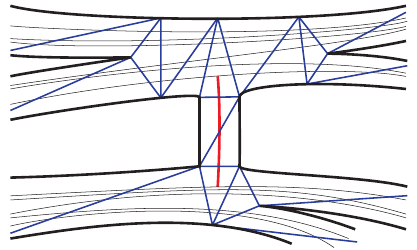}};
    \node at (-.4,-.1) [red] {$\alpha_i$};
    \node at (-3, -.2){${N}_{\arc}$};
    \node at (3.7, 1.8){$\lambda$};
\end{tikzpicture}
    \caption{A triangulation of a (snug) neighborhood of $\lambda \cup \arc$. Axioms (SH0)--(SH3) imply that $\sigma(\partial \Delta)=0$ for each triangle $\Delta$ in the triangulation, i.e., $\sigma$ is a cocycle.}
    \label{fig:triangulate_tt}
\end{figure}

It therefore suffices to prove that for each oriented triangle $\Delta$ of $\widehat{N}_{\arc}$ we have that $\sigma (\partial \Delta) = 0.$ There are three types of triangles, each of which corresponds to a different axiom of Definition \ref{def:shsh_axiom}:
\begin{itemize}
    \item If $\Delta$ is (the lift of) a triangle coming from a subdivision of a branch, then one if its sides does not intersect $\lambda$ and is thus assigned the value $0$ by (SH0). The other two sides are isotopic rel $\lambda$, cross $\lambda$ with different orientations, and are assigned the same value by (SH1). Therefore $\sigma (\partial \Delta) = 0.$
    Similarly, if $\Delta$ comes from a neighborhood of $\alpha$, then the edges transverse to $\alpha$ are assigned the arc weight $c_i$ (with opposite signs) while the other edge gets zero weight, so $\sigma (\partial \Delta) = 0.$
    \item Now suppose $\Delta$ is (the lift of) a triangle corresponding to a switch of $N_{\arc}$ with $\partial \Delta = k_1+k_2-k$. Then since the concatenation of $k_1$ and $k_2$ is isotopic transverse to $\lambda$ to $-k$, axiom (SH2) implies
    \[\sigma(k_1) + \sigma(k_2) - \sigma(k) = 0\]
    and again $\sigma (\partial \Delta) = 0.$
    \item Finally, suppose that $\Delta$ is (the lift of) a triangle corresponding to a point of $\lambda \cap \arc$, so $\partial \Delta$ is some signed combination of the (canonically oriented) lifts of arcs $k, \ell$ and $t$ where $t$ is a standard transversal and $k$ is isotopic rel endpoints and transverse to $\lambda$ to $\ell \cup t$. Without loss of generality we assume that $\Delta$ is positively oriented; then depending on the configuration of $k$, $t$ and $\ell$ we have either
    \[\ell-k +t = 0 \text{ or } \ell - t - k = 0\]
    (as in Figure \ref{fig:shshdefs_SH3}). In either case, axiom (SH3) implies that $\sigma(\partial \Delta)=0$.
\end{itemize}
We have therefore shown that $\sigma (\partial \Delta) = 0$ for every triangle of a triangulation and hence $\sigma$ is indeed a 1-cocycle on $\widehat{N}_{\arc}$ rel boundary, finishing the proof of the lemma.
\end{proof}

\para{Measuring arcs along curves}
We will also want to associate a number $\sigma(k)$ to certain arcs $k$ that have non-empty intersection with $\arc$; this quantity should be invariant under suitable isotopy transverse to $\lambda$ respecting the combinatorics of intersections with $\arc$. 

So suppose $\mu\subset \lambda$ is an isolated leaf, i.e. a simple closed curve.
We say that an arc $k$ transverse to $\lambda\cup \arc$ and contained in an annular neighborhood of $\mu$ is {\em non-backtracking} if any lift $\tilde k$ of $k$ to the universal cover intersects the entire preimage $\tilde \mu$ of $\mu$ exactly once and $\tilde k$ crosses each lift of an arc of $\arc$ at most once.

If $k$ is a non-backtracking arc, then one may orient $k$ and give $\mu$ the orientation that makes $k$ start to the right of $\mu$. 
Record the sequence of arcs $\beta_1, ..., \beta_m$ crossed by $k$, in order (note that arcs of $\arc$ may repeat in this sequence). 
Then up to isotopy, we may assume that $k$ is a concatenation of standard transversals $t_1, ..., t_m$ together with a small segment $k_0$ disjoint from $\arc$ crossing $\mu$ from right to left.
Compare Figure \ref{fig:measure_isolated}.

Since $k$ is non-backtracking, the points $\beta_1\cap \mu, ..., \beta_m \cap \mu$ make progress around $\mu$ either in the positive direction or the negative direction.
Take $\varepsilon = +1$ in the former case and $\varepsilon = -1$ in the latter, then define 
\begin{equation}\label{eqn:def_non_backtracking_arc}
\sigma(k): = \sigma(k_0)+ \varepsilon \sum_{j = 1}^m c_j
\end{equation}
where $c_j$ is the weight corresponding to the arc $\beta_j$.
Note that the value of $\varepsilon$ only depends on $k$ and not on its orientation, as reversing its orientation also reverses the orientation of $\mu$.

\begin{figure}[ht]
\centering
\begin{tikzpicture}
    \draw (0, 0) node[inner sep=0] {\includegraphics{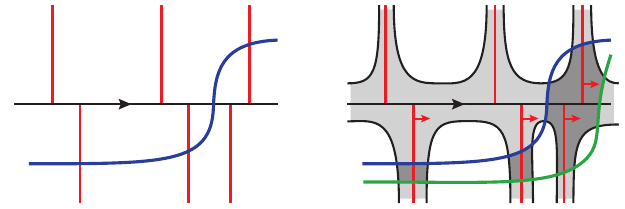}};
    \node at (-5.5, 1){$\varepsilon = +1$};
    \node at (-3, -.7)[blue]{$\tilde k$};
    \node at (-4.3, -1.5)[red]{$\beta_1$};
    \node at (-2.4, -1.5)[red]{$\beta_2$};
    \node at (-1.35, 1.5)[red]{$\beta_3$};
    \node at (-5.5, 0){$\widetilde \mu$};
    \node at (2.5, -.7)[blue]{$\tilde k$};
    \node at (2.5, -1.6)[green]{$\tilde{k}'$};
    \node at (2.1, 1){$\widetilde{N}_{\arc}$};
\end{tikzpicture}
    \caption{Since $k$ makes progress around $\mu$ in the positive direction, $\varepsilon =+1$.}
    \label{fig:measure_isolated}
\end{figure}

\begin{lemma}\label{lem:measure_isolatedleaf}
Suppose that $k$ and $k'$ are non-backtracking arcs transverse to $\lambda\cup\arc$ contained in an annular neighborhood of a simple closed curve component $\mu$ of $\lambda$.
If there exist lifts $\tilde{k}$ and $\tilde{k}'$ to $\widetilde{S}$ whose endpoints lie in the same component of $\widetilde{S} \setminus (\tlambda \cup \tilde{\arc})$ and $k$ is isotopic to $k'$ transverse to $\lambda$, then $\sigma(k) = \sigma(k')$. 
\end{lemma}
\begin{proof}
Fix a snug neighborhood $N_{\arc}$ of $\lambda \cup \arc$; then we need only show that $k|_{N_{\arc}}$ and $k'|_{N_{\arc}}$ define homologous cycles in the orientation cover.

We can find an isotopy $[0,1]^2 \rightarrow \widetilde{S}$ between lifts of $k$ and $k'$ (transverse to $\lambda$) that leaves the endpoints in the same component of $\widetilde{S} \setminus \widetilde{N}_{\arc}$.  
Such an isotopy then descends to $S$ under the covering projection.
The intersection of the image of each transverse arc with ${N}_{\arc}$ defines a cycle in the relative homology group, and this family of cycles is constant along the isotopy.

Since $\mu$ is orientable, an annular neighborhood of $\mu$ lifts homeomorphically to $\widehat N_{\arc}$, as do $k$ and $k'$. Therefore, the isotopy between $k$ and $k'$ (and the homology between their restrictions) also lifts to the orientation cover $\widehat{N}_{\arc}$, showing that the (lifts of the) restrictions of $k$ and $k'$ are homologous there as well.
Compare Figure \ref{fig:measure_isolated}.
\end{proof}

\section{The structure of shear-shape space}\label{sec:shsh_structure}
In this section, we investigate the global structure of the space of shear-shape cocycles.  
Whereas Bonahon's transverse cocycles assemble into a vector space, the space $\SH(\lambda)$ of all shear-shape cocycles is more complex when $\lambda$ is not maximal, forming an principal $\cH(\lambda)$-bundle over $\Base$ (Theorem \ref{thm:shsh_structure}).

After understanding the structure of shear-shape space, we define an intersection form on $\SH(\lambda)$ (Section \ref{subsec:Th_form}) and use it to specify the ``positive locus'' $\SH^+(\lambda)$ (Definition \ref{def:SH+}) which we show in Sections \ref{sec:flat_map} through \ref{sec:shsh_homeo} serves as a global parametrization of both $\MF(\lambda)$ and $\T(S)$. 

\subsection{Bundle structure}\label{subsec:shsh_structure}
Lemma \ref{lem:shsh_compat} of the previous section parametrizes all shear-shape cocycles which are compatible with a given weighted arc system. In this section, we analyze how these parameter spaces piece together to get a global description of the space of all shear-shape cocycles for a fixed lamination.

Let $G$ be a topological group.  A principal $G$-bundle is a fiber bundle whose fibers are equipped with a transitive, continuous $G$-action with trivial point stabilizers together with a bundle atlas whose transition functions are continuous maps into $G$.  We remind the reader that a principal $G$-bundle does not typically have a natural ``zero section,'' but instead, any local section of the bundle defines an identification of the fibers with $G$ via the $G$-action.  Moreover, any two sections define local trivializations of the bundle that differ by an element of $G$ in each fiber.

\begin{theorem}\label{thm:shsh_structure}
Let $\lambda \in \ML(S)$. The space $\SH(\lambda)$ forms a 
principal $\calH(\lambda)$-bundle over $\Base$ whose fiber over $\arcwt \in \Base$ is $\SH(\lambda;\arcwt)$.
\end{theorem}

\begin{proof}
There is an obvious map from $\SH(\lambda)$ to $\Base$ given by remembering only the values $\sigma(t_i)$ of transversals to the arcs.  For a given choice $\sigma_0$ in the fiber $ \SH(\lambda;\arcwt)$ over $\arcwt$, Lemma \ref{lem:shsh_compat} identifies $\SH(\lambda;\arcwt)$ with $\cH(\lambda)$ via the assignment $\sigma\mapsto \sigma - \sigma_0$.

For any filling arc system $\arc$ of $S\setminus\lambda$, the space $\SH^\circ(\lambda; \arc)$ of shear-shape cocycles with underlying arc system $\arc$ is naturally identified with the open orthant
\begin{equation}\label{eq:tt_octant}
\left\{\sigma \in H^1(\widehat{N}_{\arc}, \partial \widehat{N}_{\arc}; \RR)^- :
\sigma(t_i^{(j)}) > 0 \hspace{4pt} \forall i, j=1,2\right\},
\end{equation}
where $N_{\arc}$ is a snug neighborhood of $\lambda\cup \arc$ on $S$. 

Consider the open cell $\sB^\circ(\arc)\subset \Base$ defined as all those weighted arc systems with support equal to a maximal arc system $\arc$.
Using cohomological coordinates \eqref{eq:tt_octant} for $\SH^\circ(\lambda;\arc)$, we can find a continuous section $\sigma$ of $\SH^\circ(\lambda;\arc)\to \sB^\circ(\arc)$.  Then 
\begin{align*}
    \phi_\sigma:\sB^\circ(\arc)\times \cH(\lambda) & \to \SH^\circ(\lambda;\arc) \\ 
    (\arcwt, \eta) &\mapsto \sigma(\arcwt) + \eta
\end{align*}
is a homeomorphism preserving fibers of the natural projections.  For another choice of section $\sigma'$, we have 
\[\phi_{\sigma}\inverse (\phi_{\sigma'}(\arcwt, \eta)) = (\arcwt, \eta +\sigma'(\arcwt) - \sigma(\arcwt)).\]
Evidently, the map $\arcwt \mapsto \sigma'(\arcwt) - \sigma(\arcwt) \in \cH(\lambda)$ is continuous.  

If $N'_{\arc}$ is another snug neighborhood of $\lambda \cup \arc$, then $N_{\arc}$ and $N'_{\arc}$ share a common deformation retract.  The composition of the linear isomorphisms induced on cohomology by inclusion of the deformation retract preserves the orthants defined as in \eqref{eq:tt_octant} as well as fibers of projection to $\Base$.  This proves that the principal $\cH(\lambda)$-structure of the bundle lying over $\sB^\circ(\arc)$ does not depend on the snug neighborhood whose cohomology coordinatizes $\SH^\circ (\lambda;\arc)$. \\

To show that the principal $\cH(\lambda)$-bundle structures over all cells of $\Base$ glue together nicely, we find a continuous section of $\SH(\lambda)\to \Base$ near any given weighted arc system $\arcwt$.
Indeed, if $\arc\subset \arcb$, then inclusion $N_{\arc}\hookrightarrow N_{\arcb}$ of snug neighborhoods defines a  map on cohomology.  This map restricts to a linear isomorphism on the kernel of the evaluation map on the transversals to  $\arcb\setminus \arc$.  
Thus, the closure
\begin{equation} \label{eqn:tt_octantcl} \SH(\lambda;\arcb) = \bigcup_{\substack{\arcb \supseteq \arc \\ \arc \text{ fills $S\setminus \lambda$}}} \SH^\circ(\lambda;\arc)
\end{equation} 
of $\SH^\circ(\lambda;\arcb)$ in $\SH(\lambda)$ may be realized as an orthant in $H^1(\widehat{N}_{\arcb}, \partial \widehat{N}_{\arcb}; \RR)^-$ with some open and closed faces; one of the closed faces corresponds to $\SH^\circ(\lambda;\arc)$.
\footnote{When every component of $S\setminus \lambda$ is simply connected, the empty set is a filling arc system. When this is the case, $\sB^\circ(\emptyset)$ is identified with a point, while $\SH(\lambda;\emptyset) = \cH(\lambda)$.}

Since the complex $\Arcfill(S\setminus\lambda)$ is locally finite, there are only finitely many arcs $\beta_1, ..., \beta_k$ disjoint from $\arc$.  Let $U\subset \Base$ be a small neighborhood of $\arcwt$ and  $\sigma$ be a continuous section of $\SH(\lambda;\arc)\to \sB^\circ(\arc)\cap U$.  
For each $i$, after including $\SH^\circ(\lambda;\arc)$ as a face of $\SH(\lambda;\arc\cup \beta_i)$, we may extend $\sigma$ continuously on $U\cap \sB^\circ (\arc\cup \beta_i)$.  
Continuing this process, eventually extending $\sigma$ to higher dimensional cells meeting $U$, we end up with a continuous section $U\to \SH(\lambda)$, as claimed.  
As before, trivializations defined by two different sections differ by a continuous function $U\to \cH(\lambda)$; this completes the proof of the theorem.
\end{proof}

Since every bundle over a contractible base is trivial, this implies that

\begin{corollary}\label{cor:shsh_dimension}
Shear-shape space $\SH(\lambda)$ is homeomorphic to $\RR^{6g-6}$.
\end{corollary}
\begin{proof}
Let $\cutsurf_1, \ldots, \cutsurf_m$ denote the complementary components of $\lambda$, where $\cutsurf_j$ has genus $g_j$ with $b_j$ closed boundary components and $k_j$ crowns of types $\{ c_1^j, \ldots, c_{k_j}^j\}$.
By Lemmas \ref{lem:base_dim} and \ref{lem:Teich_crown}, we know that $\Base$ is homeomorphic to a cell of dimension
\[- n_0(\lambda)+ \sum_{j=1}^m \dim ( \T(\cutsurf_j))
= - n_0(\lambda)+ \sum_{j=1}^m 
\left( 6g_j-6 + 3b_j + \sum_{i=1}^{k_j} (c_i^j + 3) \right) \]
Lemmas \ref{lem:shsh_compat} and \ref{lem:Eulerchar_crowns} together imply that $\SH(\lambda; \arcwt)$ is an affine $\cH(\lambda)$-space of dimension
\[n_0(\lambda) -\chi(\lambda)  = n_0(\lambda)+ \frac{1}{2} \sum_{j=1}^m \sum_{i=1}^{k_j} c_{i}^j \]
Putting these dimension counts together via Theorem \ref{thm:shsh_structure}, we see that $\SH(\lambda)$ is homeomorphic to a cell of dimension
\[\sum_{j=1}^m \left( 6g_j-6 + 3b_j + \frac{3}{2} \sum_{i=1}^{k_j} ( c_i^j + 2) \right)
= \frac{3}{2\pi} \sum_{j=1}^m \text{Area}(\cutsurf_j) = \frac{3}{2\pi} \text{Area}(S) = 6g-6,\]
where the first equality follows from \eqref{eqn:area_crown}.
\end{proof}

\subsection{Intersection forms and positivity}\label{subsec:Th_form}
Now that we have a global description of shear-shape space, we restrict our attention to a certain positive locus $\SH^+(\lambda)$ inside of $\SH(\lambda)$. The main result of this section is Proposition \ref{prop:SH+_structure}, in which we identify $\SH^+(\lambda)$ as an affine cone bundle over $\Base$.

\para{Positive transverse cocycles} 
We begin by recalling the definition of positivity for transverse cocycles, as developed in \cite[\S6]{Bon_SPB} . Fixing some $\lambda \in \ML(S)$, we recall that a transverse cocycle for $\lambda$ may be identified with a relative cohomology class of the orientation cover $\widehat{N}$ of a snug neighborhood $N$ of $\lambda$ (Definition \ref{def:trans_cohom}).  The intersection pairing of $\widehat{N}$ therefore induces a anti-symmetric bilinear pairing
\[\ThH: \calH(\lambda) \times \calH(\lambda) \rightarrow \RR\]
\label{ind:ThH}
called the {\em Thurston intersection/symplectic form}. This form is nondegenerate when $\lambda$ is maximal, and more generally, when $\lambda$ cuts $S$ into polygons each with an odd number of sides \cite[\S3.2]{PennerHarer}.

Each transverse measure for $\lambda$ is in particular a transverse cocycle. Using the intersection form one can therefore define a positive cone $\calH^+(\lambda)$
inside of $\calH(\lambda)$ with respect to the (non-atomic) measures supported on $\lambda$. Write
\[\lambda = 
\lambda_1 \cup \ldots \cup \lambda_L
\cup \gamma_1 \cup \ldots \cup \gamma_M\]
where the $\gamma_m$ are all weighted simple closed curves and the $\lambda_\ell$ are minimal measured sub-laminations whose supports are not simple closed curves. Then set
\begin{equation}\label{eqn:defH+}
\calH^+(\lambda) := \{ \rho \in \calH(\lambda) : \ThH(\rho, \mu) > 0 \text{ for all } \mu \in \bigcup_{\ell=1}^L \Delta(\lambda_\ell)\},
\end{equation}
where $\Delta(\lambda_\ell)$ denotes the collection of measures supported on $\lambda_\ell$.
\label{ind:Delta(lambda)}

The reason for this involved definition is that the Thurston form is identically $0$ exactly when the underlying lamination is a multicurve.
Therefore, if the support of $\lambda$ contains a simple closed curve $\gamma$, the pairing of $\gamma$ with every transverse cocycle supported on $\lambda$ is $0$.
\footnote{This is because the components of the orientation cover are all annuli, whose first (co)homologies all have rank 1. For non-curve laminations, the homology has higher rank and so can support a nonzero intersection form.}

On the other hand, so long as $\lambda$ is not a multicurve then the Thurston form is not identically 0. In fact, the cone $\calH^+(\lambda)$ splits as a product
\[\calH^+(\lambda) = \bigoplus_{\ell=1}^L \calH^+(\lambda_\ell) \oplus \bigoplus_{m=1}^M \calH(\gamma_m).\]
As $\lambda$ supports at most $3g-3$ (projective classes of) ergodic measures, each $\cH^+(\lambda_{\ell})$ is a cone with a side for each (projective class of) ergodic measure supported on $\lambda_{\ell}$.

When $\lambda$ is a multicurve, then there are no $\lambda_\ell$'s and so the condition of \eqref{eqn:defH+} is empty. As such, in this case we have that the space of positive transverse cocycles is the entire twist space:
\[\calH^+(\gamma_1 \cup \ldots \cup \gamma_M) =
\calH(\gamma_1 \cup \ldots \cup \gamma_M)=
\bigoplus_{m=1}^M \calH(\gamma_m) \cong \RR^M.\]
Therefore, we see that no matter whether $\gamma$ is a multicurve or not, the space $\calH^+(\lambda)$ is a convex cone of full dimension (where we expand our definition of ``cone'' to include the entire vector space).

\para{Positive shear-shape cocycles}
We now repeat the above discussion for shear-shape cocycles. By Definition \ref{def:shsh_cohom}, any shear-shape cocycle $(\sigma, \arc)$ may be identified with a relative cohomology class of the orientation cover $\widehat{N}_{\arc}$ of a neighborhood $N_{\arc}$ of $\lambda \cup \arc$. As above, the intersection pairing of $\widehat{N}$ then defines a pairing between any two shear-shape cocycles with underlying arc system contained inside of $\arc$. However, if the underlying arc systems of $\sigma, \rho \in \SH(\lambda)$ are not nested then there is no obvious way to pair the two cocycles.

While it does not make sense to pair two arbitrary shear-shape cocycles, we can always pair shear-shape cocycles with transverse cocycles. Recall from (the discussion before) Lemma \ref{lem:shsh_compat} that $\cH(\lambda)$ naturally embeds as a subspace of the cohomology of the neighborhood $\widehat{N}_{\arc}$ defining a shear-shape cocycle and may be identified with the kernel of the evaluation map on transversals to $\arc$. Therefore, the intersection pairing on $\widehat{N}_{\arc}$ gives rise to a function
\[\ThSH: \SH(\lambda) \times \calH(\lambda) \rightarrow \RR\]
\label{ind:ThSH}
which we also refer to as the {\em Thurston intersection form}. Throughout the paper, we will differentiate between the different intersection forms by indicating their domains in subscript.

We record some of the relevant properties of $\ThSH$ below.

\begin{lemma}\label{lem:ThSHprops}
The Thurston intersection form $\ThSH$ is a $\Mod(S)[\lambda]$--invariant continuous pairing which is homogeneous in the first factor and linear in the second. Moreover, for any $\arcwt \in \Base$ and $\rho \in \calH(\lambda)$, the function
\[\ThSH(\cdot, \rho): \SH(\lambda; \arcwt) \rightarrow \RR\]
is an affine homomorphism inducing $\ThH(\cdot, \rho)$ on the underlying vector space $\calH(\lambda)$.
\end{lemma}
\begin{proof}
We begin by showing that the form is actually well-defined. Suppose first that $\arc$ is maximal; then since the (homological) intersection form is natural with respect to deformation retracts, and any two snug neighborhoods of $\lambda \cup \arc$ share a common deformation retract, we see that the form does not depend on the choice of neighborhood. 

Now suppose that $\arcb$ is a filling arc system that is a subsystem of two different maximal arc systems $\arc_1$ and $\arc_2$. Then one can take a snug neighborhood $N_{\arcb}$ of $\lambda \cup \arcb$ which includes into neighborhoods $N_i$ of $\lambda \cup \arc_i$ for $i=1, 2$. Now since the (homological) intersection form is also natural with respect to inclusions, we see that the Thurston form must be as well. Therefore, for any $\sigma \in \SH(\lambda; \arcb)$ and $\rho \in \calH(\lambda)$ it does not matter if we compute $\ThSH(\sigma, \rho)$ in $N_{\arcb}$, $N_1$, or $N_2$.

Now that we have established that $\ThSH$ is well-defined, the other properties follow readily from properties of the (homological) intersection form. Since the homological intersection pairing is linear in each coordinate, we get that $\ThSH$ is in particular linear in the second coordinate.
Similarly, for any $\arcwt \in \Base$ and any two $\sigma_1, \sigma_2 \in \SH(\lambda; \arcwt)$ we know that $\sigma_1 - \sigma_2$ is a transverse cocycle, and again by linearity of the homological intersection form we get that 
\[\ThSH(\sigma_1, \rho) - \ThSH(\sigma_2, \rho) = \ThH(\sigma_1-\sigma_2, \rho)\]
for all $\rho \in \calH(\lambda)$. Thus $\ThSH$ is affine on each $\SH(\lambda; \arcwt)$.

Finally, to see that the map $\ThSH(\cdot, \rho)$ is continuous for a fixed $\rho$, we recall that for any maximal arc system $\arc$, the space $\SH^\circ(\lambda; \arc)$ of shear-shape cocycles with underlying arc system $\arc$ may be realized as an open octant in cohomological coordinates \eqref{eq:tt_octant}, and this parametrization extends to its closure $\SH(\lambda;\arc)$.

Since the intersection pairing on cohomology is continuous, we see that for each maximal arc system $\arc$ the function $\ThSH(\cdot, \rho)$ is continuous on $\SH(\lambda; \arc)$. But now since we have checked that the value of $\ThSH(\cdot, \rho)$ does not actually depend on the neighborhood, it agrees on the overlaps of closures  $\SH(\lambda; \arc)$ for maximal $\arc$. Therefore, since the cell structure of $\Base$ is locally finite we may glue together the functions $\ThSH(\cdot, \rho)$ (which are continuous on each $\SH(\lambda; \arc)$) to get a globally continuous function on $\SH(\lambda)$.
\end{proof}

With this intersection form in hand, we may now define a positive locus with respect to the set of measures supported on $\lambda$.

\begin{definition}\label{def:SH+}
The space of {\em positive shear-shape cocycles} $\SH^+(\lambda)$ is the set 
\[\SH^+(\lambda) = \{ \sigma \in \SH(\lambda) : \ThSH(\sigma, \mu) > 0 \text{ for all } \mu \in \Delta(\lambda)\}.\]
\end{definition}

Observe the difference between the definition above and the one appearing in \eqref{eqn:defH+}: any positive shear-shape cocycle must also pair positively with all simple closed curves $\gamma_m$ appearing in the support of $\lambda$. The essential difference between the two cases is that additional branches of $\taua$ coming from the underlying arc system allows a shear-shape cocycle to meet each $\gamma_m$ without being completely supported on $\gamma_m$. 
Indeed, one can check that the contribution to the Thurston form coming from 
the intersection of $\arc$ with a simple closed curve component of $\lambda$ is always positive (compare \eqref{eqn:ThSH_tt}). In particular, the positivity condition is automatically fulfilled for any measure supported on a curve component of $\lambda$.

On each cohomological chart \eqref{eq:tt_octant} or \eqref{eqn:tt_octantcl} it is clear that $\SH^+(\lambda)$ is an open cone cut out by finitely many linear inequalities (one for each ergodic measure supported on $\lambda$, plus positivity of arcs weights).
However, this does not yield a global description of $\SH^+(\lambda)$. In order to get one, we must show that the linear subspaces cut out by the positivity conditions intersect the $\cH(\lambda)$ fibers transversely.

\begin{proposition}\label{prop:SH+_structure}
The space $\SH^+(\lambda)$ is an affine cone bundle over $\Base$ with fibers isomorphic to $\calH^+(\lambda)$.
\end{proposition}

By an affine cone bundle, we mean that there is a (non-unique) section $\sigma_0: \Base \to \SH(\lambda)$ such that
\[\SH^+(\lambda) \cap \SH(\lambda; \arcwt)
= \sigma_0(\arcwt) + \calH^+(\lambda)\]
for every $\arcwt \in \Base$. Moreover, any two such sections differ by a continuous map $\Base \to \cH(\lambda)$.

\begin{proof}
Choose mutually singular ergodic measures $\mu_1, \ldots, \mu_N, \gamma_1, \ldots, \gamma_M$ on $\lambda$ that span $\Delta(\lambda)$, where the supports of the $\mu_n$ are non-curve laminations and the $\gamma_m$ are all simple closed curves. Pick an arbitrary $\sigma \in \SH(\lambda; \arcwt)$, and define 
\[C(\sigma) := \left\{
\rho \in \calH(\lambda) 
\, \big| \,
\ThH(\rho, \mu_n) > -\ThSH(\sigma, \mu_n)
\text{ for all } n = 1, \ldots, N \right\}.\]
By linearity of $\ThH$ on $\calH(\lambda)$, together with the fact that the pairing $\ThH(\cdot, \mu_n)$ is not identically $0$ since the support of $\mu_n$ is not a simple closed curve, this is an intersection of $N$ affine half-spaces which do not depend on our choice of ergodic measures $\mu_i$ in their projective classes. Again by linearity, we see that this is just a translate of $\calH^+(\lambda)$ and hence is a cone of full dimension.

Now since $\ThSH(\cdot, \mu_j)$ is an affine map on $\SH(\lambda; \arcwt)$ for each $j$, we see that 
\[\sigma + C(\sigma) = 
\left\{
\eta \in \SH(\lambda; \arcwt) 
\, \big| \,
\ThSH(\eta, \mu_n) > 0
\text{ for all } n = 1, \ldots, N \right\}
= \SH^+(\lambda) \cap \SH(\lambda; \arcwt)\]
is an affine cone of full dimension (where the last equality holds because the positive discussion is automatically fulfilled for each $\gamma_m$). It is a further consequence of affinity that this identification does not depend on the choice of $\sigma$. The bundle structure then follows from continuity of $\ThSH$.
\end{proof}

\section{Train track coordinates for shear-shape space}\label{sec:tt_shsh}
In this section, we introduced train track charts for shear-shape cocycles.
In Section \ref{subsec:tt_for_transverse}, we recall Bonahon's realization of transverse cocycles to a lamination in the weight space of a train track that snugly carries it.
In Section \ref{subsec:tt_intersection}, we reinterpret the cohomological coordinate charts  \eqref{eq:tt_octant} for $\SH^\circ(\lambda;\arc)$ by ``smoothing'' $\lambda \cup \arc$ onto a train track $\taua$ (Construction \ref{constr:stand_smooth}) and realizing $\SH^\circ(\lambda;\arc)$ as an orthant in the weight space of $\taua$ (Proposition \ref{prop:ttcoords}).
This construction also has the added benefit of converting axiom (SH3) of Definition \ref{def:shsh_axiom} into a simpler additivity condition; this is convenient for computations and provides an explicit formula \eqref{eqn:ThSH_tt} for the Thurston intersection pairing.  We rely on this formula in Section \ref{subsec:flatmap} to show that foliations transverse to $\lambda$ define positive shear-shape cocycles (Proposition \ref{prop:Il_takes_int_to_Thurston}).

Later, in Section \ref{subsec:PIL}, we explain how the PIL structure of $\SH(\lambda)$ is manifest in train track coordinates and provides a canonical measure in the class of Lebesgue. When $\lambda$ is maximal, this measure is a constant multiple of the symplectic volume element induced by $\ThH$.
Finally, in Section \ref{subsec:horospherical}
we consider how train track charts facilitate an interpretation of $\SH(\lambda)$ as organizing the fragments of the cotangent space to $\ML$ at $\lambda$.

\begin{remark}
We advise the reader that two different types of train tracks appear below: those which carry transverse cocycles for $\lambda$ and give coordinates on the fiber $\SH(\lambda; \arcwt)$, and those which carry shear-shape cocycles and give coordinates on the total space $\SH(\lambda)$.
\end{remark}

\subsection{Train track coordinates for transverse cocycles}\label{subsec:tt_for_transverse}
We begin by recalling how transverse cocycles can be parametrized by weight systems on (snug) train tracks. The advantage of these coordinates is that they determine the cocycle with only finitely many values (a main benefit of the cohomological Definition \ref{def:trans_cohom}), but do so using unoriented arcs on the surface, not the orientation cover (a main benefit of the axiomatic Definition \ref{def:trans_axiom}).

Let $\tau$ be a train track snugly carrying a geodesic lamination $\lambda$ and $\sigma$ a transverse cocycle, thought of as a function on transverse arcs. For each branch $b$ of $\tau$, pick a tie $t_b$.
Then one can assign to $b$ the weight $\sigma(t_b)$; by Axiom (H1) this value does not depend on the choice of tie, and by Axiom (H2) these weights necessarily satisfy the switch conditions. Therefore, any transverse cocycle can be represented by a weight system on $\tau$, and in fact this map is an isomorphism.

\begin{proposition}[Theorem 11 of \cite{Bon_THDGL}]
\label{prop:tt_trans}
Let $\tau$ be a train track snugly carrying a geodesic lamination $\lambda$. Then the map $\sigma \mapsto \{\sigma(t_b)\}_{b \in b(\tau)}$ is a linear isomorphism between $\calH(\lambda)$ and $W(\tau)$, the space of all (real) weights on $\tau$ satisfying the switch conditions.
\end{proposition}

On a given train track snugly carrying $\lambda$, the Thurston intersection form $\ThH$ is easily computable in terms of the weight systems. To wit, if $\sigma, \rho \in \calH(\lambda)$ then their intersection is equal to 
\begin{equation}\label{eqn:ThH_tt}
    \ThH(\sigma, \rho) = \frac{1}{2} \sum_{s} \left| \begin{matrix}
    \sigma(r_s) & \sigma(\ell_s) \\
    \rho(r_s) & \rho(\ell_s)
    \end{matrix}\right| 
\end{equation}
where the sum is over all switches $s$ of $\tau$ and $r_s$ and $\ell_s$ are the half-branches which leave $s$ from the right and the left, respectively. Compare \cite[\S3.2]{PennerHarer}.

\subsection{Train track coordinates for shear-shape cocycles}\label{subsec:tt_intersection}
In order to imitate the above construction for shear-shape cocycles, we first must explain how to build a train track from $\lambda$ and a filling arc system $\arc$ on its complement.

Suppose that $\tau$ carries $\lambda$ snugly; then the complementary components of $\tau\cup \arc$ correspond to those of $\lambda\cup \arc$.
A {\em smoothing of $\tau \cup \arc$} is a train track $\taua$ which is obtained by choosing tangential data at each of the points of $\tau \cap \arc$ and isotoping each arc of $\arc$ to meet $\tau$ along the prescribed direction.  
Each component of $S\setminus \tau$ inherits an orientation from $S$, which in turn gives an orientation to the boundary (of the metric completion) of each subsurface.  
A smoothing $\taua$ is {\em standard} if for each switch of $\taua$ with an incoming half branch corresponding to an arc $\alpha_i \in \arc$, the incoming tangent vector to $\alpha_i$ is pointing in the positive direction with respect to the boundary  orientation of the component of $S\setminus \tau$ containing $\alpha_i$; see Figure \ref{fig:standard_smoothing}.
\label{ind:taua}

\begin{figure}[ht]
\centering
    \includegraphics{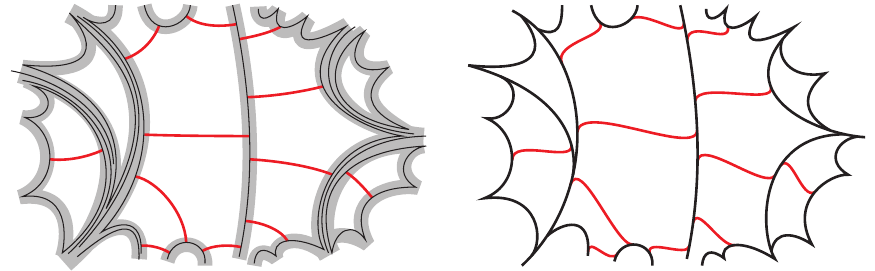}
    \caption{Left: a geometric train track neighborhood of $\tlambda$ together with an arc system $\widetilde{\arc}$.  Right: The (preimage of the) standard smoothing $\taua$.}
    \label{fig:standard_smoothing}
\end{figure}

Recall (Construction \ref{const:geometric_tt}) that a geometric train track $\tau$ constructed from a hyperbolic structure $X\in \T(S)$,  $\lambda\in \ML(S)$, and $\epsilon>0$ is obtained as the leaf space of the orthogeodesic foliation restricted to an $\epsilon$-neighborhood of $\lambda$ in $X$ (for small enough values of $\epsilon$). 

\begin{construction}[Geometric standard smoothings]\label{constr:stand_smooth}
Let $\lambda\in \ML(S)$ and $X$ be a hyperbolic metric on $S$. Let $\arc$ be a filling arc system in $S\setminus\lambda$, realized orthogeodesically on $X$.
For small enough $\epsilon>0$, $\arc\cap \mathcal N_{\epsilon}(\lambda)$ lies in a finite collection of leaves of $\Ol(X)$ and so each end of each arc of $\arc$ defines a point in the quotient $\tau = \mathcal N_{\epsilon}(\lambda)/\sim$, where $\sim$ is the equivalence relation induced by collapsing the leaves of $\Ol(X)|_{\epN\lambda}$.

The geometric standard smoothing $\taua$ is then obtained by attaching $\arc$ onto the geometric train track $\tau$ at these points and smoothing in the standard way.
\end{construction}

Since $\arc$ is filling, the components of $X\setminus(\lambda\cup \arc)$ are topological disks.  In a geometric standard smoothing $\taua$, each complementary disk incident to an arc $\alpha$ of $\arc$ has at least one spike corresponding to an ends of that $\alpha$.  Since no arc of $\arc$ joins asymptotic geodesics of $\lambda$, the complementary polygons all have at least $3$ spikes and so we see that $\taua$ is indeed a train track.

\begin{remark}
A geometric standard smoothing keeps track of the intersection pattern of $\lambda$ with $\arc$ on ``either side'' of $\tau$, and the endpoints of $\arc$ on a geometric train track $\tau_\epsilon\subset X$ constructed from $\lambda$ by a parameter $\epsilon>0$ as in Construction \ref{constr:stand_smooth} are stable as $\epsilon\to 0$. 
\end{remark}

A standard smoothing $\taua$ is reminiscent of the construction of completing $\lambda$ to a maximal lamination $\lambda'$ by ``spinning'' the arcs of $\arc$ around the boundary geodesics of complementary subsurfaces to $\lambda$ in the positive direction to obtain spiraling isolated leaves of $\lambda'$ in bijection with the arcs of $\arc$.
In Proposition \ref{prop:ttcoords} below, we observe that by smoothing $\arc$ onto $\tau$ in a standard way, axiom (SH3) allows us to assign weights to the branches of $\taua$ in such a way that the switch conditions are satisfied.
Thus, for a shear-shape cocycle carried by $\taua$,  the weights deposited on the branches $\arc\subset \taua$ encode ``shape'' data, rather than ``shear'' data.  
As such, we do not think of a standard smoothing as corresponding to the completion of $\lambda$ to a maximal lamination $\lambda'$.

\begin{proposition}
\label{prop:ttcoords}
Every shear-shape cocycle $(\sigma, \arcwt) \in \SH(\lambda)$ may be represented by a weight system $w_{\arc}(\sigma)$ on a 
standard
smoothing $\taua$ that also carries $\lambda$.
Moreover, the map $\sigma \mapsto w_{\arc}(\sigma)$ extends to a linear isomorphism
\[H^1(\widehat{N}_{\arc}, \partial \widehat{N}_{\arc}; \RR)^- \cong W(\tau_{\arc})\]
where $N_{\arc}$ is a neighborhood of $\lambda \cup \arc$, $\widehat{N}_{\arc}$ is its orientation cover, and $H^1(\widehat{N}_{\arc}, \partial \widehat{N}_{\arc}; \RR)^-$ is the $-1$ eigenspace for the covering involution $\iota^*$.
\end{proposition}

In particular, this isomorphism realizes $\SH(\lambda; \arc)$ and $\SH^+(\lambda, \arc)$ as convex cones (with some open and some closed faces) inside of $W(\taua)$.

\begin{proof}
Let $\taua$ be a standard smoothing of $\tau \cup \arc$ and for each branch $b$ of $\taua$, let $t_b$ denote a tie transverse to $b$. Evaluating a shear-shape cocycle $\sigma$ on $t_b$ yields an assignment of weights
\[w_{\arc}(\sigma): b \rightarrow \sigma(t_b).\]
By axiom (SH1) of Definition \ref{def:shsh_axiom}, this weight system does not depend on the choice of tie. 

To check that $w_{\arc}(\sigma)$ satisfies the switch conditions, we observe that there are two types of switches of $\taua$: those that come from switches of $\tau$ and those that come from smoothings of points of $\lambda \cap \arc$. Axiom (SH2) implies that the switch condition holds at each of the former, while axiom (SH3) together with our choice of smoothing ensures that $w_{\arc}(\sigma)$ satisfies the switch conditions at each of the latter. Compare Figure \ref{fig:SH3smooth}.

\begin{figure}[ht]
    \centering
\begin{tikzpicture}
    \draw (0, 0) node[inner sep=0] {\includegraphics{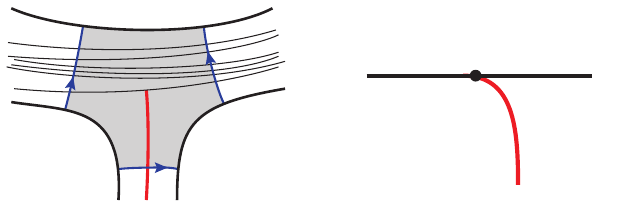}};
    \node at (-3.8, 1.6)[blue]{$k$};
    \node at (-1.9, 1.6)[blue]{$\ell$};
    \node at (-2, -1.1)[blue]{$t_i$};
    \node at (-2.8, -2.5){$[k] = [t_i]+[\ell]$};
    \node at (1.8,.9){$\sigma(k)$};
    \node at (4,.9){$\sigma(\ell)$};
    \node at (4,-.7){$\sigma(t_i)$};
    \node at (2.8, -2.5){$\sigma(k) = \sigma(t_i)+\sigma(\ell)$};
\end{tikzpicture}
    \caption{A standard smoothing of a geometric train track. The equation in homology encoded by axiom (SH3) becomes an additivity condition on the train track.}
    \label{fig:SH3smooth}
\end{figure}

We note that this discussion does not rely on the positivity of $\sigma$ on standard transversals, and so can be repeated to realize an arbitrary element of $H^1(\widehat{N}_{\arc}, \partial \widehat{N}_{\arc}; \RR)^-$ as a weight system on $\taua$.
\end{proof}

Let $\arcwt = \sum c_i \alpha_i$; then on any smoothing $\taua$ the identification of Proposition \ref{prop:ttcoords} restricts to an isomorphism
\[\SH(\lambda; \arcwt) \cong 
\{ w \in W(\tau_{\arc}) : w(b_i) = c_i \}\]
where $b_i$ is the branch of $\taua$ corresponding to $\alpha_i$. Indeed, these coordinates together with the parametrization of transverse cocycles by weight systems on $\tau \prec \taua$ (Proposition \ref{prop:tt_trans}) give another proof that the difference of any two shear-shape cocyles compatible with a given $\arcwt \in \Base$ is a transverse cocycle (Lemma \ref{lem:shsh_compat}).

\begin{remark}
The metric residue condition (Lemma \ref{lem:sum_res=0}) is still visible in train track coordinates, though it is somewhat obscured. Indeed, suppose that $\lambda$ contains an orientable component carried on a component $\zeta$ of the geometric train track $\tau$; fix an arbitrary orientation of $\zeta$.

Take a geometric standard smoothing $\taua$ of $\tau \cup \arc$. Reversing the tangential information as necessary, we can then construct a (non-standard) smoothing of $\tau \cup \arc$ so that every arc of $\arc$ is a small branch entering $\zeta$ according to the orientation.
Moreover, by reversing the sign of the weight on each arc which has its smoothing data modified, this non-standard smoothing still carries shear-shape cocycles as a weight systems.
But then by conservation of mass the total sum of the weights on the branches entering $\zeta$ must be $0$.

Hence in this setting the metric residue condition manifests as a condition embedded in the recurrence structure of smoothings.
\end{remark}

The extended intersection form on $\SH(\lambda)$ also has a nice formula in terms of train tracks. Let $\tau$ be a (trivalent) train track snugly carrying $\lambda$ and let $\taua$ be a standard smoothing of $\tau \cup \arc$; then for $\sigma \in \SH(\lambda)$ and $\rho \in \calH(\lambda)$, we have
\begin{equation}\label{eqn:ThSH_tt}
    \ThSH(\sigma, \rho) = \frac{1}{2} \sum_{s} \left| \begin{matrix}
    \sigma(r_s) & \sigma(\ell_s) \\
    \rho(r_s) & \rho(\ell_s)
    \end{matrix}\right| 
\end{equation}
where the sum is over all switches $s$ of $\taua$ and $r_s$/$\ell_s$ are the right/left small half-branches. The proof of this formula is the same as that of \eqref{eqn:ThH_tt} and is therefore omitted; the only thing to note in this case is that the value does not change if one completes $\arc$ by adding in arcs of zero weight.

\subsection{Piecewise-integral-linear structure}\label{subsec:PIL}
A piecewise linear manifold is said to be \emph{piecewise-integral-linear} or \emph{PIL} with respect to a choice of charts if the transition functions are invertible piecewise-linear maps with integral coefficients.  
The  track charts that we have constructed from standard smoothings in this section endow each cell $\SH(\lambda;\arc)$ with a PIL structure which clearly extends over all of $\SH(\lambda)$ (compare \cite[\S3.1]{PennerHarer}).

The points of the integer lattice in $W(\taua)$ are invariant under coordinate transformation, thus the \emph{integer points} $\SH_{\ZZ}(\lambda)\subset \SH(\lambda)$  are well defined.  

The PIL structure defined by train track charts gives a canonical measure $\mu_{\SH}$ in the class of the $(6g-6)$-dimensional Lebesgue measure on $\SH(\lambda)$. 
Namely, if $B\subset \SH(\lambda)$ is a Borel set, then 
\begin{equation}\label{eqn:interger_measure}
\mu_{\SH}(B): = \lim_{R\to \infty} \frac{\# R\cdot B \cap  \SH_{\ZZ}}{R^{6g-6}}.
\end{equation}
Since the symplectic intersection form $\ThSH$ is constant \eqref{eqn:ThH_tt} in a train track chart, the volume element defined by the $(3g-3)$-fold wedge product $\wedge \ThSH$ is a constant multiple of $\mu_{\SH}$ on each chart.  

We note that $\Base$ is cut out of $|\Arcfill(S\setminus\lambda)|$ by linear equations with integer coefficients, as is each cell of $|\Arcfill(S\setminus\lambda)|$.
Therefore, the integer lattice $\SH_{\ZZ}(\lambda)$ restricts to a integer lattice in the bundle $\SH(\lambda;\arc)$ over every cell $\mathscr B (\arc)$.  Thus we obtain a natural volume element on the bundle over the $k$-skeleton of $\Base$, whenever it is not empty.

\subsection{Duality in train track coordinates}\label{subsec:horospherical}

We now take a moment to discuss shear-shape coordinates from the point of view of train track weight spaces; this discussion is motivated by that in \cite{Th_stretch}, and is meant to clarify how shear-shape cocycles fit into the broader theory of train tracks.

We begin by recalling the analogy between shear coordinates for Teichm{\"u}ller space and the ``horospherical coordinates'' for hyperbolic space.
As observed by Thurston \cite[p. 42]{Th_stretch}, projecting the Lorentz model 
\[\HH^n = \{ x_1^2 + \ldots + x_n^2 - x_{n+1}^2 = -1 \mid x_{n+1} > 0\} \]
to $\langle x_1, \ldots, x_n\rangle$ along a family of parallel light rays gives a parametrization for $\HH^n$ in terms of a half-space. In these coordinates, horospheres based at the boundary point $\xi \in \partial_{\infty}\HH^n$ corresponding to the choice of light ray are mapped to affine hyperplanes and geodesics from $\xi$ are mapped to rays from the origin.
\footnote{We remark that this coordinate system is in some sense dual to the paraboloid model of \cite[Problem 2.3.13]{Th_book}. Horospherical coordinates place an observer looking out from the center of a family of expanding horospheres, whereas the paraboloid model places an observer at another boundary point looking in.}

When $\lambda$ is maximal and uniquely ergodic, Bonahon and Thurston's shear coordinates similarly realize $\T(S)$ as the space of positive transverse cocycles $\cH^+(\lambda)$, in which planes parallel to the boundary are level sets of the hyperbolic length of $\lambda$ and rays through the origin are Thurston geodesics.
Equivalently, if $\tau$ is a train track carrying $\lambda$ then shear coordinates identify $\T(S)$ as a half-space inside $W(\tau)$.

However, shear coordinates are no longer induced by a global projection. Instead, as noted by Thurston, they can be thought of as a map that takes a hyperbolic structure $X$ to (the 1-jet of) its length function with respect to a given lamination.
Shear coordinates are then a map not into $W(\tau)$ but into its dual space $W(\tau)^*$ (which can be identified with $W(\tau)$ via the non-degenerate Thurston symplectic form).
The image cone is then the positive dual
\footnote{i.e., those elements of $W(\tau)^*$ which pair positively with every element in $\Delta(\lambda)$ via the intersection form.}
of the cone of measures on $\lambda$.

This formalism then indicates how shear coordinates generalize to maximal but non-uniquely ergodic laminations. The map is the same, but now the positive dual of $\Delta(\lambda)$ has angles obtained from the intersection of hyperplanes: one for each ergodic measure on $\lambda$. Rays in the cone still correspond to geodesics, and affine planes parallel to the bounding planes correspond with the level sets of hyperbolic length of the ergodic measures on $\lambda$.

Our shear-shape coordinates come into play when $\lambda$ is not maximal. In this case, one can go through the above steps for each maximal train track $\tau$, obtained from a snug train tack carrying $\lambda$ by adding finitely many branches. Since $\lambda$ is carried on a proper subtrack of $\tau$ its cone of measures lives in a proper subspace $E \subset W(\tau)$.
Taking the positive dual of $\Delta(\lambda)$ and applying the isomorphism $W(\tau) \cong W(\tau)^*$ induced by the Thurston form then realizes Teichm{\"u}ller space as a cone $C$ in $W(\tau)$. By definition $C \cap E$ is exactly $\cH^+(\lambda)$, and one can check this demonstrates $C$ as an affine $\cH^+(\lambda)$ bundle.

However, the base of this bundle structure is not canonically determined, in part because $E < W(\tau)$ is generally not symplectic. Moreover, the same hyperbolic structure is parametrized by elements in many different maximal completions, and to achieve $\Mod(S)$-equivariance one needs to understand how to compare coordinates for different completions.
Shear-shape space is designed to solve both of these problems, picking out geometrically meaningful completions and gluing together the corresponding cones all while preserving the bundle structure.

Indeed, the shear-shape coordinates defined in Section \ref{sec:hyp_map} below associate to each hyperbolic structure a natural finite set of completions (corresponding to standard smoothings of snug train tracks plus geometric arc systems) together with a weight system on each completion.
The discussion of this section (Proposition \ref{prop:ttcoords} especially) then implies that the associated shear-shape cocycle is independent of the choice of completion, and that the corresponding train track charts glue together according to the combinatorics of $\Base$.
In this picture, level sets of the hyperbolic length now correspond to bundles over $\Base$ whose fibers are affine subspaces parallel to the boundary of $\cH^+(\lambda)$, while rays in $\SH^+(\lambda)$ correspond to scaling both the coordinate in $\Base$ as well as the coordinate in $\cH^+(\lambda)$.

\section{Shear-shape cooordinates for transverse foliations}\label{sec:flat_map}

We now show how the familiar period coordinates for a stratum of quadratic differentials can be reinterpreted as shear-shape coordinates.
The main construction of this section is that of the map
\[\Il : \Fol^{uu}(\lambda) \rightarrow \SH(\lambda)\]
which records the vertical foliation of a quadratic differential and should be thought of as a joint extension of \cite[Theorem 6.3]{MirzEQ} and \cite[Theorem 1.2]{MW_cohom}.

The idea is straightforward: given some quadratic differential $q \in \Fol^{uu}(\lambda)$, the complement $S \setminus Z(q)$ of its zeros deformation retracts onto a neighborhood $N_{\arc(q)}$ of $\lambda \cup \arc(q)$ for some filling arc system $\arc(q)$ (whose topological type reflects the geometry of $q$).
We may therefore identify the period coordinates of $q$ as a relative cohomology class in (the orientation cover of) $N_{\arc(q)}$ with complex coefficients. The imaginary part of this class corresponds to $\lambda$, while its real part is the desired shear-shape cocycle $\Il(q)$.

The only obstacle to this plan is in showing that $S \setminus Z(q)$ can actually be identified with a neighborhood of $\lambda \cup \arc(q)$.
To overcome this, we recall first in Section \ref{subsec:aq} how to reconstruct the topology of $S \setminus \lambda$ from the horizontal separatrices of $q$; this guarantees that all relevant objects have the correct topological types.
We then describe in Section \ref{subsec:flatmap} how to build from $S \setminus Z(q)$ a train track $\taua$ snugly carrying $\lambda \cup \arc(q)$ (Lemma \ref{lem:flattt_carries_lambda}); this in particular allows us to identify $S \setminus Z(q)$ as a neighborhood of $\lambda \cup \arc(q)$. We may then define $\Il(q)$ using the strategy outlined above and identify it as a weight system on $\taua$ (Lemma \ref{lem:pers_as_ttwts}).

Section \ref{subsec:Ilglobal} contains a discussion of the global properties of the map $\Il$: piecewise linearity, injectivity, and its behavior with respect to the intersection pairing. 
In this section, we also record Theorem \ref{thm:Ilhomeo}, which states that $\Il$ is a homeomorphism onto $\SH^+(\lambda)$. For purposes of convenience, the proof of this theorem is deduced from our later (logically independent) work on shear-shape coordinates for hyperbolic structures (Sections \ref{sec:hyp_overview}--\ref{sec:shsh_homeo}). See Remark \ref{rmk:Il_surjective}.

\subsection{Separatrices and arc systems}\label{subsec:aq}
Given a quadratic differential with $|\Im(q)|=\lambda$, our first task towards realizing $|\Re(q)|$ as a shear-shape cocycle is to build a filling arc system $\arc(q)$ on $S \setminus \lambda$ that encodes the horizontal separatrices of $q$.
We begin by recalling how to recover the topology of $S \setminus \lambda$ from the realization of $\lambda$ as a measured foliation on $q$.

Recall that a {\em boundary leaf} $\ell$ of a component of $S \setminus \lambda$ is a complete geodesic contained in its boundary.
\label{ind:boundaryleaves}
Note that infinite boundary leaves of $S \setminus \lambda$ are in 1-to-1 correspondence with leaves of $\lambda$ which are isolated on one side, while finite boundary leaves (i.e., closed boundary components) are in 2-to-1 correspondence with closed leaves of $\lambda$.
\footnote{This is true because we have insisted that $\lambda$ support a measure, and so no non-closed leaf may be isolated from both sides.}

The corresponding notion for measured foliations is that of {\em singular leaves}.
\label{ind:singleaves}
Let $\mathcal{F}$ be a measured foliation on $S$ and $\widetilde{\mathcal{F}}$ denote its full preimage to $\widetilde{S}$ under the covering projection; then a bi-infinite geodesic path of horizontal separatrices $\ell$ is a singular leaf of $\widetilde{\mathcal{F}}$ if for every saddle connection $s$ comprising $\ell$, the separatrices adjacent to $s$ leave from the same side of $\ell$ (i.e., always from the left or always from the right); see \cite[Figure 2]{Levitt}.

There is a fundamental correspondence between boundary leaves of a lamination and singular leaves of a foliation which we record below. Heuristically, collapsing the complementary regions of a lamination yields a foliation; the deflation map of Section \ref{subsec:deflation} is a geometric realization of this phenomenon.
Again, compare \cite[Figure 2]{Levitt} as well as \cite[Lemma 2.1]{Minsky_harmonic}.

\begin{lemma}\label{lemma:fol_lam_leaves}
Let $\lambda$ be a measured lamination on $S$ and let $\mathcal{F}$ be a measure-equivalent measured foliation.
Then there is a one-to-one, $\pi_1(S)$--equivariant correspondence between the boundary leaves of $\widetilde{S} \setminus \tilde \lambda$ and singular leaves of $\widetilde{\mathcal{F}}$.
Moreover, singular leaves of $\widetilde{\mathcal{F}}$ that share a common separatrix correspond to boundary leaves of the same component of $\widetilde{S} \setminus \tilde \lambda$.
\end{lemma}

This lemma in particular allows us to read off the topological type of $S \setminus \lambda$ from the horizontal separatrices of $q$.
Set $\Xi(q)$ to be the union of the horizontal separatrices of $q$, equipped with the path metric. This 1-complex also comes equipped with a ribbon structure (that is, a cyclic ordering of the edges incident to each vertex) and by thickening each component of $\Xi(q)$ according to this ribbon structure we see that $\Xi(q)$ can be regarded as a spine for the components of $S \setminus \lambda$. 
\label{ind:Xi(q)}

Our construction of $\arc(q)$ then records the dual arc system to the spine $\Xi(q)$ of $S \setminus \lambda$.

\begin{construction}\label{constr:aq}
Let $q$ be a quadratic differential on $S$ with $|\Im(q)| = \lambda$. 
By the correspondence of Lemma \ref{lemma:fol_lam_leaves}, each horizontal separatrix of $q$ corresponds to a pair of boundary leaves of the same component of $S \setminus \lambda$. 
Each infinite separatrix corresponds to a pair of asymptotic boundary leaves, while non-asymptotic boundary leaves are glued along horizontal saddle connections. 
Dual to each horizontal saddle connection of $\Xi(q)$ is a proper isotopy class of arcs on $S \setminus \lambda$, and we set $\arc(q)$ to be the union of all of these arcs.
\end{construction}

Since $\Xi(q)$ is a spine for $S \setminus \lambda$ and $\arc(q)$ consists of arcs dual to its compact edges, we quickly see that

\begin{lemma}\label{lem:aq_toptype}
The arcs of $\arc(q)$ are disjoint and fill $S \setminus \lambda$.
\end{lemma}
\begin{proof}
Each component of $\widetilde{S} \setminus \widetilde{\lambda}$ has a deformation retract onto the universal cover $\widetilde{\Xi}$ of a component of $\Xi(q)$. In particular, as the interiors of the edges of $\widetilde{\Xi}$ are disjoint, duality implies that the arcs of $\tilde{\arc}(q)$ can all be realized disjointly. As this picture is invariant under the covering transformation, this implies that the arcs are disjoint downstairs in $S \setminus \lambda$.

Similar considerations also imply that the arc system is filling: let $\cutsurf$ be a component of $S \setminus \lambda$ with universal cover $\widetilde{\cutsurf}$ with spine $\widetilde{\Xi}$.
By construction, the edges of $\tilde{\arc}(q)$ in $\widetilde{\cutsurf}$ are dual to the edges of $\widetilde{\Xi}$.
Since $\Xi(q)$ is a spine for $S \setminus \lambda$, any loop in $\cutsurf$ is homotopic to a union of saddle connections, implying that any nontrivial loop must pass through an edge of $\arc(q)$. 
Hence $\arc(q)$ fills $S \setminus \lambda$.
\end{proof}

\subsection{Period coordinates as shear-shape cocycles}\label{subsec:flatmap}
Now that we understand the relationship between $\lambda$ and the horizontal data of $q$, it is easy to build objects $\mathsf{T}^* \setminus \mathsf{H}^*$ and $\mathsf{T}^*$ on $q$ of the same topological type as $\lambda$ and $\lambda \cup \arc(q)$.
However, it is not immediate to actually identify these objects as neighborhoods of $\lambda$ and $\lambda \cup \arc(q)$.
Below, we deduce this from the stronger statement that they admit smoothings onto train tracks snugly carrying $\lambda$ and $\lambda \cup \arc(q)$; compare \cite[Sections 5.2 and 5.3]{MirzEQ}.

\begin{construction}[Train tracks from triangulations]\label{constr:ttfromtri}
Let $\mathsf{H}$ denote the set of all horizontal saddle connections on $q$ and let $\mathsf{T}$ be a triangulation of $q$ containing $\mathsf{H}$. Let $\mathsf{T}^*$ be the $1$-skeleton of the dual complex to $\mathsf{T}$ and let $\mathsf{H}^*$ denote the edges of $\mathsf{T}^*$ dual to $\mathsf{H}$. Note that $\mathsf{T}^*$ is trivalent by definition.

Let $\Delta$ denote a triangle of $\mathsf{T}$ with dual vertex $v_\Delta$ in $\mathsf{T}^*$. Using the $|q|$-geometry of $\Delta$ we may assign tangential data to $v_\Delta$ as follows (compare Figures \ref{fig:ttfromtri} and \ref{fig:tri_switch}).
\begin{itemize}
    \item If no edge of $\Delta$ is horizontal, then a unique edge $e$ has largest (magnitude of) imaginary part. Assign tangential data to $v_\Delta$ so that the dual edge to $e$ is a large half-branch.
    \item Otherwise, some edge of $\Delta$ is horizontal and the other two edges have the same imaginary parts. In this case, we choose tangential data so that the horizontal edge corresponds to a small half-branch and leaves the large half-branch from the right, as seen by the large half-branch.
\end{itemize}
We denote the resulting train track by $\tau_{\arc}$. The subgraph $\mathsf{T}^* \setminus \mathsf{H}^*$ can also be converted into a train track $\tau$ by deleting the branches of $\tau_{\arc}$ dual to $\mathsf{H}$.
\end{construction}

\begin{remark}
We observe that the edges of $\mathsf{H}^*$ correspond to the arcs of $\arc(q)$ and $\taua$ is a standard smoothing of $\tau \cup \arc(q)$.
Our convention for ``standard'' ensures that additivity in period coordinates corresponds to additivity in train track coordinates.
\end{remark}

\begin{figure}[ht]
    \centering
    \includegraphics{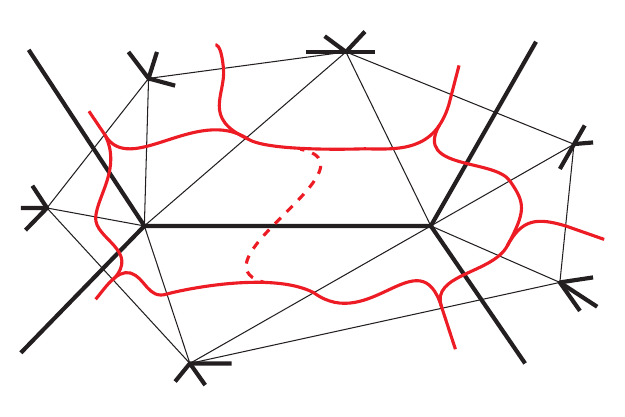}
    \caption{An example of the train track $\taua$ around a saddle connection. The thick black lines are stems of horizontal separatrices of $q$ while the light black lines are non-horizontal edges of the triangulation $\mathsf{T}$. The dashed line is a branch of $\taua \setminus \tau$.}
    \label{fig:ttfromtri}
\end{figure}

By construction, the graph $\mathsf{T}^*$ (equivalently, the train track $\taua$) is a deformation retract of $S \setminus Z(q)$. Similarly, $\mathsf{T}^* \setminus \mathsf{H}^*$ (and $\tau$) are deformation retracts of the complement of the horizontal saddle connections.
Together with our discussion above, this implies that $\tau$ has the same topological type as $\lambda$ and $\taua$ has the same topological type as $\lambda \cup \arc(q)$.

In order to actually realize these objects as neighborhoods of $\lambda$, we observe that we can build an explicit carrying map from (a foliation measure equivalent to) $\lambda$ onto $\tau$.

\begin{lemma}\label{lem:flattt_carries_lambda}
The train track $\tau$ carries $\lambda$ snugly. The weight system on $\tau$ that specifies $\lambda$ is exactly the (magnitude of) the imaginary parts of the periods of the edges of $\mathsf T$.
\end{lemma}
\begin{proof}
Let all notation be as above and let $\mathcal F$ denote the (singular) horizontal foliation of $q$. 

One can directly build a homotopy of the nonsingular leaves of $\mathcal F$ onto $\tau$: in a neighborhood of each edge $e$ of $\mathsf T \setminus \mathsf H$ there is a homotopy of the leaves of $\mathcal F$ onto the branch of $\tau$ dual to $e$.
Now any leaf of $\mathcal F$ which passes through a triangle $\Delta$ of $\mathsf T$  does so (locally) only twice and must pass through the side of $\Delta$ with the largest imaginary part, which corresponds to a large half-branch of $\tau$.
The complement of the separatrix meeting the vertex of $\Delta$ opposite to the side with largest imaginary part separates the (locally) non-singular leaves of $\mathcal F$ passing through $\Delta$ into two packets that can be homotoped onto $\tau$, respecting the smooth structure at the switch dual to $\Delta$; compare Figure \ref{fig:tri_switch}.

Now the horizontal foliation $\mathcal F$ of $q$ is measure equivalent to $\lambda$, and so as $\tau$ carries $\mathcal F$ it carries $\lambda$ (snugness follows as $\tau$ and $\lambda$ have the same topological type). The statement about the weight system follows from our description of the carrying map.
\end{proof}

Now that we have identified $\tau$ as a snug train track carrying $\lambda$, we may in turn identify a neighborhood of $\lambda \cup \arc(q)$ with (a thickened neighborhood of) $\taua$. With this correspondence established, we may now define $\Il(q)$ as the image of the real part of the period coordinates of $q$ under the natural isomorphism on cohomology.

\begin{construction}[Definition of $\Il(q)$]\label{constr:defIl}
Let $S, \lambda, q, \arc(q)$, and $\taua$ be as above, Set $M_{\arc}$ to be a thickened neighborhood of $\taua$ (in the flat metric defined by $q$) and let $N_{\arc}$ be a snug neighborhood of $\lambda \cup \arc(q)$ (taken in some auxiliary hyperbolic metric).
Perhaps by shrinking $N_{\arc}$, we may assume it embeds into $M_{\arc}$ as a deformation retract (this follows by snugness).

Now $\taua$ is itself a deformation retract of $S \setminus Z(q)$, so the inclusion $M_{\arc}\hookrightarrow S\setminus Z(q)$ is a homotopy equivalence;  composing inclusions
$N_{\arc}\hookrightarrow M_{\arc} \hookrightarrow S\setminus Z(q)$
and lifting to the orientation covers yields the isomorphism
\begin{equation}\label{eq:Ildef}
H^1(\widehat{S}, Z(\sqrt{q}); \CC)
\xrightarrow{j^*}
H^1(\widehat{N}_{\arc}, \partial \widehat{N}_{\arc}; \CC)
\end{equation}
where the hats denote the corresponding orientation covers. As the composite retraction respects the covering involution $\iota$, this isomorphism also identifies $-1$ eigenspaces for $\iota^*$. We therefore define
\[\Il(q) = \Re( j^* \Per(q))\]
where $\Per(q)$ are the period coordinates for $q$, and where the real part is taken relative to the natural splitting $\CC = \RR \oplus i\RR$.
\end{construction}

\begin{remark}
From the above construction, a basis consisting of branches for the weight space of $\taua$ (equivalently, a basis for $H_1(\widehat{N}_{\arc}, \partial \widehat{N}_{\arc};\ZZ)$ of dual arcs)
picks out a  basis for $H_1(\widehat{S}, Z(\sqrt q);\ZZ)$.
Moreover, each relative cycle is realized geometrically as a saddle connection (as opposed to concatenations, thereof).
\end{remark}

To see that $\Il(q)$ is indeed a shear-shape cocycle, we need only observe that the values on standard transversals to $\arc(q)$ are all positive. This follows essentially by definition of the orientation cover and construction of $\arc(q)$.
To wit: if $\alpha$ is an arc of $\arc(q)$ dual to a saddle connection $s$, and $t$ is a standard transversal to $\alpha$, then the canonical lifts of $t$ are mapped to those of $s$ under the isomorphism \eqref{eq:Ildef}.
As the periods of $\sqrt{q}$ increase as you move along the (oriented) horizontal foliation of $(\widehat{S}, \sqrt{q})$, this implies that the value of $\Il(q)$ on either of the lifts of $t$ is exactly the length of the saddle connection $s$.

Therefore, we see that the weighted arc system underlying $\Il(q)$ is none other than 
\[\arcwt(q) := \sum_{\alpha \in \arc(q)} c_\alpha \alpha\]
where $c_\alpha$ is the $|q|$-length of the horizontal saddle connection dual to the arc $\alpha$.
\label{ind:arcwtq}

\begin{remark}\label{rmk:natural_construction}
Naturality of all of the isomorphisms involved quickly implies that this construction does not depend on the choice of initial triangulation $\mathsf T$.
Indeed, suppose that $\mathsf{T}_1$ and $\mathsf{T}_2$ are two triangulations giving rise to train tracks $\tau_1$ and $\tau_2$ and hence shear-shape cocycles $\sigma_1$ and $\sigma_2$. Since both $\tau_i$ carry $\lambda \cup \arc(q)$ snugly, Lemma \ref{lem:flattt_carries_lambda} implies that they have a common refinement $\tau$. Lifting the inclusions 
\[N(\tau \cup \arc(q)) \hookrightarrow
N(\tau_i \cup \arc(q)) \hookrightarrow
S \setminus Z(q)\]
to their orientation covers and drawing the appropriate commutative diagram of cohomology groups, we see that the shear-shape cocycles built from each $\mathsf{T}_i$ coincide as weight systems on the common refinement $\tau$.
\end{remark}

For use in the sequel, we record below the weight systems on $\taua$ corresponding to $\lambda$ and $\Il(q)$.
The proof follows by combining the constructions above with the discussion in Section \ref{sec:tt_shsh} and is therefore left to the scrupulous reader. See also Figure \ref{fig:tri_switch}.

For a complex number $z$, define
\[[z]_+ = \left\{
\begin{array}{cl}
     z & \text{if } \arg(z) \in [0,\pi) \\
     -z & \text{if } \arg(z) \in [\pi, 2\pi).
\end{array} \right.\]
\label{ind:[z]+}
Observe that $[z]_+ = [-z]_+$ for all $z \in \CC$.

\begin{lemma}\label{lem:pers_as_ttwts}
Let all notation be as above, and for each edge $e$ of $\mathsf{T}$ let $b_e$ denote the branch of $\taua$ dual to it. Then the assignment
\[b_e \mapsto \left[ \int_e \sqrt{q} \right]_+\]
defines a complex weight system $w(q)$ on $\taua$ satisfying the switch conditions. 
Moreover,
\[\Im(w(q))=\lambda \text{ and }\Re(w(q)) = \Il(q).\]
\end{lemma}

\begin{figure}[ht]
    \centering
    
\begin{tikzpicture}
    \draw (0, 0) node[inner sep=0] 
    {\includegraphics{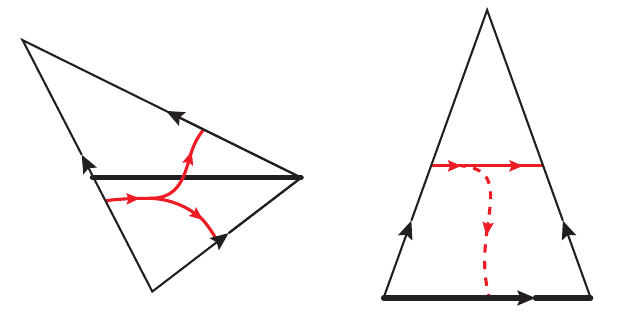}};
    \node at (-3.7, -.8){\color{red} $b_s$};
    \node at (-1.7, .7){\color{red} $\ell_s$};
    \node at (-1.4, -1.7){\color{red} $r_s$};
    \node at (0, -.4){$o$};
    \node at (-5.15, 2){$o_\ell$};
    \node at (-3, -2.3){$o_r$};
    \node at (1.7, -.2){\color{red} $b_s$};
    \node at (4.2, -.2){\color{red} $\ell_s$};
    \node at (2.7, -2){\color{red} $r_s$};
    \node at (2.6, 2.4){$o_{\ell}$};
    \node at (1, -2.3){$o_r$};
    \node at (4.9, -2.3){$o$};
\end{tikzpicture}
    \caption{Local pictures of the different types of switches of $\taua$. Here we have illustrated the images of each triangle under the holonomy map. The orientation of each edge should be interpreted as indicating the value of $[\,\cdot\,]_+$ so that the edge vector is exactly the complex weight assigned to the dual branch of $\taua$.
    The graphical conventions of this figure mirror those of Figure \ref{fig:ttfromtri}.}
    \label{fig:tri_switch}
\end{figure}

\subsection{Global properties of the coordinatization}
\label{subsec:Ilglobal}

In this section, we show that the map $\Il$ defined above gives a global coordinatization of $\MF(\lambda) \cong \Fol^{uu}(\lambda)$. First, we record certain global properties of this map; as it is defined by reinterpreting period coordinates as shear-shape cocycles, it preserves many of the structures imposed by period coordinates.

For example, it follows by construction that $\Il$ respects the stratification of each space. That is, if $q \in \QT(k_1, \ldots, k_n) \cap \Fol^{uu}(\lambda)$, then the spine dual to $\arc(q)$ has vertices of valence $k_1+2, \ldots, k_n+2$.
In a similar vein, since both $\Fol^{uu}(\lambda)$ and $\SH(\lambda)$ have local cohomological coordinates (which induce PIL structures) we can deduce the following:

\begin{lemma}\label{lem:Il_basics}
For any $\lambda \in \ML(S)$, the map
$\Il$ is $\Mod(S)[\lambda]$--equivariant and PIL.
\footnote{We recall that a PL map between PIL manifolds is itself PIL if it sends integral points to integral points.}
\end{lemma}

\begin{proof}
Equivariance follows from the naturality of our construction: all combinatorial data (arc systems, train tracks, etc.) can be pulled back to a reference surface equipped with $\lambda$, so changing the marking by an element of $\Mod(S)[\lambda]$ acts by transforming the combinatorial data on the reference surface.

The piecewise-linear structure on $\Fol^{uu}(\lambda)$ (respectively, $\SH(\lambda)$) is given by period coordinates (respectively, cohomological coordinates in a neighborhood/train track coordinates) and so the map is by construction piecewise-linear. Integrality comes from the fact that a homotopy equivalence induces an isomorphism on cohomology with $\ZZ$-coefficients, hence takes integral points to integral points. 
\end{proof}

The Thurston intersection pairing gives us a powerful tool to understand constraints on the image of $\Il$; in particular, $\Il(q)$ must be a {\em positive} shear-shape cocycle.
Indeed, the tangential structure of the train track $\taua$ at each switch provides us with an identification of each triangle $\Delta$ of $\mathsf{T}$ with an oriented simplex. With respect to this orientation, we can compute the area of $\Delta$ by taking (one half of) the cross product of two of its sides.
Comparing the formula for the cross product with the Thurston intersection pairing \eqref{eqn:ThSH_tt} then allows us to see that the intersection of $\lambda$ and $\Il(q)$ is exactly the area of $q$; compare \cite[Lemma 5.4]{MirzEQ}.

\begin{proposition}\label{prop:Il_takes_int_to_Thurston}
For all $\eta \in \MF(\lambda)$ and all $\mu \in \Delta(\lambda)$,
\[
\ThSH(\Il(\eta), \mu)
= i(\eta, \mu).
\]
In particular, $\Il ( \MF(\lambda) ) \subseteq \SH^+(\lambda).$
\end{proposition}

The proof of this lemma is made technical by the fact that if $\mu$ and $\mu' \in \Delta(\lambda)$ are ergodic but not projectively equivalent then they are mutually singular. 
To deal with this difficulty, we build a flat structure on the subsurface filled by $\mu$ by integrating against $\lambda + t\mu$ and $\Il(\eta)$ for small $t$.
The triangulation $\mathsf{T}$ then induces a combinatorially equivalent triangulation of this new flat structure by saddle connections, allowing us to compare the area of this new flat metric (computed via cross products) with the Thurston form on our original train track $\taua$. This inverse construction will also be used in the proof of Proposition \ref{prop:Il_inj}.

\begin{proof}
We begin by observing that since $\mu \in \Delta(\lambda)$, there is a union of minimal components of the horizontal foliation of $q(\eta, \lambda)$ that supports $\mu$. Call this subfoliation $\mathcal{F}$ and let $Y$ denote the subsurface filled by $\mathcal{F}$ on $q(\eta, \lambda)$.
Note that $\partial Y$ must be a union of horizontal saddle connections, hence is contained in any triangulation $\mathsf{T}$ used to define $\taua$.
In particular, $\mathsf{T}|_Y$ is a triangulation of $Y$. 

Since $\eta$ and $\lambda$ are realized transversely on $q(\eta, \lambda)$ and this specific realization of $\eta$ is non-atomic (as any closed leaves of $\eta$ have become vertical cylinders), we can compute the intersection number between $\eta$ and any measure $\mu$ supported on $\mathcal{F}$ as 
\begin{equation}\label{eqn:intersection_1}
i(\eta, \mu) 
= \int_S \eta \times \mu
= \int_Y \eta \times \mu.
\end{equation}

We now build a new flat structure on $Y$ whose conical singularities coincide with those of $Y$; the salient feature is that $\mathsf{T}|_Y$ can be straightened out to a triangulation by saddle connections on the new singular flat structure that reflects the geometry of $\lambda + t\mu$.
To construct the new singular flat structure, we build charts from a neighborhood of each triangle $\Delta \subset \mathsf{T}|_Y$ to $\CC$ and describe the transitions.

Each triangle $\Delta$ of $\mathsf T$ is dual to a switch $s$ with an edge that is dual to a large half-branch $b$ incident to $s$.  Orient $\taua\cap \Delta$ so that a train traveling along $b$ toward $s$ is moving in the positive direction.
The other edges $r$ and $\ell$ of $\Delta$ are dual to the half-branches of $\taua$ to the right and left of $s$, respectively.  The vertices $o_r$, $o_\ell$ are adjacent to $r$ and $\ell$, respectively, and the vertex $o$ is opposite $b$; see Figure \ref{fig:tri_switch}.
On the interior of each triangle $\Delta$, we orient the leaves of $\mathcal F$ parallel to $b$.  The leaves of $\eta$ are given the orientation so that the ordered basis of tangent vectors to $\lambda$ and $\eta$ at each point agree with the underlying orientation of $S$.  With this orientation, the measures $\eta$ and $\lambda$ induce smooth real $1$-forms $d\eta$ and $d\lambda$  that look locally like $dx$ and $dy$, respectively (as opposed to $|dx|$ and $|dy|$, respectively).

Restricted to the interior of $\Delta$, the local orientation of the leaves of $\eta$ also gives the measure $\mu$ the structure of a measurable $1$-form that we call $d\mu$. Spreading out the measure on a closed leaf of $\mu$ over the horizontal cylinder of $\lambda$ corresponding to its support as necessary, we get that the map 
\[F_t:p\in \Delta \mapsto \int_{\gamma_p} d\eta +id(\lambda + t\mu) \in \CC\]
obtained by integrating along a path $\gamma_p$ from $o_r$ to $p$ is isometric along leaves of $\mathcal F$ and non-decreasing along leaves of $\eta$. 
We compute
\[F_t(o) = I_\lambda (\eta)(r)+i(\lambda + t\mu)(r)
\text{ and }
F_t(o_\ell) = I_\lambda(\eta)(b)+ i (\lambda + t\mu)(b).\] 
Transverse invariance and additivity of $\mu$ gives
\begin{equation} \label{eqn:F_switch}
F_t(o_\ell) - F_t(o) = I_\lambda(\eta)(\ell)+i(\lambda + t\mu)(\ell).
\end{equation}
Since the pair $(F_0(o), F_0(o_\ell))$ forms a positively ordered basis for $\CC$ (equivalently, since the triangle $\Delta$ is positively oriented), the pair $(F_t(o), F_t(o_\ell))$ is also positively oriented for small enough $t$. 
Let $\Delta'_t$ be the convex hull of $(F_t(o_r), F_t(o), F_t(o_\ell))$.

The area of $\Delta'_t$ may now be computed as half the cross product of $F_t(o)$ and $F_t(o_\ell)$. 
Using equation \eqref{eqn:F_switch} and linearity of the cross product, we have the formula
\begin{equation}\label{eqn:cross_area}
    \Area(\Delta'_t) = \frac{1}{2} 
    \left| \begin{matrix}
    \Il(\eta)(r) & \Il(\eta)(\ell) \\
    \lambda + t\mu(r) & \lambda + t\mu(\ell)
    \end{matrix}\right| > 0.
\end{equation}

Now for each $\Delta$ and any small enough $t$ the map $F_t$ may be extended to an open set $U(\Delta)$ in $Y \setminus Z(q)$ that contains $\Delta$ (minus its vertices) and so that for every $p \in U(\Delta)$ there is a unique non-singular $|q|$-geodesic segment $\gamma_p$ joining $o_r$ to $p$.
We claim that moreover, we may choose $U(\Delta)$ so that $\Delta'_t \subset F(U(\Delta))$; see Figure \ref{fig:cantor_tri}.

\begin{figure}[ht]
\centering
\begin{tikzpicture}
    \draw (0, 0) node[inner sep=0] 
    {\includegraphics{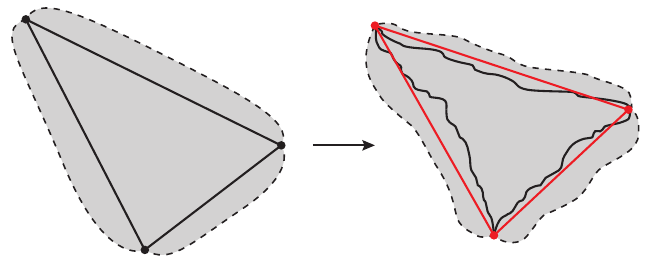}};
    \node at (-2.9, -.3){\large $\Delta$};
    \node at (-2, 1.5){$U(\Delta)$};
    \node at (-1.1, -.3){$o$};
    \node at (-2.9, -1.7){$o_r$};
    \node at (-4.5, 1.4){$o_\ell$};
    \node at (.3, -.6){\large $F_t$};
    \node at (3.5, .3)[red]{\large $\Delta'_t$};
\end{tikzpicture}
    \caption{Integrating against $\eta$ and $\lambda + t\mu$ defines a new flat structure on triangles. These charts piece together to give a new half-translation structure on the subsurface filled by $\mu$.}
    \label{fig:cantor_tri}
\end{figure}

If not, there is some vertex $v$ of $\mathsf T_Y\setminus \Delta$ such that $F_t(v)\in \Delta'_t \setminus F_t(\Delta)$.  Indeed, by construction, $U(\Delta)$ is a star-shaped neighborhood about the vertex $o_r$ of $\Delta$, so there is a saddle connection joining $o_r$ to $v$. 
This saddle connection passes through or shares a vertex of an edge $e$ of $\Delta$.  Moreover, we may find $v$ so that the triangle $\Delta_v$ formed by $e$ and $v$ is singularity free and contained in $U(\Delta)$.  But now, the straightening $\Delta_v'$ of $F_t(\Delta_v)$ in $\CC$ lies inside $\Delta'_t$ with the wrong orientation since $F_t(v)$ lies between $F_t(e)$ and the corresponding edge of $\Delta'_t$. This is a contradiction to the fact that $F_t$ is non-decreasing along leaves of $\eta$, alternatively, to the fact that the straightenings $\Delta'_t$ are all positively oriented for small enough $t$. So we may assume that $\Delta_t'\subset F(U(\Delta))$.

If $\Delta_1\subset \mathsf T_Y$ shares an edge with $\Delta$, then the construction of the map $F_t$ on $\Delta_1$ agrees with $F_t$ on $U(\Delta)\cap U(\Delta_1)$ up to multiplication by $\pm1$ (depending on the configuration of the switches dual to $\Delta$ and $\Delta_1$) and translation by the period of the arc connecting the basepoints $o_r$ of each triangle.
Thus these triangles glue up to a half-translation structure on $Y\setminus Z$  equipped with a triangulation by saddle connections corresponding to $\mathsf T|_Y$.

In our new flat structure on $Y$, $\lambda + t\mu$ is measure equivalent to the horizontal foliation and (the restriction of) $\eta$ is equivalent to the vertical foliation. Hence we obtain for any $t$ small enough that
\[\int_Y\eta\times (\lambda + t\mu)
= \sum_{\Delta \in \mathsf{T}|_Y}\Area(\Delta_t')
= \sum_{\Delta \in \mathsf{T}|_Y} \frac{1}{2} 
    \left| \begin{matrix}
    \Il(\eta)(r) & \Il(\eta)(\ell) \\
    \lambda + t\mu(r) & \lambda + t\mu(\ell)
    \end{matrix}\right| = \ThSH(\Il(\eta), \lambda + t\mu)\]
where the second equality follows from \eqref{eqn:cross_area} and the third from \eqref{eqn:ThSH_tt}.
Combining this with formula \eqref{eqn:intersection_1} and the linearity of the Thurston intersection form (Lemma \ref{lem:ThSHprops}), we get that
\[i(\eta, \mu)
= \frac{1}{t}\left( \int_Y \eta\times \lambda + t\mu - \int_Y \eta \times \lambda \right)
= \frac{1}{t}\big( \ThSH(\Il(\eta), \lambda + t\mu) - \ThSH(\Il(\eta), \lambda) \big)
= \ThSH(\Il(\eta), \mu),\]
completing the proof of the lemma.
\end{proof}

From the proof of Proposition \ref{prop:Il_takes_int_to_Thurston} we can also extract the following, which allows us to reconstruct a (triangulated) quadratic differential from a sufficiently positive shear-shape cocycle, inverting Construction \ref{constr:ttfromtri}.

\begin{lemma}\label{lem:every_switch_positive}
Let $\tau$ be a train track snugly carrying $\lambda$ and let $\taua$ be a standard smoothing of $\lambda \cup \arc$.
Suppose that $\sigma \in \SH(\lambda)$ is represented by a weight system on $\taua$ so at every switch $s$ of $\taua$, the contribution
\[\frac{1}{2} \left| \begin{matrix}
    \sigma(r_s) & \sigma(\ell_s) \\
    \lambda(r_s) & \lambda(\ell_s)
    \end{matrix}\right|\]
of $s$ to $\ThSH(\sigma, \lambda)$ is positive.
Then there exists a quadratic differential $q \in \Fol^{uu}(\lambda)$ so that $\Il(q) = \sigma$ and the dual triangulation to $\taua$ is realized by saddle connections on $q$.
\end{lemma}
\begin{proof}
The assumption that the contribution at each switch is positive implies that the basis $(F(o), F(o_\ell))$ is positively oriented at each switch, and so we can build a positively-oriented triangle $\Delta$ with the prescribed periods. These glue together into the desired quadratic differential.
\end{proof}

In particular, we can locally invert $\Il$ by building a quadratic differential out of triangles whose edges have specified periods, so we see that $\Il$ is injective.

\begin{proposition}\label{prop:Il_inj}
For any $\lambda \in \ML(S)$, the map 
$\Il$ is a homeomorphism onto its image.
\end{proposition}
\begin{proof}
To see that $\Il$ is injective, we observe that Lemma \ref{lem:every_switch_positive} provides a (left) inverse map $\Delta_\lambda$ to $\Il$.
Indeed, suppose that $\sigma = \Il(q)$ for some $q$ and pick a triangulation $\mathsf{T}$ as in Construction \ref{constr:ttfromtri}; let $\taua$ denote the dual train track. Applying Lemma \ref{lem:every_switch_positive} then constructs a quadratic differential $q'$ on which each edge of $\mathsf{T}$ is realized as a saddle connection. Since $q$ and $q'$ have the same periods with respect to the same geometric triangulation, they must be equal.

To prove that $\Il$ is continuous, we first observe that $\Il$ is by definition continuous on the closure $\SH(\lambda; \arc(q))$ of any cell, as it is induced by a continuous mapping on the level of cohomology.  
In general, we need only exploit this fact together with a standard reformulation of sequential continuity: a function $f: X \rightarrow Y$ is continuous if and only if every convergent sequence $x_n \rightarrow x$ has a subsequence $x_{n_k}$ so that $f(x_{n_k}) \rightarrow f(x)$.

So let $q_n \rightarrow q \in \Fol^{uu}(\lambda)$.  The polyhedral structure of $\SH(\lambda)$ is locally finite, so for $n$ large enough, $\Il(q_n)$ is contained in a finite union of cells.
After passing to a subsequence $q_{n_k}$, we may assume that $q_{n_k}$ all share the same underlying (maximal) arc system $\arcb$ completing $\arc$. 
In particular, $\Il( q_{n_k}) \in  \SH(\lambda;\arcb)$ for all $k$ and so $\Il( q_{n_k}) \rightarrow \Il(q)$ follows from continuity on cells. 
Therefore $\Il$ is a continuous injective map between Euclidean spaces of the same dimension (Proposition \ref{prop:SH+_structure} and Corollary \ref{cor:shsh_dimension}) and so invariance of domain guarantees it is a homeomorphism onto its image.
\end{proof}

\para{The image of $\Il$}
In light of Lemma \ref{lem:every_switch_positive}, to show that $\Il$ surjects onto $\SH^+(\lambda)$ it would suffice to show that every positive shear-shape cocycle can be realized as a weight system on a train track where every switch contributes positively to the intersection form.
However, it is rather complicated to show that every positive shear-shape cocycle admits such a representation (see the discussion in Remark \ref{rmk:Il_surjective} just below).

Instead, we deduce this fact using the commutativity of Diagram \eqref{diagram} and the results appearing in Sections \ref{sec:hyp_overview}--\ref{sec:shsh_homeo} coordinatizing hyperbolic structures by shear-shape cocycles.
We emphasize, however, that Theorem \ref{thm:Ilhomeo} is logically independent from the work done in Sections \ref{sec:hyp_overview}--\ref{sec:shsh_homeo} that leads to its proof.
We include the statement here (as opposed to after Section \ref{sec:shsh_homeo}) to provide some closure to our discussion of the  parametrization of $\MF(\lambda)$ by shear-shape coycles.

\begin{theorem}\label{thm:Ilhomeo}
The map $I_\lambda: \Fol^{uu}(\lambda) \to \SH^+(\lambda)$ is a homeomorphism.
\end{theorem}
\begin{proof}
In Section \ref{sec:hyp_map} we define the geometric shear-shape cocycle $\sigl(X)\in \SH(\lambda)$ associated to a hyperbolic metric $X\in \T(S)$ and show (Theorem \ref{thm:diagram_commutes}) that $\sigl(X) = \Il(\Ol(X))$.  
In Section \ref{sec:shsh_homeo} we prove Theorem \ref{thm:hyp_main}, which states that the map $\sigl: \T(S) \to \SH^+(\lambda)$ is a homeomorphism. In particular, $\sigl$ is surjective and hence so is $\Il$. Together with
Proposition \ref{prop:Il_inj} this implies the theorem.
\end{proof}

\begin{remark}\label{rmk:Il_surjective}
If $\lambda$ is a maximal lamination, one can deduce surjectivity of $\Il$ by appealing to the theory of ``tangential coordinates'' for measured foliations transverse to $\lambda$. 
In general, given $\tau$ snugly carrying $\lambda$,
tangential coordinates can be constructed as a quotient of $\RR^{b(\tau)}$ by a vector subspace spanned by vectors that model the change of length of branches of a train track on either side of a switch after a small ``fold'' or ``unzip.'' When $\lambda$ is maximal, there is a linear isomorphism from shear coordinates to tangential coordinates via the symplectic pairing $\ThH$; we refer the interested reader to  \cite[Section 9]{Th_stretch} or \cite[\S3.4]{PennerHarer} for details.

The transverse weights defined by the measure of $\lambda$ on $\tau$ together with positive
\footnote{Here, positive means that there is a representative of the tangential data that is positive on each branch of $\tau$.}
tangential data give $\tau$ the structure of a bi-foliated Euclidean band complex.  
If the tangential data satisfy a collection of triangle-type inequalities, this band complex can be ``zipped up'' to obtain a bi-foliated flat surface with conical singularities.
When defined, the linear transformation mapping tangential coordinates to shear coordinates preserves the intersection number, hence positivity.

A standard positivity argument (see \cite[Proposition 9.7.6]{Thurston:notes} or  \cite[Theorem 9.3]{Th_stretch}) shows that any tangential data with positive intersection with $\lambda$ has a positive representative, hence defines a foliation transverse to $\lambda$. In particular, the map from $\MF(\lambda)$ to the space of tangential coordinates with positive intersection with $\lambda$ is surjective. As the space of tangential coordinates with positive intersection is isomorphic to the $\cH^+(\lambda)$, this completes the proof of surjectivity in the maximal case.

This being considered, even in the case when $\lambda$ is maximal ``it is harder to see the [positivity] inequalities satisfied by the shear coordinates [than the tangential coordinates]'' \cite[p. 45]{Th_stretch} and it is not clear how to run the ``standard positivity argument'' without passing through tangential coordinates. We have therefore chosen to prove Theorem \ref{thm:Ilhomeo} in a way that avoids developing a theory of tangential coordinates dual to shear-shape cocycles.
Instead, we take advantage of the relationship between the Thurston intersection form on $\SH(\lambda)$ and the length of $\lambda$ on a given hyperbolic surface, as exploited in the proof of Theorem \ref{thm:hyp_main} (see in particular Claim \ref{clm:measure}).
\end{remark}

\section{Flat deformations in shear-shape coordinates} \label{sec:flatflows}

The identification of Section \ref{sec:flat_map} between periods of saddle connections and the values of the shear-shape cocycle $\Il(q)$ immediately allows us to transport certain flows on $\Fol^{uu}(\lambda)$ to shear-shape space. Moreover, Theorem \ref{thm:Ilhomeo} affords a new perspective on the ``tremor deformations'' of \cite{CSW} (see Definition \ref{def:Il_trem}).

\para{The horizontal stretch}
We begin by observing that the space $\SH^+(\lambda)$ carries a natural $\RR_{>0}$ action given by scaling both the underlying arc system $\arcwt$ and the values assigned to test arcs (equivalently, the corresponding cohomology class or the weights on a train track realization).
Using our correspondence between period coordinates and shear-shape cocycles (Lemma \ref{lem:pers_as_ttwts}), we see that this dilation expands the real part of each period, so the corresponding flat deformation is just a horizontal stretch.
\footnote{This is just the Teichm{\"u}ller geodesic flow normalized so that the horizontal foliation remains constant. Applying the standard geodesic flow takes $(\Il(q), \lambda)$ to $(e^{t/2}\Il(q), e^{-t/2}\lambda)$.}

\begin{lemma}\label{lem:Il_geo}
Let $q \in \Fol^{uu}(\lambda)$; then 
\begin{equation}\label{eqn:Il_geo}
\Il\left( \begin{pmatrix}
    e^t & 0 \\
    0 & 1
\end{pmatrix}
q
\right) = e^t \Il(q)
\end{equation}
for all $t \in \RR$.
\end{lemma}

In particular, we see that our coordinatization linearizes the expansion of the strong unstable foliation under the Teichm{\"u}ller geodesic flow.

\para{Horocycle flow and tremors}
We now consider the horocycle flow on $\Fol^{uu}(\lambda)$, which is just the restriction of the standard horocycle flow $h_s$ to the strong unstable leaf. An easy computation shows that for every saddle connection $e$ of $q$, one has
\begin{equation}\label{eqn:pers_hor}
\left[\int_e \sqrt{h_s q} \right]_+ = 
\left( \Re \left[\int_e \sqrt{q} \right]_+ + s \Im  \left[\int_e \sqrt{q} \right]_+ \right) 
+ i \Im \left[\int_e \sqrt{q} \right]_+
\end{equation}
(here we have invoked the $[\cdot]_+$ function to avoid fussing over square roots and orientations).

With the help of Lemma \ref{lem:pers_as_ttwts} we may translate this into the language of transverse and shear-shape cocycles to observe

\begin{lemma}\label{lem:Il_hor}
The map $\Il$ takes horocycle flow to translation by $\lambda$ in a time preserving way. In symbols,
\[\Il(h_s q) = \Il(q) + s \lambda.\]
\end{lemma}

More generally, we can perform a similar deformation for {\em any} measure $\mu$ supported on $\lambda$, resulting in the {\em tremor flow} along $\mu$.
First defined by Chaika, Smillie, and Weiss in the context of Abelian differentials, the {\em tremor} $\trem_{\mu}(q)$ of a quadratic differential $q = q(\eta, \lambda)$ by a measure $\mu \in \Delta(\lambda)$ is the unique quadratic differential specified by shearing $\eta$ by $\mu$ and leaving $\lambda$ fixed.
\label{ind:tremors}
Why this makes sense (note that $\eta$ and $\mu$ may not fill $S$) and why it can be continued for all time present significant technical challenges in \cite{CSW} (see \S\S 4 and 13 therein). However, when considered in our coordinates (and restricted to a leaf of the unstable foliation) tremors become quite simple.

For a given lamination $\lambda$, let $|\Delta(\lambda)|_\pm$
\label{ind:signedmeasures}
denote the vector space of all signed transverse measures on $\lambda$; this is naturally a vector subspace of $\mathcal H (\lambda)$ of dimension at most $3g-3$ with basis consisting of the length 1 (with respect to some auxiliary hyperbolic metric) ergodic measures on $\lambda$.

\begin{definition}\label{def:Il_trem}
Let $q \in \Fol^{uu}(\lambda)$ and let $\mu \in |\Delta(\lambda)|_\pm$. Then the tremor $\trem_{\mu}(q)$ of $q$ along $\mu$ is the unique quadratic differential specified by
\begin{equation}\label{eqn:Il_trem}
\Il(\trem_{\mu}(q)) = \Il(q) + \mu. 
\end{equation}
Note that the fact that $\Il(q) + \mu \in \SH^+(\lambda)$ follows by affinity of the Thurston form (Lemma \ref{lem:ThSHprops}).
\end{definition}

\begin{remark}
Technically, the deformation considered above is a ``non-atomic tremor'' in the language of \cite{CSW}. One can also consider ``atomic tremors,'' which transform $q$ by twisting along certain admissible loops of horizontal saddle connections.

In shear-shape coordinates, these admissible loops correspond to certain simple closed curves in the complementary subsurfaces. Atomic tremors are then realized by appropriately shearing the underlying arc system $\arcwt(q)$ along the curves and transporting the transverse cocycle using the affine connection coming from train-track coordinates. Of course, one can also define tremors along more complicated laminations contained in $S \setminus \lambda$ as well.
\end{remark}

For the convenience of the reader familiar with the terminology of \cite{CSW}, we have included a dictionary which translates between our notation and theirs (at least when the horizontal lamination is filling --- when it is not, one must replace $\Delta(\lambda)$ with a subset of the zero set of $\lambda$ and take more care). See Figure \ref{fig:CSWdict}.

\begin{figure}[ht]
\begin{centering}
\bgroup
\def\arraystretch{1.5}
\begin{tabular}{|*{2}{>{\centering\arraybackslash}p{.3\linewidth}|}}
\hline
Shear-shape cocycles & Foliation cocycles \\
\hline 
  $\Delta(\lambda)$   & $C^+_q$ \\
\hline
  $|\Delta(\lambda)|_\pm$   & $\mathcal{T}_q$ \\
\hline 
  $\ThSH(\eta, \mu) = i(\mu_+, \eta)- i(\mu_-, \eta)$ 
  & signed mass $L_q(\mu)$  \\
\hline 
  $i(\mu_+, \eta) + i(\mu_-, \eta)$
  & total variation $|L|_q(\mu)$ \\
\hline 
\end{tabular}
\egroup
\end{centering}
\caption{Translating between our language of shear-shape cocycles and the ``foliation cocycles'' of \cite{CSW}. Throughout, we assume that $q = q(\eta, \lambda)$ where $\lambda$ is filling (equivalently, $q$ has no loops of horizontal saddle connections). We have written a signed transverse measure $\mu$ as $\mu = \mu_+ - \mu_- \in |\Delta(\lambda)|_\pm$, where $\mu_\pm \in \Delta(\lambda)$.}
\label{fig:CSWdict}
\end{figure}

We can now immediately deduce certain properties of the tremor map from the structure of $\SH^+(\lambda)$ and the intersection pairing.
While we will not use these results in the sequel, we have chosen to include them in order to to demonstrate the utility of our new perspective on these deformations.
For example, using our coordinates one can easily deduce that (non-atomic) tremors leave horizontal data invariant and hence can be continued indefinitely while remaining in the same stratum.

\begin{lemma}
For any $q \in \Fol^{uu}(\lambda)$ and $\mu \in |\Delta(\lambda)|_\pm$, the tremor path $\trem_{t\mu}(q)$ is defined for all time and is completely contained in $\SH(\lambda; \arcwt(q))$. In particular, $\{\trem_{t\mu}(q)\}$ always remains in the same stratum.
\end{lemma}

\begin{remark}
The above Lemma is one specific instance of a much more general phenomenon. The global description of $\Fol^{uu}(\lambda)$ afforded by shear-shape coordinates allows one to formulate a general criterion for extending affine period geodesics, a topic which the authors hope to address in future work.
\end{remark}

Using our interpretation of tremors as translation, it is similarly easy to describe how tremors interact with other flat deformations. Compare with Propositions 6.1 and 6.5 of \cite{CSW}. We leave proofs to the reader, as they follow immediately from \eqref{eqn:Il_trem} and \eqref{eqn:Il_geo}.

\begin{lemma}
Let $q \in \Fol^{uu}(\lambda)$. Then for any $\mu \in |\Delta(\lambda)|_{\pm}$ and for
$g_t = \begin{pmatrix}
e^{t/2} & 0 \\ 0 & e^{-t/2}
\end{pmatrix}$,
we have that
\[g_t \trem_\mu (q) = \trem_{e^{t/2} \mu} ( g_t (q)).\]
Additionally, for any $\mu_1, \mu_2  \in |\Delta(\lambda)|_\pm$, we have that
\[\trem_{\mu_1}(q) \trem_{\mu_2}(q) = \trem_{\mu_1 + \mu_2}(q) = \trem_{\mu_2}(q) \trem_{\mu_1}(q).\]
In particular, tremors commute with the horocycle flow.
\end{lemma}

\section{Shear-shape coordinates for hyperbolic metrics}\label{sec:hyp_overview}

We now parametrize hyperbolic structures on $S$ by shear-shape cocycles for a measured geodesic lamination $\lambda$. 
With respect to the Lebesgue measure on $\ML(S)$, the generic lamination cuts a hyperbolic surface into ideal triangles. As all ideal triangles are isometric, Bonahon and Thurston's shearing coordinates need only take into account the ``shear'' between pairs of complementary triangles to describe a hyperbolic structure. 
As our objective is to generalize these coordinates to laminations with arbitrary topology, we must therefore combine the data of the geometry of hyperbolic metrics in complementary subsurfaces with the shearing data between them.
Shear-shape space $\SH(\lambda)$ is well suited to this task.

In the following Sections \ref{sec:hyp_map}--\ref{sec:shsh_homeo}, we explain how to associate a ``geometric shear-shape cocycle'' to a hyperbolic metric and prove that the space of positive shear-shape cocycles coordinatizes Teichm{\"u}ller space:

\begin{theorem}\label{thm:hyp_main}
The map $\sigl: \T(S)\to \SH^+(\lambda)$ that associates to a hyperbolic metric its geometric shear-shape cocycle is a stratified real-analytic homeomorphism.
\end{theorem}

As detailed in the Introduction, combining this theorem with Theorems \ref{thm:GM} and \ref{thm:Ilhomeo} implies that the orthogeodesic foliation map $\Ol$ is a homeomorphism, and consideration of the earthquake/horocycle flows in $\SH^+(\lambda)$ coordinates then proves the conjugacy on slices (Theorem \ref{mainthm:orthohomeo}).

We remark that the stratified regularity of $\sigl$ and $\Ol$ is the best one can expect, since the adjacency of strata of differentials is not analytic (as there are multiple inequivalent ways to ``break up a zero'').
Compare with \cite[Theorem D]{Dumas_skin}, in which it is shown that for a fixed Riemann surface $Z$, the identification $Q(Z) \cong \ML$ guaranteed by the Hubbard-Masur theorem \cite{HubMas} is stratified real-analytic.

\para{Fixed complementary subsurfaces}
By definition (see Section \ref{subsec:shsh_hyp}), the weighted arc system $\arcwt(X)$ underlying $\sigl(X)$ exactly identifies the geometry of $X \setminus \lambda$ via Theorem \ref{thm:arc=T(S)_crown}. Setting 
\[\T(S;\arcwt):= \{X\in \T(S): \arcwt(X) = \arcwt\},\]
Theorem \ref{thm:hyp_main} therefore implies that $\T(S;\arcwt)$ is nonempty if and only if $\arcwt \in \Base$.

\begin{remark}
The authors do not know a proof of this fact that does not factor through Theorem \ref{thm:hyp_main} except in some special cases (for example, when the complement of $\lambda$ is polygonal, or when $\lambda$ is a union of simple closed curves).
\footnote{One can of course complete $\lambda$ to a maximal lamination and then specify the shear coordinates on each of the added leaves, but then one must be very careful to ensure that these shears satisfy the relations coming from the metric residue condition. The argument then requires an involved computation with train tracks carrying the completed lamination.}
\end{remark} 

In fact, since $\SH^+(\lambda)$ is an affine cone bundle over $\Base$ (Proposition \ref{prop:SH+_structure}), we see that

\begin{corollary}
For each $\arcwt\in \Base$, the set $\T(S; \arcwt)$ is a real-analytic submanifold of $\T(S)$ and the restriction of $\sigl$ to 
\[\T(S;\arcwt) \to \SH^+(\lambda; \arcwt)\cong \cH^+(\lambda)\]
is a real-analytic homeomorphism.
\end{corollary}

In this setting, the correspondence between $\T(S;\arcwt)$ and $\cH^+(\lambda)$ is a natural generalization of shear coordinates, since the complementary subsurfaces to $\lambda$ are always isometric.
In fact, the shape-shifting deformations built to deform $X$ by some $\ac \in \cH(\lambda)$ (see the proof sketch of Theorem \ref{thm:hyp_main} just below) restrict to cataclysms/shear maps in the sense of \cite[Section 5]{Bon_SPB}.  In particular, if $\ac$ represents a measure supported on $\lambda$, then the shape-shifting deformation determined by $\ac$ is part of an earthquake in $\ac$ (Corollary \ref{cor:eq=translation}); if $\ac$ is a multiple of $\sigl(X)$, the shape-shifting transformation can sometimes be identified with part of a (generalized) stretch ray (Propositions \ref{prop:stretch_reg} and \ref{prop:quad_geo}).

In addition to being non-empty, $\T(S;\arcwt)$ is structurally rich; the authors hope to explore this space further in future work. Of particular interest is the (degenerate) Weil--Petersson pairing on this locus and its relation with the Thurston symplectic form and Masur--Veech measures. 

\para{A sketch of the proof}
Since the proof of Theorem \ref{thm:hyp_main} spans several sections (two of which consist of involved constructions of the relevant objects), we devote the remainder of this section to a broad-strokes outline of the arguments involved.
Our exposition throughout these sections is mostly self-contained, but we sometimes refer to \cite{Bon_SPB} for proofs and to \cite{Th_stretch} for inspiration.

We begin in Section \ref{sec:hyp_map} by defining the map $\sigl$.
Under the correspondence established in Theorem \ref{thm:arc=T(S)_crown}, we associate to $X$ the weighted arc system $\arcwt(X)$ recording the hyperbolic structure on $X \setminus \lambda$.
We cut $X$ along the (ortho)geodesic realization of $\lambda\cup\arc$ into a union of (degenerate) right-angled polygons, and measure the shear between certain pairs of polygons.
We then argue using train tracks that it suffices to record the shearing data comprising $\sigl(X)$ on short enough arcs $k$ transverse to $\lambda$ and disjoint from $\arc(X)$.
The value of $\sigl(X)$ on short $k$ may then be defined by isotoping $k$ to a path connecting vertices of the spine $\Sp$ and built of segments alternating between leaves of $\lambda$ and of $\Ol(X)$, then measuring the total (signed) length along $\lambda$.
These measurements are equivalent to Bonahon and Thurston's method of measuring shears (via the horocyclic foliation) when $k$ is short enough, but cannot be globally derived from theirs due to obstructions coming from complementary subsurfaces.

The proof that $\sigl$ is a homeomorphism then follows the same general steps as appear in \cite{Bon_SPB}.
After proving that $\sigl$ is injective and lands inside of $\SH^+(\lambda)$ (Proposition \ref{prop:shsh_inj} and Corollary \ref{cor:sigl_into_SH+}), we then show that it is open (Theorem \ref{thm:shsh_open}) and proper.
Since $\SH^+(\lambda)$ is a cell (Proposition \ref{prop:SH+_structure}), invariance of domain then implies that $\sigl$ must be a homeomorphism.

Our proof of injectivity mirrors that of \cite[Theorem 12]{Bon_SPB} with an additional invocation of Theorem \ref{thm:arc=T(S)_crown}. For properness we mostly appeal to \cite[Theorem 20]{Bon_SPB} but need to discuss complications that arise from the piecewise-linear structure of shear-shape space.
Similarly, our broad-strokes strategy to prove openness parallels that of \cite[\S 5]{Bon_SPB}, in that we build a ``shape-shifting coycle'' $\varphi_{\ac}$ for all small enough deformations $\ac$ of $\sigl(X)$ (see Section \ref{sec:shapeshift_def}). Deforming $X$ by post-composing its charts to $\mathbb{H}^2$ with $\varphi_{\ac}$ then yields a surface $X_\ac$ with $\sigl(X_\ac) = \sigl(X) + \ac$.

It is in the construction of $\varphi_{\ac}$, performed in Section \ref{sec:shapeshift_def}, where our discussion truly diverges from \cite{Bon_SPB} and \cite{Th_stretch}.
When $\lambda$ is maximal, one can  specify $\varphi_{\ac}$ by shearing $X$ along the leaves of $\lambda$ (i.e., performing a cataclysm).
Even in the maximal case this procedure is delicate, hinging on the convergence of infinite products of small M{\"o}bius transformations (compare Section \ref{subsec:shsh_spikes}). In the non-maximal case, we must also simultaneously account for the changing shapes of complementary subsurfaces (which also introduces extra complications into the shearing deformations since the shapes of spikes are changing). See the introduction to Section \ref{sec:shapeshift_def} for a more granular description of the construction of $\varphi_\ac$.

\section{Measuring hyperbolic shears and shapes}\label{sec:hyp_map}

In this section, we take our first steps towards proving Theorem \ref{thm:hyp_main} by describing how to associate to any hyperbolic surface $X$ a \emph{geometric shear-shape cocycle} $\sigl(X)$ in a natural way; this yields the map
\[\sigl: \T(S) \to \SH(\lambda).\]
After fixing some notational conventions that we will use throughout the sequel, we define $\sigl(X)$ by first specifying its underling arc system $\arcwt(X)$ in a variety of equivalent ways.
After doing so, we define the shear between ``nearby'' hexagons analogously to Bonahon and Thurston; placing all of this data onto a standard smoothing $\taua$ of a geometric train track is therefore enough to specify $\sigl(X)$ (Lemma \ref{lem:shshtt}).

We then show that the data of shears between any two nearby hexagons can be recovered from the weight system on $\taua$, even if those hexagons are not ``visible'' to $\taua$ (Lemma \ref{lem:ttsdontmatter}).
This in particular implies that our choice of $\taua$ does not actually matter, and hence $\sigl(X)$ is well-defined.

We then conclude the section by proving some initial properties of $\sigl$. Proposition \ref{prop:shsh_inj} shows that the map is injective following an argument of Bonahon, while in Theorem \ref{thm:diagram_commutes} we show that our map captures the geometry of the orthogeodesic foliation.

\subsection{Preliminaries and notation}\label{subsec:hypnot}
In this section, we discuss the geometry of a geodesic lamination on a hyperbolic surface and fix notation in preparation of our definition of the geometric shear-shape cocycle of a hyperbolic structure.

Throughout, we use the symbol $\lambda$ to refer to both the measured lamination $\lambda$ and its support, realized geodesically with respect to any number of hyperbolic metrics. We refer to Remark \ref{rmk:geod_lamination} for a discussion of how to relax the assumption that $\lambda$ is measured.
We reserve $S$ to denote a topological surface and $\Sigma$ the topological type of a component of $S \setminus \lambda$, while $X$ and $Y$ will denote their hyperbolic incarnations.
We also adopt the following family of notational conventions: the expression $g\subset\lambda$ means that $g$ is a leaf of $\lambda$, and $Y\subset X\setminus\lambda$ means that $Y$ is a component of (the metric completion of) $X\setminus \lambda$, etc. 
The notation of \cite{Bon_SPB} is used as inspiration, since we will make direct appeals to the results therein.  However, our situation requires more care, since we have more objects to keep track of.
A key difference is that we will focus not on the relative shear between complementary subsurfaces of $X\setminus \lambda$, but on the relative positioning of pairs of boundary leaves of $\lambda$, equipped with a natural collection of basepoints determined by the orthogeodesic foliation.

\para{Hexagons} Given $X \in \T(S)$ and $\lambda\in \ML(S)$, realize $\lambda$ geodesically on $X$. Construct the orthogeodesic foliation $\Ol(X)$ on $X$ with piecewise geodesic spine $\Sp$ and dual arc system $\arc=\arc(X)$, realized orthogeodesically with respect to $X$ and $\lambda$. 
The union $\lambdaa = \lambda \cup \arc$ is a geometric object on $X$ that fills; that is, the metric completion of $X\setminus \lambdaa$ is a union of geometric pieces that are topological disks, possibly with some points on the boundary removed corresponding to spikes.   We lift the situation to universal covers $\tlambdaa\subset \tX$, where we have also the full preimages $\tSp$, $\tlambda$, $\widetilde\arc$, etc., of various objects.  

Let $\cH$ be the vertex set of $\tSp$; we will sometimes refer to $v\in \cH$ as a \emph{hexagon}.
\label{ind:hexagons}
Indeed, to $v$ there is associated a component $H_v$ of $\tX\setminus \tlambdaa$ which is generically a degenerate right-angled hexagon, though $H_v$ may also be a regular ideal or right-angled polygon, for example. 
We reiterate that, by abuse of terminology, {\em any complementary component $H_v$ of $\tX\setminus \tlambdaa$ is called a hexagon}, no matter its shape.

If $\{H_v: v \in \cH\}$ contains components that are not degenerate right-angled hexagons in the usual sense, then $\arc$ corresponds to a simplex of $\Arcfill(S \setminus \lambda)$ of non-maximal dimension (or the empty set, if $\lambda$ is filling and $\arc$ is empty).
One may always include $\arc$ in a maximal arc system $\arcb$, which necessarily defines a simplex of full dimension.  The complementary components of $\tX\setminus \widetilde{\lambda_{\arcb}}$ are now degenerate right-angled hexagons in the usual sense, and gluing them in pairs along $\arcb\setminus\arc$ gives the more general ``hexagons'' of $\tX\setminus \tlambdaa$. We will often tacitly choose and work with a maximal arc system containing the original when convenient.

\para{Pointed geodesics}
We now define a natural family of basepoints associated to boundary leaves of $\tlambda$. For $v\in \cH$ and its associated hexagon $H_v$, define the \emph{$\lambda$-boundary} $\partiall H_v$ of $H_v$ to be the set of leaves of $\tlambda$ that meet $\partial H_v$. \label{ind:lambdaboundary}

For $v\in \cH$ and $g$ a leaf of $\partiall H_v$, define $p_v$ to be the orthogonal projection of $v$ to $g$. Observe that $v$ and $p_v$ lie along the same (singular) leaf of $\Ol(X)$.
\label{ind:ptdgeo}
The orientation of $S$ gives $H_v$ an orientation and hence orients $\partial H_v$; this yields an orientation-preserving, isometric identification of $(g, p_v)$ with $(\RR, 0)$.
We refer to points on a based geodesic by their signed distance to the basepoint, so that $0$ refers to $p_v$ while $\pm x$ refer to the points at signed distance $\pm x$ from $p_v$.

For a pair $v \neq w\in \cH$ not in the same component of $\tSp$, there is a unique geodesic $g_v^w\in \partiall H_v$ that separates $v$ from $w$.  Symmetrically, there is such a pointed geodesic $g_w^v\in \partiall H_w$ separating $w$ from $v$.  Note that $g_v^w=g_w^v$ occurs if and only if this leaf is isolated, and by the assumption that $\lambda$ is measured, projects to a simple closed curve component of $\lambda$. Even in this case, the points $p_v$ and $p_w$ are in general different. 

\subsection{The shear-shape cocycle of a hyperbolic structure}\label{subsec:shsh_hyp}
Our first task towards defining the geometric shear-shape cocycle $\sigl(X)$ of a hyperbolic structure $X$ is to construct a weighted, filling arc system $\arcwt(X) \in \Base$ which records the shapes of the complementary subsurfaces.

With the technology we have developed up to this point, we now have many ways of constructing $\arcwt(X)$, all of which are easily seen to be equivalent:
\label{ind:arcwt(X)}
\begin{itemize}
\item To each $\alpha \in  \arc(X)$, we associate the weight $c_\alpha:= i(\Ol(X), e_\alpha)$, where $e_\alpha$ is the edge of $\Sp$ dual to $\alpha$. Equivalently, $c_\alpha$ is the length of the projection of $e_\alpha$ to either of the two leaves of $\lambda$ to which it is closest. Then set $\arcwt(X) = \sum c_\alpha \alpha.$
\item Each component $Y \subset X \setminus \lambda$ is naturally endowed with a hyperbolic structure; by Theorem \ref{thm:arc=T(S)_crown} this metric corresponds to a weighted, filling arc system in $|\Arcfill(Y, \partial Y)|_{\RR}$, and we let $\arcwt(X)$ denote the union of these arc systems over all components of $X \setminus \lambda$.
\item Let $q$ be the quadratic differential with $|\Re(q)| = \Ol(X)$ and $|\Im(q)|=\lambda$; then set $\arcwt(X) = \arcwt(q)$.
\end{itemize}
The final definition together with the results of Section \ref{sec:flat_map} implies that $\arcwt(X) \in \Base$ for every hyperbolic structure $X$ on $S$.
In the interest of providing the reader with geometric intuition for this condition, we have included an alternative, purely hyperbolic-geometric proof of this fact below.

\begin{lemma}\label{lem:aXBase}
With notation as above, $\arcwt(X)\in \Base$.
\end{lemma}
\begin{proof}
By Theorem \ref{thm:arc=T(S)_crown}, it suffices to show that for each minimal, orientable component $\mu$ of $\lambda$, the sum of the metric residues of the crown ends of $X \setminus \lambda$ incident to $\mu$ is $0$. If $\mu$ is a simple closed curve, then the metric residue is just equal to the (signed) lengths of the boundary components resulting from cutting along $\mu$, which clearly must match.

So assume that $\mu$ is not a closed curve and pick an orientation. Construct a geometric train track $\tau$ snugly carrying $\mu$ as in Construction \ref{const:geometric_tt}; then $\tau$ inherits an orientation from the inclusion of $\mu$ and so has well defined left- and right-hand sides.
As in Section \ref{subsec:ortho_foliation}, every branch $b$ of $\tau$ has a well-defined length along $\lambda$ which we denote by $\ell_\tau(b)>0$.
At each switch $s$ of $\tau$, let $h_s$ be the leaf of the horocyclic foliation of $\epN\mu$ projecting to $s$. By assumption of snugness, the spikes of $S \setminus \tau$ correspond with the spikes of $S \setminus \mu$, so the union of the $h_s$ truncates each spike of each crown end incident to $\mu$ by $h_s$. 

Each crown incident to $\mu$ inherits an orientation from the chosen orientation on $\mu$, and we now compute the total metric residue with respect to these orientations and the truncations induced by the $h_s$'s.
Recall that the metric residue of an oriented crown $\mathcal C$ is the alternating sum of the lengths of the geodesic boundary segments running between the truncation horospheres (Definition \ref{def:metric_res}).
Each such geodesic segment defines a co-oriented trainpath $(b_1 \cdot ... \cdot b_n,\pm)$ in $\tau$ (i.e., a trainpath and a distinguished side, left or right, corresponding to $+$ and $-$, respectively) which runs along the entirety of a smooth component of the boundary of $X \setminus \tau$. Using this identification, we may compute that the corresponding contribution to the total metric residue is given by $\pm \sum_{i}\ell_\tau(b_i)$.

Finally, we observe that every branch of $\tau$ is a subpath of exactly two smooth boundary edges of $X \setminus \tau$ (corresponding to its left and right sides). Therefore, the sum of the metric residues of all of the crown ends incident to $\mu$ is the sum of the contributions of the corresponding co-oriented trainpaths, which is necessarily $0$ since each branch is counted twice, once with positive and once with negative sign. Thus $\arcwt(X) \in \Base$.
\end{proof}

\para{Shears between nearby hexagons}
Our second step towards defining $\sigl(X)$ is to determine how to record shearing data between two hexagons that lie in different components of $\tX \setminus \tlambda$ yet are close enough together. Except for sign conventions (see Remark \ref{rmk:hypsh_sign}), our discussion is essentially identical to Bonahon's definition of shearing between the plaques of a maximal lamination \cite[\S2]{Bon_SPB}. 
Our restriction to pairs of nearby hexagons reflects the fact that if two hexagons are far apart, a path connecting them may meet a subsurface of $\tX \setminus \tlambda$ in a variety of ways.

Given $v,w\in \cH$, consider the associated pointed geodesics $(g_v^w, p_v)\in \partiall H_v$ closest to $H_w$ and $(g_w^v, p_w)\in \partiall H_w$ closest to $H_v$. We say that the geodesic segment $k_{v,w} \subset \tX$ joining $p_v$ to $p_w$ is a \emph{simple piece} if $k_{v,w}$ projects to a simple geodesic segment in $X$ and $k_{v,w}$ bounds a spike in every hexagon that it crosses.  That is,  if $k_{v,w}$ crosses $H_u$ for some $u\in \cH$, then $k_{v,w}\cap H_u$ bounds a triangle in $H_u$, two sides of which lie on asymptotic leaves $g_u^v$ and $g_u^w$ defining a spike of $\tlambda$. 
If $k_{v,w}$ is a simple piece, then we say that $(v,w)$ is a \emph{simple pair}.
\label{ind:simple}

We observe that if $v, w \in \cH$ are close enough together and lie in different components of $\tSp$, then $(v,w)$ is a simple pair.
The exact value of ``close enough'' is unimportant, but we note that it suffices for $d(H_u, H_v)$ to be smaller than the length of the shortest arc of $\arc(X)$.

Now following Bonahon \cite[Section 2]{Bon_SPB}, let $\Lambda_{v,w}$
\label{ind:sepleaves}
be the leaves of $\tlambda$ that separate $g_v^w$ from $g_w^v$, equipped with the linear order $<$ induced by traversing $k_{v,w}$ from $p_v$ to $p_w$. Since $(v,w)$ is a simple pair, the subset of those leaves that are also the boundary of a complementary component of $\tX\setminus \tlambda$ come in pairs that are asymptotic in one direction. 
The partial horocyclic foliations on the wedges bounded by pairs of asymptotic boundary leaves extend across the leaves of $\Lambda_{v,w}$, foliating the region bounded by $g_v^w$ and $g_w^v$. In particular, the leaf of the horocyclic foliation containing $p_v$ meets $g_w^v$ (and the leaf containing $p_w$ meets $g_v^w$).

Since the orthogeodesic foliation is equivalent to the horocyclic foliation in spikes, we see that if $(v,w)$ is a simple pair then the leaf of $\cO_{\Lambda_{v,w}}(\tX)$ containing $p_v$ meets $g_w^v$ (and the leaf containing $p_w$ meets $g_v^w$). In fact, simplicity implies that $\cO_{\Lambda_{v,w}}(\tX)$ foliates the ``quadrilateral'' bounded by $g_w^v$, $g_v^w$, and the two leaves of $\cO_{\Lambda_{v,w}}(\tX)$ containing $p_v$ and $p_w$.

\begin{definition}\label{def:nearbyshear}
Suppose that $(v, w)$ is a simple pair of hexagons. Using the orientation conventions of Subsection \ref{subsec:hypnot}, identify the corresponding pointed geodesics $(g_v^w, p_v)$ and $(g_w^v, p_w)$ with $(\RR, 0)$.
Now since the hexagons are close enough, the singular leaf of $\cO_{\Lambda_{v,w}}(\tX)$ containing $p_v$ meets $g_w^v$ in some point $r \in \RR$, and we set $\sigl(X)(v,w) = -r$. See Figure \ref{fig:nearbyshear}.
\end{definition}

\begin{figure}[ht]
    \centering
\begin{tikzpicture}
    \draw (0, 0) node[inner sep=0] {\includegraphics{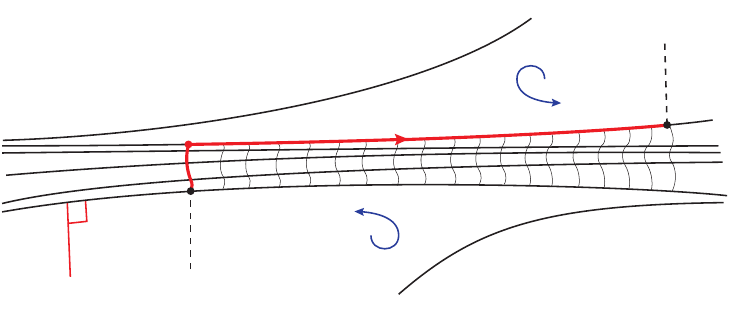}};
    \node at (-2.6, -.8){$p_v$};
    \node at (-2.9, -2.1){$v$};
    \node at (4.8, .75){$p_w$};
    \node at (5.1, 2.1){$w$};
    \node at (-3, 1.3){\color{red} $r=- \sigl(X)(v,w)$};
    \draw[line width=0.25mm,red] (-3, .4) -- (-3, 1);
\end{tikzpicture}
    \caption{Computing the shears between two nearby hexagons $v$ and $w$.  In this example, $r<0$, so $\sigl(X)(v,w)>0$.}
    \label{fig:nearbyshear}
\end{figure}

It is not hard to see that $\sigl(X)(v,w)$ remains the same if we flip the roles of $v$ and $w$. Indeed, following along the leaves of the orthogeodesic foliation defines an orientation reversing isometry 
from a subsegment of $g_v^w$ to a subsegment of $g_w^v$ that takes $t\mapsto r-t$. In particular, $p_v$ maps to a point on $g_w^v$ that is positioned $r$ signed units away from $p_w$, and so we see that $\sigl(X)(v,w) = \sigl(X)(w,v)$.

\begin{remark}\label{rmk:hypsh_sign}
Our choice to set $\sigl(X)(v,w) = -r$ instead of $+r$ records ``how far along $g_w^v$ you must travel from $r$ to get to $p_w$.'' Though this convention is the opposite of what appears in \cite{Bon_SPB}, it allows us to combine the data of $\sigl(X)(v,w)$ and $\arcwt(X)$ into a system of train track weights on a standard smoothing (Construction \ref{const:shsh_hyp_tt} just below). Our convention also parallels our choice of $[\cdot]_+$ function when measuring periods of a quadratic differential (Lemma \ref{lem:pers_as_ttwts}), which makes the relationship between the hyperbolic geometry of $(X, \lambda)$ and the flat geometry of $q(\Ol(X), \lambda)$ more transparent.
\end{remark}

Below, we give an elementary estimate that will be used in the proof of Proposition \ref{prop:shsh_inj}; compare with \cite[Lemma 8]{Bon_SPB}.

\begin{lemma}\label{lem:length_shear_bound}
Suppose that $(v, w)$ is a simple pair of hexagons. Let $(g_v^w, p_v)$ and $(g_w^v,p_w)$ be the associated pointed geodesics.  Then the geodesic segment $k_{v,w}$ joining $p_v$ to $p_w$ satisfies
\[|\sigl(X)(k_{v,w})|\le \ell(k_{v,w}).\]
\end{lemma}
\begin{proof}
As $(v,w)$ is simple, the partial orthogeodesic foliation $\cO_{\Lambda_{v,w}}(\tX)$ foliates the region $U$ bounded by $g_w^v$, $g_v^w$, and the two leaves of $\cO_{\Lambda_{v,w}}(\tX)$ containing $p_v$ and $p_w$.
This foliation gives rise to a $1$-Lipschitz retraction $\pi$ from $U$ to $g_w^v$ defined by following the leaves of the orthogeodesic foliation to $g_w^v$. The image $\pi(k_{v,w})$ is then equal to the segment of $g_w^v$ joining $p_w$ to the point labeled by $\sigl(X)(v,w)$, which has length $|\sigl(X)(v,w)|$.  The lemma follows.
\end{proof}

\para{Hyperbolic shearing as train track weights}
Now that we have explained how to record the shapes of $X \setminus \lambda$ (Lemma \ref{lem:aXBase}) and the shears between nearby hexagons (Definition \ref{def:nearbyshear}), we can package this information together to define the {\em geometric shear-shape cocycle} $\sigl(X)$ of a hyperbolic structure $X$. 

Below, we realize the shape and shear information specified above as a weight system on a standard smoothing of a geometric train track carrying $\lambda$; this strategy allows us to specify $\sigl(X)$ by a finite collection of information.
Once we show that the weights are well-defined and satisfy the switch conditions, we then invoke Proposition \ref{prop:ttcoords} to interpret this weight system as an (axiomatic) shear-shape cocycle (see Definition \ref{def:hypshsh}).
This reinterpretation in turn makes it apparent that our initial choice of train track does not matter.

Using Construction \ref{const:geometric_tt}, choose a geometric train track $\tau \subset X$ that carries $\lambda$ snugly and let $\taua$ be a standard smoothing of $\tau \cup \arc(X)$ (see Construction \ref{constr:stand_smooth}). Note that the components of $\tX\setminus \ttaua$ are in bijection with the set of hexagons $\cH$, and that the assumption that $\tau$ carries $\lambda$ snugly ensures that if two hexagons correspond to adjacent components of $\tX\setminus \ttaua$ then they either share an edge of $\tilde{\arc}$ or form a simple pair. We recall that two hexagons are a simple pair if the geodesic connecting their basepoints passes only through spikes of $S \setminus \lambda$.

\begin{construction}\label{const:shsh_hyp_tt} Fix $\taua\subset X$ as above. We then associate a weight system $w(X) \in \RR^{b(\taua)}$ as follows:
\begin{itemize}
\item To each branch corresponding to $\alpha \in  \arc$, assign the weight $c_\alpha = i(\Ol(X), e_\alpha)$, where $e_\alpha$ is the edge of $\Sp$ dual to $\alpha$.
\item For each branch $b \subset \taua$ that dos not correspond to an arc of $\arc$, choose a lift $\tilde b \in \ttaua$. Let $v,w\in \cH$ denote the vertices of $\tSp$ corresponding to the hexagons adjacent to $\tilde b$, and set $w(X)(b) = \sigl(X)(v,w)$.
\end{itemize}
\end{construction}

\begin{lemma}\label{lem:shshtt}
Let $X, \lambda, \arc$ and $\taua$ be as above. Then the edge weights $w(X) \in \RR^{b(\taua)}$ given by Construction \ref{const:shsh_hyp_tt} satisfy the switch conditions.
\end{lemma}
\begin{proof}
Reference to Figure \ref{fig:switch} will provide clarity throughout. We note that $\taua$ is generically trivalent, but may be $4$-valent if there are arcs $\alpha_1, \alpha_2\in \arc$ whose endpoints on $\lambda$ lie on a common leaf of the orthogeodesic foliation. We give an argument only for the trivalent switches of $\taua$, and leave it to the reader to make the necessary adjustments for $4$-valent switches (the statement for $4$-valent switches can also be deduced by continuity).

Let $s$ be a trivalent switch; then standing at $s$ and looking into the spike, there are small half-branches exiting $s$ on our right and left; call these $r$ and $\ell$, respectively. By our convention on standard smoothings, every half-branch of $\taua$ corresponding to an arc of $\arc$ is a right small half-branch.

If no branch of $s$ corresponds to an arc of $\arc$, then the arguments appearing in \cite[Section 2]{Bon_SPB} imply that the weights satisfy the switch conditions, because the orthogeodesic foliation is equivalent to the horocycle foliation in near $s$.  See also \cite[Section 6]{Pap:extension} for a discussion more similar in spirit to ours.

Otherwise, the right small half-branch $r$ is labeled by some $\alpha\in \arc$. Let $b$ be the large half-branch incident to $s$. Give names also to the hexagons incident to $s$ and their distinguished points on $b$ or $\ell$ by projection; they are $N, SE$, and $SW\in \cH$, and $p_{N}$, $p_{SE}$, $p_{SW}$ respectively, where $b$ and $\ell$ form part of the boundary of $N$, $\ell$ and $r$ form part of the boundary of $SE$, and $r$ and $b$ form part of the boundary of $SW$. See Figure \ref{fig:switch}.

Now take $d= d(p_{SW}, p_{SE})$, which is equal to $w(X)(r)=c_\alpha >0$ by definition. Define also
\[s_1:=|w(X)(b)|= d_\tau(p_{SW}, p_{N}) \text{ and } s_2:=|w(X)(\ell)|= d_\tau(p_N, p_{SE}).\]
Here, $d_\tau$ is understood to mean the distance between leaves of the orthogeodesic foliation near $\tau$, measured along any leaf of $\lambda$ (see Section \ref{subsec:ortho_foliation} for an explanation of why this value is well-defined).

\begin{figure}[ht]
    \centering
\begin{tikzpicture}
    \draw (-4, 0) node[inner sep=0] {\includegraphics{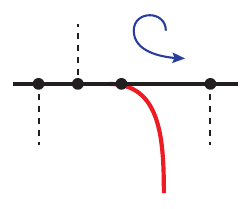}};
    \draw (4, -.3) node[inner sep=0] {\includegraphics{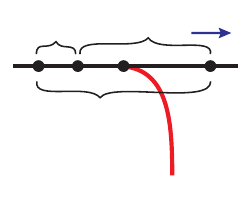}};
    \node at (-4.8, 1.6){$N$};
    \node at (-5.5, -1){$SW$};
    \node at (-2.5, -1){$SE$};
    \node at (-2.1, 0){$p_{SE}$};
    \node at (-4.4, .6){$p_{N}$};
    \node at (-5.9, 0){$p_{SW}$};
    \node at (-6, .7){$b$};
    \node at (-2, .7){$\ell$};
    \node at (-3.6, -1.4){\color{red} $\alpha$};
    \node at (-4, 0){$s$}; 
    \node at (2.4,1.2){$s_1 = w(X)(b)$};
    \node at (3.4,-.45){$d = w(X)(\alpha)$};
    \node at (5,1.2){$s_2 = - w(X)(\ell)$};
\end{tikzpicture}
    \caption{Left: A local picture of $\tau$ near $s$.  Right: Case (3). The switch condition is satisfied because $s_1 = d-s_2$.}
    \label{fig:switch}
\end{figure}

There are $3$ kinds of configurations for the projection points $p_{SW}, p_{N}$ and $p_{SE}$ that determine the signs of $w(X)(b)$ and $w(X)(\ell)$:

\begin{enumerate}
\item The point $p_N$ precedes both $p_{SW}$ and $p_{SE}$ with respect to the orientation of $\tau$ on induced by $H_N$, so that
\[ w(X)(b) = -s_1 \text{ and } w(X)(\ell) = -s_2 \text{ with }s_2>s_1.\]
In this case, $d = s_2-s_1$ and so $d-s_2 = -s_1$, which is exactly the switch condition.
\item Both $p_{SW}$ and $p_{SE}$ precede $p_N$, so that
\[ w(X)(b) = s_1 \text{ and } w(X)(\ell) = s_2 \text{ with }s_1>s_2.\]
This possibility gives that $d=s_1-s_2$ and so $d+s_2 = s_1$.
\item The point $p_{SW}$ precedes $p_N$ which in turn precedes $p_{SE}$, so that $w(X)(b) =s_1$ and $w(X)(\ell) = -s_2$. In this case, $d = s_1+s_2$ and so $d-s_2 = s_1$, which is again the switch condition.
\end{enumerate}
Therefore, no matter the configuration of points $p_N, p_{SW}$, and $p_{SE}$, we see that the switch conditions are fulfilled at $s$, completing the proof of the lemma.
\end{proof}

\begin{remark}
It is important to note that $w(X)$ is generally {\em not} the same as the weight system coming from the shear coordinates of a completion of $\lambda$ (unless $\lambda$ was maximal to begin with).
\end{remark}

Invoking Proposition \ref{prop:ttcoords} and Lemma \ref{lem:aXBase}, we see that the weight system $w(X)$ defines a shear-shape cocycle with underlying arc system $\arcwt(X)$.

\begin{definition}\label{def:hypshsh}
The {\em geometric shear-shape cocycle} $(\sigl(X), \arcwt(X))$ of a hyperbolic metric $X$ is the unique shear-shape cocycle for $\lambda$ corresponding to the weight system $w(X)$ of Construction \ref{const:shsh_hyp_tt}.
\end{definition}

The rule that assigns to a hyperbolic structure its geometric shear-shape cocyle therefore defines a map 
\begin{align*}
    \sigl: \T(S)&\to \SH(\lambda) \\
    X & \mapsto \sigl(X),
\end{align*}
which is the subject of the rest of this article.

\para{Train track independence}
We have employed the language of train tracks for convenience --- the ties of a train track are a useful class of measurable arcs in the sense that they can be made transverse to $\lambda$ and disjoint from $\arc$ (or record the weight associated to an arc of $\arc$).  However, Construction \ref{const:shsh_hyp_tt} and Definition \ref{def:hypshsh} {\it a priori} depend on the choice of geometric train track $\taua$ carrying $\lambda$. 

Now that we have identified the weight system $w(X)$ with the shear-shape cocycle $\sigl(X)$, however, we can invoke both the axiomatic and cohomological interpretations (Definitions \ref{def:shsh_cohom} and \ref{def:shsh_axiom}) to see that the value of $\sigl(X)$ on any arc $k$ transverse to $\lambda$ but disjoint from $\arc$ does not depend on the choice of geometric train track. Indeed, let $k$ be any such arc; then by transverse invariance (axiom (SH1)) we may replace $k$ with a concatenation of short geodesics, all of which are transverse to $\lambda$ but disjoint from $\arc$. By additivity (axiom (SH2)), it therefore suffices to show that the value of $\sigl(X)$ on any short geodesic disjoint from $\arc$ does not depend on the train track.

\begin{lemma}\label{lem:ttsdontmatter}
Let $k$ be a short enough geodesic segment on $X$ that is transverse to $\lambda$. Lift $k$ to an arc $\tilde{k}$ on $\tX$ and let $v$ and $w$ be the hexagons containing the endpoints of $\tilde{k}$; then 
\[\sigl(X)(k) = \sigl(X)(v,w)\]
where on the left $\sigl(X)$ represents the axiomatic shear-shape cocycle and on the right $\sigl(X)$ represents the shear between nearby hexagons (Definition \ref{def:nearbyshear}).
In particular, $\sigl(X)(k)$ does not depend on the choice of train track employed in Definition \ref{def:hypshsh}.
\end{lemma}

In fact, the conclusion of this lemma holds for {\em all} simple pairs.

\begin{proof}
So long as $k$ is short enough (shorter than all arcs of $\arc(X)$) we know that $(v,w)$ is a simple pair. Using axiom (SH1), we may therefore isotope $k$ through arcs transverse to $\lambda$ but disjoint from $\arc$ to an arc $k'$, defined to be the concatenation of $k_{v,w}$, the geodesic connecting the points $p_v$ and $p_w$ on the boundary geodesics $g_v^w$ and $g_w^v$, together with segments of the orthogeodesic foliation inside each hexagon $H_v$ and $H_w$.

Let $\tau$ be a geometric train track snugly carrying $\lambda$ defined with parameter $\epsilon$; then the collapse map $\pi: \epN\lambda \rightarrow \tau$ takes $k'$ to a train route on $\tau$, hence on $\taua$. Orient $k'$ (and hence also the train route $\pi(k')$) so that it travels from $v$ to $w$. Let $v= u_1, u_2, \ldots, u_N=w$ denote the sequence of hexagons corresponding to regions of $\tX \setminus \ttaua$ bordering this train route, so that the regions corresponding to $u_i$ and $u_{i+1}$ both meet the same subsegment of $\pi(k')$. Let $p_i$ denote their corresponding projections onto $\lambda$. Note that since $\pi(k')$ is carried on $\tau \prec \taua$, no pair of subsequent hexagons $u_i$ and $u_{i+1}$ lies in the same component of $\tSp$. This plus the construction of the train track implies that $(u_i, u_{i+1})$ is a simple pair, and we can measure the shear $\sigl(X)(u_i, u_{i+1})$ (up to sign) as the distance along the train track between $\pi(p_i)$ and $\pi(p_{i+1})$.

Now given $\taua$ carrying $\lambda$, we observe that $k'$ also determines a (pair of) relative cycle(s) in the corresponding (orientation cover of the) $\epsilon$-neighborhood of $\lambdaa$. The value $\sigl(X)(k) = \sigl(X)(k')$ is then equal to the value of the cohomological shear-shape cocycle evaluated on either of the oriented lifts $\hat{k}'$ of $k'$ which cross the lift of $\lambda$ with positive local orientation. We may therefore express
\[ [\hat{k}'] = [t_1] - [t_2] + [t_3] - \ldots \pm [t_{N-1}]\]
where $t_i$ is a (lift of a) tie corresponding to the branch of the train track connecting the regions corresponding to $u_i$ and $u_{i+1}$, lifted to the orientation cover to have positive intersection with $\lambda$. See Figure \ref{fig:hypttindep}.

\begin{figure}[ht]
    \centering
\begin{tikzpicture}
    \draw (0, 0) node[inner sep=0] {\includegraphics{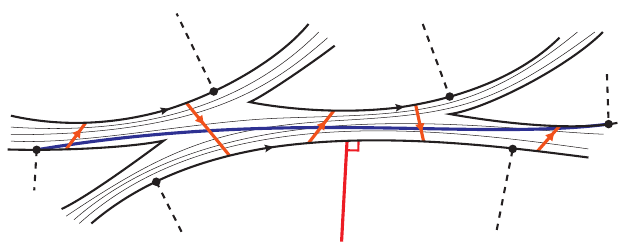}};
     \node at (-5, -0.7){$p_1$};
     \node at (-4.8, -1.4){$v = u_1$};
     \node at (-1.5, 1){$p_2$};
     \node at (-2.5, 2){$u_2$};
     \node at (-2.7, -1.3){$p_3$};
     \node at (-2, -2){$u_3$};
     \node at (2.5, .8){$p_4$};
     \node at (1.7, 1.7){$u_4$};
     \node at (3.1, -.6){$p_5$};
     \node at (3, -2){$u_5$};
     \node at (5.3, .3){$p_6$};
     \node at (5, 1){$u_6=w$};
     \node at (.8, -2){\color{red}$\alpha$};
     \node at (-3.7, .4){\color{orange}$[t_1]$};
     \node at (-1.4, -.8){\color{orange}$-[t_2]$};
     \node at (.5, .6){\color{orange}$[t_3]$};
     \node at (1.9, -.6){\color{orange}$-[t_4]$};
     \node at (4, -.8){\color{orange}$[t_5]$};
     \node at (4.25, .3){\color{blue}$k_{v,w}$};
\end{tikzpicture}
    \caption{Measuring the shear of a small arc using a geometric train track. By isotoping $k$ to a proper arc in the geometric train track neighborhood and then expressing its relative homology class as a sum of the branches, we can compute its shear as the alternating sum of shears between adjacent hexagons.}
    \label{fig:hypttindep}
\end{figure}

But now by construction,
we know that $\sigl(X)$ evaluated on $[t_i]$ is just the shear $\sigl(X)(u_i, u_{i+1})$. In turn, this shear is equal to the signed distance along the train track between $\pi(p_i)$ and $\pi(p_{i+1})$ (where the sign is determined by the local orientation of $\lambda$). Combining this with the expression for $[\hat{k}']$ above, we see that $\sigl(X)(k)$ is exactly equal to the signed distance along the train track between $\pi(p_1)$ and $\pi(p_N)$, which is the shear $\sigl(X)(v,w)$.
\end{proof}

We note that in the proof above, the cohomological interpretation of shear-shape cocycles provides a convenient workaround for the obstacle that the train route with dual transversals $t_1, ..., t_{N-1}$ is not in general isotopic to $k$ through arcs transverse to $\lambda$. Regardless, the relative homology class defined by $ k'\cap \epN\lambda$ is homologous to a linear combination of $\{t_i\}$ in the orientation cover of $\epN\lambda$.

\begin{remark}
The lemma above can also be proved by splitting any two geometric train tracks to a common subtrack \cite[Theorem 2.3.1]{PennerHarer}. Each splitting sequence can then be realized in the orthogeodesically-foliated neighborhood $\epN \lambda \subset X$ by cutting along compact paths in the spine associated to a spike, as in \cite[Section 3]{BonZhu:HD}. Splits induce maps on weight spaces, and so Lemma \ref{lem:ttsdontmatter} is essentially equivalent to the statement that Construction \ref{const:shsh_hyp_tt} is compatible with splitting and collapsing.  See also \cite[Lemma 6]{Bon_THDGL}.
\end{remark}

\para{The cocycle as a map on pairs}
It will be convenient to repackage the data provided by $\sigl(X)$ in yet another form, which also explains our choice of notation in Definition \ref{def:nearbyshear}. 

If $v, w\in \cH$ can be joined by a Lipschitz continuous segment $k_{v,w}$ which is transverse to $\lambda$, disjoint from $\arc$, and meets no leaf of $\tlambda$ twice, then we say that $(v,w)$ is a \emph{transverse pair} and that $k_{v,w}$ is a {\em transversal}. If $(v,w)$ is a transverse pair, we say that $r$ is \emph{between} $v$ and $w$ if there is a transversal $k_{v,w}$ that decomposes as a concatenation of transversals $k_{v,w} = k_{v,r}\cdot k_{r,w}$. Finally, we define
\[\sigl(X)(v,w):= \sigl(X)(k_{v,w})\]
and declare that $\sigl(X)(v,v) = 0$. Observe that if $(v,w)$ are a simple pair then this agrees with our definition of the shear between nearby hexagons (Definition \ref{def:nearbyshear}).

\begin{lemma}\label{lem:shsh_onpairs}
The shear-shape cocycle $\sigl(X)$ defines a map on transverse pairs that satisfies
\begin{enumerate}
    \item {($\pi_1$-invariance) } for each $\gamma\in \pi_1(X)$, we have $\sigl(X)(\gamma v, \gamma w) = \sigl(X)(v,w)$.
    \item {(finite additivity)} If $(v,w)$ is a transverse pair and $r$ is between $v$ and $w$, then 
    \[\sigl(X)(v,w) = \sigl(X)(v,r)+\sigl(X)(r,w).\]
    \item {(symmetry)} $\sigl(X)(v,w) = \sigl(X)(w,v)$.
\end{enumerate}
\end{lemma}

The proof of this lemma is simply a consequence of unpacking the definitions and showing that two different choices of transversals give the same shear values; the latter statement is just a repeated application of Axiom (SH3).

\subsection{Injectivity and positivity}\label{subsec:hypshsh_props}
We now record some initial structural properties of the map $\sigl$ defined above.
In particular, we demonstrate that $\sigl$ is injective and interacts coherently with the orthogeodesic foliation map $\Ol$ and the shear-shape coordinatization $\Il$ of transverse foliations.

Observe that injectivity of $\sigl$ is equivalent to the statement that if two hyperbolic structures have the same complementary subsurfaces and same gluing data along $\lambda$, then they must be isometric.
As the horocyclic and orthogeodesic foliations are equivalent in spikes of complementary subsurfaces, the proofs of \cite[Lemma 11 and Theorem 12]{Bon_SPB} may be invoked {\em mutatis mutandis}.
We outline this argument below for the convenience of the reader, and direct them to \cite{Bon_SPB} for a more thorough discussion of the estimates involved. We remark that this strategy also appears in the proof of Proposition \ref{prop:stretch_reg}, where we use it to piece together Lipschitz-optimal homeomorphisms along $\lambda$.

\begin{proposition}\label{prop:shsh_inj}
The map $\sigl: \T(S) \to \SH(\lambda)$ is injective.
\end{proposition}
\begin{proof}[Sketch of proof]
Fix homeomorphisms $(\tS, \tlambda)$ with $(\tX_i, \tlambda)$ that lift the markings $S\to X_i$ and so that each component $\widetilde{\Sigma} \subset \tS\setminus \tlambda$ maps homeomorphically to a component $\widetilde{Y}_i\subset \tX_i\setminus \tlambda$ for $i = 1,2$.

Suppose that $\sigl(X_1) = \sigl(X_2)$; then in particular $\arcwt(X_1) = \arcwt(X_2)$ and so by Theorem \ref{thm:arc=T(S)_crown}, the complementary subsurfaces $\overline{X_1 \setminus \lambda}$ and $\overline{X_2 \setminus \lambda}$ are isometric.
Therefore, for a given component $\Sigma \subset S \setminus \lambda$, we can find an $\pi_1(\Sigma)$ equivariant isometry $\varphi_\Sigma: \widetilde{Y}_1\to \widetilde{Y}_2$.
Define $\varphi: \tX_1\setminus \lambda \to \tX_2\setminus \lambda$ to be the union of these maps on each complementary component; by construction, $\varphi$ is an isometry.

We need to show that $\varphi$ extends to a $\pi_1(S)$-equivariant isometry $\varphi: \tX_1 \to \tX_2$. To prove this, we apply the arguments of \cite[Lemma 11]{Bon_SPB}, which we summarize presently. The first step is to construct a locally Lipschitz continuous extension of $\varphi$; this step employs the length bound of Lemma \ref{lem:length_shear_bound} and some elementary hyperbolic geometry, and the arguments of the first ten paragraphs of \cite[Lemma 11]{Bon_SPB} may be applied verbatim.

As in Bonahon's original proof, we now show that $\varphi$ is actually $1$-Lipschitz, given that it is locally Lipschitz.
We first show that $\varphi$ does not increase the length of leaves of the orthogeodesic foliation.

Given any segment $\ell$ of a leaf of the orthogeodesic foliation $\widetilde{\Ol(X_1)}$, the length of $\ell$ restricted to any hexagon $H_u$ where $u\in \cH$ is completely determined by the isometry type of $H_u$ and the distance along $\tlambda$ from $p_u\in \partiall H_u$. As $\sigl(X_1)$ determines the shape of $X_1 \setminus \lambda$, we can recover this information and hence determine the length of $\ell\cap H_u$ just from the data of $\sigl(X_1)$.

From $\sigl(X_2)=\sigl(X_1)$, we deduce that the length of $\ell$ in any hexagon of $\tX_1$ is equal to the length of $\varphi(\ell)$ in the corresponding hexagon of $\tX_2$. 
Moreover, since $\varphi$ is locally Lipschitz, the $1$-dimensional Lebesgue measure of $\varphi(\ell) \cap \varphi(\tlambda)$ is at most the $1$-dimensional Lebesgue measure of $\ell\cap \tlambda$.  By a now classical fact the latter is zero \cite{BS}, hence so is the former.
Therefore, the length of $\ell$ in $X_1$ is equal to the length of $\ell$ in $X_2$.

Now there is a path joining any two points in $\tX_1$ built from geodesic segments and segments of leaves of the orthogeodesic foliation. The argument above shows that $\varphi$ preserves the lengths of such paths, so $\varphi$ is globally $1$-Lipschitz. The construction is completely symmetric, so $\varphi\inverse$ is $1$-Lipschitz as well. Now every $1$-Lipschitz homeomorphism between metric spaces with $1$-Lipschitz inverse is necessarily an isometry, and equivariance of $\varphi$ is immediate from the construction. Therefore $X_1$ and $X_2$ must be isometric.
\end{proof}

\para{The diagram commutes}
We have now developed sufficient technology to prove that the geometric shear-shape cocycle of a hyperbolic metric is the same as the shear-shape cocycle associated to its orthogeodesic foliation. In other words, Diagram \eqref{diagram} commutes. Compare with \cite[Proposition 6.1]{MirzEQ}.

\begin{theorem}\label{thm:diagram_commutes}
For all $\lambda \in \ML$ and all $X\in \T(S)$ we have $\sigl(X) = \Il \circ \Ol(X)$.
\end{theorem}
\begin{proof}
Fix a standard smoothing $\taua$ of a geometric train track $\tau$ for $\lambda$ on $X$. Our approach is to compute both $\sigl(X)(b)$ and $\Il\circ\Ol(X)(b)$ for each branch $b$ of $\taua$. These numbers will coincide, so by Proposition \ref{prop:ttcoords}, $\sigl(X) = \Il\circ\Ol(X)$.

Let $\mathsf{T}_X \subset X$ be the piecewise geodesic triangulation of $X$ whose vertices are the vertices of $\Sp$, so that there is an edge between $v,w \in \Sp$ if the corresponding regions of $X \setminus \taua$ share a branch. This recipe generically yields a triangulation, but may have quadrilaterals in the case that two points of $\arc(X) \cap \lambda$ lie on the same leaf of $\Ol(X) \cap \epN\lambda$. In this case, we may either choose a smaller initial neighborhood to define our geometric train track so that this does not occur, or these points correspond to arcs that meet an isolated leaf of $\lambda$ on either side; in the latter case, choose either diagonal that crosses the quadrilateral to include into $\mathsf{T}_X$.
Observe that each edge of $\mathsf{T}_X$ is either transverse to $\Ol(X)$ or a segment of a leaf (on the off chance that two adjacent regions have exactly $0$ shear between them).

Let $q = q(\Ol(X), \lambda)$, and recall that Proposition \ref{prop:deflation} provides a homotopy equivalence $\Defl : X \to q$ in the correct homotopy class satisfying ${\Defl}_*\Ol(X) = V(q)$ and ${\Defl}_*\lambda = H(q)$ both leafwise and measurably.
Furthermore, $\Defl$ maps $\mathsf{T}_X$ to a (topological) triangulation of $q$ with vertices at its zeros.
It therefore remains to show that $\sigl(X)$ evaluated on a branch of $\taua$ is the same as $\Il(q)$ evaluated on the dual edge of this triangulation.

Now by definition, $\arcwt(X) = \arcwt(q)$, so consider a branch $b$ of $\taua$ not corresponding to an arc of the arc system.
Dual to $b$ there is an edge $e$ of the triangulation $\Defl(\mathsf{T}_X)$ which is transverse to the orthogeodesic foliation $\Ol(X)$ on $q$ (since $\mathsf{T}_X$ was transverse to $\Ol(X)$ on $X$).
Up to sign, the value of $\Il(q)$ on $b$ is the magnitude of the real part of the period of $e$, which is just the geometric intersection number $i(\Ol(X), e)$ by transversality.

On the other hand, we have that $\sigl(X)(b)$ is equal to the shear between the two hexagons on either side of $b$. This in turn is equal to the geometric intersection number $i(\Ol(X), k_{v,w})$ up to sign, where $k_{v,w}$ is the geodesic connecting the vertices $p_v$ and $p_w$ of $\lambda \cap \arc(X)$. Since $\Defl$ takes $k_{v,w}$ to an arc transversely isotopic to $e$, we have that $|\sigl(X)(b)| = |\Il(q)(b)|$.

Finally, to show that the signs are equal, fix matching orientations on $k_{v,w}$ and $e$. These induce a local orientations on the leaves of $\lambda$ so that the algebraic intersection of $\lambda$ with $k_{v,w}$, respectively $e$, is positive. In turn, this induces a local orientation on the leaves of $\Ol(X)$ near $k_{v,w}$, respectively $e$, and our sign conventions are equivalent to stipulating that the sign is positive if $k_{v,w}$, respectively $e$, crosses $\Ol(X)$ from left to right and negative if it crosses from right to left (compare \cite[\S5.2]{MirzEQ}). In particular, the signs agree and so $\sigl(X)(b) = \Il(q)(b)$ for all branches $b$, completing the proof of the theorem.
\end{proof}

\begin{corollary}\label{cor:sigl_into_SH+}
For all $\mu \in \Delta(\lambda)$, we have an equality 
\[\ThSH(\sigl(X), \mu) = i(\Ol(X),\mu) = \ell_{X}(\mu)>0.\]
In particular, $\sigl(\T(S))\subseteq \SH^+(\lambda)$.
\end{corollary}

\begin{proof}
The first equality is a direct consequence of Theorem \ref{thm:diagram_commutes} and Proposition \ref{prop:Il_takes_int_to_Thurston}.  The second equality was proved in Lemma \ref{lem:length_computation}.
\end{proof}

\section{Shape-shifting cocycles}\label{sec:shapeshift_def}

In the previous section, we explained how to associate to each hyperbolic structure $X$ a shear-shape cocycle $\sigl(X)$.
In this one, we explain how to upgrade a small deformation $\ac$ of the cocycle into a deformation of the hyperbolic structure; this is eventually used to prove that $\sigl: \T(S) \rightarrow \SH^+(\lambda)$ is open (Theorem \ref{thm:shsh_open} below). The main issue that we need to overcome is that we must simultaneously change the geometry of the non-rigid components of $X\setminus \lambda$ while shearing these subsurfaces along one another.  

The goal of this section is therefore to build, for every small enough deformation $\ac$ of $\sigl(X)$, a $\pi_1(S)$-equivariant {\em shape-shifting cocycle} that records how to adjust the relative position of geodesics of $\lambda$:
\[\varphi_\ac: \partiall\cH \times \partiall\cH  \to \Isom^+\tX\]
where $\partiall\cH := \{(h_v, p_v) \subset \partiall H_v: v\in \cH\}$ is the set of boundary geodesics of $\tlambda$ equipped with basepoints obtained from projections of the vertices of $\tSp$. See Proposition \ref{prop:shapeshift_cocycle}.

In Section \ref{subsec:shapeshift_deform} below, we explain how to modify the developing map $\tX \rightarrow \mathbb{H}^2$ according to $\varphi_\ac$, resulting in a new (equivariant) hyperbolic structure $X_\ac$ with geometric shear-shape cocycle $\sigl(X) + \ac$ (Lemma \ref{lem:shsh_correct}).
By fixing a pointed geodesic $(h_v, p_v) \in \partiall \cH$ we identify $\Isom^+(\tX)$ with $T^1 \tX$, so that the projection of $\{\varphi_\ac((h_v, p_v),(h_w, p_w)) \mid (h_w, p_w)\in \partiall \cH\}$ to $\tX$ is then be the geodesic realization of $\tlambda$ in the new metric $\tX_\ac$. 

When the deformation $\ac$ preserves $\arcwt(X)$, the cocycle $\varphi_\ac$ corresponds to a cataclysm map: the complementary components of $\tX\setminus\tlambda$ are sheared along the leaves of $\tlambda$ and map isometrically into the deformed surface $X_\ac$.  When $\ac$ alters $\arcwt(X)$, we must shear the complementary subsurfaces while also simultaneously changing their shape, introducing complications not present in Bonahon and Thurston's original considerations.

\para{Deforming the cocycle}
We first make explicit what we mean by a deformation of a shear-shape cocycle; we quantify what we mean by ``small'' in Section \ref{subsec:shsh_spikes}.

Observe that if $\sigma$ and $\sigma'$ in $\SH^+(\lambda)$ are close, then by Proposition \ref{prop:SH+_structure} we know that their underlying weighted arc systems $\arcwt$ and $\arcwt'$ are close in $\Base$.
In particular, the corresponding unweighted arc systems $\arc$ and $\arc'$ must both live in some common top-dimensional cell of $\Base$, i.e., must both be contained in some common maximal arc system $\arcb$.
Let $\tau$ be some snug train track for $\lambda$ and let $\tau_{\arcb}$ be a standard smoothing of $\tau \cup \arcb$. By Proposition \ref{prop:ttcoords}, we may then identify $\sigma$ and $\sigma'$ as weight systems on $\tau_{\arcb}$; the difference $\sigma - \sigma' \in W(\tau_{\arcb})$ is then a deformation of $\sigma$.

In general, if $(\sigma, \arcwt) \in \SH^+(\lambda)$ and $\arcb$ is any maximal arc system containing the support of $\arcwt$, then the deformations we consider in this section are those $\ac \in W(\tau_{\arcb})$ such that $\sigma + \ac \in W(\tau_{\arcb})$ corresponds to a positive shear-shape cocycle.
Passing between equivalent definitions of shear-shape cocycles, we see that we may also think of $\ac$ as a ``shear-shape cocycle with negative arc weights.''
\label{ind:ac}
The underlying weighted arc system of any deformation $\ac$ will be denoted by $\acarc$; while its coefficients are not necessarily positive, they will satisfy the zero total residue condition of \eqref{eqn:ressum=0} by construction.

By Theorem \ref{thm:arc=T(S)_crown}, the arc system $\arcwt + \acarc$ gives each component of $S \setminus\lambda$ a new complete hyperbolic metric $Y$ with (non-compact) totally geodesic boundary.
Since the supports of $\arcwt$ and $\arcwt + \acarc$ are both contained inside of some common maximal $\arcb$, one may set up a correspondence between the complementary components of $X \setminus \lambdaa$ with the components of $Y \setminus \text{supp}(\arcwt + \acarc)$ (adding in weight 0 edges as necessary).

\para{A blueprint}
To help guide the reader through this rather intricate construction, we include here a top-level overview of the necessary steps, together with an outline of the section.
Briefly, our strategy is to explicitly define $\varphi_\ac$ on two types of pairs of pointed geodesics: the ``simple pairs'' between which the orthogeodesic foliation is comparable to the horocyclic, and the pairs which live in the boundary of a common subsurface.
Piecing together these basic deformations then allows us to define $\varphi_\ac$ on arbitrary pairs of pointed geodesics.

Our construction of $\varphi_\ac$ for simple pairs parallels Bonahon's construction of shear maps \cite[Section 5]{Bon_SPB}, and as such requires a detailed analysis of the geometry of the spikes of $\tX \setminus \tlambda$. We therefore devote Section \ref{subsec:geom_spikes} to recording a number of useful notions and estimates from \cite{Bon_SPB}. In this section, we also introduce the ``injectivity radius of $X$ along $\lambda$,'' which measures the length of the shortest curve carried on a maximal snug train track for $\lambda$ and plays a crucial role in our convergence estimates.

After these preliminary considerations, we turn in Section \ref{subsec:shsh_spikes} to the actual construction of $\varphi_\ac$ on simple pairs.
As in \cite{Bon_SPB}, the map is defined by adjusting the lengths of countably many horocyclic arcs in an appropriate neighborhood of $\lambda$, compensating for changing shears between hexagons.
Unlike in \cite{Bon_SPB}, we must also adjust the arcs to account for the changing shapes of each of the spikes (as we are deforming the complementary subsurfaces).
Convergence of the resulting infinite product of parabolic transformations is delicate; our approach follows \cite[Section 5]{Bon_SPB} with influence from the more geometric approach of \cite{Th_stretch}.
An accessible treatment of Thurston's construction of ``cataclysm coordinates'' can be found in \cite[Section 3.5]{PapadopTheret:Teich_Thurston}.

We then turn in Sections \ref{subsec:shsh_hex} and \ref{subsec:shsh_spine} to defining $\varphi_\ac$ on pairs of geodesics in the boundary of the same hexagon or the same complementary subsurface, respectively.
It is here that our work significantly differs from that of Bonahon and Thurston.
In these sections we also develop the idea of ``sliding'' a deformed complementary subsurface along the original; this viewpoint allows us to easily demonstrate a number of otherwise nontrivial relations between M{\"o}bius transformations (see Propositions \ref{prop:cocycle_hex}, \ref{prop:cocycle_deg_hex}, and \ref{prop:subsurf_cocycle}).

Finally, in Section \ref{subsec:shapeshift_total} we build the shape-shifting cocycle $\varphi_\ac$ from these pieces; the cocycle relation (Proposition \ref{prop:shapeshift_cocycle}) then follows from the cocycle relations for pieces and the separation properties of $\tlambda$.

\begin{note}
We remark that throughout this section and the next, we consider isometries via their action on a pointed geodesic, and compositions should be read from right to left.
\end{note}

\subsection{Geometric control in the spikes}\label{subsec:geom_spikes}
We first record some useful definitions and associated geometric estimates. These estimates play a crucial role in establishing convergence of the infinite products appearing in Section \ref{subsec:shsh_spikes} below. Many of our definitions follow Bonahon's, but in order to contend with the fact that the complementary subsurfaces of $\lambda$ are not always isometric, we must relate certain constants to the geometry of $\lambda$ on $X$ (see Lemma \ref{lem:decay_gaps}, in particular).

Our discussion will take place with certain data fixed. Choose a hyperbolic surface $X\in \T(S)$ and a measured lamination $\lambda\in \ML(S)$. Let $\epsilon>0$ be small enough so that an $\epsilon$-geometric train track $\tau$ on $X$ carries $\lambda$ snugly. The standard smoothing $\taua$ for the arc system $\arc = \arc(X)$ provides us with a vector space $W(\taua)$ that models $\SH(\lambda; \arc)$.
With $\taua$ fixed, we endow the vector space of weights on branches of $\taua$ with the sup norm $\| \cdot \|_{\taua}$, and restrict this norm to the weight space $W(\taua)$.

Let $k_b$ be an oriented geodesic transverse to a branch $b\in \tau$ that also avoids $\arc$.
Following Bonahon, we define the \emph{divergence radius} or \emph{depth} $r_b(d) \in \ZZ_{> 0}$ of a component $d$ of $k_b \setminus \lambda$ to be ``how long the leaves of $\lambda$ incident to $d$ track each other,'' as viewed by $\tau$.
\label{ind:divrad}

More precisely, lift everything to the universal cover $\tX$. By convention, set $r_b(d) = 1$ if $d$ contains one of the endpoints of $k_b$.
Otherwise, $d$ is contained in a spike of $H_v$ for some $v\in \cH$, i.e., $d$ connects a pair of asymptotic geodesics $g_d^-$ and $g_d^+$.
The divergence radius $r_b(d)$ is then the largest integer $r\ge 1$ such that $\pi(g_d^+)$ and $\pi(g_d^-)$ successively cross the same sequence of branches
\[b_{-r+1}, b_{-r+2}, ... , b_0, ... b_{r-2}, b_{r-1}\]
of $\ttau$, where $b_0$ is the lift of $b$ meeting $\tilde k_b$ and $\pi: \epN\tlambda \rightarrow \ttau$ is the collapse map. 
By equivariance, $r_b(d)$ is clearly independent of the choice of lift $\tilde k_b$ of $k_b$.

\begin{remark}
After projecting back down to $\tau \subset X$, either
$b_{-r+1}\cdot ... \cdot b_0$ or $b_0\cdot ...\cdot b_{r-1}$ defines a train route $\gamma_d$ in $\tau$ that starts at $b$ and terminates by ``opening up'' into the projection of $H_v$ in $X$.  That is, the geodesics $g_d^+$ and $g_d^-$ diverge from each other (at scale $\epsilon$) at the terminus of $\gamma_d$. 
\end{remark}

Now there are boundedly many spikes of $X \setminus \lambda$, and for each $r \ge 1$ each spike may contain at most $1$ component $d\subset k_b\setminus \lambda $ with depth exactly $r$. This gives us the following bound:

\begin{lemma}[Lemma 4 of \cite{Bon_SPB} and Lemma 5 of \cite{BonSoz}]\label{lem:bdd_gaps}
For any branch $b$ of $\tau$ and any transversal $k_b$, the number of components $d$ of $k_b\setminus \lambda$ with $r_b(d) = r$ is at most $6|\chi(S)|$.
\end{lemma}

The train track interpretation of the depth of a segment also allows us to bound the value of a shear-shape cocycle $\ac$ in terms of its weights on a snug train track and the depth of its endpoints.

More specifically, for each component $d$ of $k_b\setminus \lambda$, let $k_b^d$ be the subarc of $k_b$ joining the initial point of $k_b$ to any point of $d$.
\label{ind:kbd}
Then for any combinatorial deformation $\ac$ and $b$ a branch of $\taua$, there is an explicit formula for $\ac(k_b^d)$ as a linear function of the weights of $\ac$ on $\taua$ with at most $r_b(d)$ terms \cite[Lemma 6]{Bon_THDGL}. Conceptually, this formula arises by splitting $\taua$ open along the spike $s$ containing $d$,  until $d$ is ``visible'' in some new track $\taua'$ carried by $\taua$ (see also the proof of Lemma \ref{lem:ttsdontmatter}).

The exact expression for $\ac(k_b^d)$ will not be important for us; instead, we record the following estimate, which follows by considering the growth of edge weights upon splitting.

\begin{lemma}[Lemma 6 of \cite{Bon_SPB} and Lemma 6 of \cite{BonSoz}]\label{lem:lin_growth_gaps}
Let $k_b$ be a transversal of a branch $b$. Then 
\[|\ac(k_b^d)|\le \|\ac\|_{\taua}r_b(d)\]
for every $\ac\in \SH(\lambda; \arc)$ and every component $d$ of $k_b\setminus \lambda$.
\end{lemma}
We remark that our definitions of $\|\cdot \|_{\taua}$ and $r_b(\cdot)$ make the bound given in Lemma \ref{lem:lin_growth_gaps} hold without a topological multiplicative factor, as in \cite{Bon_SPB}.

\para{Geometric estimates on depth}
The depth of a component $d$ of $k_b \setminus \lambda$ is proportional to the distance from a lift $\tilde d$ to the vertex $u \in \cH$ inside of the corresponding spike. The constant of proportionality in turn depends on how quickly the spike of $H_u$ containing $\tilde d$ returns to $k_b$ on $X$; we now identify a quantity that will allow us to estimate this constant.

Let $k$ be any geodesic arc transverse to $\lambda$ such that each lift $\tilde k$ to $\tX$ bounds a spike in every hexagon that it crosses; equivalently, the endpoints of $\tilde k$ lie in a simple pair of hexagons. As in Section \ref{subsec:shsh_hyp}, it suffices for $k$ to be shorter than the shortest arc of $\arc(X)$. Now for each leaf $g$ of $\tlambda$, there is a bound $R_k(g) > 0$ for the distance in $g$ between intersections of $g$ with different lifts $\tilde{k}_1$ and $\tilde{k}_2$ of $k$. Indeed, any two lifts of $k$ meeting $g$ differ by a deck transformation $\gamma\in \pi_1(X)$ determined by a path in $X$ that traces along the projection of a segment in $g$ and then closes up along $k$.

We then define the \emph{injectivity radius of $X$ along $\lambda$} to be
\[\inj_\lambda(X) := 
\inf_{k\pitchfork \lambda}
\inf_{g \subset \tlambda} R_k(g)\]
\label{ind:injradlam}
where the infimum is taken over all transverse arcs $k$ whose endpoints lie in a simple pair of hexagons.

Equivalently, the injectivity radius of $\lambda$ may also be computed by taking a $\epsilon$ so that the geometric train track $\tau_{\max}$ built from $\epN \lambda$ is snug and so that for all $\epsilon' > \epsilon$, the train track built from $\epN \lambda$ is the same (not just equivalent) to $\tau_{\max}$, as follows.
\label{ind:taumax}
\footnote{Observe that any $\epsilon$ sufficiently close to the supremum of $\epsilon$ for which $\epN \lambda$ is snug satisfies these conditions.}

For each branch of $\tau_{\max}$, choose a tie $t_b$ (that is, a leaf of the orthogeodesic (or horocyclic) foliation restricted to $\epN \lambda$ that is transverse to $b$). The injectivity radius along $\lambda$ is then equal to the infimum of the recurrence times of $\lambda$ to any $t_b$. Using the ``length along a geometric train track'' function $\ell_{\tau_{\max}}$ defined in  Section \ref{subsec:ortho_foliation}, we may therefore write
\begin{equation}\label{eqn:injl_via_tts}
\inj_\lambda(X) = 
\inf_{\gamma \prec {\tau_{\max}}} \ell_{\tau_{\max}} (\gamma)
\end{equation}
where the infimum is taken over all simple closed curves $\gamma$ carried on the train track ${\tau_{\max}}$.

\begin{remark}
The length of the hyperbolic systole of $X$ is clearly a lower bound for $\inj_\lambda(X)$, which is therefore positive. However, $\inj_\lambda(X)$ can be much larger than the length of the systole.

For example, if $\lambda$ does not fill the surface then there can be a disjoint curve of arbitrarily small length. In addition, $X$ may have a very short curve $\gamma$ transverse to $\lambda$, and if $\lambda$ does not twist around $\gamma$, then $\inj_\lambda(X)$ is necessarily very large.
\end{remark}

We can now relate the geometry of small arcs to their depth and injectivity radius along $\lambda$.

\begin{lemma}[Lemmas 3 and 5 of \cite{Bon_SPB} and Lemma 4 of \cite{BonSoz}]\label{lem:decay_gaps}
Given a branch $b$ of a geometric train track $\tau$ constructed from $\lambda$ on $X$ and a short transversal $k_b$, there exists $B>0$ such that the following holds. For every component $d$ of $k_b\setminus \lambda$ with depth $r_b(d)$,
\[\ell_{X}(d)\le Be^{-D_\lambda(X)r_b(d)},\]
where $D_\lambda(X) =\inj_\lambda(X)/9|\chi(S)|$. 
\label{ind:Dlam}
\end{lemma}
\begin{proof}
The idea is the same as in the reference, but our constants are different.
Small geodesic arcs meeting a spike $s$ of a hexagon $H_v$ transversely and far away from the vertex $v$ look like horocycles, which have length that decays exponentially in distance from $v$.
Therefore, we just need to give a lower bound for the distance between $d$ and $v \in H_v$ along the spike $s$ in terms of $\inj_\lambda(X)$ and the topological complexity of $S$.

Consider the train path $\gamma_d$ starting at $b$ that defines $r_b(d)$. By definition, $\gamma_d$ traverses exactly $r_b(d)$ branches of $\tau$ (counted with multiplicity).
Now $\gamma_d$ decomposes as a concatenation of maximal sub--train paths with embedded interiors, each forming a \emph{simple loop} in $\tau$.
\footnote{A simple loop on a train track is a carried curve which traverses each branch at most once.}

The depth $r_b(d)$ is thus bounded above by the number of consecutive simple loops in $\gamma_d$ times the size of the longest simple loop in $\tau$. The size of a simple loop in $\tau$ is in turn bounded above by the number of branches of $\tau$, which is at most $9|\chi(S)|$.
Finally, since each simple loop in $\gamma_d$ is carried on $\tau \prec {\tau_{\max}}$ it must have length at least $\inj_\lambda (X)$ by \eqref{eqn:injl_via_tts}.

Putting the above estimates together, we see that the distance between $v$ and $d$ in $H_v$ is at least
\[\inj_\lambda(X)\cdot \#\{\text{simple loops in $\gamma_d$}\}\ge \frac{\inj_\lambda(X)r_b(d)}{\text{size of the longest simple loop in $\tau$}}\ge \frac{\inj_\lambda(X)}{9|\chi(S)|}r_b(d),\]
and the lemma follows.
\end{proof}

\subsection{Shape-shifting in the spikes}\label{subsec:shsh_spikes}
Our discussion now begins to diverge from \cite{Bon_SPB}.
While pairs of asymptotic geodesics are all isometric, the spikes of $X \setminus \lambda$ come with extra decoration, namely, a choice of horocycle at each cusp (equivalently, basepoints which lie on a common leaf of the orthogeodesic foliation).
In this section, we explain how to use these decorations to define the shape-shifting cocycle $\varphi_\ac$ on pairs of basepointed geodesics coming from simple pairs of hexagons.

We remind the reader that $X$, $\lambda$, and $\taua$ are fixed so that geometric objects like geodesic segments, hexagons, arcs of $\arc(X)$, etc. are understood to live in and be realized (ortho)geodesically on $X$. Throughout this section we will fix $\arcwt = \arcwt(X)$ and use it to denote both a weighted arc system as well as the induced metric on $S \setminus \lambda.$ Finally, we recall that $\ac$ is a combinatorial deformation of $\sigl(X)$ which changes $\arcwt$ by $\acarc$; we will refer to the deformed hyperbolic structure on $S \setminus \lambda$ by $\arcwt + \acarc$ and its hexagonal pieces by $G_u$ for $u \in \cH$.

\para{Shapes of spikes} The group $\PSL_2(\RR)$ acts transitively on pairs of asymptotic geodesics but, having done so, cannot further act on the family of horocycles based at the spike. To measure this failure, we associate below a geometric parameter which records the placement of basepoints in each spike.

Suppose that $u \in \cH$ is a hexagon of $\tX \setminus \tlambdaa$ and $s$ is a spike of $H_u$, that is, a pair of asymptotic geodesics $g$ and $g'$. Both $g$ and $g'$ come with basepoints $p$ and $p'$ obtained by projecting $u$ to these geodesics.
We then associate to $s$ the number $h_{\arcwt}(s)$ which measures the length of either of the orthogeodesic leaves connecting $u$ to $p$ or $p'$:
\[h_{\arcwt}(s):= d(p,u) = d(p',u).\]
\label{ind:has}
Our notation reflects the fact that this function clearly depends only on the geometry of $X \setminus \lambda$ and not the shearing along $\lambda$. 
The reader familiar with the literature will observe that this parameter is essentially an orthogeodesic version of the ``sharpness functions'' appearing in \cite{Th_stretch}.

In order to measure the difference in sharpness functions between the realizations of $s$ in $\arcwt$ and in the deformed metric $\arcwt + \acarc$, we superimpose the hexagons $H_u$ and $G_u$ and measure the distance between their boundary basepoints.

More concretely, choose an arbitrary orientation $\vec{s}$ of the spike $s$ and fix realizations of both $H_u$ and $G_u$ inside of $\mathbb{H}^2$. As $\PSL_2(\RR)$ acts simply transitively on triples in $\partial \mathbb{H}^2$, there is a unique isometry that takes the realization of $s$ in $G_u$ to its realization in $H_u$. The vertex $u$ of $\Sp$ is realized in both $H_u$ and $G_u$; let $p$ and $q$ denote the projections of these points to one of the boundary geodesics $g$ of $s$. See Figure \ref{fig:spikeshape}.

\begin{lemma}\label{lem:spike_param}
With all notation as above, the signed distance from  $q$ to $p$ along $g$ is 
\begin{equation}\label{eqn:spike_param}
 f_{X, \ac}(\vec{s}) := \varepsilon \log \left( \frac{\tanh{h_{\arcwt}(s)}}{\tanh{h_{\arcwt + \acarc}(s)}}\right) \in \RR,
\end{equation}
where $\varepsilon = +1$ if $\vec{s}$ is oriented towards the shared ideal endpoint, and $\varepsilon = -1$ otherwise.
\end{lemma}

\begin{figure}[ht]
    \centering
\begin{tikzpicture}
    \draw (0, 0) node[inner sep=0] {\includegraphics{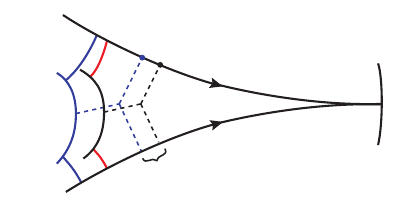}};
    \node at (-1.1, 1.1)[blue]{$q$};
    \node at (-.7, .9){$p$};
    \node at (1, .4){$\vec{s}$};
    \node at (-.6, -1.2){$f_{X,\ac}(\vec{s})$};
    \node at (-2.6, 0)[blue]{$G_u$};
    \node at (-.75, 0){$H_u$};
\end{tikzpicture}
    \caption{Superimposing hexagons to measure the difference in the shapes of their spikes.}
    \label{fig:spikeshape}
\end{figure}

The parameter $f_{X, \ac}(\vec{s})$ plays a crucial role below in our definition of the shape-shifting map on spikes. In our convergence estimates, we will also need to consider the parameter
\begin{equation}\label{eqn:normac_def}
\|\ac\|_{\vec{s}} := \max_s | f_{X, \ac}(\vec{s})| < \infty
\end{equation}
which quantifies the maximum distance that the deformation $\ac$ moves a basepoint in a spike.

\begin{proof}
We compute in the upper half plane; up to isometry, we may assume that $s$ is bound by the imaginary axis $V = i\RR_{>0}$ and its  translate $2 + V$; the spine of the orthogeodesic foliation in this spike is a subsegment of the vertical line $1+V$. With this choice fixed, the projections $p$ and $q$ of $u$ to $V$ may be identified with $i e^a$ and $i e^b$ for some $a$ and $b$, respectively. Without loss of generality, we may also assume that $V$ is oriented upwards (towards $\infty$); the opposite choice of orientations simply reverses all signs at the end of the computation.

Now for $t\ge 0$, the path $t\mapsto \tanh t + i \sech t$ is the unit speed parametrization of the orthogeodesic emanating from $V$ at $i$. Observe that the isometry  $z \mapsto e^a z$ stabilizes $V$ and takes this segment to an orthogeodesic segment emanating from $i e^a = p$ which is distance $a$ from $i$. Since the orthogeodesic segment through $p$ meets the spine $1+V$ after traveling distance $h_{\arcwt}(s)$ (by definition), this implies that 
\[e^a = \tanh h_{\arcwt}(s).\]
Similarly, we have that $e^b = \tanh h_{\arcwt+\acarc}(s)$. Together, these imply that
\[\frac{\tanh{h_{\arcwt}(s_u)}}{\tanh{h_{\arcwt+\acarc}(s_u)}} =e^{a-b}.\]
Taking logarithms, we see that $a-b$ is the signed distance from $q$ to $p$ along $V$, as claimed.
\end{proof}

\begin{remark}
Note that by Theorem \ref{thm:arc=T(S)_crown}, the parameter $f_{X, \ac}(\vec{s})$ varies analytically in $\acarc$ (hence $\ac$).
\end{remark}

\para{Orientation conventions}
We now specialize to the case where $(v,w)$ is a simple pair of hexagons with associated oriented geodesic $k_{v,w}$ running between $p_v^w$ on $g_v^w$ (the projection of $v$ to the boundary leaf of $\partiall H_v$ closest to $w$) and $p_w^v$ on $g_w^v$.

Each leaf $g \subset \tlambda$ crossed by $k_{v,w}$ inherits an orientation by declaring that {\em turning right} onto $g$ while traveling from $v$ to $w$ along $k_{v,w}$ is the positive direction. We remark that if $k_{v,w}$ crosses a hexagon $H_u$, then the induced orientation of $g_u^w$, the geodesic in $\partiall H_u$ closest to $w$, is the {\em opposite} of the orientation of $g_u^w$ induced as a part of the boundary of $H_u$. On the other hand, the two orientations on $g_u^v$ induced by $k_{v,w}$ and coming from $H_u$ agree. This is an artifact of our sign convention for measuring shears; see Remark \ref{rmk:hypsh_sign}. 

If $g$ is a complete oriented geodesic in the hyperbolic plane and $t\in \RR$, we let $T_g^t$ be the hyperbolic isometry stabilizing $g$ and acting by oriented translation distance $t$ along $g$. The opposite orientation of $g$ will be denoted $\bar g$, so that $T_{\bar g}^t = T_g^{-t}$.
\label{ind:spikes}

For an oriented spike $\vec{s} = (g_u^v, g_u^w)$, its opposite orientation is $\cev{s} = (\bar g_u^w,\bar  g_u^v)$. In particular, we note that if $\vec{s}$ is an oriented spike of $H_u$ crossed by $k_{v,w}$, then $\cev{s}$ is an oriented spike crossed by $k_{w,v}=\bar k_{v,w}$.

\para{Shape-shifting in spikes}
Suppose $(v,w)$ is a simple pair and suppose $u$ is between $v$ and $w$. Let $\vec{s} = (g_u^v, g_u^w)$ be the spike of $u$ crossed by $k_{v,w}$ with basepoints $p_v$ and $p_w$.
We define the elementary \emph{shaping transformation} $A(\vec{s})\in \Isom^+(\tX) = \PSL_2\RR$ determined by $X$,  $\ac$, and $s$ to be
\begin{equation}\label{eqn:spike_shape}
    A(\vec{s}) : = T_{g_u^v}^{f_{X, \ac}(\vec{s})}\circ T_{g_u^w}^{-f_{X, \ac}(\vec{s})}.
\end{equation}
Ultimately, the element $A(\vec{s})$ will be the value of the shear-shape cocycle $\varphi_\ac$ on the pair $(g_u^v, g_u^w)$; see just below for an explanation of how we think of $A(\vec{s})$ as ``changing the shape'' of $s$.

Observe that $A(\vec{s})$ is a parabolic transformation preserving the common ideal endpoint of $s$.
A familiar computation shows that in the spike determined by $g_u^v$ and $A(\vec{s})g_u^w$, the  orthogeodesics emanating from $p_v$ and $A(\vec{s})p_w$ meet at a point distance $h_{\arcwt+\acarc}(s)$ from each (supposing that the deformation is small enough).

To the oriented spike $\vec{s}$ of $u$, we also associate the \emph{elementary shape-shift}
\begin{equation}\label{eqn:spike_shsh}
    \varphi(\vec{s}) : = T_{g_u^v}^{\ac(v,u)}\circ A(\vec{s})\circ T_{g_u^w}^{-\ac(v,u)}
    = T_{g_u^v}^{\ac(v,u) + f_{X, \ac}(\vec{s})}\circ  T_{g_u^w}^{-(\ac(v,u)+f_{X, \ac}(\vec{s}))}
\end{equation}
where we recall that the value $\ac(v,u)$ is obtained by thinking of $\ac$ as a function on transverse pairs ({\`a} la Lemma \ref{lem:shsh_onpairs}).
Note that $\varphi(s)$ depends on our reference point $v$: whereas $A(\vec{s})$ is eventually identified as a value of the shape-shifting cocycle $\varphi_\ac$, the elementary shape-shifts $\varphi(\vec{s})$ are only building blocks for values of $\varphi_\ac$.

For the opposite orientation $\cev{s} = (\bar g_u^w, \bar g_u^v)$, we check
\begin{equation}\label{eqn:spike_shape_inverse}
    A(\cev{s} ) = T_{\bar g_u^w}^{f_{X, \ac}(\cev{s})}T_{\bar g_u^v}^{-f_{X, \ac}(\cev{s})}
    = T_{ g_u^w}^{f_{X, \ac}(\vec{s})} T_{ g_u^v}^{-f_{X, \ac}(\vec{s})}=A(\vec{s})\inverse.
\end{equation}
Since $\ac(v,u) = \ac(u, v)$, we may similarly observe that $\varphi(\cev{s}) = \varphi(\vec{s})\inverse$.

Take $\cH_{v,w}$ to be the set of hexagons between $v$ and $w$ equipped with the linearing order $u_1<u_2$ induced by the orientation of $k_{v,w}$.
Let $\fcH\subset \cH_{v,w}$ be any finite subset and order its elements $\fcH = \{u_1, ... u_n\}$. For short, we denote hexagons $H_i:=H_{u_i}$, spikes $s_i := \vec{s}_{u_i}$, geodesics $g_i^v := g_{u_i}^v$, etc.

To the finite, ordered set $\underline{\cH}$ we associate the product
\begin{equation}\label{eqn:first_adj}
\varphi_{\underline{\cH}} : = \varphi(s_1) \circ ... \circ \varphi(s_n) \circ T_{g_w^v}^{\ac(v,w)}\in \Isom^+(\tX).
\end{equation}
The goal of the rest of the section is then to extract a meaningful limit from $\varphi_{\underline{\cH}}$ as  $\underline{\cH}$ increases to $\cH_{v,w}$.
Ultimately, this limit is how we will define the shape-shifting cocycle $\varphi_\ac$ on the boundary geodesics $g_v^w$ and $g_w^v$ corresponding to the simple pair $(v,w)$.

\begin{remark}
In the case that $\lambda$ is maximal, each $H_i$ is an ideal triangle and so $\arcwt = \arcwt + \acarc = \emptyset$. In this case, each spike parameter $f_{X, \ac}(s_i)$ is $0$ and we recover the formula from \cite[p. 255]{Bon_SPB}.
\end{remark}

\para{Geometric explanation of \eqref{eqn:spike_shsh}}
Before proving convergence, however, let us explain the intuition behind the formulas above. In order to interpret $A(\vec{s})$ and $\varphi(\vec{s})$ as deformations of the hyperbolic structure $X$, we will switch our viewpoint to think of them as values of a deformation cocycle, and so as affecting the placement of pointed geodesics relative to each other. For brevity, let $f_{X, \ac}(\vec{s}) = t$.

Let us focus first on the shaping transformation $A(\vec{s})$. The oriented spike $\vec{s}$ in the hexagon $H_u$ is formed by two pointed geodesics $(g_u^v, p_u^v)$ and $(g_u^w, p_u^w)$. Fixing our viewpoint at $(g_u^v, p_u^v)$, we may think of $A(\vec{s})$ as deforming $\tX$ by holding $(g_u^v,p_u^v)$ fixed and identifying $(g_u^w, p_u^w)$ with $A(\vec{s}) \cdot (g_u^v, p_u^v)$.
This has the overall effect of ``widening'' the spike $s$ so that its sharpness parameter increases from $h_{\arcwt}$ to $h_{\arcwt + \acarc}$.

If instead we fix our basepoint to be outside of $H_u$, say at the basepoint $p_v^w$ on $g_v^w \subset \partiall H_v$, then this transformation can viewed as a composition of left and right earthquakes.
Let $Q^w$ and $Q^v$ denote the half-spaces to the left of the oriented geodesics $g_u^w$ and $g_u^v$, respectively. Note that $Q^w \subset Q^v$.
The deformation $A(\vec{s})$ may then be thought of as first transforming all geodesics of $\tlambda$ that lie in $Q^w$ by $T_{g_u^w}^{-t}$; this has the effect of breaking $\tX$ open along $g_u^w$ and sliding $Q_w$ to the left by distance $t$ along $g_u^w$ while keeping $\tX\setminus Q^w$ fixed.
The deformation then further transforms all geodesics in $Q^v$ by $T_{g_u^v}^{t}$; this is equivalent to the right earthquake with fault locus $g_u^v$ that slides $Q^v$ to the right while keeping $\tX\setminus Q^v$ fixed. The cumulative effect is then that the spike $s$ has been ``pushed'' in the direction of $\vec{s}$ by distance $t$.
See Figure \ref{fig:simple_cat}.

\begin{remark}
We give one final interpretation of $A(\vec{s})$ as ``sliding $G_u$ along $H_u$'' in the proof of Proposition \ref{prop:cocycle_deg_hex} below (see also Figure \ref{fig:slide_deghex}), once we have set up the framework to understand the utility of this viewpoint.
\end{remark}

\begin{figure}[ht]
    \centering
\begin{tikzpicture}
    \draw (0, 0) node[inner sep=0] {\includegraphics{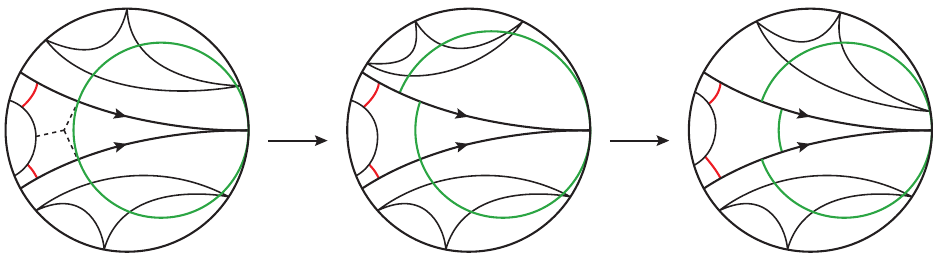}};
    \node at (-7.85,1.1){$g_u^w$};
    \node at (-7.85, -1.1){$g_u^v$};
    \node at (-5.9, -1){$v$};
    \node at (-5.7, 1.2){$w$};
    \node at (-7, .2){$u$};
    \node at (-2.9, -.7){\large $T_{g_u^w}^{-t} |_{Q^w}$};
    \node at (2.9, -.7){\large $T_{g_u^v}^{t} |_{Q^v}$};
\end{tikzpicture}
    \caption{The effect of $A(\vec{s})$ when considered as a composition of left and right earthquakes.}
    \label{fig:simple_cat}
\end{figure}

In particular, note that the shear from $H_v$ to $H_u$ measured from $p_v^w$ to the image of $p_u^v$ under this composition of earthquakes has increased by $t = f_{X, \ac}(\vec{s})$.
Therefore, if we let $q_u^v$ denote the basepoint on $g_u^v$ corresponding to the hexagon $G_u$, then the shear from $H_v$ to $H_u$ measured from $p_v^w$ to the image of $q_u^v$ under the deformation is exactly the original shear $\sigl(X)(v,u)$ between $v$ and $u$.

The elementary shape-shift $\varphi(\vec{s})$ can be interpreted in much the same way, but now the spike should be pushed distance $f_{X, \ac}(\vec{s}) + \ac(v, u)$ so that the resulting shear (measured between $p_v^w$ and the image of $q_u^v$) is exactly $\sigl(X)(v, u) + \ac(v, u)$.

Finally, the composition \eqref{eqn:first_adj} can be thought of as a composition of the operations described above (read from right to left).
Therefore, $\varphi_{\fcH}$ first performs a right earthquake along $g_v^w$ by $\ac(v, w)$, then performs an elementary shape-shift to pushing the spike $s_n$ by $\ac(v, u_n) + f_{X,\ac}(s_n)$, then performs a shape-shift for $s_{n-1}$, etc.
Observe that if $q_i^v$ denotes the basepoint in $g_i^v$ corresponding to $G_{u_i}$, then by construction the shear between $v$ and each $u_i$ measured from $p_v^w$ to the image of $q_i^v$ under the composite deformation is exactly the desired shear $\sigl(X)(v, u_i) + \ac(v, u_i)$.

Assuming the convergence of $\varphi_{\fcH}$ to a limit $\varphi_{v,w}$ (a step performed just below), we see that the placement of $\varphi_{\fcH}(g_w^v, p_w^v)$ limits to that of $\varphi_{v,w}(g_w^v, p_w^v)$.
This in turn will be the placement of the geodesic $(g_w^v, p_w^v)$ relative to $(g_v^w, p_v^w)$ straightened in the deformed surface $\tX_\ac$; see Lemma \ref{lem:shsh_correct}.

\para{Convergence}
We now consider the limiting behavior of $\varphi_{\underline{\cH}}$ as $\underline{\cH} \to \cH_{v,w}$; that a limit exists is almost exactly the content of \cite[Lemma 14]{Bon_SPB}. 
We give a proof here for convenience of the reader and to make sure that we are extracting the correct radius of convergence, i.e., that the modifications in the cusps actually do not affect the radius of convergence (even though there are countably many contributions from changing the shape of each spike).

Recall from Lemma \ref{lem:decay_gaps} that the function $D_\lambda(X) =  \inj_\lambda(Y)/9|\chi(S)|$ gives a bound for the rate of decay of the length of a piece of a leaf of $\Ol(X)$ in terms of its divergence radius.

\begin{lemma}[compare Lemma 14 of \cite{Bon_SPB}]\label{lem:simple_piece_convergence}
If $\|\ac \|_{\taua}< D_\lambda(X)$, then $\varphi_{\fcH}$ converges to a well-defined isometry $\varphi_{v,w}$ as $\fcH$ tends to $\cH_{v,w}$.
\end{lemma}

\begin{definition}\label{def:shapeshift_simple}
The limiting isometry $\varphi_{v,w}$ is called the \emph{shape-shifting map} for the simple pair $(v,w)$.
\end{definition}

\begin{remark}
After combining all of our deformations in Section \ref{subsec:shapeshift_total}, the shape-shifting map $\varphi_{v,w}$ will be identified as the value of the shape-shifting cocycle $\varphi_\ac$ on the pair $(g_v^w, g_w^v)$. However, due to the asymmetry of our current definition, it is not clear that $\varphi_{v,w}\inverse = \varphi_{w,v}$. See Lemma \ref{cor:simple_inverse}.
\end{remark}

\begin{proof}[Proof of Lemma \ref{lem:simple_piece_convergence}]
For brevity, we set $D=D_\lambda(X)$ for the remainder of the proof.

Identify $\tX$ with $\HH^2$ and $\Isom^+(\tX)$ with the unit tangent bundle $T^1\HH^2$ so that the identity $I$ is the vector over $p_v\in \tX$ that is tangent to $g_v^w$ and pointed in the positive direction with the orientation on $g_v^w$ induced by $k_{v,w}$. Equip $T^1\HH^2$ with a left-invariant metric $d$ that is right-invariant with respect to the stabilizer of $p_v$. Finally, for $A\in \Isom^+(\tX)$, let $\|A\| := d(I, A)$, so that $\|AB\|\le \|A\|+\|B\|$ holds by the triangle inequality.

We first show that $\varphi_{\underline{\cH}}$ stays in a compact set in $\Isom^+(\widetilde{X})$. Using boundedness, we then show that any sequence $\fcH \rightarrow \cH$ is in fact Cauchy with respect to $d$, hence converges.\\

We start by bounding the lengths of segments of the form $k_{v,w}\cap H_u$, where $u\in \cH_{v,w}$.  To this end, construct a geometric train track $\tau$ from $\lambda$ on $X$, and assume that the projection of $k_{v,w}$ meets $\tau$ transversely.  
Subdivide $k_{v,w}$ into arcs $k_1, ..., k_m$ whose projections meet $\tau$ once in branches $b_1, ..., b_m$.  
For all  but finitely many $u\in \cH_{v,w}$, we have $k_{v,w}\cap H_u\subset k_j \setminus \tlambda$ for some $j=1, ..., m$. 

If $d\subset k_j\setminus \tlambda$, we set $r(d)$ to be the depth $r_{b_j}(d)$ of $d$ with respect to $b_j$ and $r(d) =1$, otherwise.
By Lemma \ref{lem:decay_gaps}, there is $B>0$ such that for all $u\in \cH_{v,w}$,
\[\ell(k_{v,w}\cap H_u)\le B e^{-D r(k_{v,w}\cap H_u)}.\]
With this estimate in mind, our next task is to give a uniform bound on $\|\varphi_{\underline \cH}\circ T_{g_w^v}^{-\ac(v,w)}\|$ for all finite $\underline \cH\subset \cH_{v,w}$. 
For each $i$, let $\gamma_i \in \Isom^+(\tX)$ be the isometry corresponding to the tangent vector over $k_{v,w}\cap g_i^v$ pointing toward the positive endpoint of $g_i^v$.  
Unpacking definitions, we may therefore write the shape-shift $\varphi(s_i)$ as
\[\varphi(s_i)= \gamma_i T_{g_v^w}^{\ac(v, u_i) +f_{X, \ac}(s_i)} T_{h_i}^{-(\ac(v, u_i) + f_{X, \ac}(s_i))}\gamma_i\inverse,\]
where $h_i := \gamma_i\inverse g_i^w$. 

An explicit computation (in the upper half plane model, say) shows that 
\begin{align*}
    \left\|T_{g_v^w}^{\ac(v, u_i) + f_{X, \ac}(s_i)} T_{h_i}^{-(\ac(v, u_i) + f_{X, \ac}(s_i))}\right\| &\le (e^{|\ac(v, u_i) + f_{X, \ac}(s_i)|}-1)\ell(k_{v,w}\cap H_i)\\
    &\le Be^{|\ac(v, u_i) + f_{X, \ac}(s_i)|-Dr(k_{v,w}\cap H_i)}.
\end{align*}
By Lemma \ref{lem:lin_growth_gaps} and the triangle inequality, we have that
\[|\ac(v, u_i) + f_{X, \ac}(s_i)|\le \|\ac\|_{\taua}r(k_{v,w}\cap H_i)+\|\ac\|_{\vec{s}}\]
and so we conclude that
\[\left\|\gamma_i\inverse \varphi(s_i)\gamma_i\right\| \le B' e^{r(k_{v,w}\cap H_i)(\|\ac\|_{\taua} -D)}\]
for $B' = B e^{\|\ac\|_{\vec{s}}}$.

Notice now that conjugation by $\gamma_i$ changes the reference point of our calculation at a distance in the plane at most $\ell(k_{v,w})$, so the effect of $\gamma_i\inverse\varphi(s_i)\gamma_i$ on $g_i^v\cap k_{v,w}$ is a displacement by $e^{\ell(k_{v,w})}$ times the quantity indicated above.
Since this is independent of $\fcH$, we have that 
\begin{equation}\label{eqn:spike_shift_small}
\|\varphi(s_i)\|= O\left(e^{r(k_{v,w}\cap H_i)(\|\ac\|_{\taua} -D)}\right)
\end{equation}
for any spike $s_i$ corresponding to any hexagon $u$ between $v$ and $w$.

Expanding out $\varphi_{\fcH}$ in terms of the $\varphi(s_i)$ (see \eqref{eqn:first_adj}), we have that
\begin{align*}
    \left\|\varphi_{\underline \cH}\circ T_{g_w^v}^{-\ac(v,w)}\right\| = \left\|\prod_{i = 1}^{n}\varphi(s_i)\right\| 
     \le \sum_{i = 1}^n \|\varphi(s_i)\|  
    = O\left( \sum_{i = 1}^n e^{r(k_{v,w}\cap H_i)(\|\ac\|_{\taua} -D)}\right).
\end{align*}
Since there are a uniformly bounded number of gaps with given depth (Lemma \ref{lem:bdd_gaps}), the last expression is bounded by the sum of at most $6|\chi(S)|$ many geometric series which are convergent so long as $\|\ac\|_{\taua}<D$.
Therefore, we see there is a compact set $K$ in $\Isom^+(\tX)$ so that $\varphi_{\fcH} \in K$ for any finite subset $\underline \cH\subset \cH_{v,w}$.\\

Now that we have shown the family of isometries $\{\varphi_{\fcH}\}$ to be uniformly bounded, we can show that any sequence of refinements is in fact Cauchy.
So suppose that $\fcH_n$ increases to  $\cH_{v,w}$  and $|\fcH_n| = n$. 
By construction, we may therefore write 
\[\varphi_n = \psi \psi' \text{ and }\varphi_{n+1} = \psi\varphi(s_u)\psi',\]
where $\underline{\cH}_{n+1} = \underline{\cH}_n \cup\{u\}$ and $\psi, \psi' \in K$.
Writing 
$\varphi_{n+1} =\psi\psi'\varphi(s_u)[\varphi(s_u)\inverse, {\psi'}\inverse],$
we have that
\[d(\varphi_{n},\varphi_{n+1}) = \left\|\varphi(s_u)\left[\varphi(s_u)\inverse, {\psi'}\inverse\right]\right\|.\]

The zeroth order term in the Taylor expansion near the identity for the function $X\mapsto \|[X, {\psi'}\inverse]\|$ is $0$, because  $[I, {\psi'}\inverse]=I$.  Since ${\psi'}\inverse$ stays in a compact set,   
\[\left\|\left[\varphi(s_u)\inverse, {\psi'}\inverse\right]\right\| = O(\|\varphi(s_u)\|)\]
(see \cite[Theorem 4.1.6]{Th_book} or \cite[Lemma 1.1 of Lecture 2]{Gelander}).

Combining this estimate with the triangle inequality and \eqref{eqn:spike_shift_small}, we get that
\[ d(\varphi_n, \varphi_{n+1})= O(\|\varphi(s_u)\|)= O\left(e^{r(k_{v,w}\cap H_u)(\|\ac\|_{\taua} - D)}\right).\]
Now there are at most $6|\chi(S)|$ many $u\in \cH_{v,w}$ with $r(k_{v,w}\cap H_u)= r$ (Lemma \ref{lem:bdd_gaps}), so as $n \to \infty$ we must have that $r \to \infty$, and hence $d(\varphi_n, \varphi_{n+1}) \to 0$.
Moreover, since this goes to $0$ exponentially quickly, the sequence is in fact Cauchy.
This completes the proof of the Lemma.
\end{proof}

\para{Shape-shifting as a limit of signed earthquakes}
Here we give a different description of the shape-shifting map which forgoes approximations by ``pushing spikes'' in favor of approximations by left and right simple earthquakes (compare \cite[Section III]{EM}).
While this reformulation is symmetric and geometrically meaningful, it comes at the cost of restricting which approximating sequences $\fcH \rightarrow \cH$ actually yield convergent sequences of deformations $\varphi_{\fcH}$. See also the remark at the top of page 261 in \cite{Bon_SPB}.

Let $(v,w)$ be a simple pair and fix a geometric train track $\tau$ snugly carrying $\lambda$. So long as $\tau$ is built from a small enough neighborhood, we may assume that the geodesic $k_{v,w}$ is transverse to the branches of $\tau$. Then for each integer $r\ge 0$, let $\cH_{v,w}^r$ denote the set of hexagons such that $k_{v,w}\cap H_u$ has depth at most $r$ with respect to the branches of $\tau$. Order
\[\cH_{v,w}^r = (u_0 = v, u_{1}, ..., u_{n}, u_{{n+1}} = w),\]
and for each $i= 0 , \ldots, n$, choose a geodesic $h_i$ that separates the interior of $H_{i}$ from the interior of $H_{i+1}$. Orient each $h_i$ so that it crosses $k_{v,w}$ from left to right and set
\begin{equation}\label{eqn:adj_r}
\varphi^r_{v,w}= T_{h_0}^{\ac(u_0, u_1)}\circ A(s_1)\circ T_{h_1}^{\ac(u_1, u_2)}\circ A(s_2)\circ ...  \circ A(s_n) \circ T_{h_n}^{\ac(u_n, u_{n+1})}.
\end{equation}
We now wish to show that $\varphi^r_{v,w} \to \varphi_{v,w}$ as $r \to \infty$. As in the case of $\varphi_{\fcH} \to \varphi_{v,w}$, this argument will parallel that of \cite{Bon_SPB}, with the extra complicating factor of the adjustments $A(s_i)$ to the shape of cusps.

The interpretation of \eqref{eqn:adj_r} as a deformation cocycle is now similar to that of $\eqref{eqn:first_adj}$, but is now a combination of spike-shaping transformations together with simple earthquakes.

Let us give a description of the action of this deformation on the pointed geodesic $(g_w^v,p_w^v)$ in $\partiall H_w$ closest to $v$.
Reading the formula from right to left, we can obtain $\varphi_{v,w}^r (g_w^v, p_w^v)$ by first breaking $\tX$ along $h_n = g_w^v$ and sliding the closed half-space containing $H_w$ signed distance $\ac(u_n,u_{n+1})$, keeping the open half-space containing $H_v$ fixed.
Applying the spike shaping transformation $A(s_n)$ then preserves the natural basepoints $p_n^v$ and $p_n^w$ but increases the sharpness parameter $h_{\arcwt}(s_n)$, making it match that of the spike in the hexagon $G_{u_n}$ in the deformed metric $\arcwt + \acarc$.
We then simply continue moving from $w$ to $v$ (i.e., backwards along $k_{v,w}$), alternating between signed earthquakes in the $h_i$ and shaping the next spike until we reach $g_v^w$.
Note that unlike the deformations associated to $\varphi_{\fcH}$, each step of the process requires only local information about the spike $s_i$ and the shear between $u_i$ and $u_{i+1}$.

\begin{lemma}[Lemma 16 of \cite{Bon_SPB}]\label{lem:intuitive_adjustment}
So long as $\|\ac\|_{\taua}<D_\lambda(X)/2$, we have that
$\lim_{r\to \infty} \varphi_{v,w}^r = \varphi_{v,w}$.
\end{lemma}
\begin{proof}
Using additivity, $\ac(u_i, u_{i+1}) = \ac(v,u_{i+1}) - \ac(v,u_i)$, we observe that
\begin{equation}\label{eqn:first_replacement}
    \varphi_{v,w}^r =
    \left( T_{h_0}^{\ac(v,u_1)} A(s_1) T_{h_1}^{-\ac(v,u_1)} \right) 
    \left(T_{h_1}^{\ac(v,u_2)} A(s_2)  T_{h_2}^{-\ac(v,u_2)} \right)
    \dots
    \left(T_{h_{n-1}}^{\ac(v, u_n)} A(s_n) T_{h_n}^{-\ac(v,u_n)}\right)
    T_{h_n}^{\ac(v,w)}.
\end{equation}
So $\varphi_{v,w}^r$  is obtained from $\varphi_{\cH_{v,w}^r}$ by replacing each term of the form
\[\varphi(s_i) = T_{g_i^v}^{\ac(v,u_i)}A(s_i)T_{g_i^w}^{-\ac(v,u_i)}\]
with  
\[\phi(s_i) : = T_{h_{i-1}}^{\ac(v,u_i)} A(s_i) T_{h_i}^{-\ac(v, u_i)}\]
and $T_{g_w^v}^{\ac(v,w)}$ with $T_{h_n}^{\ac(v,w)}$.   

The basic estimate we need is approximately how close $\varphi(s_i)$ is to $\phi(s_i)$ in $\Isom^+(\tilde X)$ as $r$ tends to infinity.  For this we will want to understand how closely  $h_{i-1}$ approximates $g_i^v$ near its intersection with $k_{v,w}$, as well as for $g_i^w$ and $h_i$.  

By construction, $h_i$ must be between $g_{i}^w$ and $g_{i+1}^v$ for each $i = 1, ..., n$ and $h_0$ is between $g_v^w$ and $g_1^v$.  But $g_i^w$ and $g_{i+1}^v$ follow the same edge path of length $2r$ in $\tau \subset \taua$, for otherwise, there would be a another $u\in \cH_{v,w}^r$ such that $H_u$ separates $H_i$ from $H_{i+1}$. Thus $h_i$ follows the same edge path and fellow travels $g_i^w$ and $g_{i+1}^v$ for length at least $O(2rD_\lambda(X))$;  using negative curvature, we have that $h_i$ is $O(e^{-D_\lambda(X)r})$ close to both $g_i^w$ and $g_{i+1}^v$ near $k_{v,w}$.  

From closeness of these geodesics from the previous paragraph (and our estimates for $\|\varphi(s_i)\|$ from Lemma \ref{lem:simple_piece_convergence}) it is possible to obtain the basic estimate
\[\|\phi(s_i)\inverse \varphi(s_i)\|= O\left(\exp\left(\|\ac\|_{\taua} r(k_{v,w}\cap H_i)-rD_\lambda(X)\right)\right),\]
which is small when $\|\ac\|_{\taua}<D_\lambda(X)$.  Notice that we have also used the fact that the adjustment parameter associated to each spike $\ac(s_i)$ is uniformly bounded; that said, even if it grew linearly in $r$ we would obtain the same estimate (up to a multiplicative factor).

The rest of the argument ensuring that $\varphi_{v,w}^r$ and $\varphi_{\cH_{v,w}^r}$ have the same limit as long as $\|\ac\|_{\taua} < D_\lambda(X)/2$ follows \cite[Lemma 16]{Bon_SPB} and is omitted. We remark that the factor of $1/2$ appearing at the end is a relic of the techniques used in \cite[Lemma 16]{Bon_SPB}.
\end{proof}

The following simple fact was not apparent from the definition of $\varphi_{v,w}$ due to its lack of symmetry. Fortunately, the approximation of $\varphi_{v,w}$ by $\varphi_{v,w}^r$ gives us a symmetric description of $\varphi_{v,w}$.  

\begin{corollary}\label{cor:simple_inverse}
If $(v,w)$ is simple and $\|\ac\|_{\taua}<D_\lambda(X)/2$, then 
\[\varphi_{w,v} = \varphi_{v,w}\inverse.\]
\end{corollary}
\begin{proof}
We observe first that $\cH_{v,w}^r = \cH_{w,v}^r$, so the each term of $\varphi_{v,w}^r$ appears in $\varphi_{w,v}^r$ with the opposite orientation.
Now by \eqref{eqn:spike_shape_inverse}, the inverse of the shaping transformation of an oriented spike is equal to the shaping transformation of the same spike with opposite orientation.
Therefore we have that $\varphi_{v,w}^r = (\varphi_{w,v}^r)\inverse$ for all $r$, and the equality holds as we take $r\to \infty$.
\end{proof}

\subsection{Shape-shifting in hexagons}\label{subsec:shsh_hex}
In this section, we explain how to define the shape-shifting cocycle $\varphi_{\ac}$ on pairs of basepointed geodesics that lie in the boundary of a common hexagon; this will encode the change in hyperbolic structure on $X \setminus \lambda.$

While in this setting we do not have to worry about delicate convergence results, we must be more diligent about recording the placement of basepoints on each geodesic of $\partiall H_u$. Moreover, the cocycle condition (Propositions \ref{prop:cocycle_hex} and \ref{prop:cocycle_deg_hex}) only becomes apparent once we reinterpret the shaping deformations defined below as ``sliding the deformed hexagon along the original.''

Throughout this section, we have extended both $\arcwt$ and $\arcwt + \acarc$ to some common maximal arc system $\arc$ by adding in arcs of weight $0$ as necessary. We remind the reader that $\ac(\alpha)$ denotes the coefficient of $\alpha$ in $\acarc$.

\para{Notations and orientations}
Let $H_u \subset \tX\setminus \tlambdaa$ be a nondegenerate right-angled hexagon and enumerate the $\lambda$-boundary components of $H_u$ as 
$\partiall H_u = \{(h_1, p_1), (h_2, p_2), (h_3, p_3)\},$
cyclically ordered about $u$. Let $\alpha_i \in \widetilde{\arc}$ be the orthogeodesic arc opposite to $h_i$, and denote by $p_{ij}$ the vertex of $H_u$ meeting both $h_i$ and $\alpha_j$. See Figure \ref{fig:hexagon}. If $H_u$ is a degenerate hexagon (i.e., a pentagon with one ideal vertex or a quadrilateral with two) then we label only those points and geodesics which appear in its boundary.

Each choice of orientation $\vec{\alpha}_1$ of $\alpha_1$ induces orientations of $h_2$ and $h_3$ so that $\alpha_1$ leaves from the left-hand side of $h_j$ and arrives on the right-hand side of $h_k$ for $\{j,k\} = \{2,3\}$; an example is pictured in Figure \ref{fig:hexagon}. Observe that the opposite orientation $\cev{\alpha}_1$ induces the opposite orientations on $h_2$ and $h_3$. Throughout this section, we also adopt similar conventions for each orientation of $\alpha_2$ and $\alpha_3$.

\begin{figure}[ht]
    \centering
\begin{tikzpicture}
    \draw (0, 0) node[inner sep=0] {\includegraphics{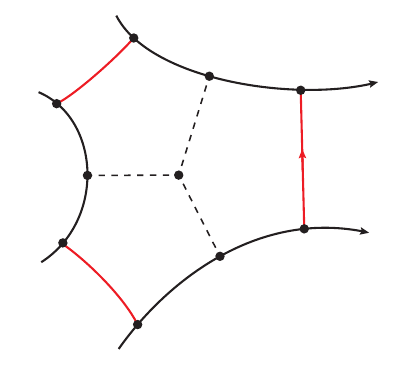}};
    \node at (-.2, .1) {$u$};
    \node at (.1, 2) {$p_3$};
    \node at (1.55, 1.85) {$p_{31}$};
    \node at (1.9, .5)[red]{$\vec{\alpha}_1$};
    \node at (1.7, -1.1){$p_{21}$};
    \node at (.5, -1.5){$p_2$};
    \node at (-.8, -2.5){$p_{23}$};
    \node at (-2, -1.8)[red]{$\alpha_3$};
    \node at (-2.8, -.9){$p_{13}$};
    \node at (-2.35, .1){$p_1$};
    \node at (-2.8, 1.1){$p_{12}$};
    \node at (-2.1, 2)[red]{$\alpha_2$};
    \node at (-.85, 2.5){$p_{32}$};
    \node at (-.8,1.8){$h_3$};
    \node at (.8, -.7){$h_2$};
    \node at (-1.9, -.55){$h_1$};
\end{tikzpicture}
    \caption{Distinguished points on a hexagon $H_u$ and induced orientations on $h_2, h_3 \in \partiall H_u$.}
    \label{fig:hexagon}
\end{figure}

Recall that (by Theorem \ref{thm:arc=T(S)_crown}) the deformation $\ac$ induces a new metric on $\tX \setminus \tlambda$ denoted by $\arcwt + \acarc$ and which contains a hexagon $G_u$ corresponding to $H_u$. The corresponding basepointed $\lambda$-boundary geodesics and vertices of $G_u$ will be denoted by $(g_i, q_i)$ and $q_{ij}$, respectively. We adopt similar orientation conventions as above for the realizations of $\alpha_j$ in $G_u$.

\para{Shapes of Hexagons}
Paralleling our discussion for spikes, we first need to define geometric parameters that measure the shape of the hexagon as well as the difference of the placements of the basepoints $p_i$ and $q_i$ on the geodesics $h_i$ and $g_i$.
For concreteness, we only consider $\alpha_1$ below; the parameters for $\alpha_2$ and $\alpha_3$ are defined symmetrically.

We begin by associating to $\alpha_1$ the parameter 
\[\ell_{\ac} (\alpha_1) := \ell_{\arcwt + \acarc}(\alpha_1)-\ell_{\arcwt} (\alpha_1) \in \RR.\]
\label{ind:ls(arc)}
which measures the difference in the hyperbolic length of $\alpha_1$ in the metric determined by $\arcwt + \acarc$ versus in the original metric $\arcwt$ induced by $X$.

Now fix an orientation $\vec{\alpha}_1$ of $\alpha_1$; as above, this induces orientations of the geodesics $h_2$, $h_3$, $g_2$, and $g_3$.
Let $d_{\arcwt}(\vec{\alpha}_1, u)$ be the signed distance from $p_2$ to $p_{21}$ on $h_2$;
\label{ind:arcdist}
\footnote{The parameter $d_{\arcwt}(\vec{\alpha}_1, u)$ is called the ``$t$-coordinate'' of the arc $\alpha_1$ in the hexagon $H_u$ in \cite{Luo}.  See also \cite[Proposition 2.10]{Mondello}, where a formula is given in terms of the lengths of $\{\alpha_i: i = 1,2,3\}$.}
the local symmetry of the orthogeodesic foliation implies that $d_{\arcwt}(\vec{\alpha}_1, u)$ can also be computed as the signed distance from $p_3$ to $p_{31}$ on $h_3$.
Define similarly $d_{\arcwt+\acarc}(\vec{\alpha}_1, u)$ as the distance from $q_2$ to $q_{21}$ on $g_2$ (equivalently, the signed distance from $q_3$ to $q_{31}$ on $g_3$).

To all of this information, we associate the parameter 
\[f_{X, \ac}(\vec{\alpha}_1, u) : = 
d_{\arcwt+\acarc}(\vec{\alpha}_1, u) - d_{\arcwt}(\vec{\alpha}_1, u) \in \RR\]
\label{ind:arcparam}
which measures the difference in how far $u$ is from $\alpha_1$ in $G_u$ versus in $H_u$.
More precisely, considering $H_u$ and $G_u$ in the hyperbolic plane, we can use an element of $\PSL_2(\RR)$ to line up $(h_2, p_{21})$ with $(g_2, q_{21})$ so that the basepoints and orientations agree.
The parameter $f_{X, \ac}(\vec{\alpha}_1, u)$ then measures the distance from $q_{2}$ to $p_{2}$ along $h_2=g_2$. 
See Figure \ref{fig:edgetotalshapechange}.
Of course, symmetry shows that it is equivalent to align $(h_3, p_{31})$ with $(g_3, q_{31})$ and measure the signed distance from $q_3$ to $p_3$ along $h_3=g_3$. 
\begin{figure}[ht]
    \centering
\begin{tikzpicture}
    \draw (0, 0) node[inner sep=0] {\includegraphics{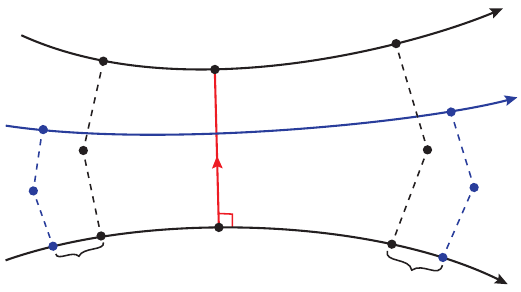}};
    \node at (-4.2, -.8)[blue]{$G_u$};
    \node at (-3.3,-1.4)[blue]{$q_2^u$};
    \node at (-2.6, -.1){$H_u$};
    \node at (-2.4, -1.2){$p_2^u$};
    \node at (-3, -2.2){$f_{X, \ac}(\vec{\alpha}_1, u)$};
    \node at (-1, -.8)[red]{$\vec{\alpha}_1$};
    \node at (4,-.7)[blue]{$G_v$};
    \node at (3,-1.5)[blue]{$q_2^v$};
    \node at (2.5, -.1){$H_v$};
    \node at (2.1, -1.3){$p_2^v$};
    \node at (2.5, -2.34){$-f_{X, \ac}(\vec{\alpha}_1, v)$};
    \node at (4, 2){$h_3$};
    \node at (4, .5)[blue]{$g_3$};
    \node at (4.4, -1.9){$h_2 = g_2$};
\end{tikzpicture}
\caption{The parameter $f_{X, \ac}(\vec{\alpha}_1, u)$ for two adjacent hexagons. We have decorated the basepoints on $h_2 = g_2$ with a superscript to emphasize their dependence on the hexagon.}
\label{fig:edgetotalshapechange}
\end{figure}

Note that reversing orientations reverses signs, so that
$d_{\arcwt}(\vec{\alpha}_1, u) = - d_{\arcwt}(\cev{\alpha}_1, u)$ and hence 
\[f_{X, \ac}(\vec{\alpha}_1, u) = - f_{X, \ac}(\cev{\alpha}_1, u).\]

The parameters associated to the hexagons which border a given arc are related in the following way:

\begin{lemma}\label{lem:adjhex_params}
Let $\alpha$ be any edge of $\widetilde{\arc}$ and let $H_u$ and $H_v$ be its adjoining hexagons. Then 
\[f_{X, \ac}(\vec{\alpha}, u) + f_{X, \ac}(\cev{\alpha}, v) = \ac(\alpha)\]
where the orientation $\vec{\alpha}$ is chosen so that $u$ is on its left (equivalently, $\cev{\alpha}$ is oriented so that $v$ is on its left).
\end{lemma}
\begin{proof}
The proof is an exercise in unpacking the definitions and being careful with orientations; compare Figure \ref{fig:edgetotalshapechange}.
Let $p_2^u$ and $p_2^v$ denote the projections of $u$ and $v$ to either of geodesics common to $\partiall H_u$ and $\partiall H_v$, and let $q_2^u$ and $q_2^v$ play similar roles for $G_u$ and $G_v$.

We can then write
\begin{align*}
f_{X, \ac}(\vec{\alpha}, u) + f_{X, \ac}(\cev{\alpha}, v)
& = f_{X, \ac}(\vec{\alpha}, u) - f_{X, \ac}(\vec{\alpha}, v)\\
& = d_{\arcwt+\acarc}(\vec{\alpha}, u) - d_{\arcwt}(\vec{\alpha}, u) - d_{\arcwt+\acarc}(\vec{\alpha}, v) + d_{\arcwt}(\vec{\alpha}, v)\\
& = d_h(p_2^u, p_2^v) - d_g(q_2^u, q_2^v) = \ac(\alpha)
\end{align*}
where we recall that $\ac(\alpha)$ denotes the coefficient of $\alpha$ in $\acarc$ and where $d_h$ and $d_g$ represent the signed distance measured along $h_2$ and $g_2$, equipped with the orientation induced by $\vec{\alpha}$.
\end{proof}

\begin{remark}
Using Theorem \ref{thm:arc=T(S)_crown} and some hyperbolic trigonometry, one may show that $f_{X, \ac}(\vec{\alpha}_1, u)$ depends analytically on both $\arcwt$ and $\acarc$ for fixed $\alpha_1$ and $u$.
\end{remark}

\para{Shaping Hexagons}
To the hexagon $H_u$ and oriented arc $\vec{\alpha}_1$ in its boundary, we associate the \emph{shaping transformation} $A(\vec{\alpha}_1, u)$ by
\begin{equation}\label{eqn:adjust_hexagon}
A(\vec{\alpha}_1, u) :=
T_{h_2}^{-f_{X, \ac}(\vec{\alpha}_1, u)}
\circ T_{\vec{\alpha}_1}^{\ell_{\ac}(\alpha_1)}
\circ T_{h_3}^{f_{X, \ac}(\vec{\alpha}_1, u)} \in \Isom^+(\tX),
\end{equation}
where $T_{\vec{\alpha}_1}$ denotes translation along the complete oriented geodesic extending $\vec{\alpha}_1$. The shaping transformation is explicitly constructed so that if $H_u$ and $G_u$ are superimposed with $(h_2,p_2) = (g_2, q_2)$, then 
\[A(\vec{\alpha}_1, u)(h_3, p_3) = (g_3, q_3).\]
This claim is not immediately apparent from the expression of \eqref{eqn:adjust_hexagon}, but is easy to verify once we reinterpret $A(\vec{\alpha}_1, u)$ as ``sliding $G_u$ along $H_u$.''

To wit, suppose that we superimpose $H_u$ and $G_u$ so that $(h_3,p_3) = (g_3, q_3)$.
Now consider what happens as we apply $A(\vec{\alpha}_1, u)$ to $G_u$ while holding $H_u$ fixed;
the first term $T_{h_3}^{f_{X, \ac}(\vec{\alpha}_1, u)}$ translates $G_u$ along $h_3$ so that $q_{31}=p_{31}$, and the right angle formed by $\alpha_1$ and $g_3$ in $G_u$ lines up with the same angle in $H_u$.
The transformation $T_{\vec{\alpha}_1}^{\ell_\ac(\alpha_1)}$ then slides $T_{h_3}^{f_{X, \ac}(\vec{\alpha}_1, u)} G_u$ along $\alpha_1$ so that $(h_2, q_{21})=(g_2,p_{21})$.
Finally, $T_{h_2}^{-f_{X, \ac}(\vec{\alpha}_1, u)}$ slides  $T_{\vec{\alpha}_1}^{\ell_{\ac}(\alpha_1)}
T_{h_3}^{f_{X, \ac}(\vec{\alpha}_1, u)} G_u$ along $h_2=g_2$ so that $q_2$ lines up with $p_2$. See Figure \ref{fig:slide_hexagons}.

\begin{figure}[ht]
    \centering
    \begin{tikzpicture}
    \draw (0, 0) node[inner sep=0] {\includegraphics{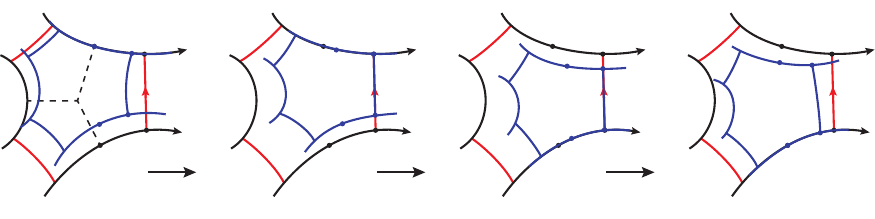}};
    \node at (-5.4,1.3){\small $p_3=q_3$};
    \node at (-6.45,.5)[blue]{\small $G_u$};
    \node at (-7.3, 0.15){\small $H_u$};
    \node at (-4.6, .3)[red]{\small $\vec{\alpha}_1$};
    \node at (-4, 1){\small $h_3$};
    \node at (-4, -.4){\small $h_2$};
    \node at (-5.6, -.85){\small $p_2$};
    \node at (-.6,1.2){\small $p_{31}=q_{31}$};
    \node at (3.3,-.6){\small $p_{21}=q_{21}$};
    \node at (6.4,-.85){\small $p_{2}=q_{2}$};
    \node at (-4,-1.5){$T_{h_3}^{f_{X, \ac}(\vec{\alpha}_1, u)}$}; 
    \node at (-.5,-1.5){$T_{\vec{\alpha}_1}^{\ell_{\ac}(\alpha_1)}$}; 
    \node at (3.7,-1.5){$T_{h_2}^{-f_{X, \ac}(\vec{\alpha}_1, u)}$}; 
    \end{tikzpicture}
    \caption{How the shaping transformation $A(\vec{\alpha}_1, u)$ slides $G_u$ along $H_u$.}
    \label{fig:slide_hexagons}
\end{figure}

Summarizing, we have shown that $A(\vec{\alpha}_1, u)$ takes a superimposition of $G_u$ on $H_u$ with $(h_3,p_3) = (g_3, q_3)$ to another superimposition with $(h_2,p_2) = (g_2, q_2)$. In particular, this implies that applying $A(\vec{\alpha}_1, u)$ to $(h_3,p_3)$ takes it to the position of $(g_3, q_3)$ in the latter placement of $G_u$, which is what we claimed.

\begin{remark}
An elementary hyperbolic geometry argument similar to that in the proof of  Lemma \ref{lem:spike_param} shows that if $\alpha_1$ in $X$ degenerates to an oriented spike $\vec{s}$ then the corresponding geometric parameter $f_{X, \ac}(\vec{\alpha}_1, u)$ limits to the parameter $f_{X, \ac}(\vec{s})$.
In particular, along this degeneration the corresponding hexagon-shaping transformation $A(\vec{\alpha}_1, u)$ converges to the spike-shaping transformation $A(\vec{s})$.
\end{remark}

\para{A cocycle condition for hexagons}
A number of relations hold between the shaping transformations for different arcs and different orientations; eventually, these relations are what ensure that the deformations we are currently building package together into an honest cocycle (see Proposition \ref{prop:shapeshift_cocycle}).

First, we observe that reversing the orientation of $\alpha_1$ inverts the shaping transformation:
\begin{equation}\label{eqn:hex_inverse}
A(\cev{\alpha}_1, u)
= T_{\bar{h}_3}^{-f_{X, \ac}(\cev{\alpha}_1, u)}
\circ T_{\cev{\alpha}_1}^{\ell_{\ac}(\alpha_1)}
\circ T_{\bar{h}_2}^{f_{X, \ac}(\cev{\alpha}_1, u)}
= T_{h_3}^{-f_{X, \ac}(\vec{\alpha}_1, u)}
\circ T_{\vec{\alpha}_1}^{-\ell_{\ac}(\alpha_1)}
\circ T_{h_2}^{f_{X, \ac}(\vec{\alpha}_1, u)}
= A(\vec{\alpha}_1, u) \inverse.
\end{equation}

Now suppose that $H_u$ is on the left and $H_v$ is on the right of the oriented arc $\vec{\alpha}_1$. Combining the relation of  Lemma \ref{lem:adjhex_params} with the definition of the shaping transformation, we have that
\begin{equation}\label{eqn:adjhex_shape}
A(\vec{\alpha}_1, u) =
T_{h_2}^{\ac(\alpha_1)}
\circ A(\vec{\alpha}_1, v)
\circ T_{h_3}^{-\ac(\alpha_1)}.
\end{equation}
This equation is used frequently in Section \ref{subsec:shsh_spine} just below.

Finally, a beautiful and important relationship holds among the three shaping transformations in a single right-angled hexagon. Our proof utilizes the ``sliding'' viewpoint explained above; the statement seems difficult to prove just by writing down a string of M{\"o}bius transformations.

\begin{proposition}\label{prop:cocycle_hex}
Let $u \in \cH$ be a nondegenerate right-angled hexagon with boundary arcs $\vec{\alpha}_1, \vec{\alpha}_2, \vec{\alpha}_3$, oriented so that $H_u$ lies to the left of each $\vec{\alpha}_i$.
Then
\[A(\vec{\alpha}_3, u) \circ A(\vec{\alpha}_2, u) \circ A(\vec{\alpha}_1, u) = 1.\]
A similar statement clearly holds for any cyclic permutation of $(3, 2, 1)$.
\end{proposition}
\begin{proof}
In order to prove the lemma, we superimpose $G_u$ on top of $H_u$ so that $(g_3, q_3) = (h_3, p_3)$.
Holding $H_u$ fixed, the first shaping transformation $A(\vec{\alpha}_1, u)$ slides $G_u$ along $h_3$, then along $\alpha_1$, then along $h_2$ so that $(g_2, q_2)$ lines up with $(h_2, p_2)$.
The second shaping transformation $A(\vec{\alpha}_2, u)$ then acts on this translated copy of $G_u$ by sliding it along $h_2$, then $\alpha_2$, then $h_1$ so that $(g_1, q_1)=(h_1, p_1)$.
Finally, the last term slides $A(\vec{\alpha}_1, u) A(\vec{\alpha}_2, u) G_u$ along the edges of $H_u$ so that $(g_3, q_3)$ returns to $(h_3, p_3)$ (with the same orientation).

Therefore, since $A(\vec{\alpha}_1, u) \circ A(\vec{\alpha}_2, u) \circ A(\vec{\alpha}_3, u)$ preserves a unit tangent vector and $\Isom^+(\tX)$ acts simply transitively on $T^1 \tX$, the composition of the three shaping transformations must be trivial.
\end{proof}

A similar result holds for degenerate right-angled hexagons, with the hexagon-shaping transformation replaced with the corresponding spike-shaping transformation.

\begin{proposition}\label{prop:cocycle_deg_hex}
Suppose that $u \in \cH$ is a pentagon with two orthogeodesic arcs $\alpha_1, \alpha_2$ and one spike $s$, labeled so that $(\alpha_1, \alpha_2, s)$ runs counterclockwise around $u$.
Orient each $\alpha_j$ so that $H_u$ is on its left and orient $s$ so that it is pointing towards the ideal vertex of $H_u$. Then
\[A(\vec{s}) \circ A(\vec{\alpha}_2, u) \circ A(\vec{\alpha}_1, u) = 1.\]
Similarly, if $u \in \cH$ is a quadrilateral with one orthogeodesic edge $\alpha$ and two spikes $s_1$ and $s_2$ (labeled so that $(\alpha, s_1, s_2)$ is read counterclockwise), then
\[A(\vec{s}_2) \circ A(\vec{s}_1) \circ A(\vec{\alpha}, u) = 1\]
where all orientation conventions are as above.
\end{proposition}
\begin{proof}
We only explain how to interpret the spike-shaping transformation $A(\vec{s})$ in our ``sliding'' framework; once we have done so, the rest of the proof is completely analogous to that of Proposition \ref{prop:cocycle_hex}.

So let $\vec{s}$ be a spike of $H_u$, oriented as described; suppose that its left and right boundary geodesics are $h_3$ and $h_2$. Recall that $A(\vec{s})$ is constructed so that if we superimpose $G_u$ and $H_u$ with $(g_2, q_2) = (h_2, p_2)$ then $A(\vec{s})(h_3, p_3) = (g_3, q_3)$.
This can equivalently be interpreted by superimposing $G_u$ on $H_u$ with $(g_3, q_3) = (h_3, p_3)$; then applying the shaping transformation to $G_u$ while leaving $H_u$ fixed takes $G_u$ to another superimposition where $(g_2, q_2) = (h_2, p_2)$. See Figure \ref{fig:slide_deghex}.
\end{proof}

\begin{figure}[ht]
    \centering
    \begin{tikzpicture}
    \draw (0, 0) node[inner sep=0] {\includegraphics{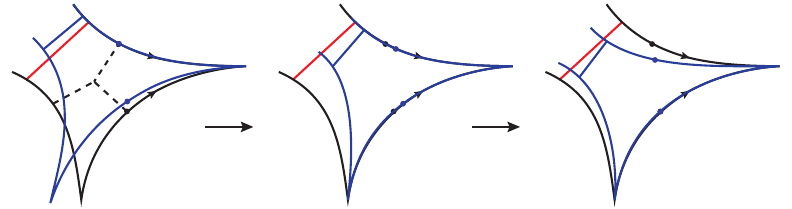}};
    \node at (-4.2,1.2){\small $p_3=q_3$};
    \node at (-5.4,.8)[blue]{\small $G_u$};
    \node at (-6.2, 0.1){\small $H_u$};
    \node at (-3, 1){\small $h_3$};
    \node at (-3, .4){\small $h_2$};
    \node at (-4.4, -.4){\small $p_2$};
    \node at (1.6,.3){\small $h_{2}=g_{2}$};
    \node at (5,-.4){\small $p_{2}=q_{2}$};
    \node at (-2.8,-.8){$T_{h_3}^{-f_{X, \ac}(\vec{s})}$}; 
    \node at (1.8,-.8){$T_{h_2}^{f_{X, \ac}(\vec{s})}$}; 
    \end{tikzpicture}
    \caption{Interpreting the spike-shaping transformation $A(\vec{s})$ as sliding $G_u$ along $H_u$. Note that in this picture, we have $f_{X, \ac}(\vec{s})<0$.}
    \label{fig:slide_deghex}
\end{figure}

\subsection{Shape-shifting along the spine}\label{subsec:shsh_spine}
In this section we package together the hexagon-shaping deformations defined in \eqref{eqn:adjust_hexagon} into deformations of entire complementary subsurfaces of $\tX \setminus \tlambda$. 
As always, we will exhibit this deformation by explaining how to adjust the positions of the pointed geodesics in the boundary of each component of $\tX \setminus \tlambda$ relative to one another.
This in turn requires some book-keeping of orientations and a liberal application of the cocycle relation (Propositions \ref{prop:cocycle_hex} and \ref{prop:cocycle_deg_hex}).

Throughout this section, we fix some component $Y$ of $\tX \setminus \tlambda$.
We remind the reader that the deformation $\ac$ induces a new hyperbolic structure $\arcwt + \acarc$ on $Y$ whose hexagons and basepointed geodesics correspond to those of $Y$.

\para{Hexagonal hulls and induced orientations}
Suppose that $v, w \in \cH$ are distinct hexagons of $Y$.
Since the corresponding component of $\tSp$ is a tree it contains a unique oriented non-backtracking edge path $[v,w]$ joining $v$ to $w$.
We then define the {\em hexagonal hull}
\label{ind:hexhull}
$H(v,w)$ of $(v,w)$ to be the union of all of the hexagons corresponding to the vertices of $[v,w]$. Define also the {\em truncated hexagonal hull} $\widehat{H}(v,w)$ by truncating each spike of $H(v,w)$ by the horocycle through the basepoints that are closest to the ideal vertex. 
Note that both $H(v,w)$ and $\widehat{H}(v,w)$ come with ($\pi_1(Y)$-equivariant) collections of basepoints on their boundaries obtained by projecting each of the vertices of $[v,w]$ onto the associated boundary geodesics.

Now for any $(h_v, p_v) \in \partiall H_v$ and $(h_w, p_w) \in \partiall H_w$, we have that
$\partial \widehat{H}(v,w) \setminus \{p_v, p_w\}$ consists of two paths $\delta_{\pm}$. We orient each of $\delta_\pm$ so that they both travel from $p_v$ to $p_w$. 
\label{ind:deltapm}
See Figure \ref{fig:hexhull}.

\begin{figure}[ht]
    \centering
    \begin{tikzpicture}
    \draw (0, 0) node[inner sep=0] {\includegraphics{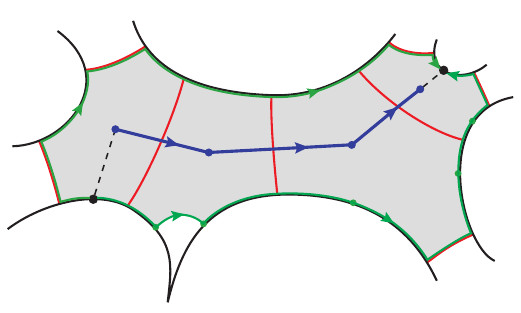}};
    \node at (-2.8, -1){$p_v$};
    \node at (3.3, 1.7){$p_w$};
    \node at (-2.5, .7)[blue]{$v$};
    \node at (2.9, 1)[blue]{$w$};
    \node at (0, 1.25){$\delta_-$};
    \node at (.3, -1){$\delta_+$};
    \node at (-2, -1.3){$p_2$};
    \node at (-.7, -1.3){$p_3$};
    \node at (1.6, -1.05){$p_4$};
    \node at (3.65, -.3){$p_5$};
    \node at (3.85, .5){$p_6$};
    \end{tikzpicture}
    \caption{The truncated hexagonal hull (shaded) of the path $[v,w]$ and the induced orientations on the paths $\delta_\pm$ from $p_v$ to $p_w$ in its boundary.}
    \label{fig:hexhull}
\end{figure}

With this induced orientation, the path $\delta_+$ passes through a sequence of basepoints
\[p_v = p_1, p_2, \ldots, p_{n+1} = p_w.\]
We then associate a shaping transformation $A_i$ to each subsequent pair of basepoints as follows:
\begin{itemize}
    \item If $p_i, p_{i+1}$ are in different hexagons, then they must lie on the same geodesic $h_i$ of $\partial Y$ and correspond to two hexagons $H_i$ and $H_{i+1}$ both adjacent to an arc $\alpha_i$. In this case, define 
    $A_i = T_{h_i}^{\ac(\alpha_i)}$
    where $h_i$ is given the orientation induced by $\delta_+$ and where we recall that $\ac(\alpha_i)$ is the coefficient of $\alpha_i$ in $\acarc$.
    \item If $p_i, p_{i+1}$ are in the same hexagon $H_{u_i}$ but do not lie on a common spike, then necessarily they lie on geodesics connected by some arc $\alpha_i$. In this case, define
    $A_i = A(\vec{\alpha}_i, u_i)$ where the orientation on $\alpha_i$ is induced from $\delta_+$.
    \item If $p_i$ and $p_{i+1}$ lie on a common spike $s_i$, then we define $A_i = A(\vec{s}_i)$, where the orientation on $s_i$ is such that the horocyclic segment of $\delta_+$ cutting off $s_i$ runs from the left of one of the oriented geodesics to the right of the other.
\end{itemize}
Finally, we then combine all of this information to define the shape-shifting transformation
\begin{equation}\label{eqn:spine_shift}
A(\delta_+) :=A_1 \circ A_2 \circ \ldots \circ A_n,
\end{equation}
where we recall that we are multiplying from right to left. Define $A(\delta_-)$ analogously; the point is, however, that the choice of $\pm$ does not matter.

\begin{lemma}
We have $A(\delta_-) = A(\delta_+)$.
\end{lemma}

\begin{definition}\label{def:shapeshift_subsurf}
We call $\varphi_{p_v, p_w}:= A(\delta_+) = A(\delta_-)$ the {\em shape-shifting map} for the pair $((h_v, p_v), (h_w, p_w))$.
\end{definition}
\begin{proof}
The proof follows by induction on the length of $[v,w]$.
If $[v,w]$ has length $0$, i.e., $v=w$, then this statement is exactly the content of the cocycle relation for hexagons (Propositions \ref{prop:cocycle_hex} and \ref{prop:cocycle_deg_hex}).

Now suppose that $[v,w]$ has length $n$ and 
let $u$ be the penultimate vertex in $[v,w]$. Let $\alpha$ denote the arc separating $u$ from $w$, and choose the orientation $\vec{\alpha}$ so that $u$ lies on its left.
Up to relabeling, we may assume that the orientation of $\delta_+$ agrees with the orientation of $\partial \widehat{H}(v,w)$.
Denote by $(h^\pm_u,p^\pm_u) \in \partial Y$ the last basepoints of $H_u$ visited by $\delta_\pm$ and let $\gamma_\pm$ denote the subpaths of $\delta_\pm$ from $p_v$ to $p^\pm_u$ in the boundary of the truncated hexagonal hull $\widehat{H}(v,u).$
Define $A(\gamma_\pm)$ analogously to $A(\delta_\pm)$. Then we may write
\begin{align*}
A(\delta_+) A(\delta_-)\inverse
& = A(\gamma_+) \,  T_{h_u^+}^{\ac(\alpha)} \, B_1 \, B_2 \, T_{h_u^-}^{-\ac(\alpha)} \, A(\gamma_-)\inverse \\
& = \big( A(\gamma_+)  A(\vec{\alpha}, u) A(\gamma_-) \inverse \big) 
\cdot 
A(\gamma_-)
\big( A(\vec{\alpha}, u)\inverse 
T_{h_u^+}^{\ac(\alpha)} \, B_1 \, B_2 \, T_{h_u^-}^{-\ac(\alpha)}
\big) 
A(\gamma_-) \inverse
\end{align*}
where $B_1$ and $B_2$ are the shaping transformations corresponding to arcs and spikes of $w$ that are different from $\alpha$ (labeled counterclockwise from $\alpha$), oriented either so that $w$ lies on the left of the arc or so that the spike points into the common ideal endpoint.

Now observe that $A(\gamma_+)  A(\vec{\alpha}, u) A(\gamma_-) \inverse$ is trivial by the inductive hypothesis, as it corresponds to the comparison between the two possible definitions of $\varphi_{p_v, p_u^-}$.
We also note that
\[ A(\vec{\alpha}, u)\inverse \,
T_{h_u^+}^{\ac(\alpha)} \, B_1 \, B_2 \, T_{h_u^-}^{-\ac(\alpha)}\]
is conjugate to 
\[T_{h_u^-}^{-\ac(\alpha)} \, A(\vec{\alpha}, u)\inverse 
T_{h_u^+}^{\ac(\alpha)} \, B_1 \, B_2
= A(\cev{\alpha}, w) B_1 B_2 = 1\]
where the first equality follows from \eqref{eqn:adjhex_shape} (note the reversals in orientations of $h_u^\pm$) and the second follows from the cocycle relation (Propositions \ref{prop:cocycle_hex} and \ref{prop:cocycle_deg_hex}).
Therefore, we see that the entire term $A(\delta_+) A(\delta_-)\inverse$ is trivial, which is what we wanted to show.
\end{proof}

\begin{remark}
The above statement can also be proven by interpreting $A(\delta_\pm)$ in terms of sliding.
In particular, let $Z$ denote the $\pi_1(Y)$-equivariant hyperbolic structure on $Y$ corresponding to the weighted arc system $\arcwt +\acarc$.
Then superimposing $Z$ on $Y$ so that $(g_w, q_w) = (h_w, p_w)$, one can consecutively apply the shaping transformations $A_i$ to $Z$ while keeping $Y$ fixed.

Doing so, we see that $A_n$ moves $Z$ so that $(g_n,q_{n}) = (h_n,p_n)$, then $A_{n-1} \circ A_n$ moves $Z$ so that $(g_{n-1}, q_{n-1}) = (h_{n-1}, p_{n-1})$, etc.
At the end of this process, we have applied $A(\delta_+)$ to $Z$ and by construction, the pointed geodesic $(g_v, q_v)$ matches up with $(h_v, p_v)$. Since the final positioning of $Z$ is the same relative to $Y$ whether we used $A(\delta_+)$ or $A(\delta_-)$, this allows us to conclude that the two compositions define the same element.
\end{remark}

\begin{remark}
While we used the distinguished boundary paths $\delta_\pm$ to define the shape-shifting map, one could in fact use {\em any} path from $p_v$ to $p_w$ in $Y \cup \widetilde{\arc}$. In this case, one must take more care to enumerate basepoints so that $p_i, p_{i+1}$ always either lie on the same geodesic or in the same hexagon.
\end{remark}

Observe that reversing the orientation $\overline{[v,w]} = [w,v]$ also reverses the sequence $p_{n+1}, \ldots, p_1$ of basepoints that the boundary paths $\overline{\delta_\pm}$ meet.
Since flipping the order of $p_i$ and $p_{i+1}$ inverts each of the $A_i$ transformations defined above, we therefore discover that 
$\varphi_{p_w, p_v} = \varphi_{p_v, p_w}\inverse$.

In a similar vein, it is not hard to see that the shape-shifting maps satisfy a cocycle relation.

\begin{proposition}\label{prop:subsurf_cocycle}
For any triple of pointed geodesics $(h_u, p_u)$, $(h_v, p_v)$, and $(h_w, p_w)$ of $\partial Y$, we have that 
\[\varphi_{p_u, p_v}\circ\varphi_{p_v,p_w} \circ \varphi_{p_w, p_u} = 1.\]
\end{proposition}
\begin{proof}
This follows immediately from the definitions when $v$ lies on either of the paths $\delta_\pm$ from $u$ to $w$.

Otherwise, note that the intersection of paths $[u,v]\cap[v,w]\cap[w,u]$ is a point $x\in \cH$.
Choosing a basepoint $p_x \in \partiall H_x$, compute the shape-shifting transformations using the boundary arcs that pass through $x$.
Then we may express $\varphi_{p_u, p_v} = \varphi_{p_u, p_x}\circ \varphi_{p_x, p_v}$ and using the observation about inverses above,  we realize that 
\[\varphi_{p_u, p_v}\circ\varphi_{p_v,p_w} \circ \varphi_{p_w, p_u}
= \varphi_{p_u, p_x} \circ
\left( \varphi_{p_x, p_v} \circ \varphi_{p_v, p_x} \right)
\circ
\left( \varphi_{p_x, p_w} \circ \varphi_{p_w, p_x} \right)
\circ \varphi_{p_x, p_u} 
=1.\]
This completes the proof of the proposition.
\end{proof}

\subsection{The shape-shifting cocycle}\label{subsec:shapeshift_total}
We now combine the shape-shifting maps for simple pairs (Definition \ref{def:shapeshift_simple}) with those for complementary subsurfaces (Definition \ref{def:shapeshift_subsurf}) into the promised shape-shifting cocycle (Proposition \ref{prop:shapeshift_cocycle}), which is well-defined as long as the combinatorial deformation $\ac$ is small enough. 
As usual, we construct a geometric train track $\tau$ from $\lambda$ on $X$ so that the weight space of $\taua$ provides a notion of size for $\ac$.

\para{Admissible routes}
For $v, w\in \cH$ and $Y$ a component of $\tX \setminus \tlambda$, we say that $Y$ is \emph{thick} with respect to $v$ and $w$ if one of the two possibilities occur:
\begin{enumerate}
    \item either $Y$ contains $v$ and/or $w$, or
    \item $v$ and $w$ lie in different components of $\tX \setminus Y$ and the boundary leaves of $Y$ closest to $v$ and $w$ are not asymptotic.
\end{enumerate}
\label{ind:thick}
Observe that in the first case, there is either no or one boundary geodesic of $Y$ separating $v$ from $w$ (depending if $v$ and $w$ are both in $Y$ or not), while in the second, there are exactly two boundary components of $Y$ separating $v$ from $w$.

Now let $v,w \in \cH$ be any pair of distinct hexagons that do not lie in the same component of $\tX \setminus \tlambda$ and let $(h_v,p_v)$ and $(h_w,p_w)$ be a pointed geodesic in $\partiall H_v$ and $\partiall H_w$.
Then there is a unique (possibly empty) sequence $h_1, \ldots, h_n$ of boundary geodesics of thick subsurfaces separating $p_v$ from $p_w$, ordered by proximity to $v$ (with $h_1$ closest).
\footnote{Note that this sequence is necessarily finite, as the distance that any geodesic travels in a thick subsurface is bounded below by the shortest arc of $\arc$ (compare the discussion of ``close enough'' pairs of hexagons in Section \ref{subsec:shsh_hyp}).}
If one of the $h_i$ lies in the boundary of two complementary subsurfaces (so corresponds to a lift of a curve component of $\lambda$) then we record it twice, one time for each of the adjoining subsurfaces.
Additionally, if either $h_v$ or $h_w$ is a boundary geodesic separating $v$ from $w$, then we do not record it as one of the $h_i$. See Figure \ref{fig:admiss_route}.

\begin{figure}[ht]
    \centering
    \begin{tikzpicture}
    \draw (0, 0) node[inner sep=0] {\includegraphics{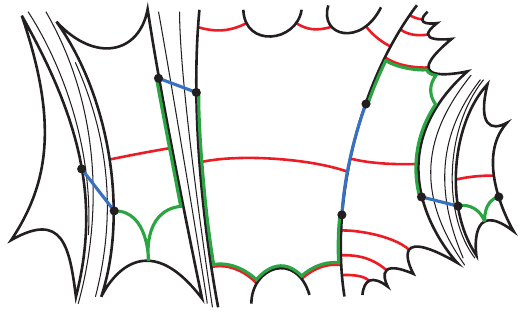}};
    The next two lines produce a grid and origin marker
    \node at (-3.3, -.2){$p_v$};
    \node at (-2.2, -.7){$p_1$};
    \node at (-2, 1.3){$p_2$};
    \node at (-.75, 1){$p_3$};
    \node at (1.1, -1){$p_4$};
    \node at (2.1, .8){$p_5$};
    \node at (2.5, -.8){$p_6$};
    \node at (3.6, -.6){$p_7$};
    \node at (4.4, -.6){$p_w$};
    \end{tikzpicture}
    \caption{Thick subsurfaces between $v$ and $w$ and an admissible route from $p_v$ to $p_w$. In the figure we have highlighted a path from $p_v$ to $p_w$ through the $p_i$; each subpath from $p_i$ to $p_{i+1}$ specifies a factor in the shape-shifting transformation.}
    \label{fig:admiss_route}
\end{figure}

We now define an {\em admissible route} from $p_v$ to $p_w$ to be any sequence of basepoints
\label{ind:admissroute}
\[p_v = p_0, \, p_1 \in h_1, \, \ldots, \, p_n \in h_n, \, p_{n+1} = p_w\]
coming from the projections of the central vertices $u_i$ of hexagons $H_{u_i}$ to $h_i \in \partiall H_{u_i}$. If any geodesic $h_i = h_{i+1}$ is repeated then we require that $v$ and $u_i$ lie on one side of $h_i$ and that $w$ and $u_{i+1}$ lie on the other.
Observe that the sequence of pairs $(u_i, u_{i+1})$ necessarily alternates between simple pairs/pairs sharing a boundary geodesic and pairs which lie in the same (thick) subsurface.

\para{Shape-shifting along admissible routes}
To any admissible route we can then define a shape-shifting transformation by concatenating the shape-shifting transformations for subsequent pairs:
\begin{equation}\label{defn:shapeshift_admiss}
\varphi_{p_v, p_w}:= \varphi_{p_0, p_{1}} \circ \ldots \circ \varphi_{p_n, p_{n+1}}
\end{equation}
where $\varphi_{p_i, p_{i+1}}$ is as in Definition \ref{def:shapeshift_simple} if $(u_i, u_{i+1})$ is simple and as in Definition \ref{def:shapeshift_subsurf} if $u_i$ and $u_{i+1}$ lie in the same subsurface. If $h_i = h_{i+1}$, then we orient $h_i$ to the right as seen from $u_i$ and set $\varphi_{p_i, p_{i+1}} = T_{h_i}^{\ac(u_i, u_{i+1})}$ (recall that we can associate a shear value to the pair $(u_i, u_{i+1})$ by \eqref{eqn:def_non_backtracking_arc}).

\begin{lemma}\label{lem:admissible_decomposition}
The shape-shifting map $\varphi_{p_v,p_w}$ is independent of the choice of admissible route (as long as it is defined).
\end{lemma}
\begin{proof}
Since the $h_i$ are uniquely determined, it suffices to change one point at a time.

So suppose that $p_i$ and $p_i'$ are both basepoints on the geodesic $h_i$;
we then demonstrate the equality
\[\varphi_{p_{i-1},p_i} \circ \varphi_{p_{i},p_{i+1}}
= \varphi_{p_{i-1},p_i'} \circ  \varphi_{p_{i}',p_{i+1}}\]
from which the lemma follows. Orient $h_i$ so that it runs to the right as seen from $v$ or $u_{i-1}$.

Without loss of generality, we may assume that the hexagons $u_i$ and $u_{i+1}$ lie in the same subsurface. Otherwise, the hexagons $u_{i-1}$ and $u_i$ lie in the same subsurface and so $(p_i, p_{i+1})$ is either simple or the points lie on the same isolated leaf. If this happens we prove that 
\[\varphi_{p_{i+1},p_i} \circ \varphi_{p_{i},p_{i-1}}
= \varphi_{p_{i+1},p_i'} \circ  \varphi_{p_{i}',p_{i-1}},\] 
which is equivalent to the equation above since each of the shape-shifting factors inverts when one flips the order of the points.\\

We first consider the shape-shifting transformations coming from comparing $p_i$ or $p_i'$ with $p_{i+1}$. By our reduction above, $u_i$ and $u_i'$ lie in the same thick subsurface $Y$.
Let $\alpha_1, \ldots, \alpha_m$ denote the arcs of $\widetilde{\arc} \cap Y$ encountered when traveling from $p_i'$ to $p_i$ along $h_i$;
then our definition of shape-shifting in subsurfaces associates the transformation 
\[\varphi_{p_i', p_i} = T_{h_i}^{\varepsilon_1 \sum^m_{j=1} \ac(\alpha_j)}\]
where $\varepsilon_1=1$ if $h_i$ is oriented from $p_i'$ to $p_i$ and $-1$ otherwise.
Combining this equation with the subsurface cocycle relation (Proposition \ref{prop:subsurf_cocycle}), we have that
\begin{equation}\label{eqn:change_admiss_subsurf}
\varphi_{p_i', p_{i+1}} = \varphi_{p_i', p_i} \circ \varphi_{p_i, p_{i+1}} =  T_{h_i}^{\varepsilon_1 \sum^m_{j=1} \ac(\alpha_j)} \circ \varphi_{p_i, p_{i+1}}.
\end{equation}

We now turn our attention to the transformation $\varphi_{p_{i-1}, p_i'}$.
Consider first the case when $(p_{i-1}, p_i)$ is simple; since $p_i$ and $p_i'$ both lie on $h_i$ this implies that $(p_{i-1},p_i')$ is also simple.
Moreover, since the geodesics $\cH_{i-1, i}$ that separate $p_{i-1}$ from $p_i$ are the same that separate $p_{i-1}$ from $p_i'$, we may write 
\[\varphi_{p_{i-1}, p_i}
= \lim_{\fcH \rightarrow \cH_{v,w}} \varphi(s_1) \circ \ldots \circ \varphi(s_n) \circ T_{h_i}^{\ac(u_{i-1}, u_i)}\]
and similarly for $\varphi_{p_{i-1}, p_i'}$.
In particular, each approximation for $\varphi_{p_{i-1}, p_i}$ differs from the approximation for $\varphi_{p_{i-1}, p_i'}$ by translation along $h_i$, and so the same is true in the limit: 
\begin{equation}\label{eqn:change_admiss_simple}
\varphi_{p_{i-1}, p_i'} = \varphi_{p_{i-1}, p_i} \circ T_{h_i}^{\ac(u_{i-1}, u_i')-\ac(u_{i-1}, u_i)}.
\end{equation}
Applying axiom (SH3) for shear-shape cocycles (Definition \ref{def:shsh_axiom}) multiple times, we compute that
\begin{equation}\label{eqn:admiss_sheardiff}
\ac(u_{i-1}, u_i')-\ac(u_{i-1}, u_i)
= \varepsilon_2 \sum^m_{j=1} \ac(\alpha_j)
\end{equation}
where $\varepsilon_2 = +1$ if $p_i$ precedes $p_i'$ along $h_i$ and $-1$ if $p_i'$ precedes $p_i$.
Combining \eqref{eqn:change_admiss_subsurf}, \eqref{eqn:change_admiss_simple}, and \eqref{eqn:admiss_sheardiff},
\[\varphi_{p_{i-1},p_i'} \circ  \varphi_{p_{i}',p_{i+1}}
= \varphi_{p_{i-1},p_i} \circ T_{h_i}^{\varepsilon_2 \sum^m_{j=1} \ac(\alpha_j)}
\circ T_{h_i}^{\varepsilon_1 \sum^m_{j=1} \ac(\alpha_j)}
\circ \varphi_{p_{i},p_{i+1}}
= \varphi_{p_{i-1},p_i} \circ \varphi_{p_{i},p_{i+1}}\]
since $\varepsilon_2 = -\varepsilon_1.$ This completes the proof of the lemma in the case when $(u_{i-1}, u_i)$ is simple.

Similarly, if $p_{i-1}$ and $p_i$ lie on the same isolated leaf of $\lambda$ then so must $p_i'$.
Unpacking the definitions shows that \eqref{eqn:change_admiss_simple} holds in this case, and Lemma \ref{lem:measure_isolatedleaf} implies that \eqref{eqn:admiss_sheardiff} does as well.
Therefore, in this case we also see that the desired equality holds.
\end{proof}

Finally, now that we have constructed shape-shifting maps for arbitrary pairs of pointed geodesics in $\partiall\cH$ we can prove that they piece together into an $\Isom^+(\tX)$-valued cocycle.

\begin{proposition}\label{prop:shapeshift_cocycle}
The map constructed from $X$ and $\ac$
\begin{align*}
\varphi_{\ac}: \partiall\cH \times \partiall\cH &\to \Isom^+(\tX)\\
((h_v, p_v), (h_w,p_w)) & \mapsto \varphi_{p_v, p_w}
\end{align*}
is a $\pi_1(X)$-equivariant 1-cocycle, as long as $\|\ac\|_{\taua}<D_\lambda(X)/2$.
\end{proposition}
\begin{proof}
That $\varphi$ is $\pi_1(X)$-equivariant means that $\varphi_{\gamma p_v, \gamma p_w} = \gamma \circ \varphi_{p_v,p_w} \circ \gamma\inverse$ for $\gamma\in \pi_1(X)$; this follows directly from the construction.

That $\varphi$ is a $1$-cocycle means it satisfies the familiar cocycle condition on triples, i.e., \[\varphi_{p_u, p_v}\circ\varphi_{p_v, p_w} = \varphi_{p_u,p_w}. \]
Note that if $p_v$ lies on some admissible route from $p_u$ to $p_w$ then this is fulfilled automatically by unpacking the definitions and invoking Lemma \ref{lem:admissible_decomposition}.

One special case of the cocycle condition is when $p_u = p_w$; in this case we must show that $\varphi_{p_v, p_w} = \varphi_{p_w, p_v} \inverse$.
To demonstrate this, observe that reversing an admissible route from $v$ to $w$ produces an admissible route from $w$ to $v$.
Moreover, by Corollary \ref{cor:simple_inverse} in the simple case and by definition in the other cases, each $\varphi_{p_i, p_{i+1}}$ also inverts when we flip $i$ and $i+1$, proving that reversing $v$ and $w$ inverts $\varphi_{p_v, p_w}$.

Now suppose that $u,$ $v,$ and $w$ are all distinct; then there exists a unique subsurface $Y$ of $\tX \setminus \tlambda$ such that each component of $\tX \setminus Y$ contains at most one of $u$, $v$, or $w$ (note that some of $u, v, w$ may be inside of $Y$). Choose basepoints $r_u$, $r_v$, and $r_w$ on the boundary components of $Y$ that are closest to $u$, $v$, and $w$ (if any $\bullet \in \{u,v,w\}$ is in $Y$ then set $r_\bullet = p_\bullet$). See Figure \ref{fig:coycle_admiss}.

\begin{figure}[ht]
    \centering
    \begin{tikzpicture}
    \draw (0, 0) node[inner sep=0] {\includegraphics{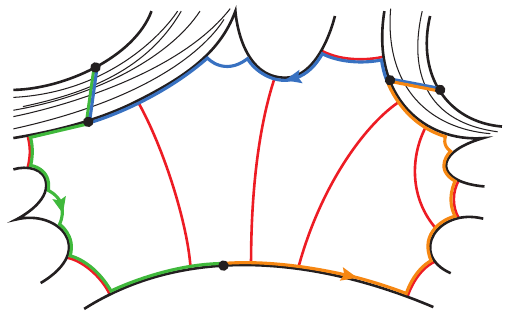}};
    \node at (-.6, -2.2){$p_v=r_v$};
    \node at (1.9, 1.2){$r_w$};
    \node at (3.4, 1.2){$p_w$};
    \node at (-2.8, .3){$r_u$};
    \node at (-2.9, 1.8){$p_u$};
    \node at (-2.7,-.7){$\varphi_{p_u, p_v}$};
    \node at (1.8, -1.7){$\varphi_{p_v, p_w}$};
    \node at (-.6, 1.2){$\varphi_{p_w, p_u}$};
    \end{tikzpicture}
    \caption{The cocycle relation for admissible routes.}
    \label{fig:coycle_admiss}
\end{figure}

Choose an admissible route from $p_u$ to $p_v$ containing $r_u$ and $r_v$, and similarly for the other two pairs.
Then by Lemma \ref{lem:admissible_decomposition} and the observation that the cocycle condition holds along admissible routes, we may write
\[\varphi_{p_u,p_v}  = 
\varphi_{p_u,r_u} \circ \varphi_{r_u,r_v} \circ \varphi_{r_v,p_v}\]
and similarly for the other two pairs. Combining all three equations and applying the cocycle relation for $Y$ (Proposition \ref{prop:subsurf_cocycle}), we see that
\begin{align*}
\varphi_{p_u,p_v} \circ \varphi_{p_v,p_w} 
&= \varphi_{p_u,r_u} \circ \varphi_{r_u,r_v} \circ \varphi_{r_v,p_v}
\circ \varphi_{p_v,r_v} \circ \varphi_{r_v,r_w} \circ \varphi_{r_w,p_w}\\
& = \varphi_{p_u,r_u} \circ \varphi_{r_u,r_w} \circ \varphi_{r_w,p_w} 
= \varphi_{p_u,p_w},
\end{align*}
finishing the proof. See Figure \ref{fig:coycle_admiss} for a graphical depiction of this argument.
\end{proof}

\section{Shear-shape coordinates are a homeomorphism}\label{sec:shsh_homeo}

We now finish the proof of Theorem \ref{thm:hyp_main} by proving that the map $\sigl: \T(S) \rightarrow \SH^+(\lambda)$ is open (Theorem \ref{thm:shsh_open}) and proper and thus, by invariance of domain, a homeomorphism.

In Section \ref{subsec:shapeshift_deform}, we use the shape-shifting cocycle $\varphi_{\ac}$, built in the previous section, to deform the representation $\rho: \pi_1(S) \rightarrow \PSL_2 \RR$ that induces the hyperbolic structure $X$. The deformed representation $\rho_{\ac}$ is then discrete and faithful (Lemma \ref{lem:deform_DF}) and the quotient surface $X_\ac$ has the desired shear-shape cocycle (Lemma \ref{lem:shsh_correct}). In particular, this gives us a continuous local inverse to $\sigl$, proving openness. These statements are similar in spirit to those in \cite{Bon_SPB}, but the specifics of our proofs are different.
In particular, instead of adjusting the relative placements of ideal triangles of $\tX \setminus \tlambda$ we adjust the relative position of pointed geodesics comprising $\tlambda$.

We then prove properness of $\sigl$ in Section \ref{subsec:shsh_homeo}, concluding the proof of Theorem \ref{thm:hyp_main}. Here we return to Bonahon's argument \cite[Theorem 20]{Bon_SPB}, but applying this strategy in our setting still requires a bit of extra care due to the polyhedral structure of $\SH^+(\lambda)$.

Finally, in Section \ref{subsec:Thurston_geos} we show that the action of $\RR_{>0}$ on $\SH^+(\lambda)$ by dilation produces lines in $\T(S)$ that can sometimes be identified with directed Thurston geodesics.

\subsection{Deforming by shape-shifting}\label{subsec:shapeshift_deform}

In this section, we show that any positive shear-shape cocycle close enough to $\sigl(X)$ is actually the geometric shear-shape cocycle of a hyperbolic structure. Compare with \cite[Proposition 13]{Bon_SPB}.

\begin{theorem}\label{thm:shsh_open}
Let $\arcb$ be a maximal arc system containing $\arc(X)$ and let $\tau_{\arcb}$ be a standard smoothing.
Then for any $\ac \in W(\tau_{\arcb})$ such that
$\| \ac \|_{\tau_{\arcb}} < D_\lambda(X)/2$ and such that $\sigl(X) + \ac$ represents a positive shear-shape cocycle, there exists $X_{\ac} \in \T(S)$ close to $X$ with 
\[\sigl(X_{\ac}) = \sigl(X) + \ac.\]
In particular, the image of $\sigl(X)$ is open in $\SH^+(\lambda)$.
\end{theorem}
The proof of this theorem appears at the end of this subsection as the culmination of a series of structural lemmas. Our strategy is to explicitly define $X_\ac$ by using the  shape-shifting cocycle constructed in Section \ref{sec:shapeshift_def} to deform the hyperbolic structure on $X$.
Before proceeding we note the following
\begin{corollary}\label{cor:eq=translation}
For all $t\in \RR$, and for all $\mu\in \Delta(\lambda)$, we have the following identity
\[\sigl(\Eq_{t\mu}(X)) = \sigl(X) + t\mu.\]
\end{corollary}

\begin{proof}
That the earthquake $\Eq_{t\mu}(X)$ is defined for all time is a consequence of countable additivity (equivalently, positivity) of $\mu$; a complete proof can be found in \cite[Section III]{EM}.
Viewing the set of measures supported on $\lambda$ as a subset of $\cH(\lambda)$, the formula is immediate from Theorem \ref{thm:shsh_open} once we note that  $\Eq_{t\mu}(X)=X_{t\mu}$, which follows from the description of $\varphi_{t\mu}$ as a limit of simple left (or right) earthquakes; see \eqref{eqn:adj_r} and Lemma \ref{lem:intuitive_adjustment}.    
\end{proof}

Fix $\ac$ as in the statement of the theorem and pick an arbitrary $v\in \cH$ and $(h_v, p_v)\in \partiall H_v$.
Identifying $\tX$ isometrically with $\HH^2$ and $(h_v, p_v)$ with a pointed line picks out a representation $\rho: \pi_1(S)\to \PSL_2\RR$ that induces $X$.
Since $\|\ac\|_{\taua}<D_\lambda(X)/2$, Proposition \ref{prop:shapeshift_cocycle} allows us to construct the shape-shifting cocycle $\varphi_{\ac}$.

We may now deform the representation $\rho$ by $\varphi_{\ac}$ by defining
\begin{align*}
    \rho_\ac: \pi_1(S) &\to \PSL_2\RR\\
    \gamma &\mapsto \varphi_{p_v, \gamma p_v} \circ \rho(\gamma).
\end{align*}
\label{ind:deformrep}
The equivariance and cocycle properties of Proposition \ref{prop:shapeshift_cocycle} ensure that $\rho_{\ac}$ is itself a representation. Indeed:
\begin{align*}
\rho_\ac(\gamma_1\gamma_2) & = \varphi_{p_v, \gamma_1\gamma_2 p_v} \circ \rho(\gamma_1 \gamma_2) \\ 
    & = \varphi_{p_v, \gamma_1p_v} \circ \varphi_{\gamma_1p_v, \gamma_1\gamma_2 p_v} \circ \rho(\gamma_1) \circ \rho(\gamma_2) \\ 
    &= \varphi_{p_v, \gamma_1p_v} \circ \rho(\gamma_1) \circ \varphi_{p_v, \gamma_2p_v} \circ \rho(\gamma_1)\inverse \circ \rho(\gamma_1) \circ \rho(\gamma_2) \\ 
    & = \rho_\ac(\gamma_1) \circ \rho_\ac(\gamma_2)
\end{align*}
for all $\gamma_1, \gamma_2 \in \pi_1(S)$.
Our goal in the rest of the section is then to show that $\rho_{\ac}$ is discrete and faithful, and that the quotient surface has the correct geometric shear-shape cocycle.

\para{Adjusting geodesics}
To show that $\rho_{\ac}$ has the desired properties, we use $\varphi_{\ac}$ to adjust the position of $\tlambda$ in $\tX$. Ultimately, these adjusted geodesics correspond to the realization of $\lambda$ on the quotient surface $\tX / \im \rho_{\ac}$.

Let $\mathscr{G}(\tX)$ be the space of geodesics in $\tX$,
\label{ind:geospace}
and let $\partial \tlambda \subset \mathscr{G}(\tX)$ denote the set of boundary leaves of $\tlambda$.
Define a map 
\[\Phi_{p_v}: \partial \tlambda \to \mathscr G(\tX),\]
\label{ind:Phi}
as follows: if $h$ is a leaf of $\partial \tlambda$, then $h = h_u$ for some $(h_u, p_u)$ in $\partiall \cH$.
The map $\Phi_{p_v}$ then takes $(h_u, p_u)$ isometrically to the pointed geodesic $\varphi_{p_v, p_u}(h_u, p_u)\subset \tX$.
Note that if $h_u = h_w$ for some other $(h_w, p_w)$ in $\partiall \cH$, then
\[\varphi_{p_v,p_u} \inverse \circ \varphi_{p_v,p_w} = \varphi_{p_u, p_w}\]
by the cocycle relation (Proposition \ref{prop:shapeshift_cocycle}) and $\varphi_{p_u, p_w}$ is by definition a translation along $h$.
Therefore $\Phi_{p_v} h_w = \Phi_{p_v} h_u$, so $\Phi_{p_v}$ is indeed well defined.

Using the fact that $S$ is closed, the following lemma follows directly from the fact that $\rho_{\ac}$ defines a representation of $\pi_1(S)$ in the same component of representations as $\rho$.  We give a hands on explanation that does not use this fact.
\begin{lemma}\label{lem:deform_DF}
The representation $\rho_{\ac}$ constructed above is discrete and faithful.
\end{lemma}
\begin{proof}
For distinct leaves $h_u$ and $h_w \in \partial \tlambda$, we claim that $\Phi_{p_v}(h_u)$ is disjoint from $\Phi_{p_v}(h_w)$.  
Indeed, by the cocycle relation,
the position of $\Phi_{p_v}(h_w)$ relative to $\Phi_{p_v}(h_u)$ is the same as the position of $\varphi_{p_u,p_w}h_w$  relative to $h_u$.
Every finite approximation of $\varphi_{p_u,p_w}$ by compositions of elementary shape-shifting transformations preserves the property that the image of $h_w$ is disjoint from $h_u$, so the same is true in the limit.

Therefore, as long as $\rho(\gamma)$ does not stabilize $h_v$ then $\Phi_{p_v}(\rho(\gamma)h_v) =\rho_{\ac}(\gamma)h_v $ is different from $ \Phi_{p_v}(h_v)=h_v$.
If $\rho(\gamma)$ is a translation along $h_v$, we can find $\gamma_0$ such that $\rho(\gamma_0\gamma\gamma_0\inverse)h_v \not= h_v$, and so we see $\rho_{\ac}(\gamma)$ does not stabilize 
$\rho(\gamma_0\gamma\gamma_0\inverse)h_v$.
In either case, this implies that $\rho_{\ac}(\gamma)$ acts nontrivially on the space of geodesics, so in particular $\rho_{\ac}(\gamma) \not= 1$, i.e., $\rho_{\ac}$ is faithful.

Since $\pi_1(S)$ is a non-elementary group and $\rho_{\ac}$ is faithful, $\im\rho_{\ac}$ is a non-elementary subgroup of isometries.
So assume towards contradiction that $\rho_{\ac}$ is indiscrete.  Then $\im\rho_{\ac}$ must be dense in $\PSL_2\RR$; see, e.g., \cite[Proposition, p. 246]{Sullivan:stability}. In particular, there is an element $\gamma\in \pi_1(S)$ such that $\rho_{\ac}(\gamma)$ is arbitrarily close to a rotation of angle $\pi/2$ around $p_v$.
Then $\rho_{\ac}(\gamma)h_v = \Phi_{p_v}\rho(\gamma)h_v$ meets $h_v$ in a point, which is impossible, because $\Phi_{p_v}(h_v)$ is either equal to or disjoint from $\Phi_{p_v}(\rho(\gamma)h_v)$. 
We conclude that $\rho_{\ac}$ is discrete, completing the proof of the lemma.
\end{proof}

By Lemma \ref{lem:deform_DF}, the quotient $X_{\ac} = \tX/\im \rho_{\ac}$
\label{ind:Xac}
is a hyperbolic surface equipped with a homeomorphism $S\to X_{\ac}$ in the homotopy class determined by $\rho_{\ac}$.
As such, $\rho_{\ac}$ induces a $(\rho, \rho_{\ac})$-equivariant homeomorphism $\partial \tX \to \partial \tX$, hence a continuous, equivariant map on the space of geodesics.
\begin{lemma}
The map $\Phi_{p_v}$ extends continuously to $\tlambda$, and $\Phi_{p_v}(\tlambda)$ descends to the geodesic realization of $\lambda$ on $X_{\ac}$. 
\end{lemma}
\begin{proof}
By equivariance, the induced map on geodesics agrees with $\Phi_{p_v}$ on $\partial \tlambda$.  The leaves of $\partial \tlambda$ are dense in $\tlambda$, so the closure of the image of $\Phi_{p_v}$ is the geodesic realization of $\tlambda$ on $\widetilde X_{\ac}$, which is invariant under the action of $\rho_{\ac}$.
\end{proof}

Since $\Phi_{p_v}(\tlambda)$ is the lift of the realization of $\lambda$ on $X_{\ac}$, we may now leverage our understanding of the shape-shifting cocycle to show that the complementary subsurfaces of $X_{\ac} \setminus \lambda$ have the desired shapes.

\begin{lemma}\label{lem:arc_correct}
We have $\arcwt(X_\ac) = \arcwt(X) + \acarc.$
\end{lemma}
\begin{proof}
Recall that by construction the unweighted arc systems of $X \setminus \lambda$ and $X_{\ac} \setminus \lambda$ are both contained in some joint maximal arc system $\arcb$, leading to an identification of hexagons of $\tX \setminus \lambda$ with those of $\tX \setminus \lambda_{\ac}$.

So let $\beta$ be an arc of $\arcb$, realized orthogeodesically in $X_{\ac}$. Let $\beta$ also denote a choice of lift, orthogonal to $\lambda_{\ac}$ in $\tX$, and let $u$ and $w$ denote the hexagons adjacent to $\beta$.
Choose either of the geodesics $g$ of $\lambda$ meeting $\beta$ and let $q_u$ and $q_w$ be the basepoints of $u$ and $w$ on $g$.
Then by the cocycle relation (Proposition \ref{prop:shapeshift_cocycle}), we know that
\[\varphi_{p_v, p_w}
= \varphi_{p_v, p_u} \circ \varphi_{p_u, p_w}\]
and so applying equivariance we see that the placement of $q_w$ relative to $q_u$ differs from the placement of $p_w$ relative to $p_u$ only by $\varphi_{p_u, p_w}$.

But now since $u$ and $w$ are in the same subsurface, we see by definition that $\varphi_{p_u, p_w}$ is translation along $g$ by exactly $\ac(\beta)$. Therefore, the distance along $\varphi_{p_u, p_w} g$ between $q_u$ and $q_w$ is exactly the distance along $g$ between $p_u$ and $p_w$ plus $\ac(\beta)$.
Translated into arc weights,
\[\sigl(X_{\ac})(\beta) = \sigl(X)(\beta) + \ac(\beta),\]
completing the proof of the lemma.
\end{proof}

Now that we know that the ``shape'' part of the data of $\sigl(X_{\ac})$ is what it is supposed to be, we need only check that the ``shearing'' data is as specified. Compare \cite[Lemma 19]{Bon_SPB}.

\begin{lemma}\label{lem:shsh_correct}
The surface $X_\ac$ has geometric shear-shape coycle $\sigl(X_\ac) = \sigl(X)+\ac$.
\end{lemma}
\begin{proof}
Observe that by the cocycle relation (Proposition \ref{prop:shapeshift_cocycle}) and the discussion in Section \ref{subsec:shsh_hyp}, it suffices to compute the change in shearing data between simple pairs.

So suppose that $(v,w)$ is simple.
For each integer $r$, recall that $\cH_{v,w}^r = (u_i)_{i=1}^n$ denotes the set of hexagons such that the intersection of the geodesic from $p_v$ to $p_w$ with $u_i$ has depth at most $r$ with respect to a fixed geometric train track.
Set $v = u_0$ and $w = u_{n+1}$, and let $h_i = g_{u_i}^v$, the pointed boundary geodesic of $u_i$ closest to $v$.
Then by Lemma \ref{lem:intuitive_adjustment}, we know that 
\[\varphi_{p_v,  p_w}^r = T_{h_0}^{\ac(u_0, u_1)}\circ A(s_1)\circ T_{h_1}^{\ac(u_1, u_2)}\circ A(s_2)\circ ...  \circ A(s_n) \circ T_{h_n}^{\ac(u_n, u_{n+1})}\]
is a good approximation of $\varphi_{p_v, p_w}$ for large enough $r$.

Now for each $r$ we can deform the hyperbolic structure on $\tX$ by $\varphi_{p_v,  p_w}^r$ (sacrificing equivariance) and measure the shear $\sigma_r(v,w)$ between $v$ and $w$ in that deformed structure.
More precisely, we recall that if $h_i'$ denotes the other geodesic in $u_i$ that separates $v$ from $w$, the spike-shaping transformation is equal to a translation along $h_i'$ then along $h_i$.
We may then deform $\tX$ by replacing each translation in the factorization of $\varphi_{p_v,  p_w}^r$ with a (right) earthquake along the same geodesic; compare with our ``geometric explanation'' of spike-shaping in Section \ref{subsec:shsh_spikes}.

Since each translation $T_{h_i}^{\ac(u_i, u_{i+1})}$ appearing in $\varphi_{p_v,  p_w}^r$ shears $\tX$ along a leaf of $\lambda$, it preserves the orthogeodesic foliation in complementary components.
Therefore, each such term in the deformation thus changes the shear between $v$ and $w$ by exactly $\ac(u_i, u_{i+1})$.

On the other hand, each spike-shaping transformation $A(s_i)$ is a parabolic transformation fixing the vertex of the spike and thus preserves horocycles based at that point. In particular, the distinguished basepoints of each $h_i$ and $h_i'$ remain on the same horocycle and hence deforming by $A(s_i)$ does not affect $\sigma_r(v,w)$.

In summary, deforming $\tX$ by the approximation $\varphi_{p_v,  p_w}^r$ changes the shear between $v$ and $w$ by
\[\sigma_r(v,w) - \sigma(v,w) = \sum_{i=0}^n \ac(u_i, u_{i+1}) = \ac(v,w)\]
where the last equality follows from finite additivity (axiom (SH2)).

Since this equality holds in each approximation and $\varphi_{p_v,  p_w}^r \to \varphi_{p_v,  p_w}$ as $r \to \infty$, the equality holds in the limit as well.
Therefore, deforming $\tX$ by $\varphi_{p_v,  p_w}$ changes the shear between $v$ and $w$ by exactly $\ac(v,w)$, which is what we needed to show.
\end{proof}

\begin{proof}[Proof of Theorem \ref{thm:shsh_open}]
As $\|\ac\|_{\taua}<D_\lambda(X)/2$, Lemmas \ref{lem:simple_piece_convergence} and \ref{lem:intuitive_adjustment} ensure that the limits in the definition of $\varphi_{p_v, p_w}$ make sense for all simple pairs $(v,w)$. Proposition \ref{prop:shapeshift_cocycle} then allows us to construct $\varphi_{\ac}$.
By Lemma \ref{lem:deform_DF}, the deformed representation $\rho_{\ac} = \varphi_{\ac} \cdot \rho$ is discrete and faithful, and by Lemma \ref{lem:shsh_correct} the quotient surface has the correct geometric shear-shape cocycle.

Finally, we observe that the values of  shape-shifting cocycle $\varphi_\ac$ all converge to the identity as $\|\ac\|_{\taua} \rightarrow 0$, and consequently $X_{\ac} \to X.$ This completes the proof of the theorem.
\end{proof}

\subsection{The global structure of the shear-shape map}\label{subsec:shsh_homeo}
We have already proven in Proposition \ref{cor:sigl_into_SH+} that the image of $\sigl$ lies in $\SH^+(\lambda)$. We now show that this containment is in fact an equality, completing the proof of Theorem \ref{thm:hyp_main}.

We proceed in two steps; the first is to show that

\begin{proposition}\label{prop:shsh_local_homeo}
The shear-shape map $\sigl$ is a homeomorphism onto its image.  
\end{proposition}

\begin{proof}
Proposition \ref{prop:shsh_inj} (injectivity of $\sigl$) allows us to invert $\sigl$ on its image and for each $X \in \T(S)$, Theorem $\ref{thm:shsh_open}$ provides us with an open neighborhood of $\sigl(X)\in \SH^+(\lambda)$ on which $\sigl\inverse$ is defined and continuous.
By Proposition \ref{prop:SH+_structure}, we have that $\SH^+(\lambda)\subset \SH(\lambda)$ is an open cell of dimension $6g-6$. Invoking invariance of domain, we get that $\sigl\inverse$ and hence $\sigl$ are local homeomorphisms.
An additional application of Proposition \ref{prop:shsh_inj} implies that $\sigl$ is globally injective, so $\sigl$ is a homeomorphism onto its image as claimed.
\end{proof}

The second step is to prove that $\sigl:\T(S)\to \SH^+(\lambda)$ is a proper map.
That is, we must show that when $X_k$ escapes to infinity in $\T(S)$, the corresponding shear-shape cocycles $\sigl(X_k)$ must diverge in $\SH^+(\lambda)$. Since proper local homeomorphisms are coverings and $\SH^+(\lambda)$ is a cell, the map $\sigl$ must be a homeomorphism.

The proof we present below is essentially just that of \cite[Theorem 20]{Bon_SPB}, but we have to address the additional complications introduced by the PL structure of $\SH^+(\lambda)$; this manifests itself in the {\em stratified} real-analytic structure of the map.
\footnote{Recall that  $\mathcal H^+(\lambda)$ is an open cone with finitely many faces in a vector space, while $\SH^+(\lambda)$ is an affine cone bundle over a piecewise linear space with no obvious way of extending the smooth structure over faces of $\Base$.}

\begin{proof}[Proof of Theorem \ref{thm:hyp_main}]
We begin by recording an estimate for the geometry of surfaces near the boundary of the image of $\sigl$ (where ``near'' is measured in a train track chart).

So suppose that $X \in \T(S)$, set $\arc = \arc(X)$, and build a standard smoothing $\taua$ carrying $\lambda$ geometrically on $X$.
Fix $\epsilon > 0$ and suppose that there exists some $\ac \in W(\taua)$ with $\| \ac\|_{\taua} < \epsilon$ such that $\sigl(X) + \ac \in \SH^+(\lambda)$ is not in the image of $\sigl$; then Theorem \ref{thm:shsh_open} implies that
\[D_{\lambda}(X)/2 \le \| \ac \|_{\taua} < \epsilon.\]
The following claim can be extracted from the proof of \cite[Theorem 20]{Bon_SPB}; we outline a proof for the convenience of the reader.

\begin{claim}\label{clm:measure} There is a transverse measure $\mu\in \Delta(\lambda)$ with $\frac{1}{9\chi(S)}\le \|\mu\|_{\taua} \le 1$ and 
\[\ell_{X}(\mu)= \omega_{\SH}(\sigl(X), \mu) < \epsilon.\]
\end{claim}

\begin{proof}[Proof of Claim \ref{clm:measure}]
If there is a simple closed curve component of $\lambda$ with length at most $\epsilon$, then we are done.
Otherwise, even though $\arcwt + \acarc$ defines a hyperbolic structure on each piece of $S \setminus \lambda$, the overall shear-shape cocycle $\sigl(X) + \ac$ does not define a hyperbolic structure on $S$ because the proof of Lemma \ref{lem:simple_piece_convergence} or Lemma \ref{lem:intuitive_adjustment} fails. Therefore, there is a simple pair $(v,w)$ for which the finite products $\varphi_{\underline \cH}$ (or $\varphi_{v,w}^r$) fail to converge as $\underline \cH$ tends to $\cH_{v,w}$ (or $r\to \infty$).

We claim that there exists $u$ between $v$ and $w$ and a spike $s = (g, h)$ of $H_u$ such that the following holds:
for any geodesic transversal $k \subset X$ to $\lambda$ meeting the spike $s$, the countably many points of $\tilde{k}\cap g \subset g$ (labeled by $r\in \NN$) exiting one end of $g$ escape at a rate strictly slower than $\epsilon(r-1)$.
In other words, there are segments $d_r\subset g$ such that $\ell_X(d_r)\le \epsilon(r-1)$ and $d_r$ meets $k$ exactly $r$ times.

If this were not the case, then as in the proof of Lemma \ref{lem:decay_gaps}, the ``gaps'' $c_r \subset k_{v,w}\setminus \lambda$ have length $\ell_X(c_r)=O(e^{-\epsilon r })$, where $c_r\cap g$ is labelled by $r\in \NN$. 
This estimate on the decay of gaps implies that $\varphi_{\underline \cH}$ converges as $\underline \cH \to \cH_{v,w}$ and that $\varphi_{v,w}^r \to \varphi_{v,w}$ as $r\to \infty$ (see the proof of Lemma \ref{lem:simple_piece_convergence}), contradicting our assumption.

Now consider the weight system $w_r$ on $\taua$ (not  satisfying the switch conditions) defined by counting the number of times $d_r$ travels along each branch of $\taua$, and dividing by the total number of branches $n_r$ that $d_r$ traverses, with multiplicity. Observe that $n_r \ge r$ by definition.
Then $\|w_r\|_{\taua}\le 1$ in the vector space $\RR^{b(\taua)}$ and $w_r$ takes value zero on branches corresponding to arcs of $\arc$.
Moreover, $w_r$ is non-negative on each branch and approaches the weight space $W(\tau)\subset \RR^{b(\tau)}$ as $r\to \infty$. Since $w_r$ are built from leaves of $\lambda$, any limit point $\mu$ defines a transverse measure supported on $\lambda$ (compare also \cite[Proposition 3.3.2]{PennerHarer}).

There are at most $9\chi(S)$ branches of $\taua$, so by the pigeonhole principal there is a branch such that each $w_r$ has mass at least $1/9\chi(S)$, and therefore so must $\mu$.
But now by construction,
\[\ell_X(\mu) =\lim_{r\to \infty} \frac{\ell_X(d_r)}{n_r} <\frac{(r-1)\epsilon}{n_r}<\epsilon, \]
providing the desired measure.
\end{proof}

Now suppose towards contradiction that $\arc$ is maximal and $\sigl(X_k)\in \SH^+(\lambda;\arc)$ is a sequence approaching some $\sigma \in\SH^+(\lambda;\arc)$ that is not in the image of $\sigl$. We may then apply the above construction to $\sigma - \sigl(X_k)$ to extract a family of measures $\mu_k$ on $\lambda$ satisfying $1/9\chi(S)\le\|\mu_k\|_{\taua}\le1$ and $\omega_{\SH}(\sigl(X), \mu_k)\to 0$.
By compactness of the set measures on $\lambda$ with norm bounded away from zero and infinity, there is some non-zero accumulation point $\mu$ of $\mu_k$.
Continuity of $\omega_{\SH}$ (Lemma \ref{lem:ThSHprops}) then gives 
\[\omega_{\SH}(\sigl(X_k), \mu_k) \to \omega_{\SH}(\sigma, \mu )=0, \]
and so we see that $\sigma \not\in \SH^+(\lambda;\arc)$, a contradiction.
Hence $\im(\sigl) \cap \SH^+(\lambda;\arc)$ is relatively closed.
On the other hand, $\sigl$ is a local homeomorphism by Proposition \ref{prop:shsh_local_homeo}, hence $\im(\sigl) \cap \SH^+(\lambda;\arc)$ is relatively open.

If we knew that the projection of $\im(\sigl)$ surjects onto $\Base$ (or at least meets each top-dimensional face) we would be done. Since we do not {\em a priori} have this information, we instead work our way out in $\Base$ cell by cell.

To wit, we may invoke Theorem \ref{thm:shsh_open} once more to deduce that $\im(\sigl)\cap \SH^+(\lambda;\arc')$ is relatively open for every filling arc system $\arc'$ that shares a common filling arc subsystem with $\arc$ (hence $\SH^+(\lambda;\arc)$ and $\SH^+(\lambda;\arc')$ intersect).
Repeating the argument above for these cells, we have that $\im(\sigl)\supset \SH^+(\lambda;\arc')$ as well.
Since $\Base$ is connected, iterating this procedure allows us to deduce that $\im(\sigl)\supset \SH^+(\lambda)$.
The reverse inclusion follows from Corollary \ref{cor:sigl_into_SH+}, so $\sigl$ is a homeomorphism onto $\SH^+(\lambda)$.\\

To address the regularity of $\sigl$, we note that while $\T(S)$ has a natural $\RR$-analytic structure, $\SH(\lambda)$ does not.  However, for each arc system $\arc$, filling or not, the open cell $\mathscr{B}^\circ(\arc)$ has a well defined analytic structure compatible with that of the analytic submanifold of $\T(S\setminus \lambda)$ that it parametrizes.
The total space of the bundle $\SH^\circ(\lambda;\arc)\to \mathscr{B}^\circ(\arc)$ also carries an analytic structure, invariant under train track coordinate--transformations
(Proposition \ref{prop:SH+_structure}); thus $\SH(\lambda)$ has a stratified $\RR$-analytic structure.

The shape-shifting cocycle $\varphi_{\ac}$, hence the surface $X_{\ac}$, then depends real-analytically on $\ac \in W(\taua)$ (where $\arc$ here is equal to the support of $\arcwt(X)$, not a maximal completion). The reason for this is clear: all elementary shape-shifting transformations are products of small parabolic transformations (see \cite[Section 9]{Th_stretch} or \cite[Theorem A]{Bon_SPB}) or translations with translation distance that are (restrictions of) real-analytic functions on (an analytic submanifold of) $\T(S\setminus \lambda)$.
These products converge absolutely to the shape-shifting cocycle, hence uniformly on compact sets to an analytic deformation. 
\end{proof}

\subsection{Dilation rays and Thurston geodesics}\label{subsec:Thurston_geos}

Using our coordinatization, we can define an extension of the earthquake flow to an action by the upper-triangular subgroup.

\begin{definition}\label{def:dilationray}
Given a measured geodesic lamination $\lambda$, a hyperbolic surface $X\in \T(S)$, and $t \in \RR$, define an analytic path of surfaces $\{X_\lambda^t\}_{t\in \RR}$ by \[ X_\lambda^t:= \sigl\inverse(e^t \sigl(X)),\]
called the \emph{dilation ray}\footnote{We are abusing terminology here by declaring that the image of $\RR$ under an analytic mapping is a ray.  Our aim is to emphasize that the dilation ray should be thought of as directed toward the future, even though it can be defined for all time.} based at $X$ directed by $\lambda$.
\end{definition}

As the earthquake flow acts by translation in coordinates (Corollary \ref{cor:eq=translation}), we see that dilation and earthquake along $\lambda$ (together with scaling the measure on $\lambda$) fit together into an action by the upper-triangular subgroup $B < \GL_2^+\RR$ on $\PT_g$.
More explicitly, we can specify an action of $B$ on $\T(S) \times \RR_{>0}\lambda$ (by homeomorphisms) by setting
\begin{equation}\label{eqn:hypPaction}
    \begin{pmatrix} 
a & b \\ 0 & c
\end{pmatrix}
\cdot (X, \lambda) := (\sigl^{-1} (a \sigl(X)+b \lambda), c\lambda).
\end{equation}
These $B$-actions assemble into a $\Mod(S)$-equivariant $B$-action on $\PT_g$ (observe that $\sigl$ depends only on the support of $\lambda$ and not the actual measure). Quotienting by the mapping class group and restricting to the unit length locus then gives a $P$-action on $\PoM_g$, and since dilation preserves the property of being regular, a $P$-action on each stratum $\PoM_g^{\text{reg}}(\sing)$. We call any such action an action by {\em stretchquakes}.
\label{ind:stretchquake}

Using the commutativity of Diagram \eqref{diagram} (Theorem \ref{thm:diagram_commutes}), we can compare \eqref{eqn:hypPaction} with the computations performed in Lemmas \ref{lem:Il_geo} and \ref{lem:Il_hor} to see that

\begin{proposition}\label{prop:Pactions}
The map $\cO$ takes the $P$ action of \eqref{eqn:hypPaction} on $\PoM_g$ to the standard $P$ action on $\QoM_g$.
\end{proposition}

While we have defined them via coordinates, it is not hard to see that dilation rays are geometrically meaningful families of surfaces. Generally, we obtain paths of surfaces along which the length of $\lambda$ scales nicely, and we can identify some dilation rays as directed lines in Thurston's asymmetric metric on $\T(S)$.

Mirzakhani  observed (see \cite[Remark p. 33]{MirzEQ}) that for a maximal lamination $\mu$, the dilation ray $t\mapsto X_\mu^t$ corresponds to the stretch path directed by $\mu$ defined by Thurston in \cite[Section 4]{Th_stretch}.    
Very roughly, stretch paths are obtained by gluing together certain expanding self-homeomorphisms of the ideal triangles comprising $X\setminus \mu$ along the leaves of $\mu$.

\begin{lemma}[Proposition 2.2 of \cite{Th_stretch}]\label{lem:poly_expands}
Let $P_n$ be a regular ideal hyperbolic $n$-gon.  For any $K\ge1$, there is a $K$-Lipschitz self-homeomorphism  $P_n\to P_n$ that maps each side to itself and expands arclength along the boundary by a constant factor of $K$.
\end{lemma}
\begin{proof}
The orthogeodesic foliation $\cO(P_n)$ is measure equivalent to a partial foliation by horocycles centered at the spikes of $P_n$.  The desired $K$-Lipschitz homeomorphism $P_n\to P_n$ is constructed by fixing the central horocyclic $n$-gon and mapping each horocyclic arc at distance $s$ from the central region to the horocyclic arc at distance $Ks$ in the same spike. 
\end{proof}

Any partition $\sing = (\kappa_1, ..., \kappa_n)$ of $4g-4$ determines a regular locus $\PT_g^{\text{reg}}(\sing)$ of pairs $(X,\lambda)$, where the complement of $\lambda$ in $X$ is a union of regular ideal $(\kappa_i+2)$-gons.
Then $\PoM_g^{\text{reg}}(\sing)$ is the moduli space of pairs where $\ell_X(\lambda)=1$.

Gluing together the expanding maps of regular polygons provides an explicit model of dilation rays in $\PoM_g(\sing)$ and identifies them with geodesics for the Thurston metric.
A survey of some basic properties of Thurston's metric as well as similarities and differences between directed stretch rays and Teich\"uller geodesics can be found in \cite{PapadopTheret:Teich_Thurston}.  
The following proposition was inspired in part by recent work of Horbez and Tao, in which they investigate the minimally displaced sets in the Thurston's metric using a similar construction \cite{HT:stretch}.

\begin{proposition}\label{prop:stretch_reg}
For any $(X,\lambda)\in \PT_g^{\text{reg}}(\sing)$, the dilation ray $\{X_\lambda^t: t\in \RR\}\subset \PT_g^{\text{reg}}(\sing)$ is a directed unit-speed geodesic in Thurston's asymmetric Lipschitz metric.
\end{proposition}

\begin{proof}
Since $\lambda$ is regular on $X$,  $\sigl(X)\in \SH^+(\lambda)$ lies in the fiber over the empty arc system.  Scaling $\sigl(X)$ preserves this arc system, so $X_\lambda^t$ is regular for all $t$.  It suffices to prove that the optimal Lipschitz constant for a map $X\to X_\lambda^t$ in the homotopy class determined by markings is $e^t$ for all $t\ge0$.

Let $H_\lambda(X)$ denote the (partial) foliation of $X$ by horocyclic arcs that is measure equivalent to $\Ol(X)$.  
The maps of Lemma \ref{lem:poly_expands} assemble to an $e^t$-Lipschitz homeomorphism $X\setminus \lambda\to X_t^{\lambda}\setminus \lambda$ such that $H_\lambda(X)$ maps to $H_\lambda(X_\lambda^t) = e^t H_\lambda(X)$ on each component (as measured foliations).
Now, using the fact that $\sigl(X_\lambda^t) = e^t\sigl(X)$, we can adapt the argument of \cite[Lemma 11]{Bon_SPB} (as sketched in Proposition \ref{prop:shsh_inj}) to show  that this map is locally Lipschitz hence extends across $\lambda$ to an $e^t$-Lipschitz homeomorphism $X\to X_\lambda^t$.

Thus $e^t$ provides an upper bound for the optimal Lipschitz constant in the homotopy class determined by markings. 
On the other hand, 
\[\ell_{X_\lambda^t}(\lambda) = \ThSH(\sigl(X_\lambda^t), \lambda) =\ThSH(e^t\sigl(X), \lambda)=e^t\ell_X({\lambda}),\]
so $e^t$ is also a lower bound for the optimal Lipschitz constant.  This completes the proof of the proposition.
\end{proof}

\begin{remark}
As in the last line of the proof of Proposition \ref{prop:stretch_reg} we always have  $\ell_{X_\lambda^t}(\lambda) = e^t \ell_X(\lambda)$ for arbitrary $\lambda\in \ML(S)$.  Thus the distance from $X$ to $X_\lambda^t$ in Thurston's metric is at least $t$.  However, we do not always know how to build $e^t$-Lipschitz proper homotopy equivalences ${X\setminus \lambda}\to {X_\lambda^t\setminus \lambda}$ (in the correct homotopy  class) that expand arclength along $\partial {X\setminus \lambda}$ by a constant factor of $e^t$.   
\end{remark}

\begin{remark}\label{rmk:PW}
{\em Added in proof:} In recent work, Pan and Wolf build new families of geodesics for the Lipschitz metric using harmonic maps \cite{PWstretch}. Their work also uses our coordinates to show that certain ``Hopf differential disks'' in $\T_g$ converge to ``stretch--earthquake disks.'' It would be interesting to know if their new geodesics coincide with the dilation rays defined here, and by extension if their stretch--earthquake disks are the same as the orbits of the stretchquake action defined here.
\end{remark}

\begin{remark}
 Our dilation rays are different from Thurston's stretch rays defined with respect to one of the finitely many maximal completions of $\lambda$ when $\lambda$ is not maximal.  This follows from the fact that $\Ol(X)\not= \cO_{\lambda'}(X)$, where $\lambda'$ is a maximal completion of $\lambda$.
\end{remark}

The map $\PT_g^{\text{reg}}(\sing) \times \RR\to \PT_g^{\text{reg}}(\sing)$ defined by the rule $(X,\lambda,t)\mapsto (X_\lambda^t,e^{-t}\lambda)$ is called the \emph{stretch flow}.
The stretch flow is $\Mod(S)$-equivariant and \[\ell_{X_\lambda^t}({e^{-t} \lambda}) = \ell_X(\lambda),\]
hence descends to  $\PoM_g^{\text{reg}}(\sing)$. 

\begin{corollary}\label{cor:stretch_ergodic}
Let $\nu$ be a $P$-invariant ergodic probability measure on $\PoM_g$. 
\begin{itemize}
    \item For $\nu$-almost every $(X,\lambda)$, the dilation ray $t \mapsto X_\lambda^t $ is a unit speed geodesic in Thurston's asymmetric metric.
    \item On a set of full $\nu$-measure, the action of the diagonal subgroup of $P$ is identified with the stretch flow and $\cO$  conjugates stretch flow to Teichm\"uller geodesic flow.
\end{itemize} 
In particular, the stretch flow is ergodic with respect to $\nu$.  
\end{corollary}
\begin{proof}
By Corollary \ref{cor:reg_ae}, $\nu$-almost every point is regular (with respect to the same topological type of lamination), so the first statement of the theorem is immediate from Proposition \ref{prop:stretch_reg}.  

The second statement is essentially a restatement of Theorem \ref{mainthm:strata} combined with the previous statement. Alternatively, in the Gardiner-Masur parameterization of $\QT_g$ (Theorem \ref{thm:GM}), the Teichm\"uller geodesic flow at time $t$ is given by $(\eta, \lambda)\mapsto (e^t\eta, e^{-t}\lambda)$, so  unraveling the definitions and using commutativity of Diagram \eqref{diagram} (Theorem \ref{thm:diagram_commutes}) gives the result.

For ergodicity, we apply Theorem \ref{mainthm:AIS} which asserts, in particular, that $\cO_*\nu$ is an ergodic $\SL_2\RR$-invariant probability measure on $\QoM_g(\sing)$.
The Howe--Moore Theorem implies that any non-compact, closed subgroup of $\SL_2\RR$ inherits ergodicity (see, e.g., \cite[Theorem 3.3.1]{ErgodicTheory}); in particular, the Teichm\"uller geodesic flow is ergodic with respect to $\cO_*\nu$.
So $\cO$ maps any stretch flow--invariant set $B$ of positive $\nu$-measure to an $\cO_*\nu$ Teichm\"uller geodesic flow--invariant set of positive measure, which must have full measure by ergodicity. Thus $\nu(B) = 1$, demonstrating ergodicity of the stretch flow.
\end{proof}

Recently, Allessandrini and Disarlo \cite{AD:stretch} constructed Lipschitz maps between some pairs of degenerate right angled hexagons that stretch alternating boundary geodesics by a constant factor. 
Recall from Section \ref{sec:arc_cx} that the Teichm\"uller space of an ideal quadrilateral is $1$ dimensional and can be described as the the cone over a pair of points corresponding to the two arcs $\alpha$ and $\beta$ that join opposite sides of $Q$.  

\begin{lemma}\label{lem:quad_stretch}
Let $Q$ be an ideal quadrilateral with weighted filling arc system $s\delta$, where $\delta \in \{\alpha, \beta\}$.  Let $Q^t$ be the quadrilateral with arc system $e^ts\delta$.  There is an $e^t$-Lipschitz surjection $Q\to Q^t$ that multiplies arclength along the boundary of $Q$ by a factor of $e^t$.
Moreover, the projection of the compact edge of the spine of $Q$ is mapped to the projection of the compact edge of the spine of $Q^t$.
\end{lemma}
\begin{proof}
Every ideal quadrilateral has an orientation preserving isometric involution swapping opposite sides. 
Thus the orthogeodesic representative of $\delta$ cuts $Q$ into $2$ isometric pieces, each of which is a right angled hexagon with two degenerate sides.  
On each piece, we can apply \cite[Lemma 6.9]{AD:stretch} to obtain maps which glue together along $\delta$ to give a map with the desired properties.
\end{proof}

We immediately obtain some new geodesics for Thurston's metric.
\begin{proposition}\label{prop:quad_geo}
If $S \setminus \lambda$ consists of ideal triangles and quadrilaterals, then for any $X\in \T(S)$, $t\mapsto X_\lambda^t$ is a directed, unit speed geodesic for Thurston's asymmetric metric.
\end{proposition}
\begin{proof}
The proof is nearly identical to the proof of Proposition \ref{prop:stretch_reg}, so we only provide a brief outline.
Construct an $e^t$-Lipschitz surjective map $X\setminus\lambda\to X_t^{\lambda}\setminus \lambda$ from the units of Lemma \ref{lem:poly_expands} and Lemma \ref{lem:quad_stretch}. For the same reason as before, this map extends continuously across the leaves of $\lambda$ and provides an $e^t$-Lipschitz homotopy equivalence $X\to X_\lambda^t$ in the homotopy class determined by markings.  
Thus $e^t$ is an upper bound for the Lipschitz constant among homotopy equivalences $X\to X_\lambda^t$ in correct homotopy class.  This is clearly an upper bound for the ratio
\[\max_{\mu\in \ML(S)}\frac{\ell_{\mu}(X_\lambda^t)}{\ell_{\mu}(X)}.\]
But $e^t$ is also a lower bound for this ratio, because the length of $\lambda$ is scaled by a factor of $e^t$.  

By a theorem of Thurston \cite[Theorem 8.5]{Th_stretch}, there is a $e^t$-Lipschitz {\it homeomorphism} $X \to X_\lambda^t$ homotopic to the map constructed above.  This completes the proof of the proposition.
\end{proof}

\begin{remark}\label{rmk:thurston_geo_general}
The proof of Proposition \ref{prop:quad_geo} clearly supplies a more general statement: If $\lambda$ is filling and cuts $X\in \T(S)$ into a regular polygons and quadrilaterals of any shape, then $t\mapsto X_\lambda^t$ is a geodesic for Thurston's metric.

There are other cases in which we can glue Lipschitz maps between degenerate right angled hexagons that can be found in the literature (e.g., \cite{AD:stretch,Pap_Yam}).  However, these other cases require additional symmetry that is not always present in our setting.  We suspect that there is a different approach that would prove that dilation rays can always be identified with Thurston geodesics, so that $\cO$ conjugates a kind of Thurston geodesic flow to Teichm\"uller geodesic flow.
\end{remark}

\section{Future and ongoing work}\label{sec:future}
There is much more to understand about the correspondence between hyperbolic and flat geometry described in this paper.
In addition to using the orthogeodesic foliation to import tools from Teichm{\"u}ller dynamics into the world of hyperbolic geometry (and vice versa), the authors expect this link to provide retroactive explanations for analogous phenomena in the two settings.

We describe a number of future directions and potential applications of the correspondence below, some of which will be addressed in forthcoming sequels.

\para{Continuity and equidistribution}
Theorem \ref{mainthm:orthohomeo} states that for a fixed lamination $\Ol: \T(S)\to \MF(\lambda)$ is a homeomorphism, but as Mirzakhani already observed \cite[p. 33]{MirzEQ}, $\cO$ cannot be continuous on $\PT_g$.
Moreover, Arana-Herrera and Wright have proven that the earthquake and horocycle flow are not topologically conjugate by any map \cite{AHW_EQ}. At fault is the basic fact that the support of a measured lamination does not vary continuously in the relevant topology.

In forthcoming work \cite{shshII}, the authors investigate the continuity properties of $\cO$ restricted to specific families of $(X, \lambda)$ with constrained geometry and topology.
On these families, the support of $\lambda$ is forced to vary continuously in the Hausdorff topology as the pair varies (in the usual topology on $\PT_g$). 
For example, each of the regular loci has this property.
With this extra geometric control in hand, we prove that $\cO$ restricts to a homeomorphism $\PT_g^{\text{reg}}(\sing)\leftrightarrow\QT_g^{\text{nsc}}(\sing)$ on each regular locus.

By imposing a stronger (yet still geometrically meaningful) topology on $\ML(S)$, we ensure the continuity of $\cO$ varying over all pairs: let $\mathfrak{ML}(S)$ denote the set of measured laminations with the ``Hausdorff $+$ measure'' topology so that measured laminations are close in $\mathfrak{ML}(S)$ if they are close both in measure and their supports are Hausdorff close. 
We prove a general phenomenon that $\cO: \T(S) \times \mathfrak{ML}(S) \to \QT(S)$ is locally H\"older continuous with respect to a nice family of locally defined metrics in geometric train track coordinates.

Our continuity arguments depend on a detailed analysis of the geometric structure of small  foliated train track neighborhoods of a lamination on a hyperbolic surface.
This analysis is sufficiently robust to produce ``enough continuity'' to deduce that $\cO$ is a Borel-measurable isomorphism, a fact which is pivotal for applications. 
The results of Section \ref{subsec:Paction} then live in a more natural setting, as well.

Combined with this work, the conjugacy of Theorems \ref{mainthm:conjugacy} and \ref{mainthm:strata} allows us to import techniques of flat geometry to the hyperbolic setting.
In particular, while $\cO$ is not continuous, its discontinuity is controlled enough that we can translate between equidistribution in $\PoT_g$ and equidistribution in $\QoT_g$ \cite{shshIII}.


\para{Symplectic structure} 
For a maximal lamination $\lambda$, Bonahon and S{\"o}zen identified the Goldman symplectic form on the Teichm\"uller component of $\text{Hom}(\pi_1S, \PSL_2\RR)/\PSL_2\RR$ (also the Weil-Petersson symplectic form) as $\sigl^*\ThH$ in shear coordinates \cite{BonSoz}.
For arbitrary $\lambda \in \ML(S)$ and $X\in\T(S)$, the shape-shifting cocycles built in Section \ref{sec:shapeshift_def}
provide an open set of deformations of the hyperbolization $[\rho: \pi_1S\to \PSL_2\RR]$ of $X$ (Theorem \ref{thm:shsh_open}).
Taking derivatives (as in \cite{BonSoz}) identifies the tangent space to $[\rho]$ with the vector space of $Ad_\rho$-invariant Lie algebra valued $1$-cocycles, yielding a reasonably explicit formula for a vector in the tangent space at $[\rho]$.
Using this formula, it is then possible to compute an expression for the Goldman symplectic form in shear-shape coordinates. 
It remains to understand precisely how $\cO$ interacts with the various natural symplectic forms on $\PM_g$, $\QM_g$, and the (degenerate) symplectic forms on strata, a question that is made technical by the lack of regularity of $\cO$.

\para{Measures}
To each $\PSL_2\RR$-invariant ergodic probability measure on $\QoM_g$, pushforward along $\cO\inverse$ produces a $P$-invariant ergodic probability measure on $\PoM_g$ (and vice-versa).  
An important class of such measures on the singular flat side is furnished by the Masur-Veech measure $\mu_{\sing}$ on a component of a stratum $\QoM_g(\sing)$.  
In \cite{shshII}, the authors give a geometric description of $\nu_{\sing}:=\cO\inverse_*(\mu_{\sing})$ on the corresponding ``stratum'' $\PoM_g^{\text{reg}}(\sing)$, which parallels \cite[Theorem 1.4]{MirzEQ} that on the principal stratum,  $\nu_{\sing}$ disintegrates into the Weil-Petersson  measure on hyperbolic surfaces and Thurston measure on laminations (up to a normalization factor).
We give an outline of the various ingredients required to make the analogous statement for $\nu_{\sing}$ with $\sing$ arbitrary.

As discussed in Section \ref{subsec:PIL}, the piecewise-integral-linear (PIL) structure on $\SH(\lambda)$ endows it with an integer lattice and distinguished measure in the class of Lebesgue.
Indeed, for each filling $\lambda$, the integer lattice in $\SH^+(\lambda)$ restricts to an integer lattice on the fiber $\cH^+(\lambda)$ over the empty arc system due to integrality of the equations defining the piecewise-linear structure of $\Base$. The empty arc system corresponds to the set of $X$ on which $\lambda$ is regular, and so the PIL structure induces a measure (in the class of Lebesgue) on this regular locus. 

We identify the kernel of the Goldman symplectic form restricted to regular loci as tangent to certain ``hyperbolic Schiffer deformations" associated to each even-gon in the complement of $\lambda$.  
These directions admit explicit descriptions as weight systems on a snug train track for $\lambda$ \cite[Appendix]{BW:symplectic} which can be reinterpreted as $1$-forms on regular loci obtained as the differentials of coordinate functions.
Using our formula for the Goldman symplectic form restricted to regular loci, we  identify the pullback of the Lebesgue measure on the fiber $\cH(\lambda)$ over the empty arc system with an analytic volume form obtained as a wedge power of the restricted symplectic form then wedged together with the distinguished $1$-forms associated to the kernel.

Using snug train tracks, one can define a $\sing$-Thurston measure on the space $\ML(\sing)$ of polygonal measured laminations of a given topological type. While this is essentially Lebesgue measure in train track coordinates for the ``measure $+$ Hausdorff'' topology, it is not locally finite in the usual topology on $\ML(S)$.
We then construct natural train track coordinate charts that give local measurable trivializations of $\PoM_g(\sing)$ and exhibit $\nu_{\sing}$ as the product of $\sing$-Thurston measure and the  Weil-Petersson type volume form.

\para{Counting}
The integral points of the PIL structure on $\SH^+(\lambda)$ correspond to integer multicurves transverse to $\lambda$, so when $\lambda$ is itself a multicurve, integral points correspond to square-tiled surfaces. 
Using our coordinates for $\Fol^{uu}(\lambda)$, one can recover the leading coefficient for the polynomial counting the number of square-tiled surfaces with given horizontal curve of bounded area (which was originally computed in \cite{AH_STS, DGZZ_freq}). In particular, since renormalized lattice point counts equidistribute to Lebesgue measure, the coefficient in question can be identified as the Lebesgue measure of (a torus bundle over) the portion of the combinatorial moduli space $\Base / \Mod(S \setminus \lambda)$ with controlled boundary lengths. Using the convergence of the renormalized Weil-Petersson form to the Kontsevich form \cite{Mondello} (which on maximal cells induces the Lebesgue volume), this framework can also be manipulated to relate counts of curves on hyperbolic surfaces and intersection numbers on $\M_{g,n}$.

Moreover, since our coordinates also record the singularity type of the associated differential, the authors expect that the same strategy can also be used to count square-tiled surfaces in a given stratum with fixed horizontal multicurve. The leading coefficient should then be the Lebesgue volume of (a torus bundle over) a compact part of the appropriate subcomplex of the combinatorial moduli space. The relationships with hyperbolic geometry and intersection theory in these settings are more subtle, however, due in part to the non-transversality of the metric residue condition cutting out $\Base$ and the combinatorial conditions specifying strata.
The authors hope to investigate these properties more fully in future work.

\para{Expanding horospheres}
Counting problems for square-tiled surfaces/curves on hyperbolic surfaces are intricately related to the equidistribution of $L$-level sets for the intersection number with/hyperbolic length of laminations as one takes $L \to \infty$.
When $\lambda$ is a multicurve, the equidistribution of such ``expanding horospheres'' to the Masur--Veech measure on the principal stratum of $\QoM_g$/ the pullback by $\cO$ of this measure on $\PoM_g$ (sometimes called {\em Mirzakhani measure}) was established in \cite{Mirz_horo, AH_horo, Liu_horo} using the geometry of the (symmetrized) Lipschitz metric, the non-divergence of the earthquake flow, and a no-escape-of-mass argument.
On the other end of the spectrum, the equidistribution of expanding horospheres for maximal $\lambda$ to $\QoM_g$ can be proven using a standard ``thickening plus mixing'' argument from homogeneous dynamics; in the flat setting this is implicit in the work of Lindenstrauss and Mirzakhani \cite{LM}, and was recently generalized in \cite[Theorem 1.6] {Forni} using different methods.
Equidistribution in the hyperbolic setting then follows from the continuity results described above.

Using our extension of Mirzakhani's conjugacy (and the continuity results described above), the same ``thickening plus mixing'' technique can be used to prove that expanding horospheres based at any $\lambda$ equidistribute to the Mirzakhani measure on $\PoM_g$. Moreover, an analogous result should also hold for strata: intersections of expanding horospheres based at $\lambda$ and the regular locus should equidistribute to the pullback to $\PoM_g$ of the Masur--Veech measure for a component of $\QoM_g(\sing)$.

\bibliography{references_SHSH}{}

\providecommand{\bysame}{\leavevmode\hbox to3em{\hrulefill}\thinspace}
\providecommand{\MR}{\relax\ifhmode\unskip\space\fi MR }
\providecommand{\MRhref}[2]{%
  \href{http://www.ams.org/mathscinet-getitem?mr=#1}{#2}
}
\providecommand{\href}[2]{#2}
\begin{thebibliography}{HTTW95}

\bibitem[AD20]{AD:stretch}
D.~Alessandrini and V.~Disarlo, \emph{Generalized stretch lines for surfaces
  with boundary}, Preprint,
  \href{http://arxiv.org/abs/1911.10431}{arXiv:1911.10431}, 2020.

\bibitem[AH20a]{AH_STS}
F.~Arana-Herrera, \emph{Counting square-tiled surfaces with prescribed real and
  imaginary foliations and connections to {M}irzakhani's asymptotics for simple
  closed hyperbolic geodesics}, J. Mod. Dyn. \textbf{16} (2020), 81--107.

\bibitem[AH20b]{AH_horo}
\bysame, \emph{Equidistribution of families of expanding horospheres on moduli
  spaces of hyperbolic surfaces}, Geom. Dedicata (2020).

\bibitem[AHW22]{AHW_EQ}
F.~Arana-Herrera and A.~Wright, \emph{The asymmetry of thurston's earthquake
  flow}, Preprint, \href{https://arxiv.org/abs/2201.04077}{2201.04077}, 2022.

\bibitem[BD17]{BonDre}
F.~Bonahon and G.~Dreyer, \emph{Hitchin characters and geodesic laminations},
  Acta Math. \textbf{218} (2017), no.~2, 201--295.

\bibitem[BE88]{BowdEpst}
B.~Bowditch and D.~B.~A. Epstein, \emph{{Natural triangulations associated to a
  surface}}, Topology \textbf{27} (1988), no.~1, 91--117.

\bibitem[Bon96]{Bon_SPB}
F.~Bonahon, \emph{Shearing hyperbolic surfaces, bending pleated surfaces and
  {T}hurston's symplectic form}, Ann. Fac. Sci. Toulouse Math. (6) \textbf{5}
  (1996), no.~2, 233--297.

\bibitem[Bon97a]{Bon_GLTHB}
\bysame, \emph{Geodesic laminations with transverse {H}\"{o}lder
  distributions}, Ann. Sci. \'{E}cole Norm. Sup. (4) \textbf{30} (1997), no.~2,
  205--240.

\bibitem[Bon97b]{Bon_THDGL}
\bysame, \emph{Transverse {H}\"{o}lder distributions for geodesic laminations},
  Topology \textbf{36} (1997), no.~1, 103--122.

\bibitem[BS85]{BS}
J.~S. Birman and C.~Series, \emph{Geodesics with bounded intersection number on
  surfaces are sparsely distributed}, Topology \textbf{24} (1985), no.~2,
  217--225.

\bibitem[Bus10]{Buser}
P.~Buser, \emph{Geometry and spectra of compact {R}iemann surfaces}, Modern
  Birkh\"{a}user Classics, Birkh\"{a}user Boston, Ltd., Boston, MA, 2010,
  Reprint of the 1992 edition.

\bibitem[BW17]{BW:symplectic}
F.~Bonahon and H.~Wong, \emph{Representations of the {K}auffman bracket skein
  algebra {II}: {P}unctured surfaces}, Algebr. Geom. Topol. \textbf{17} (2017),
  no.~6, 3399--3434.

\bibitem[CB88]{CB}
A.~J. Casson and S.~A. Bleiler, \emph{Automorphisms of surfaces after {N}ielsen
  and {T}hurston}, London Mathematical Society Student Texts, vol.~9, Cambridge
  University Press, Cambridge, 1988.

\bibitem[CEG06]{CEG}
R.~D. Canary, D.~B.~A. Epstein, and P.~L. Green, \emph{Notes on notes of
  {T}hurston}, Fundamentals of hyperbolic geometry: selected expositions,
  London Math. Soc. Lecture Note Ser., vol. 328, Cambridge Univ. Press,
  Cambridge, 2006, With a new foreword by Canary, pp.~1--115.

\bibitem[CFa]{shshII}
A.~Calderon and J.~Farre, \emph{Continuity of the orthogeodesic foliation and
  the geometry of train tracks}, In preparation.

\bibitem[CFb]{shshIII}
\bysame, \emph{Equidistribution on moduli spaces of flat and hyperbolic
  surfaces, with application to {M}irzakhani's twist torus conjecture}, In
  preparation.

\bibitem[CM14]{CMexceptional}
D.~Chen and M.~M{\"o}ller, \emph{Quadratic differentials in low genus:
  exceptional and non-varying strata}, Ann. Sci. {\'E}c. Norm. Sup{\'e}r.(4)
  \textbf{47} (2014), no.~2, 309--369.

\bibitem[CSW20]{CSW}
J.~Chaika, J.~Smillie, and B.~Weiss, \emph{Tremors and horocycle dynamics on
  the moduli space of translation surfaces}, Preprint,
  \href{http://arxiv.org/abs/2004.04027}{arXiv:2004.04027}, 2020.

\bibitem[DGZZ20]{DGZZ_freq}
V.~Delecroix, E.~Goujard, P.~Zograf, and A.~Zorich, \emph{Masur-{V}eech
  volumes, frequencies of simple closed geodesics and intersection numbers of
  moduli spaces of curves}, Preprint,
  \href{http://arxiv.org/abs/2011.05306}{arXiv:2011.05306}, 2020.

\bibitem[Do08]{Do}
N.~Do, \emph{Intersection theory on moduli spaces of curves via hyperbolic
  geometry}, Ph.D. thesis, The University of Melbourne, 2008.

\bibitem[Dum09]{DumasCP1}
D.~Dumas, \emph{Complex projective structures}, Handbook of {T}eichm\"{u}ller
  theory. {V}ol. {II}, IRMA Lect. Math. Theor. Phys., vol.~13, Eur. Math. Soc.,
  Z\"{u}rich, 2009, pp.~455--508. \MR{2497780}

\bibitem[Dum15]{Dumas_skin}
\bysame, \emph{Skinning maps are finite-to-one}, Acta Math. \textbf{215}
  (2015), no.~1, 55--126. \MR{3413977}

\bibitem[EM06]{EM}
D.~B.~A. Epstein and A.~Marden, \emph{Convex hulls in hyperbolic space, a
  theorem of {S}ullivan, and measured pleated surfaces}, Fundamentals of
  hyperbolic geometry: selected expositions, London Math. Soc. Lecture Note
  Ser., vol. 328, Cambridge Univ. Press, Cambridge, 2006, pp.~117--266.

\bibitem[EM18]{EskMirz}
A.~Eskin and M.~Mirzakhani, \emph{Invariant and stationary measures for the
  {${\rm SL}(2,\Bbb R)$} action on moduli space}, Publ. Math. Inst. Hautes
  \'{E}tudes Sci. \textbf{127} (2018), 95--324.

\bibitem[EMM15]{EMM}
A.~Eskin, M.~Mirzakhani, and A.~Mohammadi, \emph{Isolation, equidistribution,
  and orbit closures for the {${\rm SL}(2,\Bbb R)$} action on moduli space},
  Ann. of Math. (2) \textbf{182} (2015), no.~2, 673--721.

\bibitem[FK02]{ErgodicTheory}
R.~Feres and A.~Katok, \emph{Ergodic theory and dynamics of {$G$}-spaces (with
  special emphasis on rigidity phenomena)}, Handbook of dynamical systems,
  {V}ol. 1{A}, North-Holland, Amsterdam, 2002, pp.~665--763.

\bibitem[For20]{Forni}
G.~Forni, \emph{Limits of geodesic push-forwards of horocycle invariant
  measures}, Preprint,
  \href{http://arxiv.org/abs/1908.11037}{arXiv:1908.11037}, 2020.

\bibitem[Gel14]{Gelander}
T.~Gelander, \emph{Lectures on lattices and locally symmetric spaces},
  Preprint, \href{https://arxiv.org/abs/1402.0962}{arXiv:1402.0962}, 2014.

\bibitem[GM91]{GM}
F.~P. Gardiner and H.~Masur, \emph{Extremal length geometry of
  {T}eichm\"{u}ller space}, Complex Variables Theory Appl. \textbf{16} (1991),
  no.~2-3, 209--237.

\bibitem[Gup17]{Gupta_wild}
S.~Gupta, \emph{Harmonic maps and wild {T}eichm{\"u}ller spaces}, Preprint,
  \href{http://arxiv.org/abs/2011.05306}{arXiv:2011.05306}, 2017.

\bibitem[HM79]{HubMas}
J.~Hubbard and H.~Masur, \emph{Quadratic differentials and foliations}, Acta
  Math. \textbf{142} (1979), no.~3-4, 221--274.

\bibitem[HT]{HT:stretch}
C.~Horbez and J.~Tao, \emph{Nielsen-{T}hurston classification and isometries of
  the {T}hurston metric}, In preparation.

\bibitem[HTTW95]{HTTW}
Z.~Han, L.~Tam, A.~Treibergs, and T.~Wan, \emph{Harmonic maps from the complex
  plane into surfaces with nonpositive curvature}, Comm. Anal. Geom. \textbf{3}
  (1995), no.~1-2, 85--114.

\bibitem[Ker83]{Kerckhoff_NR}
S.~P. Kerckhoff, \emph{The {N}ielsen realization problem}, Ann. of Math. (2)
  \textbf{117} (1983), no.~2, 235--265.

\bibitem[KZ03]{KZ}
M.~Kontsevich and A.~Zorich, \emph{{Connected components of the moduli spaces
  of Abelian differentials with prescribed singularities}}, Invent. Math.
  \textbf{153} (2003), no.~3, 631--678.

\bibitem[Lan08]{Lanneau}
E.~Lanneau, \emph{{Connected components of the strata of the moduli spaces of
  quadratic differentials}}, Ann. Sci. l'{\'{E}}c. Norm. Sup. \textbf{41}
  (2008), no.~1, 1--56.

\bibitem[Lev83]{Levitt}
G.~Levitt, \emph{Foliations and laminations on hyperbolic surfaces}, Topology
  \textbf{22} (1983), no.~2, 119--135.

\bibitem[Liu20]{Liu_horo}
M.~Liu, \emph{Length statistics of random multicurves on closed hyperbolic
  surfaces}, Preprint,
  \href{http://arxiv.org/abs/1912.11155}{arXiv:1912.11155}, 2020.

\bibitem[LM08]{LM}
E.~Lindenstrauss and M.~Mirzakhani, \emph{Ergodic theory of the space of
  measured laminations}, Int. Math. Res. Not. IMRN (2008), no.~4, Art. ID
  rnm126, 49.

\bibitem[Luo07]{Luo}
F.~Luo, \emph{On {T}eichm\"{u}ller spaces of surfaces with boundary}, Duke
  Math. J. \textbf{139} (2007), no.~3, 463--482.

\bibitem[Mas82]{Masur_IETsMF}
H.~Masur, \emph{{Interval exchange transformations and measured foliations}},
  Ann. Math. \textbf{115} (1982), no.~1, 169--200.

\bibitem[Min92]{Minsky_harmonic}
Y.~N. Minsky, \emph{{Harmonic maps, length, and energy in Teichm{\"{u}}ller
  space}}, J. Differ. Geom. \textbf{35} (1992), 151--217.

\bibitem[Mir07]{Mirz_horo}
M.~Mirzakhani, \emph{Random hyperbolic surfaces and measured laminations}, In
  the tradition of {A}hlfors-{B}ers. {IV}, Contemp. Math., vol. 432, Amer.
  Math. Soc., Providence, RI, 2007, pp.~179--198.

\bibitem[Mir08]{MirzEQ}
\bysame, \emph{Ergodic theory of the earthquake flow}, Int. Math. Res. Not.
  IMRN (2008), no.~3, Art. ID rnm116, 39.

\bibitem[Mon09a]{Mond_handbook}
G.~Mondello, \emph{Riemann surfaces, ribbon graphs and combinatorial classes},
  Handbook of {T}eichm\"{u}ller theory. {V}ol. {II}, IRMA Lect. Math. Theor.
  Phys., vol.~13, Eur. Math. Soc., Z\"{u}rich, 2009, pp.~151--215.

\bibitem[Mon09b]{Mondello}
\bysame, \emph{Triangulated {R}iemann surfaces with boundary and the
  {W}eil-{P}etersson {P}oisson structure}, J. Differential Geom. \textbf{81}
  (2009), no.~2, 391--436.

\bibitem[MW02]{MWnondiv}
Y.~N. Minsky and B.~Weiss, \emph{Nondivergence of horocyclic flows on moduli
  space}, J. Reine Angew. Math. \textbf{552} (2002), 131--177.

\bibitem[MW14]{MW_cohom}
\bysame, \emph{Cohomology classes represented by measured foliations, and
  {M}ahler's question for interval exchanges}, Ann. Sci. \'{E}c. Norm.
  Sup\'{e}r. (4) \textbf{47} (2014), no.~2, 245--284.

\bibitem[Pap86]{PapaHam}
A.~Papadopoulos, \emph{Geometric intersection functions and {H}amiltonian flows
  on the space of measured foliations on a surface}, Pacific J. Math.
  \textbf{124} (1986), no.~2, 375--402.

\bibitem[Pap91]{Pap:extension}
\bysame, \emph{On {T}hurston's boundary of {T}eichm\"{u}ller space and the
  extension of earthquakes}, Topology Appl. \textbf{41} (1991), no.~3,
  147--177.

\bibitem[PH92]{PennerHarer}
R.~C. Penner and J.~L. Harer, \emph{Combinatorics of train tracks}, Annals of
  Mathematics Studies, vol. 125, Princeton University Press, Princeton, NJ,
  1992.

\bibitem[PT07]{PapadopTheret:Teich_Thurston}
A.~Papadopoulos and G.~Th\'{e}ret, \emph{On {T}eichm\"{u}ller's metric and
  {T}hurston's asymmetric metric on {T}eichm\"{u}ller space}, Handbook of
  {T}eichm\"{u}ller theory. {V}ol. {I}, IRMA Lect. Math. Theor. Phys., vol.~11,
  Eur. Math. Soc., Z\"{u}rich, 2007, pp.~111--204.

\bibitem[PW22]{PWstretch}
H.~Pan and M.~Wolf, \emph{Ray structures on {Teichm{\"u}ller} space}, Preprint,
  \href{https://arxiv.org/abs/2206.01371}{arXiv:2206.01371}, 2022.

\bibitem[PY17]{Pap_Yam}
A.~Papadopoulos and S.~Yamada, \emph{Deforming hyperbolic hexagons with
  applications to the arc and the {T}hurston metrics on {T}eichm\"{u}ller
  spaces}, Monatsh. Math. \textbf{182} (2017), no.~4, 913--939.

\bibitem[SB01]{BonSoz}
Y.~S\"{o}zen and F.~Bonahon, \emph{The {W}eil-{P}etersson and {T}hurston
  symplectic forms}, Duke Math. J. \textbf{108} (2001), no.~3, 581--597.

\bibitem[Sul85]{Sullivan:stability}
D.~Sullivan, \emph{Quasiconformal homeomorphisms and dynamics. {II}.
  {S}tructural stability implies hyperbolicity for {K}leinian groups}, Acta
  Math. \textbf{155} (1985), no.~3-4, 243--260.

\bibitem[Thu82]{Thurston:notes}
W.~P. Thurston, \emph{The geometry and topology of 3-manifolds}, 1982,
  Princeton University Lecture Notes.

\bibitem[Thu97]{Th_book}
\bysame, \emph{Three-dimensional geometry and topology. {V}ol. 1}, Princeton
  Mathematical Series, vol.~35, Princeton University Press, Princeton, NJ,
  1997, Edited by Silvio Levy.

\bibitem[Thu98]{Th_stretch}
\bysame, \emph{Minimal stretch maps between hyperbolic surfaces}, Preprint,
  \href{https://arxiv.org/abs/math/9801039}{arXiv:9801039}, 1998.

\bibitem[Ush99]{Ushijima}
A.~Ushijima, \emph{A canonical cellular decomposition of the {T}eichm\"{u}ller
  space of compact surfaces with boundary}, Comm. Math. Phys. \textbf{201}
  (1999), no.~2, 305--326.

\bibitem[Vee82]{Veech_IETs}
W.~Veech, \emph{{Gauss measures for transformations on the space of interval
  exchange maps}}, Ann. Math. \textbf{115} (1982), 201--242.

\bibitem[Wol83]{Wolpert}
S.~Wolpert, \emph{On the symplectic geometry of deformations of a hyperbolic
  surface}, Ann. of Math. (2) \textbf{117} (1983), no.~2, 207--234.

\bibitem[Wri18]{Wright_MirzEQ}
A.~Wright, \emph{{Mirzakhani's work on earthquake flow}}, Preprint,
  \href{http://arxiv.org/abs/1810.07571}{arXiv:1810.07571}, 2018.

\bibitem[Wri20]{Wright_Mirz}
\bysame, \emph{A tour through {M}irzakhani's work on moduli spaces of {R}iemann
  surfaces}, Bull. Amer. Math. Soc. (N.S.) \textbf{57} (2020), no.~3, 359--408.

\bibitem[ZB04]{BonZhu:HD}
X.~Zhu and F.~Bonahon, \emph{The metric space of geodesic laminations on a
  surface. {I}}, Geom. Topol. \textbf{8} (2004), 539--564.

\end{thebibliography}
\bibliographystyle{amsalpha.bst}

\pagebreak

\section*{Index}

\begin{multicols}{2}
$\acarc$ \pageref{ind:ac}

$\arc$ \pageref{ind:arcsys}

$\arc(q)$ \pageref{constr:aq}

$\arcwt(q)$ \pageref{ind:arcwtq}

$\arcwt(X)$ \pageref{ind:arcwt(X)}

$|\Arcfill(\cutsurf, \partial \cutsurf)|_{\RR}$ \pageref{def:arccx}

$A(\vec{s})$ \pageref{eqn:spike_shape}

$A(\vec{\alpha}_1, u)$ \pageref{eqn:adjust_hexagon}

Admissible routes \pageref{ind:admissroute}
\\

$\Base$ \pageref{ind:base}

Binding pairs of laminations \pageref{ind:bind}

Boundary leaves \pageref{ind:boundaryleaves}
\\

$c_\alpha$ \pageref{ind:arcwtq}, \pageref{ind:arcwt(X)}

Crowned hyperbolic surface \pageref{ind:crownedsurf}
\\

$\delta_\pm$ \pageref{ind:deltapm}

$D_\lambda(X)$ \pageref{ind:Dlam}

$d_{\arcwt}(\vec{\alpha}_1, u)$ \pageref{ind:arcdist}

$\Delta(\lambda)$ \pageref{ind:Delta(lambda)}

$|\Delta(\lambda)|_\pm$ \pageref{ind:signedmeasures}

$\partiall H_v$ \pageref{ind:lambdaboundary}

Deflation \pageref{ind:deflate}

Depth \pageref{ind:divrad}

Dilation rays \pageref{def:dilationray}

Dual arcs \pageref{ind:dualarc}
\\

$\varphi(\vec{s})$ \pageref{eqn:spike_shsh}

$\varphi_{\underline{\cH}}$ \pageref{eqn:first_adj}

$\varphi^r_{v,w}$ \pageref{eqn:adj_r}

$\varphi_{p_v, p_w}$ \pageref{def:shapeshift_subsurf}

$\varphi_{\ac}$ \pageref{prop:shapeshift_cocycle}

$\Phi_{p_v}$ \pageref{ind:Phi}

$\Fol^{uu}(\lambda)$ \pageref{ind:unstable}

$F_{\lambda}$ \pageref{ind:horfol}

$f_{X, \ac}(\vec{s})$ \pageref{eqn:spike_param}

$f_{X, \ac}(\vec{\alpha}_1, u)$ \pageref{ind:arcparam}
\\

$(g, p_v)$ \pageref{ind:ptdgeo}

$g_v^w$, $g_w^v$ \pageref{ind:simple}

$\mathscr{G}(\tX)$ \pageref{ind:geospace}

Geometric standard smoothings \pageref{constr:stand_smooth}

Geometric train tracks \pageref{ind:geott}
\\

$\cH$ \pageref{ind:hexagons}

$h_{\arcwt}(s)$ \pageref{ind:has}

$\calH(\lambda)$ \pageref{ind:transcoc}

$\calH^+(\lambda)$ \pageref{eqn:defH+}

(H0), (H1), (H2) \pageref{ind:transcocaxioms}

$H^1(\widehat{N}, \partial \widehat{N}; \RR)^-$ \pageref{ind:transcochom}

$H^1(\widehat{N}_{\arc}, \partial \widehat{N}_{\arc}; \RR)^-$ \pageref{ind:shshhom}

$H(v,w)$ \pageref{ind:hexhull}

Hexagons \pageref{ind:hexagons}

Hexagonal hulls \pageref{ind:hexhull}
\\

$\Il(q)$ \pageref{eq:Ildef}

$\inj_\lambda(X)$ \pageref{ind:injradlam}
\\

$k_b^d$ \pageref{ind:kbd}
\\

$\ell_X(b)$ \pageref{ind:ttlength}

$\ell_{\ac} (\alpha_1)$ \pageref{ind:ls(arc)}

$\widehat{\lambda} \cup \hat{\arc}$ \pageref{ind:doublecover}

$\Lambda_{v,w}$ \pageref{ind:sepleaves}
\\

Metric residue \pageref{ind:metricres}

$\MF(\lambda)$ \pageref{ind:MF(lambda)}
\\

$\epN \lambda$ \pageref{ind:geott}

$n_0(\lambda)$ \pageref{ind:orcomps}
\\

$\mathcal{O}$ \pageref{ind:O}

$\cO_{\partial Y}(Y)$ \pageref{ind:orthorelboundary}

$\Ol$ \pageref{ind:Ol}

Orientations of $\lambda \cup \arc$ \pageref{ind:orientlamplusarc}

Orientation double cover \pageref{ind:doublecover}
\\

$\PT_g$ and $\PoT_g$ \pageref{ind:PTandQT}

$\PoM_g^{\text{reg}}(\sing)$ \pageref{ind:reglocus}

Piecewise-integral-linear \pageref{subsec:PIL}

Pointed geodesic \pageref{ind:ptdgeo}
\\

$\QT_g$ and $\QoT_g$ \pageref{ind:PTandQT}

$\QoM_g^{\text{nsc}}(\sing)$ \pageref{ind:stratumNSC}
\\

$r_b(d)$ \pageref{ind:divrad}

$\res(\Crown)$ \pageref{ind:metricres}

$\res_{\arcwt}(\Crown)$ \pageref{ind:combres}

$\rho_\ac$ \pageref{ind:deformrep}
\\

$\vec{s}$ \pageref{ind:spikes}

$\ac$ \pageref{ind:ac}

$\|\ac\|_{\vec{s}}$ \pageref{eqn:normac_def}

$\sigl(X)$ \pageref{def:hypshsh}

$\cutsurf_{g,b}^{\crowns}$ \pageref{ind:crownedsurf}

$\widehat{\cutsurf}_{g,b}^{\crowns}$ \pageref{ind:compactifiedcrowned}

Singular leaves \pageref{ind:singleaves}

Shape-shifting cocycles \pageref{subsec:shapeshift_total}

Shear-shape cocyles \pageref{def:shsh_cohom}, \pageref{def:shsh_axiom}

$\SH(\lambda)$ \pageref{ind:SHspaces}

$\SH^+(\lambda)$ \pageref{def:SH+}

$\SH^\circ(\lambda; \arc)$ \pageref{ind:SHspaces}

$\SH(\lambda;{\arcwt})$ \pageref{ind:SHspaces}

(SH0), (SH1), (SH2), (SH3) \pageref{def:shsh_axiom}

Simple pairs \pageref{ind:simple}

Snug neighborhoods \pageref{ind:snugnbhd}

Spine \pageref{ind:spine}

$\Sp(Y)$ \pageref{ind:spine}

$\Sp_k(Y)$ \pageref{ind:spine}

$\Sp^0$ \pageref{ind:spinecore}

Standard transversals \pageref{ind:standard transversal}

Standard smoothings \pageref{ind:taua}

Stretchquakes \pageref{ind:stretchquake}
\\

$\taua$ \pageref{ind:taua}

$\tau_{\max}$ \pageref{ind:taumax}

$\mathsf{T}^* \setminus \mathsf{H}^*$ \pageref{constr:ttfromtri}

Thick with respect to $v$ and $w$ \pageref{ind:thick}

Ties \pageref{ind:ties}

Transverse cocycles \pageref{ind:transcochom}

Transverse pairs \pageref{lem:shsh_onpairs}

Tremors, $\trem_{\mu}(q)$ \pageref{ind:tremors}
\\ 

Valence, $\val(x)$ \pageref{ind:valence} 
\\

$\ThH$ \pageref{ind:ThH}

$\ThSH$ \pageref{ind:ThSH}

$w_{\arc}(\sigma)$ \pageref{prop:ttcoords}
\\

$\Xi(q)$ \pageref{ind:Xi(q)}

$\chi(\lambda)$ \pageref{ind:X(lambda)}

$X_{\ac}$ \pageref{ind:Xac}
\\

$[z]_+$ \pageref{ind:[z]+}
\end{multicols}

\Addresses
\end{document}